\documentclass[a4paper,10pt]{article}
\usepackage{a4wide,bm,epsfig,booktabs}

\usepackage{amsmath}
\usepackage{amsthm}
\usepackage{tikz}
\usepackage{setspace}
\usepackage{relsize}  
\usepackage{graphicx}  
\usepackage{cancel} 
\usepackage{slashed}
\usepackage{verbatim} 
\usepackage{enumerate} 
\usepackage{stmaryrd} 
\usepackage{cite}
\usepackage[flushmargin,hang]{footmisc}
\usepackage{fancyhdr} 
\usepackage[arrow, matrix,curve]{xy}
\usepackage{hyperref}
\usepackage{amssymb}
\usepackage{eucal}

\newtheorem{innercustomthm}{Theorem}
\newenvironment{customthm}[1]
  {\renewcommand\theinnercustomthm{#1}\innercustomthm}
  {\endinnercustomthm}
  
 \newtheorem{innercustompr}{Proposition}
\newenvironment{custompr}[1]
  {\renewcommand\theinnercustompr{#1}\innercustompr}
  {\endinnercustompr}
  
   \newtheorem{innercustomlem}{Lemma}
\newenvironment{customlem}[1]
  {\renewcommand\theinnercustomlem{#1}\innercustomlem}
  {\endinnercustomlem}

\newtheorem{theorem}{Theorem}
\newtheorem{proposition}{Proposition}
\newtheorem{lemma}{Lemma}
\newtheorem{corollary}{Corollary}
\newtheorem{remark}{Remark}

\newtheorem{example}{Example}
\newtheorem{definition}{Definition}
\newtheorem{convention}{Convention}

\renewcommand{\labelenumi}{\arabic{enumi})}
\renewcommand{\theenumi}{\arabic{enumi}} 
\let\origenumerate\enumerate
\def\enumerate{\origenumerate\itemsep0pt}
\let\origitemize\itemize
\def\itemize{\origitemize\itemsep0pt}

\newcommand{\id} {\mathrm{id}}
\newcommand{\e} {\mathrm{e}}
\newcommand{\sign}   {\mathrm{sgn}}
\newcommand{\I}{\mathrm{i}}

\newcommand{\FP} {\mathcal{M}}
\newcommand{\he} {\hspace{1pt}}

\newcommand{\barrow} {\:{}_{g_n}\hspace{-7pt}\boldsymbol{\rightharpoonup}}

\newcommand{\barrows} {\:{}_{\tilde{g}_n}\hspace{-7pt}\boldsymbol{\rightharpoonup}}

\newcommand{\innt} {\mathrm{int}}
\newcommand{\clos} {\mathrm{clos}}
\newcommand{\inv} {\mathfrak{i}}

\newcommand{\og} {\ovl{\gamma}}
\newcommand{\ug} {\underline{\gamma}}
\newcommand{\um} {\underline{\mu}}

\newcommand{\oI} {\ovl{I}}
\newcommand{\uI} {I}

\newcommand{\T} {\mathfrak{T}}

\newcommand{\deff} {{\it iff }}
\newcommand{\defff} {{\it iff }}

\newcommand{\x} {\mathrm{x}}

\newcommand{\dist} {\mathrm{dist}}

\newcommand{\met} {\mathrm{d}}

\newcommand{\CM} {C}
\newcommand{\im} {\mathrm{im}}
\newcommand{\dom} {\mathrm{dom}}
\newcommand{\dd} {\mathrm{d}}

\newcommand{\wt}[1] {\widetilde{#1}} 
\newcommand{\ovl}[1]{\overline{#1}} 

\newcommand{\OO} {\mathcal{O}}
\newcommand{\RR} {\mathbb{R}}
\newcommand{\CC} {\mathbb{C}}
\newcommand{\QQ} {\mathbb{Q}}
\newcommand{\ZZ} {\mathbb{Z}}
\newcommand{\CN} {\mathfrak{N}}
\newcommand{\cn} {\mathfrak{n}}
\newcommand{\ocn} {\ovl{\mathfrak{n}}}
\newcommand{\NN} {\mathbb{N}}

\newcommand{\cp} {\circ}
\newcommand{\wm} {\varphi}
\newcommand{\g} {{\vec{g}}}

\newcommand{\q} {{\vec{q}}}
\newcommand{\h} {{\vec{h}}}

\newcommand{\ccc} {{\vec{c}}}

\newcommand{\mg} {\mathfrak{g}}
\newcommand{\mc} {\mathfrak{c}}
\newcommand{\mq} {\mathfrak{q}}
\newcommand{\mh} {\mathfrak{h}}

\newcommand{\SO} {\mathrm{SO}}
\newcommand{\UE} {\mathrm{U}(1)}

\newcommand{\ppsim} {\rightharpoonup}

\newcommand{\cpsim} {\sim_\cp}
\newcommand{\psim} {\leadsto}

\newcommand{\gag}{{\gamma_{\g}^x}}
\newcommand{\gagg}{{\gamma_{\q}^y}}

\newcommand{\pii} {\boldsymbol{\pi}}
\newcommand{\ppi} {\Pi}
\newcommand{\pri} {\pi}

\newcommand{\zd} {k}
\newcommand{\EE} {\mathrm{E}}

\newcommand{\MK}{{\mathfrak{D}}}
\newcommand{\ML}{{\mathfrak{C}}}

\newcommand{\vv}{\vec{v}}
\makeatletter
\def\blfootnote{\gdef\@thefnmark{}\@footnotetext}
\makeatother

\begin{document}
\title{Symmetries of Analytic Curves}
\author{Maximilian Hanusch\thanks{e-mail:
    {\tt mhanusch@math.upb.de}}\\   
  \\
  {\normalsize\em Department of Physics}\\[-0.15ex]
  {\normalsize\em Florida Atlantic University}\\[-0.15ex]
  {\normalsize\em 777 Glades Road}\\[-0.15ex]
  {\normalsize\em FL 33431 Boca Raton}\\[-0.15ex]
  {\normalsize\em USA}}   
\date{October 17, 2022}
\maketitle

\begin{abstract}
Analytic curves are classified w.r.t.\ their symmetry under a  regular and separately analytic Lie group action $G\times M\rightarrow M$  on an analytic manifold. We prove that a non-constant analytic curve $\gamma\colon D\rightarrow M$ is either exponential or free. In the first case, $\gamma$ is  up to analytic reparametrization of the form $t\mapsto \exp(t\cdot \g)\cdot x$, whereby $\g\in\mg$ 
is unique up to non-zero scalation and addition of elements in the Lie algebra of the stabilizer $G_\gamma\equiv \{g\in G\:|\: g\cdot \gamma=\gamma\}$ of the curve $\gamma$. 
In the free case, $\gamma$ splits into countably many immersive subcurves, each of them being discretely generated by $G$.    
This means that each such subcurve $\delta\colon D\supseteq (\iota',\iota)\rightarrow M$   
is build up countably many $G$-translates of a symmetry free building block $\delta|_{\Delta}$, with $\Delta\subseteq (\iota',\iota)$ a nondegenerate  interval; whereby the following three cases can occur:\hspace*{\fill}
\begingroup
\setlength{\leftmarginii}{11pt}
\begin{itemize}
\item[$-$]
In the shift case, the building blocks are continuously distributed in $\delta$, with $\Delta$ always compact.
For $\Delta$ fixed, $\delta$ is created by iterated shifts of $\delta|_\Delta$ by some $g\in G$ and its inverse $g^{-1}$, whereby the class $[e]\neq [g]\in G\slash G_\gamma$ is uniquely determined (in particular  the same for each possible decomposition). 
\item[$-$]
In the flip case, there exist countably many building blocks; each of them is defined on a compact interval, and contained in the one and only decomposition that exists in this case. 
The curve $\delta$ is created by iterated flips at the boundary points of these building blocks, whereby the occurring transformations are generated by two non-trivial classes in $G\slash G_\gamma$. 
\item[$-$]
In the mirror case, there exists exactly one symmetry (flipping) point $\delta(\tau)$, as well as one translation class $[e]\neq [g]\in G\slash G_\gamma$. The one and only decomposition of $\delta$ is given by $\delta|_{(i',\tau]}$, $\delta|_{[\tau,i)}$, whereby $\delta|_{(i',\tau]}$ is flipped into $\delta|_{[\tau,i)}$ or vice versa (or both). 
\end{itemize}
\endgroup
\noindent
We finally extend the classification result to the analytic 1-submanifold case. Specifically, we show that an analytic 1-submanifold of $M$ is either free or (exponential, i.e) analytically diffeomorphic (via the exponential map) to the unit circle or an interval. The corresponding decomposition results in the free case are outlined in this paper, but proven in \cite{MAXDECO}.
\end{abstract}
\tableofcontents 

\section{Introduction}
\label{intro}
Curves and 1-submanifolds are the key mathematical objects of both loop quantum gravity \cite{BackLA, Thiemann} and string theory \cite{Polch}, two of the major current approaches to quantum gravity.   
In the context of a symmetry reduction of these two theories, therefore a detailed analysis of the symmetry properties of such 1-dimensional objects should be done. 
For instance, in loop quantum gravity (where the analytic category is considered) symmetries of embedded analytic curves  
have turned out to be elementary for the investigation of quantum-reduced configuration spaces \cite{InvConLQG, MAX}.  This is because  the holonomy of an invariant connection along an analytic embedded curve $\gamma$   is  completely determined by its holonomy along a symmetry free building block (if $\gamma$ is discretely generated)  or given in terms of the exponential map (if $\gamma$ is exponential), cf.\ Lemma 5.6.1 in \cite{MAX}. In this context, it is thus relevant to know whether a given analytic curve is either of these types. 
For this reason, in \cite{MAX} symmetries of analytic curves had been investigated, and the mentioned classification result was established there for particular classes of Lie group actions. Specifically, it was shown that, given an analytic Lie group action $\wm\colon G\times M\rightarrow M$ on an analytic manifold, then $\gamma$ is discretely generated if it is free and if $\wm$ is pointwise proper, cf.\ Proposition 5.23.1) in \cite{MAX}. It was furthermore shown that an embedded analytic curve is either exponential or free if $\wm$ admits only normal stabilizers, and is proper or transitive and pointwise proper, cf.\ Proposition 5.23 in \cite{MAX}.   
In this paper, we generalize these statements to arbitrary analytic curves, and  to a significantly larger class of Lie group actions. More specifically, the classification result is proven in the regular context (less restrictive than pointwise properness), and the decomposition results for free curves are established for sated actions (even less restrictive than regularity\footnote{Notably, Example \ref{podspodspopodsdsa} shows that satedness is sufficient for the decomposition result, but not for the classification result.}). We furthermore discuss uniqueness of the Lie algebra element in the exponential context; and perform a detailed analysis of the free case, i.e., we work out the different decomposition types that can occur and clarify their special properties.   

Let $\wm\colon G\times M\rightarrow M$ be a left action of a Lie group $G$ on an analytic manifold $M$ (without boundary) that is continuous in $G$, i.e., $\wm_x\colon G\ni g\mapsto \wm(g,x)\in M$ is continuous for each fixed $x\in M$.  
We say that $\wm$ is
\begingroup
\setlength{\leftmargini}{12pt}
\begin{itemize}
\item
analytic in $G$\hspace{24.7pt}\: \defff\:  
	$\wm_x\colon \hspace{3.1pt}G\ni g\mapsto\hspace{0.6pt} \wm(g,x)\in M$\hspace{1pt} is analytic for each fixed $x\in M$.
\item	
analytic in $M$\hspace{21.8pt}\: \defff\:  
	$\wm_g\hspace{0.6pt}\colon M\ni x\mapsto \wm(g,x)\in M$\hspace{1pt} is analytic for each fixed $g\in G$.
\item
separately analytic\: \defff\:   $\wm$ is analytic in $G$ and $M$.
\end{itemize}
\endgroup
\noindent   
Moreover, we say that $\wm$ is\footnote{In \cite{MAXDECO}, we use the following equivalent formulation of satedness: To each  $C\subseteq M$ with $|C|\geq 2$ and each $x\in M\setminus C$, there exists a neighbourhood $U$ 
of $x$ with $g\cdot C\subsetneq U$ for all $g\in G$.} 
\begingroup
\setlength{\leftmargini}{12pt}
\begin{itemize}
\item
\textbf{sated}\:\:\defff\:{}for $x,y,z\in M$ mutually different, we cannot have\footnote{Notably, the multiplicative action of $\RR_{>0}$ on $\RR^n$ for $n\geq 1$ is not sated.}
\begin{align}
\label{iosdiosd1}
	\textstyle\lim_n g_n\cdot y =x= \lim_n g_n\cdot z\qquad\quad\text{for a sequence}\qquad\quad  \{g_n\}_{n\in \NN}\subseteq G.
\end{align}
\vspace{-17pt}
\item
\textbf{stable}\:\:\defff\:$\lim_n g_n\cdot x=x$ for $x\in M$ and  $\{g_n\}_{n\in \NN}\subseteq G\backslash G_x$, implies that $\{h_n\cdot g_n\cdot h'_n\}_{n\in \NN}$ admits a convergent subsequence, for certain sequences   
\begin{align}
\label{iosdiosd2}
	\textstyle\{h_n\}_{n\in \NN}, \{h'_n\}_{n\in \NN}\subseteq \bigcap_{g\in G}G_{g\cdot x}\equiv G_{[x]}
\end{align}
contained in the stabilizer $G_{[x]}$ of the $G$-orbit of $x$.
\item
\textbf{regular}\:\:\defff\:$\wm$ is sated and stable. 
\end{itemize} 
\endgroup
\noindent
For instance, pointwise proper (hence proper) actions are regular, cf.\ Remark \ref{remmmmmi}.\ref{regp1}. Moreover, $\wm$ is regular if the following two conditions are fulfilled (cf.\ Remark \ref{remmmmmi}.\ref{regp3dsdsdsds}):
\begingroup
\setlength{\leftmargini}{12pt}
\begin{itemize}
\item
$(M,*)$ is a topological group, with $\wm(g,x)=\phi(g)* x$ for all $g\in G$ and $x\in M$, for a continuous group homomorphism $\phi\colon G\rightarrow M$.  
\item
$\phi\cp s=\id_{V}$ holds for a continuous map $s\colon M\supseteq V:= U\cap \phi(G)\rightarrow G$, with $U$ a neighbourhood of $e_M$.
\end{itemize}
\endgroup
\noindent
In particular, this covers the situation where $G$ (or a closed subgroup) acts in the usual way on $M\equiv G\slash H$, for $H$ a closed normal subgroup of $G$ -- The next two examples also show that regularity is a strictly  weaker condition than pointwise properness, just because the corresponding subgroups $H$ are not compact there ($n>1$): 
\begingroup
\setlength{\leftmargini}{12pt}
\begin{itemize}
\item
$G\equiv \RR^n$ and $H\equiv\ZZ^n$, hence $M\equiv \mathbb{T}^n$.   
\item
$G\equiv\RR^n$ and $H\subseteq \RR^n$ some $m$-dimensional linear subspace for $m>0$, hence $M\equiv  \RR^{n-m}$.
\end{itemize}
\endgroup
Intervals $D\subseteq \RR$ are always assumed to be nondegenerate (non-empty interior $\innt[D]$), and   
domains of curves (or maps between subsets of $\RR$) are always assumed to be intervals.  A curve $\gamma\colon D\rightarrow M$ is said to be analytic (immersive) \deff $\gamma=\tilde{\gamma}|_D$ holds for an analytic (immersion) $\tilde{\gamma}\colon \tilde{I}\rightarrow M$ that is defined on an open interval $\tilde{I}\subseteq \RR$ with $D\subseteq \tilde{I}$. 
We prove the following classification result, 
cf.\ Theorem \ref{classi}. 
\begin{customthm}{A}\label{A}
If $\wm$ is regular and separately analytic, then an analytic curve is either exponential or free. 
\end{customthm} 
\noindent
An analytic curve $\gamma\colon D\rightarrow M$ is said to be 
\begingroup
\setlength{\leftmargini}{12pt}
\begin{itemize}
\item
{\bf exponential}\:\hspace{0.5pt} \deff\: $\gamma= \exp(\rho\cdot \g)\cdot x$ holds for some $x\in M$, $\g\in \mg$, and an analytic map $\rho\colon D\rightarrow \RR$. 
\item
{\bf free}\hspace{39pt}\:\: \deff\: $\gamma$ admits a free segment -- i.e., an immersive subcurve $\delta\equiv \gamma|_{D'}$, such that the implication 
\begin{align*}
	g\cdot \delta \cpsim \delta\quad\text{for}\quad g\in G\qquad\quad\Longrightarrow\qquad\quad g\cdot \delta=\delta
\end{align*}
\vspace{-16pt}
\par
\begingroup
\leftskip=2.87cm 
\noindent
holds. 
Here and in the following, we write $\gamma \cpsim \gamma'$ for analytic immersive curves $\gamma,\gamma'$ \deff $\gamma(J)=\gamma'(J')$ holds for non-empty open intervals $J$ and $J'$  
on which $\gamma$ and $\gamma'$ are embeddings, respectively.
\par
\endgroup
\end{itemize}
\endgroup
\noindent
Additionally, we prove the following uniqueness statement, cf.\ Proposition \ref{classsiilie}.
\begin{custompr}{B}\label{B}
Assume that $\wm$ is sated and separately analytic; and let $\gamma$ be non-constant as well as exponential w.r.t.\ $x$ and $\g$. Then, $\gamma$ is exponential w.r.t.\ $y$ and $\q$ \deff 
\begin{align*}
	y\in \exp(\RR\cdot \g)\cdot x\qquad\text{and}\qquad \q\in \RR_{\neq 0}\cdot \g + \mg_\gamma\qquad\text{holds}
\end{align*}
for $\mg_\gamma$ the Lie algebra of the stabilizer $G_\gamma\equiv \{g\in G\:|\:g\cdot \gamma=\gamma\}$ of $\gamma$.
\end{custompr}
\noindent 
In  Sect.\ \ref{ghdhgg}, both results (Theorem \ref{A} and Proposition \ref{B}) are carried over to the analytic 1-submanifold case, cf.\ Theorem \ref{ghfgh} and Proposition \ref{dfgffhfh}.

Theorem \ref{A}, without the uniqueness statement in Proposition \ref{B}, was originally proven in \cite{MAX} for embedded analytic curves with compact domain;  
 namely, for the situation that $\wm$ is analytic, admits only normal stabilizers, and is proper or transitive as well as  pointwise proper, cf.\ Proposition 5.23 in \cite{MAX}. 
Then, in \cite{FH}, extensive technical efforts had been made to generalize this statement to the analytic pointwise proper case (no uniqueness result either). 
The more general Theorem \ref{A} now follows by elementary arguments from Lemma \ref{lieth}, stating that (locally) an analytic immersive curve is exponential if it fulfills a particular    approximation property (cf.\ Definition \ref{Def:LC}). We then  first derive from \eqref{iosdiosd1} that each non-free curve admits a special self similarity property -- this is done in Lemma \ref{HL} and Lemma \ref{seque}, which essentially reflect the argumentations in Lemma 5.19.2) in \cite{MAX} -- and then conclude from \eqref{iosdiosd1} and \eqref{iosdiosd2} that this self similarity property implies the mentioned approximation property, cf.\ Sect.\ \ref{regcase}. The  uniqueness statement in Proposition \ref{B} is proven in Sect.\ \ref{sdjvsdghddssad}.

In addition to the classification Theorem \ref{A}, we show that each free immersive curve $\gamma$ is discretely generated by the symmetry group if $\wm$ is sated and analytic in $M$ --   
 This was proven in Proposition 5.23.1) in \cite{MAX} for analytic pointwise proper actions and embedded analytic curves with compact domain; and then worked out in slightly more detail in \cite{FH}.\footnote{It was figured out that at most two group elements are necessary to relate the different free segments; and, formulas have been provided for the two cases discussed there, see also the points  \ref{casea} and \ref{caseb} below.}  
Specifically, we provide a detailed case analysis (cf.\ Convention \ref{nmsnmdsnmdsddscxcxcx} and Theorem \ref{sfdknfdhujfd} as well as Corollary \ref{sfdknfdhujfdd} for $\dom[\gamma]$ an arbitrary interval) that we now first state, and then explain step by step: 
\begin{customthm}{C}\label{C}
Assume that $\wm$ is sated and analytic in $M$; and let $\gamma\colon I\rightarrow M$ be immersive and free, such that $\gamma$ is not a free segment by itself. Then, $\gamma$ either admits a unique $\tau$-decomposition or a compact maximal interval.  In the second case, $\gamma$ is either positive or negative, and admits a unique $A$-decomposition for each compact maximal interval $A$. 
\end{customthm}
\vspace{-3pt}
\noindent
The used notions (maximal interval, $\tau$-decomposition, $A$-decomposition) are defined as follows:
\vspace{-3pt}
\begingroup
\setlength{\leftmargini}{18pt}
{
\renewcommand{\theenumi}{\bf\small(\arabic{enumi})} 
\renewcommand{\labelenumi}{\theenumi}
\begin{enumerate}
\item 
\label{hjdshjdhjdshjshjdsII0}
{\bf maximal interval:}  
An interval $D\subseteq I=\dom[\gamma]$ is said to be maximal \deff it is maximal w.r.t.\ the property that $\gamma|_D$ is a free segment. Notably, each such interval is necessarily closed in $I=(i',i)$,  hence either compact or of the form $(i',\tau]$ or $[\tau,i)$ for some $\tau\in I$. 
\end{enumerate}}
\endgroup
\noindent 
For the other two notions, we first need the following  two notations:
\begingroup
\setlength{\leftmargini}{11pt}
\begin{itemize}
\item
An analytic diffeomorphism is a bijection $\rho\colon D\rightarrow D'$ between intervals $D,D'\subseteq \RR$, such that both $\rho,\rho^{-1}$ are analytic (necessarily immersive). We denote  
the maximal analytic immersive extension (cf.\ Lemma \ref{maximalextension}) of $\rho$ by $\ovl{\rho}$, which is necessarily an analytic diffeomorphism to its image (as strictly monotonous, hence injective -- additionally apply Theorem \ref{ofdpofdpfdpof}).  
\item
Let $\tilde{\gamma}\colon \tilde{D}\rightarrow M$ and $\gamma\colon D\rightarrow M$ be analytic immersions, as well as  $\mu\colon \tilde{D}\supseteq \dom[\mu]\rightarrow \im[\mu]\subseteq D$ an analytic diffeomorphism. We write\: $\tilde{\gamma}\psim \gamma$\:  w.r.t.\ $\mu$\: \defff\:{}the following two conditions are fulfilled:
\begingroup
\setlength{\leftmarginii}{17pt}
\begin{itemize}
\item[$\small\mathsf{(i)}$]
We have $\tilde{\gamma}|_{\dom[\mu]}=\gamma\cp \mu$.
\item[$\mathsf{(ii)}$]
One of the following situations hold:
\begingroup
\setlength{\leftmarginiii}{11pt}
\begin{itemize}
\item
$\tilde{D}=(i',\tau]$ and $D= [\tau,i)$ are half-open, and $\mu(\tau)=\tau$ as well as $\dom[\mu]=\tilde{D}$ or $\im[\mu]=D$ holds.
\vspace{2pt}
\item
$\tilde{D}=\dom[\mu]$, $D=\im[\mu]$ are both compact.
\vspace{2pt}
\item
$\tilde{D}$ is compact, $\dom[\mu]\subset \tilde{D}$ and $D=\im[\mu]$ are half-open, and $\dom[\mu]\cap (\he\tilde{D}\backslash \innt[\tilde{D}]\he)\neq \emptyset$.
\end{itemize}
\endgroup
\end{itemize}
\endgroup
\noindent
$\boldsymbol{(\natural)}$ Notably, $\mu$ is already uniquely determined by the conditions $\mathsf{(i)}$ and $\mathsf{(ii)}$, cf.\ Sect.\ \ref{repari}. 
\end{itemize}
\endgroup
\begingroup
\setlength{\leftmargini}{18pt}
{
\renewcommand{\theenumi}{\bf\small(\arabic{enumi})} 
\renewcommand{\labelenumi}{\theenumi}
\begin{enumerate}
\setcounter{enumi}{1}
\item
\label{hjdshjdhjdshjshjdsII1}
 $\boldsymbol{\tau}${\bf{}-decomposition:} 
If $\gamma$ admits no compact maximal interval $A\subseteq I\equiv(i',i)$, 
then there exists
\begingroup
\setlength{\leftmarginii}{11pt}
\begin{itemize}
\item
 $\tau\in I$,  such that $(i',\tau]$, $[\tau,i)$ are the only maximal intervals ($\tau$ is necessarily unique).
\item 
a class $[e]\neq [g]\in G\slash G_\gamma$ with $g\cdot \gamma|_{(i',\tau]}\psim \gamma|_{[\tau,i)}$ w.r.t.\ $\mu$, such that:
\vspace{6pt}

If $g'\cdot \gamma|_J = \gamma\cp \rho$ holds for an analytic diffeomorphism $\rho\colon (i',\tau]\supseteq J\rightarrow J'\subseteq I$, then we have
\begin{align*}
\text{either}\qquad\quad 
[g']=[e]\quad\wedge\quad\ovl{\rho}|_{(i',\tau]}=\id_{(i',\tau]}\qquad\quad\text{or}\qquad\quad [g']=[g]\quad\wedge\quad \ovl{\rho}|_{\dom[\mu]}=\mu.
\end{align*}
In particular, the class $[g]$  is necessarily unique (also $\mu$ by $\boldsymbol{(\natural)}$); and a proper $G$-translate of $\gamma|_{(i',\tau]}$ can overlap $\gamma$ in exactly one way. 
\end{itemize}
\endgroup
\noindent  
For instance, if $\SO(2)$ acts via rotations on $\RR^2$, then $\gamma\colon \RR\ni t\mapsto (t,t^3)$ admits the $0$-decomposition $[g_\pi]$, for $g_\pi$ the rotation by the angle $\pi$, cf.\ Example \ref{dfggf}.
\item
\label{hjdshjdhjdshjshjdsII2}
$\boldsymbol{A}${\bf{}-decomposition:}  
If $\gamma$ admits a compact maximal interval, then each maximal interval is compact; and to each such 
compact maximal interval $A\subseteq I\equiv(i',i)$, there exists a unique $A$-decomposition of $\gamma$, i.e. a (necessarily unique) pair 
$(\{a_n\}_{n\in \cn},\{[g_n]\}_{n\in \cn})$  with\footnote{If $\cn_-=-\infty$ holds, then $\cn_-\leq n$ means $n\in \ZZ$; and analogously for $\cn_+$.} 
\begin{align*}
	\cn=\{n\in \ZZ_{\neq 0}\:|\: \cn_-\leq n\leq \cn_+\}\qquad \text{for }&\text{certain} \qquad -\infty\leq \cn_-<0< \cn_+\leq\infty
	\\[5pt]
	\text{as }&\text{well as}\\[5pt]
	\{a_n\}_{n\in \cn}\subseteq I\quad\text{with}\quad A=[a_{-1},a_1] \qquad\quad&\text{and}\qquad\quad \{g_n\}_{n\in \cn}\subseteq G\quad\text{with}\quad [e]\neq [g_{\pm1}]\in G\slash G_\gamma\\[5pt]
	\text{su}&\text{ch that}\\[5pt]
	\textstyle I=\bigcup_{n\in \cn}A_n\qquad\quad&\text{and}\qquad\quad 
	g_n\cdot \gamma|_A\psim \gamma|_{A_n}\:\:\:  \text{w.r.t.}\:\:\: \mu_n \qquad\forall\: n\in \cn\hspace{13pt}\\[5pt]
	\text{ holds for}& \text{ the Intervals}\\[5pt]
	A_{\cn_-}\equiv(i',a_{\cn_-}]\quad \text{if}\quad \cn_-\neq -\infty\qquad\:\:&\qquad\quad\:\: A_{\cn_+}\equiv[a_{\cn_+},i)\quad \text{if}\quad\cn_+\neq \infty\\[4pt]
A_n\equiv[a_{n-1},a_{n}]\quad\text{for}\quad \cn_-< n\leq -1 \qquad\quad A_0\equiv\:&[a_{-1},a_1]\qquad\quad A_n\equiv[a_n,a_{n+1}]\quad\text{for}\quad 1\leq n<\cn_+,
\end{align*}
whereby for each analytic diffeomorphism $\rho\colon A\supseteq  J\rightarrow J'\subseteq I$,  we have the implication:
\begin{align}
\label{fghhg}
	g\cdot \gamma|_{J} =\gamma\cp \rho\qquad\Longrightarrow \qquad [g]=[g_n]\:\:\:\text{and}\:\:\: \ovl{\rho}|_{\dom[\mu_n]}=\mu_n 
	\:\:\:\text{for some unique}\:\:\: n\in \cn\cup \{0\}
\end{align} 
provided we set $\mu_0:=\id_A$ and $g_0:= e$.

In particular, \eqref{fghhg} implies that the index set $\cn$ as well as the families $\{a_n\}_{n\in \cn}$, $\{[g_n]\}_{n\in \cn}$ are unique (also $\mu_n$ for each $n\in \cn$ by $\boldsymbol{(\natural)}$); and that each translate of $\gamma|_A$ overlaps $\gamma$ in a unique way. 
\vspace{3pt}

In this context, the following two cases can occur:\footnote{The different symmetry types (shifts $\cong$ positive and flips $\cong$ negative) that can occur, had already been identified in the proof of Proposition 5.23.1) in \cite{MAX}. It just was not explicitly worked out there that for a fixed curve only either of these symmetries can occur, and that the group elements are related to each other in the described way -- These statements, however, follow by simple geometric arguments from analyticity of the involved curve, cf.\ Sect.\ \ref{jsdjklsdjklsd}.}
\begingroup
\setlength{\leftmarginii}{18pt}
{
\renewcommand{\theenumii}{\bf.\small(\alph{enumii})} 
\renewcommand{\labelenumii}{\bf\small(\alph{enumii})}
\begin{enumerate}
\item
\label{casea}
$\boldsymbol{\gamma}$ {\bf positive:} \hspace*{\fill}(cf.\  Proposition \ref{sdfdfdlla})
\vspace{2pt}
 
For each fixed compact maximal interval $A$, the corresponding  diffeomorphisms $\mu_n$ are positive ($\dot\mu_n>0$), and each compact $A_n$ (i.e., $\cn_-<n<\cn_+$) is maximal.\footnote{If $\cn_-\neq -\infty$, then $A_{\cn_-}$ might not be maximal; and, if $\cn_+\neq \infty$, then $A_{\cn_+}$ might not be maximal.} 
Moreover, there exists a unique class $[h]$, such that $[g_n]=[h^n]$ holds for all $n\in \cn$; and this class is the same for each compact maximal interval $A$. In addition to that, each $t\in I$ is contained in the interior of a compact maximal interval; specifically, if $A\equiv [a',a]$ is maximal, then 
\begingroup
\setlength{\leftmarginiii}{12pt}
\begin{itemize}
\item[$-$]
for each $a<b\hspace{3.46pt}<i$,\hspace{4.7pt} there exists $a'<b'<b$ unique, such that $[b',b]$ is maximal.
\item[$-$]
for each $i'<b'<a'$, there exists $b'\hspace{0.45pt}<b\hspace{2.35pt}<a$ unique, such that $[b',b]$ is maximal. 
\end{itemize}
\endgroup 
\noindent
For instance, if $\RR$ acts via $\wm(t,(x,y)):=(t+x,y)$ on $\RR^2$, then $\gamma\colon \RR\ni t\mapsto (t,\sin(t))$ is positive, with compact maximal intervals $\{[t,t+2\pi]\}_{t\in \RR}$ as well as and $[h]=[2\pi]$, cf.\ Example \ref{dsajsu}.
\item
\label{caseb}
$\boldsymbol{\gamma}$ {\bf negative:} \hspace*{\fill}(cf.\ Proposition \ref{trhdhg})
\vspace{2pt}

Let $A$ be a fixed compact maximal interval. 
Then, the derivative of $\mu_n$ has the signature $(-1)^n$, and  $$[g_n]=\underbrace{[g_{\sigma(1\he\cdot\: \sign(n))}\cdot g_{\sigma(2\he\cdot\: \sign(n))}\cdot  {\dots}\cdot g_{\sigma(|n|\he\cdot\:  \sign(n))}]}_{\displaystyle [g_{\sigma(\sign(n))}\cdot {\dots}\cdot g_{\sigma(n)}]}\qquad\quad \forall\: n\in \cn$$ holds  for $\sigma\colon \ZZ_{\neq 0}\rightarrow \{-1,1\}$ defined by
\begin{align*}
\sigma(n):=
\begin{cases} 
	(-1)^{n-1} &\mbox{if }\: n > 0 \\ 
	(-1)^n & \mbox{if }\: n < 0. 
\end{cases} 
\end{align*}
Hence, e.g.\ $[g_{\pm 2}]=[g_{\pm1}\cdot g_{\mp1}]$, $[g_{\pm 3}]=[g_{\pm 1}\cdot g_{\mp 1}\cdot g_{\pm 1}]$, $[g_{\pm 4}]=[g_{\pm 1}\cdot g_{\mp 1}\cdot g_{\pm 1}\cdot g_{\mp 1}]$.
\vspace{4pt}

\noindent
In addition to that, the following statements hold:
\begingroup
\setlength{\leftmarginiii}{12pt}
\begin{itemize}
\item
We have\hspace{1pt} 
 $[g_{\pm 1}]\subseteq  G_{\gamma(a_{\pm 1})}\setminus G_\gamma$\hspace{1pt} as well as\hspace{1pt} $[g^{-1}_{\pm 1}]=[g_{\pm1}]$.
 \vspace{2pt}
\item
Each of the intervals $A_n$ is maximal; and they are the only maximal intervals. Hence, if $B$ is a further maximal  interval, then it equals some $A_n$; and the corresponding $B$-decomposition of $\gamma$ can be obtained from the $A$-decomposition of $\gamma$ just via \eqref{fghhg}.
\end{itemize}
\endgroup
For instance, if the euclidean group $\RR^2 \rtimes \SO(2)$ acts canonically on $\RR^2$, then $\gamma\colon \RR\ni t\mapsto (t,\sin(t))$ is negative with compact maximal interval  $A=[0,\pi]$. In this case, $[g_{-1}]$ and $[g_{1}]$ are classes of the rotations by the angle $\pi$ around $(0,0)$ and $(\pi,0)$, respectively, cf.\ Example \ref{sepocvofgvifg}.\ref{sepocvofgvifg1}.
\end{enumerate}}
\endgroup
\end{enumerate}}
\endgroup
\noindent
If $\gamma$ is non-constant analytic, then    
 $Z\equiv\{t\in \dom[\gamma]\:|\:\dot\gamma(t)=0\}$ consists of isolated points and admits no limit point in $\dom[\gamma]$, just by analyticity of $\dot\gamma$. Therefore, $\gamma$ splits into countably many analytic immersive subcurves, ``pinned together'' at the points in $Z$, cf.\ Remark \ref{rtrevgf}. Then, each of these subcurves is free as well (cf.\ Corollary \ref{fdgf}), so that our decomposition results apply to each of them separately. Anyhow, besides certain combinatorical and technical issues, a deeper investigation of the analysis of $\gamma$ at the points in $Z$ seems to be necessary, in order to prove analogous decomposition results also for the general non-constant analytic case. For  analytic 1-submanifolds, the strategy is sketched in the end of Sect.\ \ref{ghdhgg}; and the expected results \cite{MAXDECO} are stated there. 
\vspace{6pt}

\noindent
This paper is organized as follows:
\begingroup
\setlength{\leftmargini}{12pt}
\begin{itemize}
\item
In Sect.\ \ref{gffggf}, we fix the notations and collect the basic facts and definitions that we need in the main text.
\item
In Sect.\ \ref{skjsdghfsd}, we prove Theorem \ref{A} $\equiv$ Theorem \ref{classi}.
\item
In Sect.\ \ref{disgencur}, we prove the decomposition results for analytic immersive curves; specifically, Theorem \ref{C} $\equiv$ Theorem \ref{sfdknfdhujfd} as well as Corollary \ref{sfdknfdhujfdd} for arbitrary domains (intervals).
\item
In Sect.\ \ref{ghdhgg}, analytic 1-submanifolds are discussed. We prove analogues to Theorem \ref{A} and Proposition \ref{B},  and pave the way for global decomposition results for such  1-submanifolds.
\end{itemize}
\endgroup

\section{Preliminaries}
\label{gffggf}
In this section, we fix some conventions and provide several basic facts and definitions that we will need to work efficiently in the main text. 
\subsection{Conventions}
\label{dsdsaddslooasaasasalsads}
Manifolds are always assumed to be   
finite-dimensional,  
Hausdorff, and analytic. 
If $f\colon M\rightarrow N$ is a differentiable map between the manifolds $M$ and $N$, then $\dd f\colon TM\rightarrow TN$ denotes the corresponding differential map between their tangent manifolds. A differentiable map $f\colon M\rightarrow N$ is said to be immersive \deff  the restriction $\dd_xf:=\dd f|_{T_xM}\colon T_xM\rightarrow T_{f(x)}N$ is injective for each $x\in M$. Then, $f$ is said to be an embedding \deff $f$ is injective, immersive, as well as a homeomorphism to $f(M)$ when equipped with the subspace topology. 
Elements of tangent spaces are usually written with arrows, such as $\vv\in T_xM$. 
Unless explicitly stated otherwise, manifolds are assumed to be manifolds without boundary, and we do not assume  second countability.
Specifically, in what follows
\begingroup
\setlength{\leftmargini}{10pt}
\begin{itemize} 
\item
$M$ denotes a finite-dimensional Hausdorff analytic manifold  without boundary (second countability is not required); and, the  same conventions hold for the manifold structures of Lie groups.
\item
$S$ denotes a  connected,  Hausdorff,  second countable $1$-dimensional analytic manifold with boundary.
\end{itemize}
\endgroup \noindent
By an accumulation point of a topological space $X$, we understand an element $x\in X$ such that there exists a net $\{x_\alpha\}_{\alpha\in I}\subseteq X\backslash\{x\}$ with $\lim_\alpha x_\alpha=x$.
\vspace{6pt}

\noindent
Intervals are always assumed to be nondegenerate, i.e., an interval $D$ is a connected subset of $\RR$ with non-empty interior $\innt[D]$. We denote the closure of $D$ in $\RR$ by $\clos[D]$. We  say that $-\infty$ or $\infty$ is a  boundary point of $D$ \deff $\inf(D)=-\infty$ or $\sup(D)=\infty$ holds, respectively. If we write $I,J$ or $K,L$ instead of $D$, we always mean that $I,J$ are open, and that $K,L$ are compact. 
\vspace{6pt}

\noindent
A curve is a continuous map $\gamma\colon D\rightarrow X$ from an interval  $D$ to a topological space $X$. 
\begingroup
\setlength{\leftmargini}{12pt}
\begin{itemize} 
\item
If $\gamma$ is injective, then $\gamma^{-1}\colon \im[\gamma]\rightarrow \dom[\gamma]$  denotes its inverse in the sense of mappings. 
\end{itemize}
\endgroup
\noindent
An extension of $\gamma$ is a curve $\wt{\gamma}\colon I\rightarrow X$  that is defined on an open interval $I$ containing $D$, such that $\wt{\gamma}|_D=\gamma$ holds. 
If $M$ is an analytic manifold (without boundary), then a curve $\gamma\colon D\rightarrow M$ is said to be
\begingroup
\setlength{\leftmargini}{12pt}
\begin{itemize}
\item
analytic \deff it admits an analytic extension.
\item
(analytic) immersive \deff it admits an (analytic) immersive extension.
\item
an analytic embedding \deff it admits an analytic immersive extension that is a homeomorphism onto its image equipped with the subspace  topology.
\end{itemize}
\endgroup
\noindent
Analogously, a continuous map (curve) $\wt{\rho}\colon I\rightarrow \RR$ is said to be an extension of the continuous map (curve) $\rho\colon  D\rightarrow \RR$ \deff $I\subseteq \RR$ is an open interval with $D\subseteq I$ and $\rho=\wt{\rho}|_D$. Then, $\rho\colon  D\rightarrow D'\subseteq \RR$ is said to be
\begingroup
\setlength{\leftmargini}{12pt}
\begin{itemize}
\item
analytic (immersive) \deff it admits an analytic (immersive) extension.
\item
an (analytic) diffeomorphism \deff it admits an extension which is an (analytic) diffeomorphism.
\end{itemize}
\endgroup
\noindent
A diffeomorphism $\rho\colon D\rightarrow D'$ is said to be positive or negative \deff $\dot\rho(t)>0$ or $\dot\rho(t)<0$ holds for one (and then for each) $t\in \innt[D]$, respectively.\footnote{Observe  that $\rho$ is positive\slash{}negative \deff $\rho$ is strictly increasing/decreasing.} 
\vspace{6pt}

\noindent
In this paper, $\wm\colon G\times M\rightarrow M$ always denotes a left action of a Lie group $G$ on an analytic  manifold $M$ that is continuous in $G$, i.e., $\wm_x\colon G\ni g\mapsto \wm(g,x)\in M$ is continuous for each fixed $x\in M$.  
We say that $\wm$ is
\begingroup
\setlength{\leftmargini}{12pt}
\begin{itemize}
\item
analytic in $G$\hspace{24.7pt}\: \defff\:   
	$\wm_x\colon \hspace{3.1pt}G\ni g\mapsto\hspace{0.6pt} \wm(g,x)\in M$\hspace{1pt} is analytic for each fixed $x\in M$.
\item	
analytic in $M$\hspace{21.8pt}\: \defff\:  
	$\wm_g\hspace{0.6pt}\colon M\ni x\mapsto \wm(g,x)\in M$\hspace{1pt} is analytic for each fixed $g\in G$.
\item
separately analytic\: \defff\:   $\wm$ is analytic in $G$ and $M$.
\end{itemize}
\endgroup
\noindent 
We write $g\cdot x$ instead of $\wm(g,x)$ if it helps to simplify the notations. 
Notably, if $\wm$ is analytic in $M$, then $\gamma\colon D\rightarrow M$ is analytic (immersive)\slash an analytic embedding \defff   $g\cdot \gamma$ is analytic (immersive)\slash an analytic embedding for each $g\in G$. 
For $x\in M$,
\begingroup
\setlength{\leftmargini}{12pt}
\begin{itemize}
\item
$G_x\hspace{8pt}=\left\{g\in G\: \big| \: g\cdot x=x\right\}$\quad denotes the stabilizer of $x$, and $\mg_x$ the Lie algebra of $G_x$.\footnote{Observe that $G_x$ is a Lie subgroup of $G$, because $G_x$ is closed in $G$ as $\wm_x$ is continuous.} 
\item
$G\cdot x=\{g\cdot x \:|\: g\in G\}$\hspace{21.5pt}\quad denotes the orbit of $x$ under $G$, with stabilizer $G_{[x]}:=\bigcap_{g\in G} G_{g\cdot x}\subseteq G_x$. 
\end{itemize}
\endgroup
\noindent
Finally, for $x\in M$ and $\g\in \mg$, we define the curve
\begin{align}
\label{sdfdshjjww}
	\gag\colon\RR\rightarrow M,\qquad t\mapsto \exp(t\cdot \g)\cdot x
\end{align}
that is immersive \deff $\g\in \mg\backslash \mg_x$ holds, and constant elsewise; specifically, we have (cf.\ Lemma 5.6.2) in \cite{MAX}):   
\begin{lemma}
\label{lemma:expeig}
Assume that $\wm$ is analytic in $G$, and let  $\gamma\equiv \gag$ for $x\in M$ and $\g\in \mg\backslash\mg_x$. 
\begingroup
\setlength{\leftmargini}{15pt}
{
\renewcommand{\theenumi}{\sf\alph{enumi})} 
\renewcommand{\labelenumi}{\theenumi}
\begin{enumerate}
\item
$\gamma$ is analytic immersive. 
\item
If $\gamma$ is not injective, then $\gamma$ is cyclic in the sense that there exists $\pii(x,\g)\in \RR_{>0}$ unique with:  
	\begin{align*}
	\gamma(t')=\gamma(t)\quad\text{for}\quad t',t\in \RR\qquad\quad\Longleftrightarrow \qquad\quad t'=t + n\cdot \pii(x,\g) \quad\text{for some}\quad n\in \mathbb{Z}.
	\end{align*}	    	 
\end{enumerate}}
\endgroup
\end{lemma}
\begin{proof}
The proof is elementary, and can be found in Appendix \ref{appA7}.
\end{proof}
\noindent
Then, for $x\in M$ and $\g\in\mg_x$, we define: 
\begingroup
\setlength{\leftmargini}{12pt}
\begin{itemize}
\item
$\pii(x,\g):=0$, \hspace{46.8pt}and let $\ppi(x,\g):=\RR$\hspace{50.2pt} \deff\quad\:\:$\gag$ is constant, i.e.\ $\g\in \mg_x$, 
\item
$\pii(x,\g):=\infty$, \hspace{38.5pt} and let $\ppi(x,\g):=\{0\}$\hspace{45.75pt}\deff\quad\:\:$\gag$ is injective, 
\item
$\pii(x,\g)$ as in Lemma \ref{lemma:expeig}, and let $\ppi(x,\g):=\ZZ\cdot \pii(x,\g)$\quad\:\: \deff\quad\:\:$\gag$ is non-constant and not injective.
\end{itemize}
\endgroup
\noindent
Evidently, then $\pii(y,\g)=\pii(x,\g)$ as well as $\ppi(y,\g)=\ppi(x,\g)$ holds for each $y\in \im[\gamma]$.

\subsection{The Inverse Function Theorem}
\label{sdsddsdsewoewopewopweds}
We now briefly discuss some consequences of the Real Analytic Inverse Function Theorem:
\begin{theorem}
\label{ofdpofdpfdpof}
Assume that $F\colon \RR^n\supseteq U\rightarrow \RR^n$ is analytic ($U$ open) such that $\dd_p F$ is an isomorphism for some $p\in U$. Then,  there exists an open neighbourhood $V$ of $p$ with $V\subseteq U$, and an open neighbourhood $W\subseteq \RR^n$ of $F(p)$ such that $(F|_V)|^W\colon V\rightarrow W$ is an analytic diffeomorphism.
\end{theorem}
\begin{proof}
	Confer, e.g., Theorem 2.5.1 in \cite{PRIM}.
\end{proof}
\begin{lemma}
\label{dfdfdfdfdf}
Let $m\geq 1$ and $\ell\geq 0$ be given; and assume that $f\colon  \RR^m\supseteq U \rightarrow \RR^{m+\ell}$ is analytic ($U$ open), with $\dd_x f$ injective for some $x\in U$. Then, there exist $V\subseteq U$ open with $x\in V$, $W\subseteq \RR^{m+\ell}$ open with $f(V)\subseteq  W$, and an analytic diffeomorphism $\alpha\colon W\rightarrow W'$ to an open subset $W'\subseteq \RR^{m+\ell}$, such that
\begin{align*}
	\qquad\qquad\qquad(\alpha\cp f|_V)(v)=(v,0)\in \underbrace{\RR^{m}\times \RR^\ell}_{\displaystyle\cong \RR^{m+\ell}}\qquad\quad \forall\: v\in V 
\end{align*}
\vspace{-17pt}

\noindent
holds (with $\RR^0:=\{0\}$ for the case $\ell=0$). 
\end{lemma}
\begin{proof}
This is a straightforward consequence of Theorem \ref{ofdpofdpfdpof}, cf.\  Appendix \ref{appA77}.
\end{proof}

\begin{corollary}
\label{dfdsasasasassa}
Let $m\geq 1$, $\ell\geq 0$, $U\subseteq \RR^m$ open, $M$ an analytic manifold with $\dim[M]=m+\ell$, and $\iota\colon \RR^m\supseteq U \rightarrow M$ analytic with $\dd_x\iota$ injective for some fixed $x\in U$. Then, there exists an open neighbourhood $V\subseteq U$ of $x$ and an analytic chart $(O,\psi)$ of $M$ with $\iota(V)\subseteq O$, such that
\begin{align*}
	(\psi\cp \iota|_V)(v)=(v,0) \qquad\quad \forall\: v\in V 
\end{align*}  
holds; hence, $\iota|_V$ is an embedding.
\end{corollary}
\begin{proof}
Let $(\tilde{O},\tilde{\psi})$ be an analytic chart around $\iota(x)$. Shrinking $U$ around $x$ if necessary, we can assume that $\im[\iota]\subseteq \tilde{O}$ holds. Let $\alpha,W,W',V$ be as in Lemma \ref{dfdfdfdfdf}, for  $f\equiv \tilde{\psi}\cp \iota$ there. 
\begingroup
\setlength{\leftmargini}{12pt}
\begin{itemize}
\item[$\triangleright$]
We can assume $\im[\tilde{\psi}]\subseteq W$, just by restricting $\tilde{\psi}$ to $\tilde{O}\cap \tilde{\psi}^{-1}(W)$.
\item[$\triangleright$]
We can assume $\im[\tilde{\psi}]=W$, just by restricting $\alpha$ to $\im[\tilde{\psi}]$.
\end{itemize}
\endgroup
\noindent
The claim now holds for the analytic chart $(O,\psi):=(W,\alpha\cp\tilde{\psi})$.   
\end{proof}
\noindent
From Corollary \ref{dfdsasasasassa}, we obtain the following two  statements: 
\begin{lemma}
\label{lemma:BasicAnalyt0}
	Let $\gamma\colon I\rightarrow M$, $\gamma'\colon I'\rightarrow M$ be analytic curves, with $\dot\gamma(t)\neq 0$ for some fixed $t\in I$; and let $I\supseteq \{t_n\}_{n\in \NN}\rightarrow t\in I$ as well as $I'\backslash\{t'\}\supseteq \{t'_n\}_{n\in \NN}\rightarrow t'\in I'$ be sequences with $\gamma(t_n)=\gamma'(t_n')$ for all $n\in \NN$. Then, $\gamma'|_{J'}=\gamma\cp \rho$ holds for an analytic map $\rho\colon I'\supseteq J'\rightarrow B\subseteq I$ with $\rho(t')=t$, for  an open interval $J'$ containing $t'$, as well as $B\subseteq \RR$ connected. Moreover, $B\neq \emptyset$ is singleton \deff $\gamma'|_{J'}$ is constant; hence, $B$ is non-singleton (i.e.\ an interval) \deff $\gamma'|_{J'}$ is non-constant.
\end{lemma}
\begin{proof}
The proof is elementary, and can be found in Appendix \ref{appA2}. 
\end{proof}
\begin{lemma}
\label{gdfgfgf}
Let $\wm$ be analytic in $G$, and $\gamma\colon I\rightarrow M$ an analytic curve. Let $t\in I$ be given such that there exist  sequences $G\supseteq \{g_n\}_{n\in \NN}\rightarrow e$ and $I\backslash\{t\}\supseteq \{t_n\}_{n\in \NN}\rightarrow t$ with $g_n\cdot \gamma(t)=\gamma(t_n)$ for all $n\in \NN$. Then, there exists an open interval $J\subseteq I$ with $t\in J$, such that $\gamma(J)\subseteq G\cdot \gamma(t)$ holds. 
\end{lemma}
\begin{proof}
The proof is elementary, and can be found in Appendix \ref{appA1}.
\end{proof}

\subsection{Analytic Curves}
This subsection collects the most important properties of analytic curves that we will need in the main text. We start with some basic properties, and then discuss relations between immersive analytic curves. 

\subsubsection{Basic Properties}
Lemma \ref{lemma:BasicAnalyt0} provides us with the following two statements:
\begin{lemma}
    \label{lemma:BasicAnalyt1}
	 Let $\gamma\colon I\rightarrow M$, $\gamma'\colon I'\rightarrow M$ be analytic embeddings; and let $x$ be an accumulation point of $\im[\gamma]\cap \im[\gamma']$.\footnote{Hence, $x\in X:= \im[\gamma]\cap \im[\gamma']$ holds; and there exists a net $X\setminus\{x\}\supseteq \{x_\alpha\}_{\alpha\in I}\rightarrow x$ that converges w.r.t.\ the subspace topology on $X$  inherited from $M$.} Then, $\gamma(J)=\gamma'(J')$ holds for some open intervals $J\subseteq I$ and $J'\subseteq I'$ with $x\in  \gamma(J),\gamma'(J')$. 
\end{lemma}
\begin{proof}
	The claim is clear from Lemma \ref{lemma:BasicAnalyt0}.
\end{proof} 
\begin{lemma}
	\label{lemma:BasicAnalyt2}
   	Let $\gamma\colon D\rightarrow M$, $\gamma'\colon D'\rightarrow M$ be analytic embeddings with $\gamma(D)=\gamma'(D')$. Then, $\gamma=\gamma' \cp \rho$ holds for a  (necessarily unique) analytic diffeomorphism $\rho\colon D\rightarrow D'$; specifically,
$$
\rho=(\gamma'|^{\im[\gamma']})^{-1}\cp \gamma\colon D \rightarrow D'.
$$   	
\end{lemma}
\begin{proof}
The proof is elementary, and can be found in Appendix \ref{appA3}. 
\end{proof}
\noindent
We also explicitly mention the following statement, as we will use  it permanently in this paper.
\begin{lemma}
\label{corgleich}
Let $\gamma,\gamma'\colon D\rightarrow M$ be analytic curves, and $D'\subseteq D$ an interval. Then,   
\begin{align*}
	\gamma|_{D'}=\gamma'|_{D'}\qquad\quad\Longrightarrow\quad \qquad\gamma=\gamma'.
\end{align*}
\end{lemma}
\begin{proof}
Let $A\subseteq D$ denote the union of all intervals $B\subseteq \RR$, with $D'\subseteq B\subseteq D$ and $\gamma|_{B}=\gamma'|_{B}$. Then, $A$ is closed in $D$ by continuity of $\gamma,\gamma'$, as well as open in $D$ by analyticity of $\gamma, \gamma'$. Since $A$ is non-empty and $D$ is connected, we must have $A=D$. 
\end{proof}
An analytic (immersive) curve $\gamma\colon D\rightarrow M$ is said to be {\bf maximal} \deff it has no proper extension, i.e., \deff $\wt{\gamma}=\gamma$ holds for each analytic (immersive) extension $\wt{\gamma}$ of $\gamma$. Analogous conventions will also hold for analytic (immersive) maps $\rho\colon D\rightarrow \RR$, and analytic diffeomorphisms $\rho\colon D\rightarrow D'\subseteq \RR$. The domain of each such maximal $\gamma$ or $\rho$ is necessarily open; and  Lemma \ref{corgleich} yields the following statement.
\begin{lemma}
 \label{maximalextension}
    	Each analytic (immersive) curve admits a unique maximal analytic (immersive) extension.
\end{lemma}
\begin{proof}  
The proof is elementary, and can be found in Appendix \ref{appA4}.
\end{proof}
Then, for $\gamma\colon D\rightarrow M$ an analytic (immersive) curve,   $\ovl{\gamma}\colon I\rightarrow M$ will always denote its maximal analytic (immersive) extension in the following.  Analogously, if $\rho \colon D\rightarrow \RR$ is an (immersive) analytic map, then $\ovl{\rho}\colon I\rightarrow \RR$ always denotes its maximal (immersive) analytic extension. 
Observe that in the immersive case, $\ovl{\rho}|^{\im[\ovl{\rho}]}$ (hence $\rho|^{\im[\rho]}$) is automatically an analytic diffeomorphism by Theorem \ref{ofdpofdpfdpof} (as $\ovl{\rho}$ is  strictly monotonous, hence injective).

\begin{corollary}
\label{sdoppsdoods}
	Let $\rho\colon D\rightarrow \RR$ and $\rho'\colon D'\rightarrow \RR$ be analytic (immersive); and assume that $\rho|_A=\rho'|_A$ holds for an interval $A\subseteq D\cap D'$. Then, $\ovl{\rho}=\ovl{\rho}'$ holds, hence $D\cup D'\subseteq \dom[\ovl{\rho}]= \dom[\ovl{\rho}']$. 
\end{corollary}
\begin{proof}
	Let $I:=\dom[\ovl{\rho}]$, $I':=\dom[\ovl{\rho}']$, and observe that $A\subseteq I\cap I'$ holds; hence, $\ovl{\rho}|_{I\cap I'}=\ovl{\rho}'|_{I\cap I'}$ by Lemma \ref{corgleich}. Then,  $\ovl{\rho}=\hat{\rho}=\ovl{\rho}'$ holds by maximality of $\ovl{\rho},\ovl{\rho}'$, for the analytic (immersive) map \hspace{\fill} (extension of $\ovl{\rho},\ovl{\rho}'$)
\[
	\hat{\rho}\colon  I\cup I'\rightarrow \RR,\qquad t\mapsto 
	\begin{cases} 
	\ovl{\rho}(t) &\mbox{for }\:\: t\in I \\ 
	\ovl{\rho}'(t) & \mbox{for }\:\: t\in I'\backslash I. 
\end{cases} 
\qedhere
\]
\end{proof}

\subsubsection{Relations between Curves}
\label{fdggdrere}
We write $\gamma\cpsim \gamma'$ for analytic immersions $\gamma\colon D\rightarrow M$ and $\gamma'\colon D'\rightarrow M$ \deff $\gamma(J)=\gamma'(J')$ holds for open intervals $J\subseteq D$ and $J'\subseteq D'$ on which $\gamma$ and $\gamma'$ are embeddings, respectively.
\begingroup
\setlength{\leftmargini}{12pt}
\begin{itemize}
\item
We have $\gamma=\gamma'\cp \rho$ for a unique analytic diffeomorphism $\rho\colon J\rightarrow J'$, by Lemma \ref{lemma:BasicAnalyt2}. 

If it helps to simplify the argumentation, we will alternatively say that $\gamma\cpsim \gamma'$ holds w.r.t.\ $\rho$.
\item
Assume that $\wm\colon G\times M\rightarrow M$ is analytic in $M$. Then, the following equivalences hold for each $g\in G$, and each analytic diffeomorphism $\rho'\colon D''\rightarrow D'$:
\begin{align*}
g\cdot \gamma\cpsim g\cdot \gamma'	\qquad\quad\Longleftrightarrow\qquad\quad\gamma\cpsim \gamma'\qquad\quad\Longleftrightarrow\qquad\quad \gamma\cpsim \gamma'\cp \rho'.  
\end{align*}
\end{itemize}
\endgroup
\noindent
Assume that $\gamma\cpsim \gamma'$ holds w.r.t.\ $\rho$. We now want to figure out what might happen, if we try to extend such a relation to the whole domain of $\gamma$. For this, we first observe that 
\begin{align}
\label{pospodspoaop}
\gamma|_\CM=\gamma'\cp \ovl{\rho}|_\CM\qquad\quad\text{holds on the interval}\qquad\quad \CM:=D\cap \ovl{\rho}^{-1}(D')
\end{align}
by Lemma \ref{corgleich}, and that $\CM$ is maximal w.r.t.\ this property. 
For $D\equiv K$ and $D'\equiv K'$ compact, we additionally have the following statement.
\begin{lemma}
\label{lem:maxextKomp}
Let $\gamma\colon K\rightarrow M$ and $\gamma'\colon K'\rightarrow M$ be analytic immersions with $\gamma|_B=\gamma'\cp \rho$ for an  analytic diffeomorphism $\rho\colon K\supseteq B\rightarrow B'\subseteq K'$. 
Then, $\CM:=K\cap \ovl{\rho}^{-1}(K')$ is a compact interval; and we have
\begin{align}
\label{ddddd}
 a<c'\quad\:\:\: \Longrightarrow\quad\:\:\: \ovl{\rho}(c')\in\{a',b'\} \qquad\quad\:\text{as well as}\qquad\quad\: c<b\quad\:\: \Longrightarrow\quad\:\: \ovl{\rho}(c)\in\{a',b'\}
\end{align} 
for $\CM\equiv[c',c]\subseteq K\equiv[a,b]$, and $K'\equiv[a',b']$.
\end{lemma}
\begin{proof}
We first show that 
$\ovl{\rho}$ is defined on an open interval around each $t\in \clos[C]$. This is clear if $t\in \innt[C]$ holds; and, for $t\in \partial C:=\clos[C]\setminus \innt[C]$ (boundary point),  we argue as follows:   
\begingroup
 \setlength{\leftmargini}{12pt}
\begin{itemize}
\item[$\triangleright$]
We fix $C\backslash\{t\}\supseteq \{t_n\}_{n\in \NN}\rightarrow t$, hence $\{t_n\}_{n\in \NN}\subseteq \dom[\ovl{\rho}]$ with $\ovl{\rho}(\{t_n\}_{n\in \NN})\subseteq K'$. Since $K'$ is compact,   
we can assume that $\lim_n \ovl{\rho}(t_n)=t'\in K'$ converges in $K'$ (just by passing to a subsequence if necessary).
\item[$\triangleright$] 
We fix open intervals $I,I'$ with $t\in I$, $t'\in I'$, such that $\ovl{\gamma}|_I, \ovl{\gamma}'|_{I'}$ are embeddings (Corollary \ref{dfdsasasasassa}). We have
\vspace{-4pt}
\begin{align*}
	 &\textstyle\ovl{\gamma}(t_n)=\gamma(t_n)\stackrel{\eqref{pospodspoaop}}{=}\gamma'(\ovl{\rho}(t_n))=\ovl{\gamma}'(\ovl{\rho}(t_n))\qquad\quad\forall\: n\in \NN,
\end{align*}
hence $x:= \ovl{\gamma}(t)=\lim_n\ovl{\gamma}(t_n)=\lim_n\ovl{\gamma}'(\ovl{\rho}(t_n))  =\ovl{\gamma}'(t')$ is an accumulation point of $\im[\ovl{\gamma}]\cap \im[\ovl{\gamma}']$. 
By Lemma \ref{lemma:BasicAnalyt1} and Lemma \ref{lemma:BasicAnalyt2}, we can shrink $I,I'$ (around $t,t'$) in such a way that $\ovl{\gamma}|_I=\ovl{\gamma}'\cp \tau$ holds for an  analytic diffeomorphism $\tau\colon I\rightarrow I'$.
\item[$\triangleright$] 
Since $\ovl{\rho}$ is strictly monotonous (with $\lim_n \ovl{\rho}_n(t_n)=t'\in I'$ and $\lim_n t_n=t\in I\cap \partial C$), there exists an open interval $J\subseteq I$ with $t\in J$, such that $\ovl{\rho}(C\cap J)\subseteq I'$ holds. Then, $\tau$ coincides with $\ovl{\rho}$ on $\CM\cap J$, because 
\vspace{-4pt}
\begin{align*}
	\ovl{\gamma}'\cp \tau|_{\CM\cap J}=\ovl{\gamma}|_{\CM\cap J}=\gamma|_{\CM\cap J}\stackrel{\eqref{pospodspoaop}}{=}\gamma'\cp \ovl{\rho}|_{\CM\cap J}=\ovl{\gamma}'\cp \ovl{\rho}|_{\CM\cap J}\qquad\quad\Longrightarrow\qquad\quad \tau|_{\CM\cap J}=\ovl{\rho}|_{\CM\cap J}
\end{align*} 
by injectivity of $\ovl{\gamma}'$ on $I'$. Corollary \ref{sdoppsdoods} thus shows $I=\dom[\tau]\subseteq \dom[\ovl{\rho}]$. 
\end{itemize}
\endgroup
\noindent
Then, $C\equiv [c',c]$ is compact, because $\ovl{\rho}(C)\subseteq K'$ implies $\ovl{\rho}(\clos[C])\subseteq K'$ (as $K'$ is closed). 
Moreover: 
\begingroup
 \setlength{\leftmargini}{12pt}
\begin{itemize}
\item[$\triangleright$]
If $c<b$ and $\ovl{\rho}(c)\in (a',b')$ holds, then there exists an open interval $\tilde{J}\subseteq K$ with $c\in \tilde{J}$ and $\ovl{\rho}(\tilde{J})\subseteq K'$; which contradicts the definition of $\CM$. 
\item[$\triangleright$]
If $a<c'$ and $\ovl{\rho}(c')\in (a',b')$ holds, then there exists an open interval $\tilde{J}\subseteq K$ with $c'\in \tilde{J}$ and $\ovl{\rho}(\tilde{J})\subseteq K'$; which contradicts the definition of $\CM$. 
\qedhere
\end{itemize}
\endgroup
\end{proof}
Lemma \ref{lem:maxextKomp} yields the following useful corollaries:
\begin{corollary}
\label{cgdcfgd}
Assume that $\gamma|_J=\gamma'\cp \rho$ holds for analytic immersions 
$\gamma\colon [a,b]\rightarrow M$, $\gamma'\colon K'\rightarrow M$,   and an 
analytic diffeomorphism $\rho\colon J\rightarrow J'$. Then, for each fixed $t\in J$, we have the implication:
\begin{align*}
	\im[\gamma']\nsubseteq \im[\gamma] \qquad\quad\Longrightarrow \qquad\quad \gamma([a,t])\subseteq \im[\gamma']\qquad\text{or} \qquad\gamma([t,b])\subseteq \im[\gamma'].
\end{align*}
\end{corollary}
\begin{proof}
The proof is elementary, and can be found in Appendix \ref{appA5}.
\end{proof}

\begin{corollary}
\label{dfdgttrgf}
If $\gamma\colon [a',a]\rightarrow M$ is an analytic immersion, then $\gamma|_{[a',r]}=\gamma\cp \rho$ cannot hold for a negative analytic diffeomorphism $\rho\colon [a',r]\rightarrow [s,a]$ with $r,s\in (a',a)$.
\end{corollary}
\begin{proof}
The proof is elementary, and can be found in Appendix \ref{appA6}.
\end{proof}

\subsubsection{Self-Relations}
We say that the analytic immersion $\gamma\colon D\rightarrow M$ is {\bf self-related} \deff $\gamma(K)=\gamma(K')$ holds for disjoint compact intervals $K,K'\subseteq D$ on which $\gamma$ is an embedding. We now first show the following statement. 
\begin{lemma}
\label{setsdffds}
If $\gamma\colon K\rightarrow M$ is an immersion, then there does not exist a homeomorphism $\rho\colon K\rightarrow K'\subset K$ with $\gamma=\gamma\cp \rho$.
\end{lemma}
\begin{proof}
Assume that such a homeomorphism $\rho\colon K\rightarrow K'\subset K$  exists.  We set $K_n:=\rho^n(K)$ for each $n\in \NN$,\footnote{Here, $\rho^n$ denotes the $n$-fold composition of $\rho$ with itself, hence $\rho^0=\id_K$ as well as  $\rho^1=\rho$ and $K'=K_1$.} so that $K_{n+1}\subset K_n$ holds for each $n\in \NN$. We fix $k\in K\backslash K'$, and define $k_n:=\rho^n(k)\in K$ for each $n\in \NN$ (hence $k_0=k$). Then, 
\begingroup
 \setlength{\leftmargini}{12pt}
\begin{itemize}
\item[$\triangleright$]
The $k_n$ are  mutually different, because the sets $\rho^{n}(K\backslash K')=K_n\backslash K_{n+1}$ are mutually disjoint.
\item[$\triangleright$]
Since $K$ is compact with $\{k_n\}_{n\in \NN}\subseteq K$, there exists $\iota\colon \NN\rightarrow \NN$ strictly increasing, such that $\lim_n k_{\iota(n)}=k'\in K$ exists.
\end{itemize}
\endgroup
\noindent
By construction, we have 
\begin{align*}
	\gamma(k_{n+1})=\gamma(\rho(k_n))=\gamma(k_n)\qquad\forall\: n\in \NN\qquad\quad&\Longrightarrow\hspace{9.2pt}\qquad\quad \gamma(k_{n})=\gamma(k_0)\qquad \forall\: n\in \NN\\
	&\Longrightarrow\qquad\quad \gamma(k_{\iota(n)})=\gamma(k_0)\qquad\forall\: n\in \NN,
\end{align*} 
which contradicts that $\gamma$ is injective on a neighbourhood of $k'$.
\end{proof}
\begin{lemma}
\label{sdffsd}
Let $\gamma\colon I\rightarrow M$ be an analytic immersion, and $D\subset I$ an interval such that $\gamma|_D$ is self-related. 
\begingroup
\setlength{\leftmargini}{14pt}
{
\renewcommand{\theenumi}{\alph{enumi})} 
\renewcommand{\labelenumi}{\theenumi}
\begin{enumerate}
\item
\label{sdffsd1}
Let\hspace{1pt} $\tau\in I\cap \{\inf(D),\sup(D)\}$ be given. Then,  \hspace*{\fill}(\he$J_\tau,J'_\tau$ open intervals\he)
\begin{align*}
	\gamma|_{J_\tau} = \gamma\cp\rho_\tau \quad\text{holds for an analytic diffeomorphism}\quad\rho_\tau\colon  D\supseteq J_\tau\rightarrow J'_\tau\subseteq I \quad\text{with}\quad \tau\in J'_\tau. 
\end{align*}
\item
\label{sdffsd2}
We have $\gamma(J)= \gamma(J')$ for certain open intervals $J\subseteq D$ and $J'\subseteq I\backslash D$ on which $\gamma$ is an embedding.
\end{enumerate}}
\endgroup
\end{lemma}
\begin{proof}
First observe that $I\cap \{\inf(D),\sup(D)\}\neq \emptyset$ holds, because  $I$ is open with $D\subset I$. Hence, 
Part \ref{sdffsd2} is clear from Part \ref{sdffsd1} and Corollary \ref{dfdsasasasassa} (as $\rho_\tau$ is a homeomorphism and $\innt[J'_\tau\cap I\setminus D]\neq \emptyset$). It thus remains to prove Part \ref{sdffsd1}.   
For this, we assume that $\tau=\sup(D)\in I$ holds (the case $\tau=\inf(D)\in I$ is treated analogously), and fix $\epsilon>0$ with $\tau+\epsilon\in I$:
 \begingroup
\setlength{\leftmargini}{12pt}
\begin{itemize}
\item[$\triangleright$]
Since $\gamma|_D$ is self-related, Lemma \ref{lemma:BasicAnalyt2} provides an analytic diffeomorphism $\rho\colon D\supseteq [a',r]\rightarrow [s,a]\subseteq D$ for certain $D\ni a'<r<s<a\in D$, such that $\gamma|_{[a',r]}=\gamma\cp \rho$ holds. 
\item[$\triangleright$]
Corollary \ref{dfdgttrgf} applied to $\gamma|_{[a',a]}$ shows that $\rho$ (hence $\ovl{\rho}$) is positive; so that $\rho(a')=s$ and $\rho(r)=a$ holds. 
\item[$\triangleright$]
Lemma \ref{lem:maxextKomp} applied to $\gamma\equiv \gamma|_{[a',\tau]}$ and $\gamma'\equiv \gamma|_{[s,\tau+\epsilon]}$ yields\footnote{Observe $\dom[\rho]=[a',r]\subseteq [a',\tau]$ as well as $\im[\rho]=[s,a]\subseteq [s,\tau+\epsilon]$.} 
\begin{align}
\label{nmdsnmdsnmdsnmdsds98ds98s98d98ds98dsdsdsdsds}
[a',c]=[a',\tau]\cap \ovl{\rho}^{-1}([s,\tau+\epsilon])\quad\:\:\text{with}\quad\:\: \gamma|_{[a',c]}=\gamma\cp \ovl{\rho}|_{[a',c]}\quad\:\:\text{for some}\quad\:\: r\leq c\leq \tau.
\tag{$\natural$}
\end{align}   
\end{itemize}
\endgroup
\vspace{-6pt}

\noindent   
Then, the following two cases can occur:  
\begingroup
\setlength{\leftmargini}{12pt}
\begin{itemize}
\item	
If $\ovl{\rho}(c)>\tau$, the claim holds for $J_\tau:=(a',c)$, $J'_\tau:= (s,\rho(c))$, $\rho_\tau:=\ovl{\rho}|_{J_\tau}$. 
\item	
Assume $\ovl{\rho}(c)\leq \tau$: 
\begingroup
\setlength{\leftmarginii}{12pt}
\begin{itemize}
\item[$*$]
We necessarily have $c=\tau$. In fact, since $\ovl{\rho}$ is positive, Lemma \ref{lem:maxextKomp} yields the implication
\vspace{-3pt}
\begin{align*}
	c<\tau\qquad\quad\stackrel{\eqref{ddddd}}{\Longrightarrow} \qquad\quad \ovl{\rho}(c)=\tau+\epsilon>\tau.
\end{align*}
\vspace{-16pt} 
\item[$*$]
We thus have $c=\tau$ and $\ovl{\rho}(\tau)=\ovl{\rho}(c)\leq \tau$. Hence, we obtain 
\vspace{-3pt}
\begin{align*}
\gamma|_{[a',\tau]}\stackrel{\eqref{nmdsnmdsnmdsnmdsds98ds98s98d98ds98dsdsdsdsds}}{=}\gamma|_{[a',\tau]}\cp \mu\qquad\quad\text{for}\qquad\quad \mu:=\ovl{\rho}|_{[a',\tau]}\colon [a',\tau]\rightarrow [s,\ovl{\rho}(\tau)]\subset [a',\tau],
\end{align*}  
which contradicts Lemma \ref{setsdffds}.\qedhere
\end{itemize}
\endgroup
\end{itemize}
\endgroup
\end{proof}
\noindent
The following lemma provides two useful conditions for self-relatedness of curves.
\begin{lemma}
\label{sshift}
Let $\gamma\colon D\rightarrow M$ and $\gamma'\colon D'\rightarrow M$ be analytic immersions, with 
\begin{align*}
\gamma|_B=\gamma'\cp \phi\qquad&\text{for}\qquad \phi\colon B\rightarrow D'\qquad\text{an analytic diffeomorphism},\\
\gamma|_J=\gamma'\cp \psi\qquad&\text{for}\qquad \psi\colon J\rightarrow J'\qquad\hspace{3pt}\text{an analytic diffeomorphism.}
\end{align*}
Then, $\gamma$ is self-related, if we have\: $J\nsubseteq B$\: or\: $J\subseteq B$ and $\phi|_J\neq \psi$. 
\end{lemma}
\begin{proof}
Let $L'\subseteq J'\subseteq D'$ be a compact interval, such that $\gamma'|_{L'}$ is an embedding (Corollary \ref{dfdsasasasassa}). We set $D\supseteq B\supseteq K:=\phi^{-1}(L')$ and $J\supseteq \tilde{K}:=\psi^{-1}(L')$. Then, $\gamma|_K, \gamma|_{\tilde{K}}$ are embeddings with $\gamma(K)=\gamma'(L')=\gamma(\tilde{K})$. 
\begingroup
\setlength{\leftmargini}{12pt}
\begin{itemize}
\item
Assume $J\nsubseteq B$. Then, 
$\innt[J\backslash B]\neq \emptyset$ holds as $J$ is open. Thus, shrinking $J$ and $L'$ if necessary, we can assume $J\cap B=\emptyset$. Then, $\gamma$ is self-related as we have $K\cap \tilde{K}\subseteq B\cap J=\emptyset$.	
\item	
Assume $J\subseteq B$ and $\phi|_J\neq \psi$. We claim that we can shrink $L'$ such that $K\cap \tilde{K}=\emptyset$ holds; from which the claim is evident.

In fact, assume that there exists no compact interval $\tilde{L}'\subseteq L'$ with $\phi^{-1}(\tilde{L}')\cap \psi^{-1}(\tilde{L}')\neq \emptyset$. We fix $t\in L'$, and choose a sequence $\{L'_n\}_{n\in \NN}$ of compact intervals with $L'_{n+1}\subseteq L'_n\subseteq L'$ for all $n\in \NN$ as well as $\bigcap_{n\in \NN} L'_n=\{t\}$. By assumption, $\phi^{-1}(L'_n)\cap \psi^{-1}(L'_n)\neq \emptyset$ holds for all $n\in \NN$. There thus exist sequences $\{r_n\}_{n\in \NN}, \{s_n\}_{n\in \NN}\subseteq L'$ with $\lim_n r_n=t=\lim_n s_n$, such that $\phi^{-1}(r_n)=\psi^{-1}(s_n)$ holds for all $n\in \NN$, hence
\vspace{-10pt}
\begin{align*}
	\textstyle\phi^{-1}(t)=\lim_n \phi^{-1}(r_n)=\lim_n\psi^{-1}(s_n)=\psi^{-1}(t).
\end{align*}
Since, this holds for each $t\in L'$, we have $\phi^{-1}|_{L'}=\psi^{-1}|_{L'}$; hence, $\phi^{-1}|_{J'}=\psi^{-1}$ by Lemma \ref{corgleich}. This implies $\phi|_J=\psi$, which contradicts the assumptions.\qedhere
\end{itemize}
\endgroup
\end{proof}
\subsection{Reparametrizations}
\label{repari}
In this subsection, we collect some further statements concerning reparametrizations of analytic immersive curves that we will need in Sect.\ \ref{disgencur}. 
More specifically, the arguments used in the proof of Lemma \ref{lem:maxextKomp} also work for non-compact domains. 
To figure out the possible cases efficiently, in the following we write $\gamma \psim_{t,t'} \gamma'$ for analytic immersions $\gamma\colon D\rightarrow M$ and $\gamma'\colon D'\rightarrow M$ with $t\in D$ and $t'\in D'$ \deff there exists an analytic diffeomorphisms $\rho\colon D \rightarrow D'$ with $\rho(t)=t'$ and $\gamma=\gamma'\cp \rho$.\footnote{We explicitly mention that $t$ and $t'$ are not assumed to be contained in the interior of $D$ and $D'$, respectively.} This diffeomorphism is uniquely determined, because the following implication holds:
\begin{align}
\label{pofdpofdpo}
   \begin{array}{@{}lr@{}}
       \gamma|_A \psim_{t,t'} \gamma'|_{A'}\quad\:\hspace{1pt}\text{w.r.t.}\quad\: \upsilon\colon A\rightarrow A'\\
	\gamma|_B \psim_{t,t'} \gamma'|_{B'}\quad\:\text{w.r.t.}\quad\: \kappa\colon B\rightarrow B'
        \end{array}\bigg\} \qquad\quad\Longrightarrow\qquad\quad\upsilon|_{A\cap B}=\kappa|_{A\cap B}.
\end{align}
\begin{proof}[Proof of Implication \eqref{pofdpofdpo}]
\vspace{0pt}

\noindent
\begingroup
\setlength{\leftmargini}{12pt}
\begin{itemize}
\item[$\triangleright$]
There exist open intervals $I\subseteq D$, $I'\subseteq D'$ with $t\in I$, $t'\in I'$ and $\ovl{\gamma}(I)=\ovl{\gamma}'(I')$, such that $\ovl{\gamma}|_I$ and $\ovl{\gamma}'|_{I'}$ are embeddings (combine $\gamma\psim_{t,t'} \gamma'$ with Corollary \ref{dfdsasasasassa} and Lemma \ref{lemma:BasicAnalyt1}). Hence, $\ovl{\gamma}|_I=\ovl{\gamma}'\cp \mu$ holds for an analytic diffeomorphism $\mu\colon I\rightarrow I'$ by Lemma \ref{lemma:BasicAnalyt2}.
\item[$\triangleright$]
By continuity of $\upsilon$ and $\kappa$, there exists  an open interval $J\subseteq I$ with $t\in J$ and  
 $\upsilon(J\cap A),\kappa(J\cap B)\subseteq I'$. 
 Since $\ovl{\gamma}'$ is injective on $I'$, we have 
\begin{align*}
	\upsilon|_{J\cap A\cap B}=\mu|_{J\cap A\cap B}=\kappa|_{J\cap A\cap B}\qquad\quad\Longrightarrow\qquad\quad \upsilon|_{A\cap B}=\kappa|_{A\cap B}
\end{align*} 
whereby the implication is trivial for $A\cap B=\{t\}$, and holds by Lemma \ref{corgleich} in the other case.\qedhere
\end{itemize}
\endgroup
\end{proof}
\noindent
Let now $\gamma\colon D\rightarrow M$ and $\gamma'\colon D'\rightarrow M$ be fixed analytic immersions, such that $\gamma|_K=\gamma'\cp \rho$ holds for an  analytic diffeomorphism $\rho\colon K\rightarrow K'$,  with compact intervals $K\subseteq D$ and $K'\subseteq D'$.  We fix some $t\in \innt[K]$ and set $t':=\rho(t)$.  Then, the explicit form of the intervals (observe $K\subseteq C$ and $K'\subseteq C'$)
\begin{align*}
	\CM:=D\cap \ovl{\rho}^{-1}(D')\qquad\quad\text{and}\qquad\quad C':=\ovl{\rho}(C)
\end{align*}
is completely determined by Corollary \ref{sdoppsdoods} as well as the following observations:
\vspace{6pt}

\noindent
We set $K_+:=K\cap [t,\infty)$ and $K_-:=K\cap (-\infty,t]$ as well as $K'_+:=K'\cap [t',\infty)$ and $K'_-:=K'\cap (-\infty,t']$, and observe:
\begin{align*}
\begin{array}{lrllrl}
	\dot{\rho}>0\qquad\qquad\Longleftrightarrow&\qquad\quad\gamma|_{K_+}\hspace{-6pt}&\psim_{t,t'} \gamma'|_{K'_+}\quad &\text{as well as}&\quad \gamma|_{K_-}\hspace{-6pt}&\psim_{t,t'} \gamma'|_{K'_-}\\[2pt] 
	\dot{\rho}<0\qquad\qquad\Longleftrightarrow&\qquad\quad\gamma|_{K_+}\hspace{-6pt}&\psim_{t,t'} \gamma'|_{K'_-}\quad &\text{as well as}&\quad \gamma|_{K_-}\hspace{-6pt}&\psim_{t,t'} \gamma'|_{K'_+}.
\end{array}
\end{align*}
We split the intervals $D,D',C,C'$ into their positive and negative parts as well, hence:  
\begingroup
\setlength{\leftmargini}{12pt}
\begin{itemize}
\item
Let $D_+,C_+$ and $D_-,C_-$ denote the intersections of $D,C$ \hspace{5pt}with $[t,\infty)$ \hspace{3pt}and $(-\infty,t$], \hspace{3.3pt}respectively.
\item 
Let $D'_+,C'_+$ and $D'_-,C'_-$ denote the intersections of $D',C'$ with $[t',\infty)$ and $(-\infty,t'$], respectively.
\end{itemize}
\endgroup
\noindent
Then,
\begin{align*}
\ovl{\rho}\colon \CM_\pm\rightarrow C'_\pm\qquad\Longleftrightarrow\qquad \dot{\rho}>0\qquad\qquad\text{as wel as}\qquad 
	\qquad\ovl{\rho}\colon \CM_\pm\rightarrow C'_\mp \qquad\Longleftrightarrow\qquad \dot{\rho}<0
\end{align*}
holds; and we conclude:
\begingroup
\setlength{\leftmargini}{12pt}
\begin{itemize}
\item
If $\dot{\rho}>0$ holds, then we have: \qquad\quad $C_{+}=D_+$\qquad$\Longleftrightarrow$\qquad$\gamma|_{D_+}\psim_{t,t'} \gamma'|_{C'_+}$\\
\phantom{If $\dot{\rho}>0$ holds, then we have:}\hspace{1pt} \qquad\quad $C_{-}=D_-$\qquad$\Longleftrightarrow$\qquad$\gamma|_{D_-}\psim_{t,t'} \gamma'|_{C'_-}$
\item
If $\dot{\rho}<0$ holds, then we have:\hspace{0.2pt} \qquad\quad $C_+=D_+$\qquad$\Longleftrightarrow$\qquad$\gamma|_{D_+}\psim_{t,t'} \gamma'|_{C'_-}$\\
\phantom{If $\dot{\rho}<0$ holds, then we have:}\hspace{1.2pt} \qquad\quad $C_-=D_-$\qquad$\Longleftrightarrow$\qquad$\gamma|_{D_-}\psim_{t,t'} \gamma'|_{C'_+}$
\end{itemize}
\endgroup
\noindent
It thus remains to investigate what might happen if $C_+\subset D_+$ or $C_-\subset D_-$ holds. This is covered by the following lemma. 
\begin{lemma}
\label{dgfggfg}
Assume that $C_{\bullet}\subset D_\bullet$ holds for $\bullet\in \{+,-\}$. Then, the following assertions hold:
\begingroup
\setlength{\leftmargini}{12pt}
\begin{itemize}
\item
For $\dot{\rho}>0$ and $\bullet\equiv +$, we have \qquad $C'_{+}= D'_+$\qquad hence\qquad $\gamma|_{C_+}\psim_{t,t'} \gamma'|_{D'_+}$.\\
For $\dot{\rho}>0$ and $\bullet\equiv -$, we have \qquad $C'_{-}= D'_-$\qquad hence\qquad $\gamma|_{C_-}\psim_{t,t'} \gamma'|_{D'_-}$.
\item
For $\dot{\rho}<0$ and $\bullet\equiv +$, we have \qquad  $C'_{+}= D'_-$\qquad hence\qquad $\gamma|_{C_+}\psim_{t,t'} \gamma'|_{D'_-}$.\\
For $\dot{\rho}<0$ and $\bullet\equiv -$, we have \qquad  $C'_{-}= D'_+$\qquad hence\qquad $\gamma|_{C_-}\psim_{t,t'} \gamma'|_{D'_+}$.
\end{itemize}
\endgroup
\end{lemma}
\begin{proof}
We only discuss the case $C_+\subset D_+$ with $\dot{\rho}>0$ (the other cases follow analogously):  
\vspace{4pt}

\noindent
We first observe that either $C_+=[t,c]$ and $C'_+=[t',c']$ or $C_+=[t,c)$ and $C'_+=[t',c')$ holds, for certain $c>t$ and $c'>t'$. Since $C_+\subset D_+$ holds (hence $t\in D_+\subseteq \dom[\ovl{\gamma}]$), there exists an open interval $I\subseteq \dom[\ovl{\gamma}]$ with $c\in I$, such that $\ovl{\gamma}|_I$ is an embedding. 
Assume now that the statement is wrong, i.e., that $C'_+\subset D'_+$ holds (hence $t'\in D'_+\subseteq \dom[\ovl{\gamma}']$). Then, there exists an open interval $I'\subseteq \dom[\ovl{\gamma}']$ with $c'\in I'$, such that $\ovl{\gamma}'|_{I'}$ is an embedding.  
By  Lemma \ref{lemma:BasicAnalyt1}, we can shrink $I$ around $c$ as well as $I'$ around $c'$, such that $\ovl{\gamma}(I)=\ovl{\gamma}'(I')$ holds.\footnote{Observe:  If $C_+$ is of the form $[t,c)$ and $\{t_n\}_{n\in\NN}\subseteq [t,c)$ a sequence with $\{t_n\}_{n\in\NN}\rightarrow c$,  then we necessarily have 
$[t',c')\supseteq \{\ovl{\rho}(t_n)\}_{n\in \NN}\rightarrow c'$ by positivity of $\ovl{\rho}$.} Then, the same arguments as in the proof of Lemma \ref{lem:maxextKomp} show that $\ovl{\rho}$ is defined on $C_+\cup I$;   
and we conclude:
\begingroup
\setlength{\leftmargini}{12pt}
\begin{itemize}
\item[$\triangleright$]	
Assume $C_+=[t,c]$ and $C'_+=[t',c']$. Since $C_+\subset D_+$ and $C_+'\subset D'_+$ holds, there exists $\epsilon>0$ with $[t,c+\epsilon)\subseteq D_+$ and $\ovl{\rho}([t,c+\epsilon))\subseteq D'_+$, which contradicts the definition of $C$.
\item[$\triangleright$]	
Assume $C_+=[t,c)$ and $C'_+=[t',c')$. Then, $[t,c]\subseteq D_+$ and $\ovl{\rho}([t,c])=[t',c']\subseteq D'_+$ holds, which contradicts the definition of $C$.\qedhere
\end{itemize}
\endgroup
\end{proof}
\noindent
We now finally provide some notations that are adapted to the situation in Sect.\ \ref{disgencur}. There, we will be concerned with restrictions of curves to compact and half-open intervals:
\begingroup
\setlength{\leftmargini}{12pt}
\begin{itemize}
\item
 If $D=(i',\tau]$ and $D'=[\tau,i)$ holds, we  write  $\gamma \psim \gamma'$ \deff
\begin{align*}
 \gamma \psim_{\tau,\tau} \gamma' \qquad\:\: \text{or}\qquad\:\: \gamma|_{(j',\tau]} \psim_{\tau,\tau} \gamma'\quad\text{for}\quad i'<j' \qquad\:\: \text{or}\qquad\:\: \gamma \psim_{\tau,\tau} \gamma'|_{[\tau,j)}\quad\text{for}\quad j<i \qquad\:\: \text{holds.}
\end{align*}
It is then clear from the above discussions that only one of these cases can occur, and that the corresponding reals $j$ and $j'$  are uniquely determined. 
\item 
In addition to that, we write $\gamma \psim \gamma'$ \deff one of the following situations hold:
$$
\begin{array}{llllll}
\text{We have}\qquad D=[a,b],\: D'=[a',b'] \quad &\text{as well as}\qquad \hspace{18pt} \gamma &\hspace{-8pt}\psim_{a,a'} \gamma' \quad &\text{or}\qquad\hspace{20pt} \gamma &\hspace{-8pt}\psim_{a,b'} \gamma'.\quad &\footnotemark\\[4pt]
\text{We have}\qquad D=[a,b],\: D'=[a',b') \quad &\text{as well as}\qquad \gamma|_{[a,j)} &\hspace{-8pt}\psim_{a,a'} \gamma' \quad &\text{or}\qquad \gamma|_{(j',b]} &\hspace{-8pt}\psim_{b,a'} \gamma'.\quad &\\[4pt]
\text{We have}\qquad D=[a,b],\: D'=(a',b'] \quad &\text{as well as}\qquad \gamma|_{(j',b]} &\hspace{-8pt}\psim_{b,b'} \gamma' \quad &\text{or}\qquad \gamma|_{[a,j)} &\hspace{-8pt}\psim_{a,b'} \gamma'.\quad &\\[4pt]
\end{array}
$$
\footnotetext{Observe: The first case is equivalent to $\gamma \psim_{b,b'} \gamma'$; and the second case is equivalent to $\gamma \psim_{b,a'} \gamma'$.}
\vspace{-13pt}

\noindent
Again, in each of these situations, only one of the  cases on  the right side can occur. This now follows from Corollary \ref{dfdgttrgf}; whereby uniqueness of the involved reals $j$ and $j'$ is again shown by the above discussions. 

Then, instead of $\gamma\psim \gamma'$, we also write $\gamma\psim_+ \gamma'$ or $\gamma\psim_- \gamma'$ \deff one of the above cases on the left or on the right side holds, respectively.
\end{itemize}
\endgroup
\subsection{Regularity and Stabilizers}
\label{regact}
In this subsection, we collect some 
facts and definitions concerning group actions that we will need in the main text. Let thus  $\wm\colon G\times M\rightarrow M$ denote a fixed left action in the following, i.e., $G$ is a Lie group, $M$ an analytic manifold, and $\wm$ is continuous in $G$.

\subsubsection{Regularity}
\label{dsfsdsregact}
The following definitions are central:
\begin{definition}
\label{kjcxjkcxjkcx}
\noindent

\vspace{-6pt}
\begingroup
\setlength{\leftmargini}{12pt}
\begin{itemize}
\item
A point $x\in M$ is said to be {\bf sated}\:\:\defff\:{}for  $M\backslash\{x\}\ni y\neq z\in M\backslash\{x\}$,  
we cannot have
\begin{align*} 
 \textstyle\lim_n g_n\cdot y=x=\lim_n g_n\cdot z\qquad\quad\text{for a  sequence}\qquad\quad \{g_n\}_{n\in \NN}\subseteq G.
 \end{align*}
 \vspace{-17pt}
\item
A point $x\in M$ is said to be {\bf stable}\:\:\deff\:$\textstyle\lim_n g_n\cdot x=x$ for $\{g_n\}_{n\in \NN}\subseteq G\backslash G_x$, 
 implies that there exist sequences $\{h_n\}_{n\in \NN}, \{h'_n\}_{n\in \NN}\subseteq  G_{[x]}$ such that  
 $\{h_n\cdot g_n\cdot h'_n\}_{n\in \NN}$ admits a convergent subsequence.\footnote{Observe: Since $G_x$ is closed (as $\wm_x$ is continuous), this subsequence necessarily converges to some element in $G_x$.} 
\item
A point $x\in M$ is said to be {\bf regular}\:\:\defff\:{}$x$ is \emph{sated} and \emph{stable}. 
\end{itemize}
\endgroup
\noindent
The left action $\wm$ is said to be {\bf  sated\slash stable\slash regular} \deff each $x\in M$ is \emph{sated\slash stable\slash regular}.
\end{definition}

\begin{remark}
\label{remmmmmi}
\noindent

\vspace{-4pt}
\begingroup
\setlength{\leftmargini}{15pt}
{
\renewcommand{\theenumi}{\sf\arabic{enumi})} 
\renewcommand{\labelenumi}{\theenumi}
\begin{enumerate}
\item
The point $x\in M$ is  \emph{sated\slash stable\slash regular} \deff each $y\in G\cdot x$ is \emph{sated\slash stable\slash regular}.
\vspace{1pt}

In fact, for satedness, the equivalence is immediate; and, for stability, the equivalence follows from $G_{[x]}=g\cdot G_{[x]} \cdot g^{-1}$  for all $g\in G$, as well as from $G_y=g\cdot G_x\cdot g^{-1}$  for all $y\in M$ and $g\in G$ with $y=g\cdot x$. 
\item
Each closed subgroup $H$ of a Lie group $G$ acts via left multiplication regularly on $G$.
\vspace{1pt}

In fact, both satedness and stability are immediate from the implication: 
\begin{align*}
\textstyle\lim_n h_n\cdot g=g'\quad\text{ for }\quad g,g'\in G,\:\{h_n\}_{n\in \NN}\subseteq H\qquad\quad\Longrightarrow\qquad\quad \lim_n h_n= g'\cdot g^{-1}. 
\end{align*}
\vspace{-15pt}
\item
\label{regp3}
$\wm$ is \emph{sated} if there exists a $G$-invariant continuous metric $\met$ on $M$. 
\vspace{1pt}

In fact, then $\lim_n g_n\cdot y =x=\lim_n g_n\cdot z$ implies 
$0=\met(x,x)=\textstyle\lim_n\met(g_n\cdot y, g_n\cdot z)=\met(y,z)$.
\item
\label{regp3dsdsdsds}
Assume that $(M,*)$ is a group such that $*\colon M\times M\rightarrow M$ is continuous in the first argument (i.e., $M\ni x\mapsto x*y\in M$ is continuous for each fixed $y\in M$). Assume furthermore that
\begin{align*}
	\wm(g,x)\equiv g\cdot x=\phi(g)* x\qquad\quad\forall\: g\in G,\: x\in M
\end{align*}
holds for a (necessarily continuous\footnote{Since $\wm$ is assumed to be continuous in $G$, the homomorphism $\phi\colon G\ni g\mapsto\wm(g,e_M)\in M$ is necessarily continuous.}) group homomorphism $\phi\colon G\rightarrow M$. Then, we have:
\begin{align}
\label{dsoidsoidskjdskjds8787ds76ds76d76sd76ds7676ds76sd76dsd1}
	\forall\: x,y\in M,\: \{g_n\}_{n\in \NN}\subseteq G\colon\qquad \textstyle\lim_n g_n\cdot y =x \qquad \textstyle\Longleftrightarrow\qquad \lim_n \phi(g_n)=h:=x* y^{-1}\qquad\\[5pt]\nonumber
	\text{as well as}\hspace{190pt}\\[3pt]
\label{dsoidsoidskjdskjds8787ds76ds76d76sd76ds7676ds76sd76dsd2}
 G_x=\ker[\phi]\qquad\forall\: x\in M\qquad\quad\text{hence}\qquad\quad G_{[x]}=\ker[\phi]\qquad\forall\: x\in M.\hspace{53pt}
\end{align}
We obtain the following statements:
\begingroup
\setlength{\leftmarginii}{14pt}
{
\renewcommand{\theenumii}{.\sf\alph{enumii})} 
\renewcommand{\labelenumii}{\sf\alph{enumii})}
\begin{enumerate}
\item
\label{dslkjkjdskjdskjdskjsdkjkjdskjds87ds87dsa}
$\wm$ is \emph{sated}, because for $x,y,z\in M$ and $\{g_n\}_{n\in \NN}\subseteq G$, we have the implications:
$$
\textstyle\lim_n g_n\cdot y=x=\lim_n g_n\cdot z\qquad\stackrel{\eqref{dsoidsoidskjdskjds8787ds76ds76d76sd76ds7676ds76sd76dsd1}}{\Longrightarrow}\qquad x*y^{-1}=\lim_n\phi(g_n)=x*z^{-1}\qquad\Longrightarrow\qquad y=z.
$$
\item
\label{dslkjkjdskjdskjdskjsdkjkjdskjds87ds87dsb} 
\begingroup
\setlength{\leftmarginiii}{12pt}
\begin{itemize}
\item
$\wm$ is  \emph{stable}  \deff $\wm$ is \emph{stable} at $e_M\in M$, because for $x\in M$ and $\{g_n\}_{n\in \NN}\subseteq G$ we have:
\begin{align*}
\textstyle\lim_n \phi(g_n)* x= x\qquad\textstyle\stackrel{\eqref{dsoidsoidskjdskjds8787ds76ds76d76sd76ds7676ds76sd76dsd1}}{\Longleftrightarrow}\qquad\lim_n &\:\phi(g_n)=e_M\qquad\textstyle\stackrel{\eqref{dsoidsoidskjdskjds8787ds76ds76d76sd76ds7676ds76sd76dsd1}}{\Longleftrightarrow}\qquad \lim_n \phi(g_n)* e_M= e_M\quad\\[6pt] 
&\:\text{with}
\\
 G_{[x]}\stackrel{\eqref{dsoidsoidskjdskjds8787ds76ds76d76sd76ds7676ds76sd76dsd2}}{=}&\ker[\phi]\stackrel{\eqref{dsoidsoidskjdskjds8787ds76ds76d76sd76ds7676ds76sd76dsd2}}{=}G_{[e_M]}.
\end{align*}
\item
$\wm$ is \emph{stable} 
if $\phi\cp s=\id_V$ holds for a continuous map   
$s\colon M\supseteq V:= U\cap \phi(G)\rightarrow G$, with $U$ a neighbourhood of $e_M$.  
\vspace{4pt}

In fact, for $x\in M$ and $\{g_n\}_{n\in\NN}\subseteq G\setminus G_x$, we have the implication:
\begin{align}
\label{dslldsoidsoidsdspodslkdslkdslklkds86s76ds76sdds}
\textstyle\lim_n\underbrace{\phi(g_n)* x}_{\displaystyle \wm(g_n, x)}=x\qquad\quad\stackrel{\eqref{dsoidsoidskjdskjds8787ds76ds76d76sd76ds7676ds76sd76dsd1}}{\Longrightarrow}\qquad\quad \lim_n\phi(g_n)=e_M.
\end{align}
In particular, there exists $m\in \NN$ with $\{\phi(g_n)\}_{n\geq m}\subseteq U$. Then, $\phi\cp s=\id_V$ implies
$$
h'_n:=g_n^{-1}\cdot s(\phi(g_n))\in \ker[\phi]\stackrel{\eqref{dsoidsoidskjdskjds8787ds76ds76d76sd76ds7676ds76sd76dsd2}}{=}G_{[x]}\qquad\quad\forall\: n\geq m, 
$$ 
and we obtain from \eqref{dslldsoidsoidsdspodslkdslkdslklkds86s76ds76sdds} (first step)  as well as from continuity of $s$ (second step):
\begin{align*}
	\textstyle\lim_n \phi(g_n)* x=x\qquad\quad&\Longrightarrow \qquad\quad \textstyle\lim_n\phi(g_n)= e_M \\
	&\textstyle\Longrightarrow \qquad\quad  \lim_n \underbrace{s(\phi(g_n))}_{\displaystyle g_n\cdot h_n'}= s(e_M).
\end{align*} 
\vspace{-20pt}
\end{itemize}
\endgroup
\end{enumerate}}
\endgroup
\item
\label{regp1}
The point $x\in M$ is \emph{regular} if $\wm_x$ is proper.
\vspace{-6pt}
\begin{proof}
According to Appendix \ref{appA799}, properness of $\wm_x$ is equivalent to that the following implication holds:
\begin{align}
\label{lkdslkdslklkdslkdslklkdsdsds09ds9009dsdsdsdsdsdsds}
\textstyle \lim_n g_n\cdot x=y\in M\:\:\:\text{for}\:\:\: \{g_n\}_{n\in \NN}\subseteq G\qquad\Longrightarrow\qquad \{g_n\}_{n\in \NN}\:\:\:\text{admits convergent subsequence}.
\end{align}
Then, $x$ is evidently stable; and, if $\lim_n g_n\cdot y=x=\lim_n g_n\cdot z$ holds for $y,z\in M$, then $\lim_n g_n$ can be assumed to exist (just by passing to a subsequence), with what $y=z$ follows. 
\end{proof}
\vspace{-6pt}
Thus, pointwise proper ($\wm_x$ is proper for each $x\in M$), hence proper actions are regular.  
Anyhow, in general, pointwise properness is a stronger condition than  regularity as, e.g., an action cannot be pointwise proper if $G_x$ is non-compact for some $x\in M$. 
For instance (see also Example \ref{dsassayyasa}.\ref{dsassayyasa3}): 
\vspace{2pt}

Let $G$ be a Lie group with (possibly non-compact) closed normal subgroup $H$; and let $G$ act on $M\equiv G\slash H$ in the canonical way, i.e.\ via $\wm\colon (g,x)\mapsto [g\cdot x]$. Since $H$ is normal, $M$ is a Lie group, and the projection $\pri\colon G\rightarrow M$ is a Lie group homomorphism. Moreover, by general theory, there exists a local section $s\colon M\supseteq V\equiv U\rightarrow G$ with $\pri\cp s=\id_U$ for $U$ an open neighbourhood of $[e]$. Thus, $\wm$ is regular by Part \ref{regp3dsdsdsds}. Then, we obtain regular actions that admit non-compact normal stabilizers if we choose, e.g., ($n>1$): 
\begingroup
\setlength{\leftmarginii}{12pt}
\begin{itemize}
\item
$G\equiv \RR^n$ and $H\equiv\ZZ^n$, hence $M\equiv \mathbb{T}^n$.   
\item
$G\equiv\RR^n$ and $H\subseteq \RR^n$ some $m$-dimensional linear subspace for $m>0$, hence $M\equiv  \RR^{n-m}$.\hspace*{\fill}$\ddagger$
\end{itemize}
\endgroup
\end{enumerate}}
\endgroup
\end{remark}
\begin{example}
\label{dsassayyasa}
\noindent

\vspace{-4pt}
\begingroup
\setlength{\leftmargini}{15pt}
{
\renewcommand{\theenumi}{\sf\arabic{enumi})} 
\renewcommand{\labelenumi}{\theenumi}
\begin{enumerate}
\item
\label{dsassayyasa2}
	The origin is stable (as $G\setminus G_0=\emptyset$ by $G_0=G=\RR_{>0}$) but  not sated w.r.t.\ the multiplicative action of $\RR_{>0}$ on $\RR^n$. But, each point in $\RR^n\backslash \{0\}$ is regular. 
\item
\label{dsassayyasa3}
	For $\lambda\in \RR$ fixed, we consider the diagonal action \hspace*{\fill} (\:$\mathbb{T}^2= \UE\times \UE\subseteq \CC^2$)
	\begin{align*}	
	\wm\colon \RR\times \mathbb{T}^2\rightarrow \mathbb{T}^2,\qquad (\tau,(x_1,x_2))\mapsto (\e^{2\pi \tau\cdot \I}\cdot x_1, \e^{2\pi \tau\lambda\cdot \I}\cdot x_2)
	\end{align*}
	of $(\RR,+)$ on the 2-Torus $(\mathbb{T}^2,*)$, with 
	$$(x_1,x_2)*(y_1,y_2)=(x_1\cdot y_1,x_2\cdot y_2)\qquad\quad\forall\:(x_1,x_2),(y_1,y_2)\in \CC^2.$$ 
	Then,  $\wm(\tau,\x)\equiv \tau\cdot \x= \phi(\tau)*\x$ holds for all $\tau\in \RR$ and $\x\equiv (x_1,x_2)\in \mathbb{T}^2$ if we set
\begin{align*}
	\phi\colon \RR\rightarrow \mathbb{T}^2,\qquad \tau\mapsto(\e^{2\pi \tau\cdot \I},\e^{2\pi \tau\lambda\cdot \I}).
\end{align*}  	
Hence,  
 $\wm$ is sated by Remark \ref{remmmmmi}.\ref{dslkjkjdskjdskjdskjsdkjkjdskjds87ds87dsa}; and $\wm$ is stable \deff $\lambda$ is rational:
\begingroup
\setlength{\leftmarginii}{10pt}
\begin{itemize}
\item
Assume that $\lambda$ is rational. Then,  $\ker[\phi]=\ZZ\cdot m$ holds, for $m\in \NN_{>0}$ minimal with $m\cdot \lambda \in \ZZ$. Assume now  $\lim_n \tau_n\cdot \x=\x$ for some $\{\tau_n\}_{n\in \NN}\subseteq \RR$ and $\x\in \mathbb{T}^2$. Then, for each $n\in \NN$, there exists some $h_n\in G_{[\x]}=\ker[\phi]=\ZZ\cdot m$ with $h_n + \tau_n \in [0, m]$. Then, $\x$ is stable, because $\{h_n+ \tau_n\}_{n\in \NN}$ admits a convergent subsequence as $[0, m]$ is compact. 
\item
Assume that $\lambda$ is irrational. Then, $(\e^{2\pi \tau\cdot \I},\e^{2\pi \tau\lambda\cdot \I})=(1,1)\equiv e$ for $\tau\in \RR$     implies $\tau=0$, hence  $\ker[\phi]=\{0\}$. Now, there exists $\mu\in \RR_{\neq 0}$ such that $1,\mu,\mu\cdot\lambda$ are $\QQ$-independent\footnote{Elsewise, for each $\mu\in \RR$, there exist $q,q'\in \QQ$ with $\mu=\frac{q'}{1+\lambda q}$, which contradicts that $\RR$ is uncountable.};   
so that $\{\x^n\}_{n\in \ZZ}\subseteq \mathbb{T}^2$ with $\x:=(\e^{2\pi \mu\cdot\I},\e^{2\pi \mu\lambda \cdot\I})$ is dense in $\mathbb{T}^2$ by Kronecker's theorem. Then, $U\cap \{\x^n\}_{n\in \ZZ}$ is infinite for each neighbourhood $U\subseteq \mathbb{T}^2$ of $e$; so that there exists $\iota\colon \NN\rightarrow \ZZ$ injective with $\lim_n \x^{\iota(n)} = e$ and $|\iota(n+1)|>|\iota(n)|$ for all $n\in \NN$. For each $n\in \NN$  we have $\x^{\iota(n)}=\phi(\tau_n)$ with $\tau_n:=\iota(n)\cdot\mu\in \RR$, hence $\lim_n \phi(\tau_n)* e=\lim_n \x^{\iota(n)}=e$. But $\{\tau_n\}_{n\in\NN}$ cannot admit a convergent subsequence, because $|\tau_{n}-\tau_{n+m}|\geq m\cdot |\mu|$ with $|\mu|>0$ holds for all $n,m\in \NN$. Hence,  
$e$ cannot be stable as we have $G_{[e]}=\ker[\phi]=\{0\}$. \hspace*{\fill}$\ddagger$
\end{itemize}
\endgroup
\end{enumerate}}
\endgroup
\end{example} 

\subsubsection{Stabilizers}
In this brief subsection, we collect the properties of stabilizers of curves that are relevant for our discussions in the main text.
\begin{definition}
\label{defStabbi}
For a curve $\gamma\colon D\rightarrow M$, 
we define its stabilizer subgroup by
\begin{align*}
	G_\gamma:=\{g\in G\:|\: g\cdot  \gamma =\gamma\}=\textstyle\bigcap_{t\in D}G_{\gamma(t)}. 
\end{align*}
Then, $G_\gamma$ is a Lie subgroup of $G$ as closed in $G$, and we denote its Lie algebra by $\mg_\gamma$.
\end{definition}
\begin{lemma}
\label{lemma:stabi}
Let $\wm$ be analytic in $M$, and $\gamma\colon D\rightarrow M$ an   analytic curve. Then, $G_\gamma=G_{\gamma|_{B}}$ holds for each interval $B\subseteq D$.
\end{lemma}
\begin{proof}
Since $\gamma$ and $g\cdot \gamma$ are analytic, we have
	\begin{equation*}	
		g\in G_{\gamma|_{B}}\qquad\quad\stackrel{\text{def.}}{\Longrightarrow}\qquad\quad (g\cdot \gamma)|_{B}=\gamma|_{B} 
		\qquad\:\:\stackrel{\text{Lemma } \ref{corgleich}}{\Longrightarrow}\qquad\:\: g\cdot \gamma =\gamma.\qedhere
	\end{equation*}
\end{proof}
\begin{corollary}
\label{fdsfds7}
Let $\wm$ be analytic in $M$, and $\gamma\colon D\rightarrow M$, $\gamma'\colon D'\rightarrow M$ analytic curves. Assume that $\gamma|_B=\gamma'\cp\rho$ holds for a map $\rho\colon D\supseteq B\rightarrow B'\subseteq  D'$, with $B$ an interval. Then, $G_{\gamma'}\subseteq G_\gamma$ holds; as well as $G_\gamma=G_{\gamma'}$ if $\rho$ is a homeomorphism.
\end{corollary}
\begin{proof}
	For each $g\in G_{\gamma'}$ we have 
	\begin{equation*}
		g\cdot \gamma|_B=g\cdot (\gamma'\cp \rho)=\gamma'\cp \rho=\gamma|_B\qquad\stackrel{\text{Lemma } \ref{lemma:stabi}}{\Longrightarrow}\qquad g\in G_\gamma,
	\end{equation*}		 
	hence $G_{\gamma'}\subseteq G_\gamma$. Moreover, if $\rho$ is a homeomorphism, then $B'$ is an interval with  $\gamma'|_{B'}=\gamma\cp \rho^{-1}$. The same argument as above then shows $G_\gamma\subseteq G_{\gamma'}$, hence $G_\gamma=G_{\gamma'}$.
\end{proof}
\begin{corollary}
\label{fsdfsfdfs}
Let $\wm$ be analytic in $M$. 
If $\gamma\colon D\rightarrow M$ is an analytic immersion, then 
\begin{align}
\label{stabiconji}
	g\cdot \gamma\cpsim \gamma\qquad\quad\Longrightarrow\qquad\quad g^{-1}\cdot q\cdot g\in G_\gamma\qquad\forall\: q\in G_\gamma.
\end{align}
\end{corollary}
\begin{proof}
Let $J,J'\subseteq I$ be open intervals with $g\cdot \gamma(J)=\gamma(J')$; and let $q\in G_\gamma$ be fixed. Since to each $t\in J$, there exists some $t'\in J$ with $g\cdot \gamma(t)=\gamma(t')$, we have
\begin{align*}
q\cdot (g\cdot \gamma|_J)= g\cdot \gamma|_J
\qquad\quad&\stackrel{\phantom{\text{Lemma} \ref{lemma:stabi}}}{\Longrightarrow}\qquad\quad 
(g^{-1}\cdot q\cdot g)\cdot \gamma|_J=\gamma|_J\\
&\stackrel{\text{Lemma} \ref{lemma:stabi}}{\Longrightarrow}\qquad\quad\hspace{0.7pt} g^{-1}\cdot q\cdot g\in G_\gamma.\qedhere
\end{align*}
\end{proof}
\begin{lemma}
\label{stabbiii}
Assume that $\wm$ is sated and analytic in $M$. Let $\gamma\colon K\rightarrow M$ be an analytic embedding, and let $L=[\tau,\ell]\subseteq K\equiv[\tau,k]$ with $\tau<\ell$. Then, the following implication holds:  
\begin{align}
 g\cdot \gamma(L)\subseteq \gamma(K)\quad\:\text{for}\quad\: g\in G_{\gamma(\tau)} \qquad\quad\Longrightarrow\qquad\quad g\in G_{\gamma}.
\end{align}
\end{lemma}
\begin{proof}
Since $\gamma$ (hence $g\cdot \gamma$) is an embedding with $g\cdot\gamma(\tau)=\gamma(\tau)$, the condition $g\cdot \gamma(L)\subseteq \gamma(K)$ implies that
\begin{align*}
	g\cdot \gamma(L)=\gamma(L')\quad \text{holds for some} \quad L'=[\tau,\ell']\subseteq K\quad\text{with}\quad\tau<\ell'.
\end{align*}
Replacing $g$ by $g^{-1}$ if necessary, we can assume that  $L'\subseteq L$ holds. We choose $\rho\colon L\rightarrow L'$ as in Lemma \ref{lemma:BasicAnalyt2}, i.e., $\rho$ is an analytic diffeomorphism with $(g\cdot \gamma)|_L=\gamma|_{L'}\cp \rho$;  specifically, $\rho=(\gamma|^{\im[\gamma]})^{-1}\cp(g\cdot \gamma|_{L'})$. 
Then, $\rho$ is strictly increasing (positive) as $\rho(\tau)=\tau$ holds. The claim now follows if we show $\rho=\id_L$ (i.e., $L'=L$ and $g\cdot \gamma|_L=\gamma|_L$), as then $g\in G_{\gamma|_L}=G_\gamma$ holds by Lemma \ref{lemma:stabi}: 
\begingroup
\setlength{\leftmargini}{12pt}
\begin{itemize}
\item[$\triangleright$]
We have $\rho(t)\geq t$ for all $t\in L$; hence, $L=L'$ as  $\rho(\ell)=\ell'\leq \ell$ holds by $L'\subseteq L$ and positivity of $\rho$. 

In fact, assume that $\rho(t)<t$ holds for some $t\in (\tau,\ell]$. We fix $\rho(t)<s<t$, and obtain  $\rho(s)<\rho(t)<s<t$ as 
 $\rho$ is strictly increasing.    
	Applying $\rho$ successively, we obtain strictly decreasing sequences $\{\rho^n(s)\}_{n\in \NN}, \{\rho^n(t)\}_{n\in \NN}\subseteq (\tau,\ell]$ with $\rho^{n+1}(t)<\rho^n(s)<\rho^n(t)$ for all $n\in \NN$.  
	Let $v\in [\tau,\ell]$ denote their common limit. We obtain  
	\begin{align*}
	\textstyle\lim_n g^n\cdot\gamma(t)=\lim_n \gamma(\rho^n(t))= \gamma(v)=\textstyle\lim_n \gamma(\rho^n(s))=\textstyle\lim_n g^n\cdot\gamma(s)  
	\end{align*}
	which contradicts that $\gamma(v)$ is sated, as  $\gamma(t)\neq \gamma(s)$ holds by injectivity of $\gamma$ and $t\neq s$.
\item[$\triangleright$]
We have $\rho(t)=t$ for all $t\in L$, i.e.\ $\rho=\id_L$.

	In fact, assume that $\rho(t)>t$ holds for some $t\in (\tau,\ell)$. We fix $t<s<\rho(t)$ and obtain $t<s<\rho(t)<\rho(s)$ as $\rho$ is strictly increasing; whereby $\rho(t),\rho(s)\in L'=L$ holds by the previous point.  
		Applying $\rho$ successively, we obtain strictly increasing sequences $\{\rho^n(s)\}_{n\in \NN}, \{\rho^n(t)\}_{n\in \NN}\subseteq (\tau,\ell]$ with $\rho^n(t)<\rho^n(s)<\rho^{n+1}(t)$ for all $n\in \NN$. The same argument as in the previous point, yields a contradiction to satedness of $\wm$.\qedhere
\end{itemize}
\endgroup
\end{proof}

\subsection{Exponential Curves}
\label{mncmxcxooo}
In this subsection, we discuss the basic properties of the following class of analytic curves:
\begin{definition}
\label{dffdfdfdztzzuzu}
Let $\wm$ be analytic in $G$. An analytic curve $\gamma\colon D\rightarrow M$ is said to be {\bf exponential} \deff $\gamma=\gag\cp \rho$ holds for some $x\in M$, $\g\in \mg$, and an analytic map  $\rho\colon D\rightarrow \RR$. In this case, we say that $\gamma$ is exponential w.r.t.\ $(x,\g)$; or, alternatively, exponential w.r.t\ $(x,\g,\rho)$ if it helps to clarify the argumentation.  
\end{definition}

\subsubsection{Basic Properties}
In the following, let $\wm$ be analytic in $G$. 
The next remark covers the case of constant curves:
\begin{remark}[Constant Curves]
\label{remldskjdkjds}
Let $\gamma\colon D\ni t\mapsto y\in M$ be constant.
\begingroup
\setlength{\leftmargini}{15pt}
{
\renewcommand{\theenumi}{\sf\alph{enumi})} 
\renewcommand{\labelenumi}{\theenumi}
\begin{enumerate}
\item
If $\gamma$ is exponential w.r.t.\ $(x,\g,\rho)$, then one of the following two situations holds: 
\begingroup
\setlength{\leftmarginii}{12pt}
\begin{itemize}
\item
	$\g\in \mg_x$: \hspace{5.5pt}\quad $x=\gag\cp\rho=\gamma=y$, and  $\rho\colon D\rightarrow \RR$ an arbitrary analytic map.
	\vspace{1pt}
\item
	$\g\in \mg\backslash \mg_x$:\quad Lemma \ref{lemma:expeig} and the product rule yield $\dot\rho=0$. Hence,  $\rho=\tau\in \RR$ is constant,  with
	$$y=\gamma=\gag(\tau)=\exp(\tau\cdot \g)\cdot x\qquad\quad \Longrightarrow\qquad\quad x\in \exp(\RR\cdot \g)\cdot y.$$
\end{itemize}
\endgroup   
\item
Conversely, $\gamma$ is exponential w.r.t.\ $(x,\g,\rho)$ if we choose:  
\begingroup
\setlength{\leftmarginii}{12pt}
\begin{itemize}
\item
 $x=y$,\hspace{55.2pt} $\g\in \mg_x$,\hspace{17pt} $\rho\colon D\rightarrow \RR$  an arbitrary analytic map. 
 \vspace{1pt}
\item
$x\in \exp(\RR\cdot \g)\cdot y$,\quad  
$\g\in \mg\backslash \mg_x$,\quad $\rho\colon D\ni t\mapsto -\tau\in \RR$ constant, for $\tau\in \RR$ with $x=\exp(\tau\cdot \g)\cdot y$. 
\hspace*{\fill}$\ddagger$ 
\end{itemize}
\endgroup  
\end{enumerate}}
\endgroup 
\end{remark}
\noindent
Next, Lemma \ref{lemma:expeig} and Lemma \ref{corgleich} yield the following statement:
\begin{lemma}
\label{podspodsapopo}
Let $\wm$ be analytic in $G$, $\gamma=\gag\cp \rho$ as in Definition \ref{dffdfdfdztzzuzu}, and $\wt{D}\subseteq D$ an interval. Then, $\gamma|_{\wt{D}}$ is an analytic immersion \deff $\rho|_{\wt{D}}$ is an analytic diffeomorphism and $\g\in \mg\backslash \mg_x$ holds.
\end{lemma} 
\begin{proof}
	The proof is elementary, and can be found in Appendix \ref{appA8}.
\end{proof}
\noindent
We furthermore have the following statement:\footnote{Confer also Lemma 5.6.5) in \cite{MAX} for the immersive pointwise proper case.}
\begin{lemma}
\label{vd}
Let $\wm$ be analytic in $G$; and let $\gamma\colon D \rightarrow M$ be  analytic such that $\gamma(t)$ is sated for each $t\in D$. Let $\kappa\colon D\supseteq \dom[\kappa]\rightarrow \RR$ be an analytic map ($\dom[\kappa]$ an interval), with maximal analytic extension $\ovl{\kappa}$.  Then, the following implication holds:
\begin{align*}
	\gamma|_{\dom[\kappa]}= \gag\cp \kappa\qquad\quad\Longrightarrow \qquad\quad \gamma=\gag\cp\ovl{\kappa}|_{D}.
\end{align*}
\end{lemma}
\begin{proof}
Confer Appendix \ref{appA9}.
\end{proof}

\begin{remark}
We want to give an idea what might happen if $\wm$ is not sated: 
\vspace{4pt}

\noindent
Let $G\equiv\RR_{>0}$, $M\equiv \RR^n$, and $\wm\colon G\times M\ni (\lambda,x)\mapsto \lambda\cdot x\in M$ the multiplicative action. Then, $x\in M$ is sated \deff $x\neq 0$ holds (see Example \ref{dsassayyasa}.\ref{dsassayyasa2}), and the exponential map of $G$ is given by 
$$\exp\colon \mg\cong\RR\ni \lambda\mapsto \e^{\lambda}\in G=\RR_{>0} 
\quad\text{so that}\quad \gamma_\lambda^x\colon \RR \ni t\mapsto \e^{t\cdot\lambda}\cdot x\in M\quad\text{holds for all}\quad \lambda\in \mg,\: x \in M.  
$$  
For $x\neq 0$, $\lambda> 0$ and $\gamma\colon \RR\ni t\mapsto t\cdot x\in M$ we have
\begin{align*}
	\im[\gamma]=\RR\cdot x&\supset \RR_{>0}\cdot x=\im[\gamma_\lambda^x]\\[3pt]
	\gamma|_{(0,\infty)}&=\gamma_\lambda^x\cp (1\slash \lambda\cdot\ln).
\end{align*}
By the second line, we have $\gamma|_{(0,\infty)}=\gamma_\lambda^x\cp\kappa$ for the analytic diffeomorphism $\kappa:=1\slash \lambda\cdot\ln\colon (0,\infty)\rightarrow\RR$, but the first line shows that $\gamma=\gamma_\lambda^x\cp\ovl{\kappa}$ cannot hold (additionally observe  $\ovl{\kappa}=\kappa$).\footnote{Simply put, $\im[\gamma]\nsubseteq \im[\gamma_\lambda^x]$ holds, because $\lim_{t\rightarrow -\infty} \e^{t\cdot\lambda}\cdot x=0\in M$ exists  (which is possible,  because $0$ is not sated).} 
\hspace*{\fill}$\ddagger$
\end{remark}
\subsubsection{Uniqueness}
\label{sdjvsdghddssad}
We finally want to clarify w.r.t.\ which $(x,\g)$ a given analytic curve can be exponential. 
Since the constant case is already covered by Remark \ref{remldskjdkjds}, we can restrict our discussions to the non-constant case in the following. We now show the following statement:
\begin{proposition}
\label{classsiilie}
Let $\wm$ be sated and separately analytic; and let $\gamma\colon D\rightarrow M$ be non-constant analytic. Then, $\gamma$ is exponential w.r.t.\ some $(x,\g)$ \deff $\gamma$ is exponential w.r.t.\ $(y,\q)$ for all $y\in \exp(\RR\cdot \g)\cdot x$ and $\q\in \RR_{\neq 0}\cdot \g+\mg_\gamma$.
\end{proposition}
\noindent
For the proof of Proposition \ref{classsiilie}, we need the following observations (the proofs of the technical statements used  are provided at the end of this subsection): 
\vspace{6pt}

\noindent
Let $\wm$ be separately analytic; and let $\gamma\colon D\rightarrow M$ be non-constant analytic. Then, 
there exists an open interval $J\subseteq D$, such that $\gamma|_J$ is an (analytic) embedding (Corollary \ref{dfdsasasasassa}) hence immersive. Assume now that
\begin{align}
\label{ds09ds09ds09oidsoidsoidskjdslkdslkdsdsdsdsdsds}
 \gamma=\gag\cp\rho
 \qquad\text{holds for an analytic map}\qquad \rho\colon D\rightarrow \RR. 
\end{align} 
\vspace{-15pt}

\noindent
We observe the following:
\begingroup
\setlength{\leftmargini}{15pt}
{
\renewcommand{\theenumi}{\alph{enumi})} 
\renewcommand{\labelenumi}{\theenumi}
\begin{enumerate}
\item
\label{opsdopdsopsd000}
$\g\in \mg\backslash\mg_x$ holds (hence $\gag$ is immersive by Lemma \ref{lemma:expeig}), since  elsewise $\gamma$ is constant by \eqref{ds09ds09ds09oidsoidsoidskjdslkdslkdsdsdsdsdsds}. 
\item
\label{opsdopdsopsd001}
$\rho|_J\colon J\rightarrow \rho(J)=:J'$ is a diffeomorphism, as immersive by \eqref{ds09ds09ds09oidsoidsoidskjdslkdslkdsdsdsdsdsds} and the chain rule. 

Notably, 
\eqref{pofdpofdpo} shows that $\rho|_J$ is uniquely determined by its  value $\rho(t)$ for any fixed $t\in J$.  
Since $\gamma|_J$ is injective, Lemma \ref{lemma:expeig} and Lemma \ref{corgleich} imply that for $\upsilon\colon  B\rightarrow \RR$ analytic with $B\subseteq J$ an interval, we have
\begin{align*}
	\gamma|_B=\gag\cp \upsilon \qquad\quad\Longleftrightarrow\qquad\quad \ovl{\upsilon}=\Delta + \ovl{\rho}\quad\:\text{for some}\quad\: \Delta\in \ppi(x,\g).
\end{align*} 
\vspace{-15pt}
\item
\label{opsdopdsopsd11}
Corollary \ref{fdsfds7} shows $\mg_{\gag}=\mg_\gamma$, since $\rho|_J$ is a homeomorphism.
\item
\label{opsdopdsopsd13}
Assume that $\gamma=\gagg\big|_D$ holds for $y\in M$ and $\q= \lambda\cdot \g + \ccc$ with $\lambda\neq 0$ and $\ccc\in \mg_\gag$. Then, Corollary \ref{fdpofpofdpodf} together with \ref{opsdopdsopsd000} shows 
\begin{align*}
	\gagg=\gag\cp \ovl{\rho}\qquad\text{with}\qquad\ovl{\rho}\colon \RR\ni  t\mapsto \ovl{\rho}(0) + \lambda\cdot t\in \RR
	\qquad\text{hence}\qquad \im[\gagg]=\im[\gag].
\end{align*} 
Moreover, if $\wm$ is sated, then     
 Lemma \ref{ldldsldldsalldsakdshfdsf} shows that $\q\in \RR_{\neq 0}\cdot \g + \mg_\gag$ is already implied by $\gamma=\gagg\big|_D$ (observe $\q\in \mg\setminus\mg_y$ as $\gamma$ is assumed to be non-constant). 
\end{enumerate}}
\endgroup
\begin{proof}[Proof of Proposition \ref{classsiilie}]
Let $J\subseteq D$ be an open Interval such that $\gamma|_J$ is an  embedding.
\begingroup
\setlength{\leftmargini}{12pt}
{
\begin{itemize}
\item
\label{opsdopdsopsd2}	
If $\gamma$ is exponential w.r.t.\ $(x,\g,\rho)$ and $(y,\q,\rho')$, then we have
\vspace{-11pt}
\begin{align*}
\gag\cp\rho|_J=\gamma=\gagg\cp\rho'|_J\qquad\quad&\stackrel{\rm \ref{opsdopdsopsd001},\ref{opsdopdsopsd11}}{\Longrightarrow}\qquad\quad \gagg\big|_{\rho'(J)}=\gag\cp\overbrace{(\rho\cp \rho'^{-1})}^{\displaystyle=:\kappa}\quad\:\:\wedge\quad\:\:  \mg_\gag=\mg_\gamma\\
&\hspace{0.5pt}\stackrel{\rm\ref{opsdopdsopsd13}}{\Longrightarrow}\qquad\quad  \q\in \RR_{\neq 0}\cdot \g\: +\:\mg_\gag\:=\: \RR_{\neq 0}\cdot \g\: +\:\mg_\gamma\\
&\hspace{132pt}\text{and}\\[1pt]
&\phantom{\Longrightarrow}\qquad\qquad\quad\:\: y\in \im[\gagg]=\im[\gag]=\exp(\RR\cdot \g)\cdot x
\end{align*}
\vspace{-25pt}

\noindent
with $\kappa$ an analytic diffeomorphism. 
\item
Assume that $\gamma$ is exponential w.r.t.\ $(x,\g,\rho)$, hence $\mg_\gag\stackrel{\rm\ref{opsdopdsopsd11}}{=}\mg_\gamma$. Then, Corollary \ref{dsosdosdpods} shows
\begin{align*}
	\forall\:\tau\in \RR,\: \lambda\in \RR_{\neq 0},\: \ccc\in \mg_\gamma 
	\colon\qquad \gamma_{\lambda\cdot \g +\ccc}^{\gamma(\tau)}=\gag\cp\kappa[\tau,\lambda]\quad\:\:&\text{with}\quad\:\: \kappa[\tau,\lambda]\colon \RR\ni t\mapsto \tau+\lambda\cdot t\in \RR\qquad\qquad\\[1pt]
	&\hspace{-2pt}\text{hence}\\
	\gamma=\gag\:\cp\:&\rho=\gamma_{\lambda\cdot \g +\ccc}^{\gamma(\tau)}\cp(\kappa[\tau,\lambda]^{-1}\cp \rho)
\end{align*}
\vspace{-20pt}

\noindent
with $\kappa[\tau,\lambda]^{-1}\cp \rho$ an analytic diffeomorphism.  
\qedhere
\end{itemize}}
\endgroup
\end{proof}
\noindent
We now finally provide the statements used:
\begin{lemma}
\label{kljklvjkvjklvjklxcv}
Let $\wm$ be separately analytic; and let $\gamma\equiv \gag$ for $x\in M$ and $\g\in \mg$. Then, 
\begin{align*}
	 \exp(t\cdot (\lambda\cdot \g + \ccc))\cdot \gamma(\tau) = \underbrace{\exp(t\cdot \lambda\cdot \g)\cdot\gamma(\tau)}_{\displaystyle\gamma(\tau + t\cdot \lambda)}\qquad\quad  \forall\: t,\tau,\lambda\in \RR,\:\: \ccc\in \mg_\gamma.
\end{align*}  
\vspace{-17pt}
\end{lemma}
\begin{proof}
Confer Appendix \ref{appA10}.
\end{proof}
\begin{corollary}
\label{dsosdosdpods}
Let $\wm$ be separately analytic; and let $\gamma\equiv \gag$ with $x\in M$ and $\g\in \mg$. Then, 
\begin{align*}
	\gamma^{\gamma(\tau)}_{\lambda\cdot \g+\ccc}(t)= \gamma(\tau + \Delta + \lambda\cdot t)\qquad\quad\forall\: t,\tau,\lambda\in \RR,\:\: \Delta\in \ppi(x,\g),\:\: \ccc\in \mg_{\gamma}.
\end{align*}
\end{corollary}
\begin{proof}
This is clear from Lemma \ref{kljklvjkvjklvjklxcv} and the definition of $\ppi(x,\g)$. 
\end{proof}
\begin{corollary}
\label{fdpofpofdpodf}
Let $\wm$ be separately analytic; and let $\gamma\equiv \gag$ with  $x\in M$ and $\g\in \mg\setminus \mg_x$.     
Assume that 
\begin{align*}
	\gamma^{y}_{\lambda\cdot \g+\ccc}\big|_D=\gamma\cp\rho\quad\text{holds for}\quad y\in M,\:\:\lambda\neq 0,\:\: \ccc\in \mg_{\gamma}\quad\text{and}\quad \rho\colon D\rightarrow \RR\quad\text{analytic}. 
\end{align*} 
Then, 
we have\:
 	$\gamma^{y}_{\lambda\cdot \g+\ccc}=\gamma\cp \ovl{\rho}$\: with\:
	 $\ovl{\rho}\colon \RR\ni  t\mapsto \ovl{\rho}(0) + \lambda\cdot t\in \RR$. 
\end{corollary}
\begin{proof}
	For $s\in D$ fixed, we have 
	\begin{align*}
		\underbrace{\exp(s\cdot (\lambda\cdot \g+\ccc))\cdot y}_{\displaystyle\gamma^{y}_{\lambda\cdot \g+\ccc}(s)}\he=\he\underbrace{\exp(\rho(s)\cdot \g)\cdot x}_{\displaystyle\gag(\rho(s))}\qquad&\stackrel{\phantom{\text{Corollary }\ref{dsosdosdpods}}}{\Longrightarrow}\qquad y= \exp(-s\cdot (\lambda\cdot \g+\ccc))\cdot\underbrace{\exp(\rho(s)\cdot \g)\cdot x}_{\displaystyle\gag(\rho(s))}\\[-13pt]
		&\stackrel{\hspace{2pt}\text{Lemma }\ref{kljklvjkvjklvjklxcv}}{\Longrightarrow}\qquad y= \exp(\overbrace{(\rho(s)-s\cdot \lambda)}^{\displaystyle=:\tau}\he\cdot\he \g)\cdot x=\gag(\tau)\equiv\gamma(\tau)\\[2pt]
		&\stackrel{\text{Corollary }\ref{dsosdosdpods}}{\Longrightarrow}\qquad \gamma^{y}_{\lambda\cdot \g+\ccc}=\gamma^{\gamma(\tau)}_{\lambda\cdot \g+\ccc}=\gamma\cp\kappa
	\end{align*}
	for  
$\kappa\colon \RR\ni  t\mapsto \tau + \lambda\cdot t\in \RR$. Since $\kappa(s)=\rho(s)$ holds by definition, we obtain  
	\begin{align*}
   \begin{array}{@{}lr@{}}
       \gamma^y_{\lambda\cdot \g+\ccc}\big|_D \psim_{s,\rho(s)} \gamma|_{\rho(D)}\quad\:\text{w.r.t.}\quad\: \kappa|_D\\
	\gamma^y_{\lambda\cdot \g+\ccc}\big|_D \psim_{s,\rho(s)} \gamma|_{\rho(D)}\quad\:\text{w.r.t.}\quad\: \rho
        \end{array}\bigg\} 
        \qquad\quad\stackrel{\eqref{pofdpofdpo}}{\Longrightarrow}\qquad\quad\kappa|_{D}=\rho\qquad\quad\Longrightarrow\qquad\quad \kappa=\ovl{\rho}
\end{align*}
as both $\gamma$ and $\gamma^y_{\lambda\cdot \g+\ccc}$ are  immersive. 
\end{proof}
\begin{lemma}
\label{ldldsldldsalldsakdshfdsf}
Let $\wm$ be sated and separately analytic; and let $x\in M$ and $\g\in \mg\setminus\mg_x$ be such that
\begin{align*}
	\gagg\big|_D=\gag\cp\rho\quad\text{holds for}\quad y\in M,\:\: \q\in \mg\backslash\mg_y\quad\text{and}\quad \rho\colon D\rightarrow \RR\quad\text{analytic}.
\end{align*} 
Then, $\q=\lambda\cdot \g + \ccc$ holds for unique $\lambda\neq 0$ and  $\ccc\in \mg_\gag$.\footnote{Hence, we have $\gagg=\gag\cp \ovl{\rho}$ with $\ovl{\rho}\colon \RR\ni t\mapsto \ovl{\rho}(0)+ \lambda\cdot t\in \RR$, by Corollary \ref{fdpofpofdpodf}).}
\end{lemma}
\begin{proof}
The claim follows from Lemma \ref{stabbiii} and Lemma \ref{vd}, cf.\ Appendix \ref{appA11}.
\end{proof}

\section{The Classification}
\label{skjsdghfsd}
In this section, we prove our classification result Theorem \ref{classi}, stating that an analytic curve is either exponential or free if $\wm$ is regular and separately analytic. Then, under the milder assumption that $\wm$ is sated and analytic in $M$, free curves will be shown to be discretely generated by the symmetry group in Sect.\ \ref{disgencur}. 
This section is organized as follows.
\begingroup
\setlength{\leftmargini}{12pt}
\begin{itemize}
\item
In Sect.\ \ref{dpodspodspo}, we show that if $\wm$ is sated and  analytic in $G$, then an analytic immersive curve is exponential if it admits a certain  approximation property (see Definition \ref{Def:LC}).\footnote{More specifically, we show that if $\wm$  is analytic in $G$, then an analytic immersive curve is locally exponential if it admits the mentioned approximation property. The rest then follows from Lemma \ref{vd}.} 
\item	
In Sect.\ \ref{regcasee}, we introduce the notion of a free curve, and show that if $\wm$ is sated and analytic in $M$, then each analytic immersive curve that is not free has a certain self-similarity property (see Definition \ref{contgen}). 
\item
In Sect.\ \ref{regcase}, under the assumption that $\wm$ is regular,    this self-similarity property is shown to be equivalent to the approximation property introduced in Sect.\ \ref{dpodspodspo}.   
This result is already stated in the beginning of Sect.\ \ref{regcasee}, and is used there to prove the classification Theorem \ref{classi}.
\end{itemize}
\endgroup
\subsection{Exponential Curves}
\label{dpodspodspo}
In this subsection, we prove the following statement. 
\begin{lemma}
\label{liecur}
Let $\wm$ be sated and analytic in $G$, and let $\gamma\colon I\rightarrow M$ be an analytic immersion. Then, $\gamma$ is exponential if $\gamma$ is $\tau$-exponential for some $\tau\in I$. 	
\end{lemma}
\noindent
The used notion is defined as follows:
\begin{definition}
\label{Def:LC}
Let $\gamma\colon I\rightarrow M$ be analytic immersive, and let $\tau\in I$. Then, $\gamma$ is said to be {\bf $\boldsymbol{\tau}$-exponential} 
\deff there exists a {\sf faithful} sequence $G\backslash G_{\gamma(\tau)}\supseteq \{g_n\}_{n\in \NN}\rightarrow e$, such that $\tau$ is an accumulation point of 
\begin{align*}
	\T:=\{\tau< t\in I\:|\: \gamma(\tau)\barrow \gamma(t)\}.
\end{align*} 
\begingroup
\setlength{\leftmargini}{12pt}
\begin{itemize}
\item  
{\rm faithful} means that 
 there exist sequences\he\footnote{Observe that the property $\lim_n g_n=e$ is automatically implied by  \eqref{dskjdskjdskjkjdsdszudszudszus76d76ds76ds76ds76ds76dssdds}.} 
\begin{align}
\label{dskjdskjdskjkjdsdszudszudszus76d76ds76ds76ds76ds76dssdds}
\begin{split}
	(0,\infty)\supseteq \{\lambda_n\}_{n\in \NN}\rightarrow 0
	\qquad&\text{and}\qquad \mg\backslash \mg_{\gamma(\tau)}\supseteq \{\g_n\}_{n\in \NN}\rightarrow \g\in  \mg\backslash \mg_{\gamma(\tau)} \qquad\quad\\[3pt]
	&\text{with}\\[3pt]
	g_n=\exp&(\lambda_n \cdot \g_n)\qquad\forall\: n\in \NN.
\end{split}
\end{align} 
\item 
$\gamma(\tau)\barrow \gamma(t)$ means that for each $n_0\in \NN$ and  $\epsilon>0$ with $\tau< t-\epsilon$, there exist $n\geq n_0$ and $m\geq 1$ with\footnote{Simply put, $\gamma(t)$ can be ``arbitrarily well'' approximated by successive shifts of $\gamma(\tau)$ through $\gamma((\tau,t])$ by some ``arbitrarily small'' $g_n$.}
\begin{align}
\label{lieprop}
	(g_{n})^k\cdot \gamma(\tau) \in \gamma((\tau,t])\qquad\forall\:k=1,\dots,m\qquad\quad\text{as well as}\qquad\quad (g_{n})^m\cdot \gamma(\tau)\in \gamma((t-\epsilon,t]).
\end{align}
\end{itemize}
\endgroup  
\end{definition} 
\noindent
Obviously, the above definition is local in the sense that $\gamma$ is $\tau$-exponential \deff $\gamma|_J$ is $\tau$-exponential for some open interval $J\subseteq I$ containing $\tau$. 
Now, Lemma \ref{liecur} is immediate from Lemma \ref{vd} and the following statement:
\begin{lemma}
\label{lieth}
Assume that $\wm$ is analytic in $G$. 
If $\gamma\colon I\rightarrow M$ is $\tau$-exponential for some $\tau\in I$, then $\gamma|_J$ is exponential for some open interval $J\subseteq I$ containing $\tau$.
\end{lemma}
\begin{proof}
By locality, we can assume that $\gamma$ is equicontinuous as well as an embedding (Corollary \ref{dfdsasasasassa}), and furthermore that $\im[\gamma]$ is contained in a fixed chart $(O,\psi)$ around $x\equiv\gamma(\tau)$. 
Since $\wm_x\cp \exp$ is continuous and since $\lim_n\g_n=\g$ holds, there exists $0<\ell<\pii(x,\g)$ and $q_0\in \NN$ such that the images of 
\begin{align*}
 \delta:=\gag|_{L}\qquad\quad\text{and}\qquad\quad \delta_n:=\gamma_{\g_n}^x|_{L}\qquad\forall\:n\geq q_0\qquad\qquad\text{with}\qquad\qquad L:=[0,\ell]
\end{align*}
are contained in $O$. Hence, we can assume that $M=\psi(O)$ holds, and that the topology on $M$ is determined by the euclidean norm $\|\cdot\|\colon \RR^{\dim(M)}\rightarrow [0,\infty)$. Then, $\{\delta_n\}_{n\geq q_0}\rightarrow \delta$ converges uniformly on $L$, because $\exp$ is continuous and since $\{\g_n\}_{n\in \NN}\rightarrow \g$ holds. 
\vspace{6pt}

\noindent
Now, since $\gamma$ and $\delta$ are analytic embeddings\footnote{Let $\epsilon>0$ with $\ell + 2\epsilon<\pii(x,\g)$. Then, $\gag|_{(-\epsilon,\ell+\epsilon)}$ is analytic; as well as an embedding, because $\gamma|_{[-\epsilon,\ell+\epsilon]}$ is an embedding by  injectivity and continuity as well as by compactness of $[-\epsilon,\ell+\epsilon]$.}, the claim follows from Lemma \ref{lemma:BasicAnalyt1} and Lemma \ref{lemma:BasicAnalyt2} once we have shown that $x$ is an accumulation point of $\im[\delta]\cap \im[\gamma]$. 
We prove this statement by contradiction.  
Specifically, we show that  
if $x$ is not an accumulation point of $\im[\delta]\cap \im[\gamma]$, then there exists an open interval $J\subseteq\RR$ with $t\in J\subseteq I$ and $\gamma(\T\cap J)\subseteq \im[\delta]$. Since $\tau$ is an accumulation point of $\T\cap J$, this contradicts the assumption    
that $x\equiv \gamma(\tau)$ is not an accumulation point of $\im[\delta]\cap \im[\gamma]$.
\vspace{6pt}

\noindent 
Assume thus now that $x$ is not an accumulation point of $\im[\delta]\cap \im[\gamma]$:
\vspace{-2pt}
\begingroup
\setlength{\leftmargini}{12pt}
\begin{enumerate}
\item[$\triangleright$]	
 There exists a  compact interval $K\subseteq (0,\ell]$ with $\delta(K)\cap \im[\gamma]=\emptyset$. Hence, there exists an open interval  $J\subseteq\RR$ with $t\in J\subseteq I$, such that  $d:=\dist(\delta(K),\gamma(J))>0$ holds.
 \item[$\triangleright$]
 Since $\{\delta_n\}_{n\geq q_0}\rightarrow \delta$ converges uniformly, we can increase $q_0$ such that $\|\delta-\delta_n\|_\infty< d$ holds for all $n\geq q_0$, hence $\delta_n(K)\cap \gamma(J)=\emptyset$ for all $n\geq q_0$.
\item[$\triangleright$]
Since $\lim_n \lambda_n=0$ holds, we can increase $q_0$ such that
\begin{align}
\label{kjdskjkjsdiudsiudsiusdiusdds98ds9898dds98dsdscxcx}
\forall\: n\geq q_0\colon \:\:\: \exists\:  m(n)\in \NN_{>0}\colon\:\:\: m(n)\cdot \lambda_n \in K\:\:\:\text{hence}\:\:\: \delta_n(m(n)\cdot \lambda_n)\notin \gamma(J).
\end{align}
\end{enumerate}
\endgroup
\noindent
Let now $\Delta>0$ be given: 
\begingroup
\setlength{\leftmargini}{12pt}
\begin{itemize}
\item  
We choose $n_0\geq q_0$ with $\|\delta-\delta_{n}\|_\infty < \Delta$ for all $n\geq n_0$. 
\item
We choose $\tilde{\epsilon}>0$ such that:
\hspace*{\fill} (equicontinuity of $\gamma$)
\begin{align}
\label{sdlkdskjdsnmnmdsnsdhjdsdzuds87ds87sd87ds87dsdsds}
|t-t'|< \tilde{\epsilon}\quad\text{for}\quad t,t'\in  I\qquad\quad\Longrightarrow\qquad\quad\|\gamma(t)-\gamma(t')\|<\Delta.\qquad\quad
\end{align}
\end{itemize}
\endgroup 
\noindent
For $t\in \T\cap J$, we choose $0<\epsilon \leq  \tilde{\epsilon}$ with $\tau< t-\epsilon$, and let $n\geq n_0$ and $m\geq 1$ be as in \eqref{lieprop}. Then, we have
\vspace{-8pt}
	\begin{align*}
	\|\gamma(t)-\delta_{n}(m\cdot \lambda_{n})\|&=
		\|\gamma(t)-\exp(\lambda_{n}\cdot \g_n)^m\cdot x\|=\|\gamma(t)-(g_{n})^m\cdot \gamma(\tau)\|\stackrel{\eqref{lieprop},\:\eqref{sdlkdskjdsnmnmdsnsdhjdsdzuds87ds87sd87ds87dsdsds}}{<}\Delta,
	 \end{align*}
	 whereby $m\cdot \lambda_n\in L= \dom[\delta_n]$ holds,   because $m(n)\cdot\lambda_n \in K\subseteq L$ implies $m< m(n)$: 
	 $$
	\text{In fact,}\quad 	m(n)\leq m\quad\stackrel{\eqref{lieprop}}{\Longrightarrow}\quad \gamma(J)\ni (g_n)^{m(n)}\cdot \gamma(\tau)= \delta_n(m(n)\cdot \lambda_n),\quad\text{which contradicts}\quad \eqref{kjdskjkjsdiudsiudsiusdiusdds98ds9898dds98dsdscxcx}.
	 $$  
	 We thus have
	 \vspace{-6pt}
	 \begin{align}
	 \label{oidsoidsoidsds98ds98ds98ds98ds98sddsdsdssddsds}
	 	\|\gamma(t)-\delta(m\cdot \lambda_{n})\|&\leq \|\gamma(t)-\delta_{n}(m\cdot \lambda_n)\|+ \|\delta_{n}(m\cdot \lambda_{n})-\delta(m\cdot \lambda_n)\|< 2\Delta.
	\end{align}	
Since 
\eqref{oidsoidsoidsds98ds98ds98ds98ds98sddsdsdssddsds} holds for each  
$\Delta>0$,   
there exists $\{s_n\}_{n\in \NN}\subseteq L$ with $\gamma(t)=\lim_n \delta(s_n)$. Passing to a subsequence if necessary, we can assume $\lim_n s_n=s \in L$, hence $\gamma(t)=\lim_n \delta(s_n)=\delta(s)\in \im[\delta]$. 
Since $t\in \T\cap J$ was arbitrary, we have $\gamma(\T\cap J)\subseteq \im[\delta]$.	
\end{proof}

\subsection{Free Curves}
\label{regcasee}
In this subsection, we introduce the concept of a free curve,  and provide the basic facts (cf.\ Sect.\ \ref{podspoipoipds}) that we will need in Sect.\ \ref{regcase} to prove Proposition \ref{classii} that we already use in the next subsection (Sect.\ \ref{idsoidsoisoidsds98ds9898ds9dsdsdssdsds}) to establish our classification result Theorem \ref{classi}. 
\begin{proposition}
\label{classii}
	If $\wm$ is regular and separately analytic, then an analytic immersive curve is either exponential or free.
\end{proposition}
\noindent
\noindent
The used notion is defined as follows:
\begin{definition}
\label{fdf}
Let $\wm$ be analytic in $M$.
\begingroup
\setlength{\leftmargini}{12pt}
\begin{itemize}
\item
An analytic immersion $\gamma\colon D\rightarrow M$ is said to be a \textbf{free segment} \defff the following equivalence holds:
	\begin{align*}
		g\cdot\gamma\cpsim \gamma\quad\text{for}\quad g\in G\qquad\quad\Longleftrightarrow\qquad\quad g\in G_\gamma.
	\end{align*}
	Obviously, the restriction of a free segment to an interval is a free segment as well.
\item
An analytic curve $\gamma\colon D\rightarrow M$ is said to be {\bf free} \deff there exists an interval $B\subseteq D$, such that $\gamma|_{B}$ is a free segment. Obviously,
\vspace{-4pt}
\begingroup
\setlength{\leftmarginii}{12pt}
\begin{itemize}
\item[$\ast$]
$\gamma$ is non-constant if $\gamma$ is free. 
\item[$\ast$]
$\gamma$ is free if $\gamma|_{D'}$ is free for some an interval $D'\subseteq D$.
\end{itemize}
\endgroup
\end{itemize}
\endgroup
\end{definition}
\subsubsection{The Classification Result}
\label{idsoidsoisoidsds98ds9898ds9dsdsdssdsds}
From Proposition \ref{classii}, we easily obtain our classification result:
\begin{theorem}
\label{classi}
If $\wm$ is regular and separately analytic, then an analytic curve is either exponential or free; whereby the uniqueness statement from Proposition \ref{classsiilie} holds in the first case.
\end{theorem} 
\begin{proof}
Each constant curve is exponential, but not free as it cannot admit any immersive subcurve. 
Assume thus that $\gamma\colon D\rightarrow M$ is non-constant analytic.  
Then, $\gamma|_J$ is analytic immersive for some open interval $J\subseteq D$; hence, either exponential or free by Proposition \ref{classii}:
\begingroup
\setlength{\leftmargini}{12pt}
\begin{itemize}
\item
If $\gamma|_J$ is exponential, then $\gamma$ is exponential by Lemma \ref{vd}. Consequently, each analytic immersive subcurve of $\gamma$ is exponential, so that $\gamma$ cannot be free by Proposition \ref{classii}. 
\item
If $\gamma|_J$ is free, then $\gamma$ is free. Moreover, $\gamma$ cannot be exponential, because elsewise $\gamma|_J$ would be exponential, which contradicts Proposition \ref{classii} as $\gamma|_J$ is free by assumption.\qedhere
\end{itemize}
\endgroup
\end{proof}
\begin{corollary}
\label{fdgf}
Let $\wm$ be regular and separately analytic. If $\gamma\colon D\rightarrow M$ is free, then each  
subcurve of $\gamma$ is free.
\end{corollary}
\begin{proof}
If  $\gamma|_{D'}$ is not free with $D'\subseteq D$ an interval, then $\gamma|_{D'}$ is exponential by Theorem \ref{classi}, hence $\gamma$ is exponential by Lemma \ref{vd}. This, however,  contradicts   
Theorem \ref{classi} as $\gamma$ is free by assumption.
\end{proof}

\begin{remark}
\label{rtrevgf}
In Sect.\ \ref{disgencur}, we will show that each free analytic immersion $\gamma\colon I\rightarrow M$ admits a natural decomposition into countably many free segments that are mutually (and uniquely) related to each other by the group action. Now, if $\gamma$ is free and non-constant analytic, then $Z:=\{t\in I\:|\: \dot\gamma(t)=0\}$ consists of isolated points and admits no limit point in $I$. Hence, we have\:\footnote{Concerning the second point:\: If $\cn_-=-\infty$ holds, then $\cn_-\leq n$ means $n\in \ZZ$; and analogously for $\cn_+$.} 
\begingroup
\setlength{\leftmargini}{12pt}
\begin{itemize}
\item
\vspace{-3pt}
$Z=\emptyset$\:\:\: \deff\:\: $\gamma$ is analytic immersive.
\item
$Z\neq \emptyset$\:\:\:  \deff\:\: $Z=\{t_n\}_{\cn_-\leq n\leq \cn_+}$ holds with $-\infty\leq \cn_-\leq 0\leq \cn_+\leq \infty$ and $t_m< t_n$ for $m<n$ such that 
\begingroup
\setlength{\leftmarginii}{11pt}
\begin{itemize}
\item[$*$]
\vspace{-3pt}
$\cn_-=-\infty$ implies that for each $t\in I$, we have $t_n<t$ for some $\cn_-\leq n\leq \cn_+$.
\item[$*$]
\vspace{2pt}
$\cn_+=\phantom{-}\infty$ implies that for each $t\in I$, we have $t< t_n$ for some $\cn_-\leq n\leq \cn_+$.
 \end{itemize}
\endgroup
 \end{itemize}
\endgroup
\noindent
Now, the restriction of $\gamma$ to the connected components of $I\backslash Z$ 
is analytic immersive by the choice of $Z$, as well as free by Corollary \ref{fdgf}. Our decomposition results for analytic immersions thus apply to each of these subcurves separately. The problem is then to figure out in which way the corresponding decompositions glue together at the splitting points $t_n$. Alternatively, one could also drop the immersivity assumption in   Definition \ref{fdf} as well as in the definition of $\cpsim$ (see Sect.\ \ref{fdggdrere}), and then go through the argumentation in Sect.\ \ref{disgencur}. But, then one will have difficulties with certain uniqueness statements proven there, as those rely on Lemma \ref{sdffsd}, which in turn relies on Corollary \ref{dfdgttrgf}. 
For this observe that a non-constant analytic curve can ``inverse its direction'' at the points $t_n$\footnote{Specifically, this means that $\gamma|_{(t_n,s)}=\gamma\cp \rho$ holds for a negative analytic diffeomorphism $\rho\colon I\supseteq (t_n,s)\rightarrow (s',t_n)\subseteq I$ for certain $s'<t_n<s$. For instance, one obtains such a curve $\gamma$ if one composes an analytic immersive curve $(-\epsilon,r^2)\rightarrow M$ (for $r,\epsilon>0$) with the analytic map $(-r, r) \ni t\mapsto t^2\in [0,r^2)$.} (which is impossible at the points in $I\backslash Z$ by Corollary \ref{dfdgttrgf}). This issue, however, should be of rather combinatorical nature, because if $\gamma$ is immersive on $(s',t_n)$ and $(t_n,s)$, and ``inverses its direction'' at $t_n$, then $\gamma|_{(s',t_n)}$ is either a subcurve of $\gamma|_{(t_n,s)}$ or vice versa. In any case, the crucial point to be clarified first is whether the following statement holds or not:
\vspace{6pt}

\noindent
``Let $\gamma\colon I\rightarrow M$ and $\gamma'\colon I'\rightarrow M$ be non-constant analytic as well as non-immersive only at $t\in I$ and $t'\in I'$, respectively. Assume  that  $\gamma|_{I\:\cap\: (-\infty,t)}=\gamma'\cp \rho$ holds for a positive analytic diffeomorphism $\rho\colon I\cap (-\infty,t)\rightarrow I'\cap (-\infty ,t')$, and that both $\gamma$ and $\gamma'$ do not ``inverse their direction'' at $t$ and $t'$, respectively. Then, $\gamma|_{(t,t+\epsilon)}=\gamma'\cp \kappa$ holds for a positive analytic diffeomorphism $\kappa\colon (t,t+\epsilon)\rightarrow (t',t'+\epsilon')$ with $\epsilon,\epsilon'>0$.''
\hspace*{\fill}$\ddagger$
\end{remark}
\subsubsection{Elementary Properties}
\label{podspoipoipds}
In this subsection, we collect some elementary facts concerning free curves that we will need in Sect.\ \ref{regcase} to prove Proposition \ref{classii}. We start with the following straightforward observation.
\begin{lemma}
\label{grtfsghrtrez}
Assume that $\wm$ is separately analytic; and let $\gamma\colon D\rightarrow M$ be an analytic immersion. If $\gamma$ is exponential, then $\gamma$ is not free.
\end{lemma}
\begin{proof}
Let $\gamma$ be exponential, i.e., $\gamma=\gag\cp \rho$ holds for an analytic diffeomorphism ($\gamma$ is immersive) $\rho\colon D\rightarrow D'\subseteq \RR$. Then, $G_\gamma=G_\gag$ holds by Corollary \ref{fdsfds7}, and we have
\begin{align}
\label{sdffgsgaayx}
\exp(s\cdot \g)\cdot \gag(t)=\gag(s+t)\qquad\quad \forall\: s,t\in \RR.
\end{align}
Let now $B\subseteq D$ be a fixed interval. Then, there exists a compact interval $K\subseteq B$, such that  $\gamma|_K$ is an embedding. Then, $\gag\big|_{\rho(K)}$ is an embedding; and we write $L:= \rho(K)\equiv [r,r+2\epsilon]$ for $r\in \RR$ and $\epsilon>0$. Since $\gag\big|_L$ is injective, \eqref{sdffgsgaayx} implies $g:= \exp(\epsilon\cdot \g)\notin G_\gamma$ as well as 
\begin{align*}
g\cdot \gag((r,r+\epsilon))\stackrel{\eqref{sdffgsgaayx}}{=}\gag((r+\epsilon,r+2\epsilon))
\qquad\quad&\Longrightarrow\qquad\quad g\cdot \gag|_{L}\cpsim \gag|_L\\
		   &\Longrightarrow\qquad\quad g\cdot \gamma|_{B}\cpsim \gamma|_{B}.
\end{align*} 
Hence, $\gamma$ cannot be free as the interval $B\subseteq D$ was arbitrary. 
\end{proof}
\begin{definition}
\label{contgen}
	Assume that $\wm$ in analytic in $M$. Let $\gamma\colon I\rightarrow M$ be an analytic immersion, and let $\tau\in I$.   
Then, $\gamma$ is said to be {\bf continuously generated at $\boldsymbol{\tau}$} \deff to each compact interval of the form $[\tau,k]\subseteq I$, there exists some $g\in G\backslash G_{\gamma(\tau)}$ with $g\cdot \gamma|_{[\tau,k]}\cpsim \gamma|_{[\tau,k]}$.
\end{definition}
\noindent
Obviously, the above definition is local in the sense that $\gamma$ is continuously generated at $\tau$ \deff $\gamma|_J$ is continuously generated at $\tau$ for some open interval $J\subseteq I$ containing $\tau$. 

We will conclude from Lemma \ref{seque}.\ref{seque2} further below:
\begin{corollary}
\label{prop:free2} 
Let $\wm$ be sated and analytic in $M$, and let  $\gamma\colon I\rightarrow M$ be an analytic immersion. If $\gamma$ is not free, then $\gamma$ is  
continuously generated at each $\tau\in I$.
\end{corollary}
\noindent 
To establish this result, we first need to prove the following lemma.
\begin{lemma}
\label{HL}
Let $\wm$ be sated and analytic in $M$; and let  
 $\gamma\colon I\rightarrow M$ be an analytic embedding, $\tau\in I$, as well as $K\equiv[\tau,k]\subseteq I$ and  $K'\equiv[\tau,k']\subseteq I$. Then, 
\begin{align*}
	g\cdot \gamma|_K \cpsim \gamma|_{K'}\quad\:\text{for}\quad\: g\in G_{\gamma(\tau)}\backslash G_\gamma\qquad\quad\Longrightarrow\qquad\quad g\cdot \gamma(k) \in \gamma(K').
\end{align*} 
\end{lemma}
\begin{proof}
 Let $\rho\colon K\supseteq J\rightarrow J'\subseteq K'$ be an analytic diffeomorphism with $g\cdot \gamma|_J=\gamma\cp \rho$. 
Illustratively, then the following two situations can occur, depending on whether $\rho$ is positive or negative:
	\vspace{-5pt}
	
	  \hspace{45pt}
      \begin{tikzpicture}
      \draw (-1.3,0.2) node {\({\dot\rho>0\colon}\)};
		\draw[-,line width=1pt,color=red] (0,0) .. controls (0.1,0.6) and (0.7,0.4) .. (1,0.1);
		\draw[-,line width=1pt,color=red] (1,0.1) .. controls (1.1,0) and (1.2,0) .. (1.3,0);
		\draw[->,line width=1pt,color=red] (1.4,0) .. controls (1.5,0) and (1.7,0.2) .. (1.7,0.4);
		\draw[->,line width=1.2pt] (0,0) -- (2,0);
		\filldraw[black] (0,0) circle (2pt);
		\draw[color=red] (1.5,0.5) node {\(_{k}\)};
		\draw (2,-0.3) node {\(_{k'}\)};
		\draw (-0.2,-0.02) node {\(_\tau\)};
		\draw (2.25,0) node {\(\gamma\)};
		\draw[color=red] (2.25,0.5) node {\(g\cdot\gamma\)};	
		\end{tikzpicture}   
		\qquad\qquad \qquad\quad	
    \begin{tikzpicture}	
    \draw (-1.3,0.2) node {\({\dot\rho<0\colon}\)};
		\draw[-,line width=1pt,color=red] (0,0) .. controls (0,1.2) and (2,0.4) .. (1.4,0);
		\draw[->,line width=1pt,color=red] (1.1,0) .. controls (1,0) and (0.7,0) .. (0.6,0.25);
		\draw[->,line width=1.2pt] (0,0) -- (2,0);
		\filldraw[black] (0,0) circle (2pt);
		\draw[color=red] (0.4,0.17) node {\(_{k}\)};	
		\draw (2,-0.3) node {\(_{k'}\)};
		\draw (-0.2,-0.02) node {\(_\tau\)};
		\draw (2.25,0) node {\(\gamma.\)};
		\draw[color=red] (2,0.5) node {\(g\cdot \gamma\)};	
	\end{tikzpicture}
	\vspace{3pt}	
				
\noindent
More precisely, let $[c',c]\equiv\CM:=K\cap \ovl{\rho}^{-1}(K')$ be as in Lemma \ref{lem:maxextKomp}. Then, 
\begin{align*}
	\dot\rho>0\qquad\stackrel{(*)}{\Longrightarrow}\qquad c'=\tau=\ovl{\rho}(c')\qquad\Longrightarrow\qquad g\cdot \gamma([\tau,c])\subseteq \gamma([\tau,k'])\:\:\quad\stackrel{g\:\in\: G_{\gamma(\tau)}\:\:\wedge\:\:\text{Lemma } \ref{stabbiii}}{\Longrightarrow}\:\:\quad g\in G_\gamma,
\end{align*}
	which contradicts the assumption $g\notin G_\gamma$. The implication $(*)$  is obtained as follows:
\begingroup
\setlength{\leftmargini}{12pt}
\begin{itemize}
\item[$\triangleright$]
	If $\tau<c'$, then Lemma \ref{lem:maxextKomp} together with $\dot\rho>0$ implies $\ovl{\rho}(c')=\tau$; which contradicts injectivity of $g\cdot \gamma$ as   
\begin{align*}
	\ovl{\rho}(c')=\tau\quad\wedge\quad g\in G_{\gamma(\tau)}\qquad\quad\Longrightarrow\qquad\quad g\cdot \gamma(c')=\gamma(\ovl{\rho}(c'))=\gamma(\tau)=g\cdot \gamma(\tau).
\end{align*}
\item[$\triangleright$]	
	We have $c'=\tau$ by the previous point, as $\tau\leq c'$ holds by the definition of $C$. We obtain
	\begin{align*}	
	\gamma(\tau)=g\cdot \gamma(\tau)=g\cdot \gamma(c')= \gamma(\ovl{\rho}(c'))\qquad\stackrel{\gamma\text{ injective}}{\Longrightarrow}\qquad \tau=\ovl{\rho}(c').
	\end{align*}
\end{itemize}
\endgroup
\noindent
We thus have shown that $\dot\rho<0$ holds. Then, $c=k$ holds;     because
\begin{align*}
\dot\rho<0\quad\wedge\quad c<k\qquad\stackrel{\text{Lemma }  \ref{lem:maxextKomp}}{\Longrightarrow} \qquad\ovl{\rho}(c)=\tau\qquad\:\:\Longrightarrow\qquad\:\: 
g\cdot \gamma(\tau)=\gamma(\tau)=\gamma(\ovl{\rho}(c))=g\cdot\gamma(c),
\end{align*}
which contradicts injectivity of $\gamma$ as $\tau\leq c'<c$. We   obtain
\begin{equation*}
	 g\cdot \gamma(k)\stackrel{c\:=\: k}{=}g\cdot \gamma(c)=\gamma(\ovl{\rho}(c))\in \gamma(K').\qedhere
\end{equation*}
\end{proof}
\begin{lemma}
\label{seque}
Assume that $\wm$ is sated and analytic in $M$. Let $\gamma\colon I\rightarrow M$ be an analytic embedding, and let $\tau\in I$ be fixed. Then, the following assertions hold:
\begingroup
{
\renewcommand{\theenumi}{\arabic{enumi})} 
\renewcommand{\labelenumi}{\theenumi}
\setlength{\leftmargini}{16pt}
\begin{enumerate}
\item
\label{seque1} 
Let $\{g_n\}_{n\in \NN}\subseteq G\backslash G_\gamma$ as well as   $\{k_n\}_{n\in \NN}\subseteq I\cap (\tau,\infty)$ be strictly decreasing with $\lim_n k_n=\tau$, such that  
\begin{align}
\label{gffdfdaluoi}
	g_n\cdot \gamma|_{[\tau,k_n]}\cpsim \gamma|_{[\tau,k_n]}\qquad\quad\forall\: n\in \NN.
\end{align}
Then, there exists $m\in \NN$ with $\{g_n\}_{n\geq m}\subseteq  G\backslash G_{\gamma(\tau)}$.	
\item
\label{seque2} 
Let $\tau<d \in I$ be given  such that $\gamma|_{[\tau,k]}$ is not a free segment for each $\tau<k\leq d$. Then, $\gamma$ is continuously generated at $\tau$.
\end{enumerate}}
\endgroup
\end{lemma}
\begin{proof}
\begingroup
\setlength{\leftmargini}{16pt}
\begin{enumerate}
\item
Assume that the claim is wrong. Then, passing to a subsequence if necessary, we can assume that $\{g_n\}_{n\in \NN}\subseteq G_{\gamma(\tau)}\backslash G_\gamma$ holds. 
For $n\in \NN$ and $k\in[k_n,k_0]$, we obtain from \eqref{gffdfdaluoi} and Lemma \ref{HL} (with $K\equiv[\tau,k]$ and $K'\equiv[\tau,k_n]$):
\begin{align*}
g_n\cdot \gamma|_{[\tau,k]}\stackrel{\eqref{gffdfdaluoi}}{\cpsim}\gamma|_{[\tau,k_n]}
\qquad\stackrel{\text{Lemma } \ref{HL}}{\Longrightarrow}\qquad 
g_n\cdot \gamma(k)\in \gamma([\tau,k_n]).
\end{align*}
This shows $g_n\cdot \gamma([k_n,k_0])\subseteq \gamma([\tau,k_n])$ for all $n\in \NN$; hence
\begin{align*}
g_n\cdot \gamma([k_1,k_0])\subseteq g_n\cdot \gamma([k_n,k_0])\subseteq \gamma([\tau,k_n])\qquad\quad \forall\: n\geq 1,
\end{align*}
which contradicts that $\gamma(\tau)$ is sated (as $\gamma$ is continuous).
\item
We choose $\{k_n\}_{n\in \NN}\subseteq (\tau,d]\subseteq I$  strictly decreasing with $\lim_nk_n=\tau$, as well as $\{g_n\}_{n\in \NN}\subseteq  G\backslash G_\gamma$ with $g_n\cdot \gamma|_{[\tau,k_n]}\cpsim \gamma|_{[\tau,k_n]}$ for all $n\in \NN$. Then,  Part \ref{seque1} yields $m\in \NN$ with $\{g_n\}_{n\geq m}\subseteq  G\backslash G_{\gamma(\tau)}$, from which the claim is clear.\qedhere
\end{enumerate}
\endgroup
\end{proof}
\noindent
We are ready for the proof of Corollary \ref{prop:free2}.
\begin{proof}[Proof of Corollary \ref{prop:free2}]
Given $\tau\in I$, we fix an open interval $J\subseteq I$,   such that $\tilde{\gamma}:=\gamma|_J$ is an analytic embedding (Corollary \ref{dfdsasasasassa}). Then $\tilde{\gamma}$ is not free, since elsewise $\gamma$ would be free (see Definition \ref{fdf}). Thus, $\tilde{\gamma}$ is continuously generated at $\tau$ by Lemma \ref{seque}.\ref{seque2}, hence $\gamma$ is continuously generated at $\tau$ by locality. 
\end{proof}
\subsection{The Regular Case}
\label{regcase}
In this subsection, we prove Proposition \ref{classii}. For this, it suffices to show the following statement.
\begin{proposition}
\label{fdfdsfds}
Let $\wm$ be sated and separately analytic; and let $\gamma\colon I\rightarrow M$ be continuously generated at $\tau\in I$.  If $\gamma(\tau)$ is stable, then $\gamma$ is exponential. 
\end{proposition}
\noindent
In fact, Proposition \ref{classii}  immediately follows   if we combine Proposition \ref{fdfdsfds} with  Corollary \ref{prop:free2} and Lemma \ref{grtfsghrtrez}:
\begin{proof}[Proof of Proposition \ref{classii}]
\begingroup
\setlength{\leftmargini}{12pt}
\begin{itemize}
\item
If $\gamma\colon D\rightarrow M$ is exponential, then $\gamma$ is not free by Lemma \ref{grtfsghrtrez}. 
\item
Assume that $\gamma$ is not free, and let $I\subseteq D$ a  fixed open interval. Then, $\gamma|_I$ is not free (since elsewise $\gamma$ would be free), hence continuously generated at some (each) $\tau\in I$ by Corollary \ref{prop:free2}. Thus,  $\gamma|_I$ is exponential by Proposition \ref{fdfdsfds}, hence $\gamma$ is exponential by Lemma \ref{vd}. \qedhere
\end{itemize}
\endgroup 
\end{proof}
\noindent
Before we start to the prove of Proposition \ref{fdfdsfds}, we  now first want to discuss what might happen if the point $\gamma(\tau)$ is not stable: 
\begin{example}
\label{podspodspopodsdsa}
We consider the situation in Example \ref{dsassayyasa}.\ref{dsassayyasa3} for $\lambda$ irrational; hence, $\wm$ is sated but no point is stable. Then, $\gamma\colon \RR\ni t\mapsto (1,\e^{2\pi t\cdot \I}) \in \mathbb{T}^2$ is continuously generated at $0$ but not exponential: 
\begingroup
\setlength{\leftmargini}{12pt}
\begin{itemize}
\item
$\gamma$ cannot be exponential, because $\im[\gamma]$ is not contained in the orbit of $e\equiv (1,1)\in \mathbb{T}^2$ under $\wm$. To see this, let $\mu\in \RR$ be such that $1,\lambda,\mu$ are $\QQ$-independent. We have the following implications:
\begin{align*}
\im[\gamma]\subseteq \RR\cdot e\qquad\Longrightarrow\qquad \gamma(\mu)=\wm(t,e)\quad\text{for some}\quad t\in \RR\qquad\Longrightarrow\qquad(1,\e^{2\pi \mu\cdot \I})=(\e^{2\pi t\cdot \I},\e^{2\pi t\lambda\cdot \I}),
\end{align*}
whereby $1= \e^{2\pi t\cdot \I}$ implies $t\in \ZZ$, so that  
$\e^{2\pi \mu\cdot \I}=\e^{2\pi t\lambda\cdot \I}$ implies that $1,\lambda,\mu$ are $\QQ$-dependent.
\item
$\gamma$ is continuously generated at $0$. To see this, we set $v:=\e^{2\pi\lambda\cdot \I}$ and observe the following:
\begingroup
\setlength{\leftmarginii}{12pt}
\begin{itemize}
\item
$\{v^n\}_{n\in \ZZ}\subseteq \UE$ is dense by Kronecker's theorem, as $1,\lambda$ are $\QQ$-independent. 
\item
$\forall\: n\in \ZZ\colon$\quad\:\:$n\cdot \gamma(t)=(\e^{2\pi n\cdot\I},\e^{ 2\pi n\lambda\cdot \I} \cdot \e^{ 2\pi t\cdot \I} )=(1,v^n\cdot\e^{2\pi t\cdot \I })$\hspace*{\fill}$(*)$
\end{itemize}
\endgroup 
\noindent  
Let now $0<k<1$ be given. We set $\epsilon:=k/3$, and find (by  denseness) some $n\in \ZZ$ and $-\epsilon<s<\epsilon$ with $v^n=\e^{2\pi s\cdot \I}$. Then, we have\hspace*{\fill}(observe $[s+\epsilon,s+2\epsilon]\subseteq [0,k]$ for the second implication)
\begin{align*}
	n\cdot \gamma(t)\stackrel{(*)}{=}\gamma(s+t)\:\:\quad\forall\: t\in \RR\qquad\quad&\Longrightarrow\qquad\quad n\cdot \gamma([\epsilon,2\epsilon])=\gamma([s+\epsilon,s+2\epsilon])\\[4pt]
	&\Longrightarrow\qquad\quad  n\cdot \gamma|_{[0,k]}\cpsim \gamma|_{[0,k]}\quad\text{with}\quad n\in G\backslash G_{\gamma(0)},
\end{align*}
\vspace{-22pt}

\noindent
which proves the claim. 
\hspace*{\fill}$\ddagger$

\end{itemize}
\endgroup
\end{example}
\noindent
\subsubsection{Some Facts and Definitions}
\label{saposapoaposksalksasa}
\begin{convention}
\label{kjdskjdskjd}
For the rest of this section, we assume that $\wm$ is  analytic $M$. 
\end{convention}
\begin{definition}
\label{dslkdslklkdskldsdsds09ds0909ds09dsdscxcx}
Let $\gamma\colon I\rightarrow M$ be an analytic embedding, and  let $\tau\in I$. A $(\gamma,\tau)$-collection is a family  
$\{(g_n,K_n,J_n)\}_{n\in \NN}$ with
\begin{align*}
g_n\in G,\qquad 
K_n\text{ a compact interval with } \tau\in K_n,\qquad
J_n \text{ an open interval with } \tau\in J_n
\end{align*} 
for all $n\in \NN$,   
such that  
\begin{align}
\label{conni}
\begin{split}
\textstyle\bigcap_{n\in \NN}K_n =\{\tau\}\qquad\quad
&\text{holds with}\qquad\quad
I \supset K = K_0\supset J_0 \supset K_1\supset J_1 \supset  K_2\supset J_2\supset  \dots
\qquad\qquad\\[4pt]
%
&\text{as well as}\qquad\quad 
g_n\cdot \gamma|_{J_n}\cpsim \gamma|_{J_n}\qquad\quad \forall\: n\in \NN.
\end{split}
\end{align}
\end{definition}
\begin{remark}
\label{jsakjsakjs} 
In the situation of Definition \ref{dslkdslklkdskldsdsds09ds0909ds09dsdscxcx},  
let $\iota\colon \NN\rightarrow\NN$ be strictly increasing. 
\begingroup
{
\setlength{\leftmargini}{17pt}
\renewcommand{\theenumi}{{\arabic{enumi}})} 
\renewcommand{\labelenumi}{\theenumi}
\begin{enumerate}
\item
\label{jsakjsakjsb}
The family $\{(g_{\iota(n)},K_n,J_n)\}_{n\in \NN}$ is again a $(\gamma,\tau)$-collection, because for $n,k\in \NN$ we have 
\begin{align*}
	J_{n+k}\subseteq J_{n}\qquad\quad\text{hence}\qquad\quad g_{n+k}\cdot \gamma|_{J_n}\cpsim \gamma|_{J_n}.
\end{align*}
\vspace{-20pt}

\item
\label{ddiudsdsmncxycxc}
The family  $\{(g_{\iota(n)},K_{\iota(n)},J_{\iota(n)})\}_{n\in \NN}$ is again a $(\gamma,\tau)$-collection. Moreover, 
 if $\{(g_{n},K_n,J_n)\}_{n\in \NN}$ admits one of the following additional properties
\begin{align}
\label{fdfdgfgf}
g_n\cdot \gamma(\tau)\in \gamma(J_n)\hspace{42pt}\qquad\quad &\forall\: n\in \NN,\\[4pt]
\label{eq:schachtelung}
	g_{n+1}\cdot \gamma(K_{n+1})\subseteq \gamma(J_{n})\subseteq \gamma(K_{n})\subseteq\gamma(K)\quad\qquad &\forall\: n\in \NN,\\[4pt]
\label{ldslsddsoidsoid}
	g_n\cdot \gamma(K_n)\subseteq \gamma(K)\hspace{46pt}\qquad\quad &\forall\: n\in \NN,
\end{align}
\noindent
then the same holds true for (the subsequence) $\{(g_{\iota(n)},K_{\iota(n)},J_{\iota(n)})\}_{n\in \NN}$.\footnote{Of course, this means that the corresponding equation  holds with $g_\ell\equiv g_{\iota(\ell)}$, $K_\ell\equiv K_{\iota(\ell)}$, $J_\ell\equiv J_{\iota(\ell)}$ for $\ell\in \NN$.} \hspace*{\fill}$\ddagger$
\end{enumerate}}
\endgroup
\end{remark}
\noindent
\begin{lemma}
\label{lemma:simpliiiiii}
Let $\wm$ be sated; and let $\{(g_n,K_n,J_n)\}_{n\in \NN}$ be a  $(\gamma,\tau)$-collection. Then, for each compact neighbourhood $A\subseteq I$ of $\tau$, there exists $m\in \NN$ and a compact neighbourhood $B\subseteq I$ of $\tau$ with $g_n\cdot \gamma(B)\subseteq \gamma(A)$ for all $n\geq m$.
\end{lemma}
\begin{proof}
The proof is elementary, and can be found in Appendix \ref{appB1}. 
\end{proof}
\begin{corollary}
\label{lemma:simpli}
Let $\wm$ be sated; and let $\{(g_n,K_n,J_n)\}_{n\in \NN}$ be a  $(\gamma,\tau)$-collection. 
\begingroup{
\setlength{\leftmargini}{16pt}
\renewcommand{\theenumi}{{\arabic{enumi}})} 
\renewcommand{\labelenumi}{\theenumi}
\begin{enumerate}
\item
\label{lemma:simpli1}
For each compact neighbourhood $L\subseteq I$ of $\tau$, there exists $n_0\in \NN$ with $g_n\cdot \gamma(K_n)\subseteq \gamma(L)$ for all $n\geq n_0$. 
\item
\label{lemma:simpli2}
There exists $n_0\in \NN$ and a compact neighbourhood $L\subseteq K$ of $\tau$ with $g_n\cdot \gamma(L)\subseteq \gamma(K)$ for all $n\geq n_0$.
\end{enumerate}}
\endgroup
\end{corollary}
\begin{proof}
\begingroup
\setlength{\leftmargini}{16pt}
\begin{enumerate}
\item
Apply Lemma \ref{lemma:simpliiiiii} to $A\equiv L$; and choose $n_0\geq m$ such large that $K_n\subseteq B$ holds for all $n\geq  n_0$.
\item
Apply Lemma \ref{lemma:simpliiiiii} to $A\equiv K$; and set $L:= B\cap A$ as well as $n_0:= m$. \qedhere
\end{enumerate}
\endgroup
\end{proof}
\noindent
\subsubsection{{Some Technical Preparation}}
In this subsection, we derive some first statements from our discussions in Sect.\ \ref{saposapoaposksalksasa} that we will also need in the proof of Lemma \ref{lemmfreeneig} in Sect.\ \ref{pofdpofdpofdpofdpofd}. We now start with some further definitions.  
\begin{definition}
\label{dshdefetezezeeeewwewe434343}
A $(\gamma,\tau)$-approximation is a $(\gamma,\tau)$-collection  $\{(g_n,K_n,J_n)\}_{n\in \NN}$  
with $\{g_n\}_{n\in \NN}\subseteq G\backslash G_{\gamma(\tau)}$.    
\begingroup
\setlength{\leftmargini}{16pt}
{
\renewcommand{\theenumi}{\arabic{enumi})} 
\renewcommand{\labelenumi}{\theenumi}
\begin{enumerate}
\item
\label{dshdefetezezeeeewwewe4343431}
Let $\{(g_n,K_n,J_n)\}_{n\in \NN}$ be a $(\gamma,\tau)$-approximation such that \eqref{fdfdgfgf} holds. Then, $g_n\notin G_{\gamma(\tau)}$ implies
\begin{align*}
 g_n\cdot \gamma(\tau)=\gamma(\tau+\Delta_n)\quad\text{with}\quad \tau+\Delta_n\in J_n\subseteq K \quad\text{for some unique}\quad \Delta_n\neq 0.
\end{align*}
We say that $g_n$ shifts $\tau$ to the left\slash right \deff $\Delta_n<0\slash\Delta >0$ holds.
\item
\label{dshdefetezezeeeewwewe4343432}
Let $\{(g_n,K_n,J_n)\}_{n\in \NN}$ be a $(\gamma,\tau)$-approximation such that \eqref{ldslsddsoidsoid} holds. Then, $g_n\cdot \gamma|_{J_n}=\gamma\cp\rho_n$ holds for the analytic diffeomorphism (Lemma \ref{lemma:BasicAnalyt2})
\begin{align*}
	\rho_n= (\gamma|^{\im[\gamma]})^{-1}\cp (g_n\cdot \gamma|_{J_n})\colon J_n\rightarrow I_n\subseteq K\subseteq I\qquad\quad \forall\: n\in \NN
\end{align*}
as $\gamma$ and $g_n\cdot \gamma|_{J_n}$ are analytic embeddings. 
We say that $g_n$ is positive\slash  negative \deff 
$\rho_n$ is positive\slash  negative. 
\end{enumerate}}
\endgroup
\end{definition}
\noindent
The following statement is a straightforward consequence of  Remark \ref{jsakjsakjs} and Corollary \ref{lemma:simpli}.\ref{lemma:simpli1}.
\begin{lemma}
\label{dds}
Let $\wm$ be sated; and let $\gamma\colon I\rightarrow M$ be an analytic embedding that is continuously generated at $\tau\in I$. Then, there exists a $(\gamma,\tau)$-approximation $\{(g_n,K_n,J_n)\}_{n\in \NN}$ such that \eqref{fdfdgfgf},  \eqref{eq:schachtelung}, \eqref{ldslsddsoidsoid} hold.
\end{lemma}
\begin{proof}
The proof is elementary, and can be found in Appendix \ref{appB2}.
\end{proof}
\begin{remark}
\label{opdspodspodspodspodsds0909ds09dsdsdsdsds}
Assume that we are in the situation of Lemma \ref{dds}, i.e., 
 $\{(g_n,K_n,J_n)\}_{n\in \NN}$ is a $(\gamma,\tau)$-approximation such that \eqref{fdfdgfgf},  \eqref{eq:schachtelung}, \eqref{ldslsddsoidsoid} hold.  
 We have $g_n\cdot \gamma(J_n)\subseteq \gamma(J_{n-1})\subseteq \gamma(I)$ for each $n \geq 1$ by \eqref{eq:schachtelung}.  Hence, Definition \ref{dshdefetezezeeeewwewe434343}.\ref{dshdefetezezeeeewwewe4343432} implies that  
$g_n\cdot \gamma|_{J_n}=\gamma\cp\rho_n$ holds for the (unique) analytic diffeomorphism 
\begin{align*}
	\rho_n= (\gamma|^{\im[\gamma]})^{-1}\cp (g_n\cdot \gamma|_{J_n})\colon J_n\rightarrow I_n\subseteq J_{n-1}\subseteq I\qquad
	\text{for each}\qquad n\geq 1.
\end{align*}
\vspace{-33pt}

\hspace*{\fill}$\ddagger$
\end{remark}  
\noindent
Combining Remark \ref{opdspodspodspodspodsds0909ds09dsdsdsdsds}  with Corollary \ref{lemma:simpli}.\ref{lemma:simpli1} and Lemma \ref{seque}.\ref{seque1}, we obtain the following statement.
\begin{lemma}
\label{podspods}
Let $\wm$ be sated; and let $\gamma\colon I\rightarrow M$ be an analytic embedding that is continuously generated at $\tau\in I$. Then, there exists a $(\gamma,\tau)$-approximation $\{(g_n,K_n,J_n)\}_{n\in \NN}$ such that \eqref{fdfdgfgf} and  \eqref{ldslsddsoidsoid} hold, with $g_n$ positive for each $n\in \NN$.
\end{lemma}
\begin{proof}
Confer Appendix \ref{appB4}. 
\end{proof}

\subsubsection{$\boldsymbol{\delta}$-Approximations}
\begin{convention}
For the rest of this section, we assume that $\wm$ is  sated and analytic in $M$. 
Moreover, $\delta\colon \ovl{I}\rightarrow M$ denotes a fixed analytic immersion that is continuously generated at $\ovl{\tau}\in \ovl{I}$, with $x\equiv\delta(\ovl{\tau})$ additionally stable. 
\end{convention}
\begin{definition}
\label{ddskjdsjkdsjd}
A $\delta$-approximation $(\gamma,\rho)\vee \{(g_n,K_n,J_n)\}_{n\in \NN}$ is a $(\gamma,\tau)$-approximation $\{(g_n,K_n,J_n)\}_{n\in \NN}$ together with an analytic diffeomorphism 
$$\rho\colon \ovl{I}\supseteq I'\rightarrow I''\subseteq \dom[\gamma]\qquad\text{such that}\qquad \delta|_{I'}=\gamma\cp \rho \qquad\text{and} \qquad \rho(\ovl{\tau})=\tau\qquad\text{holds},
$$ 
for $I'$ an open interval containing $\ovl{\tau}$ as well as $I''$ a bounded open interval containing $\tau$. 
\end{definition} 
\begin{remark}
\label{remmmma}
Proposition \ref{fdfdsfds} follows immediately from  Lemma \ref{vd} and Lemma \ref{liecur}  once we have shown that if $\wm$ is sated and analytic in $M$, then there exists a $\delta$-approximation $(\gamma,\rho)\vee \{(g_n,K_n,J_n)\}_{n\in \NN}$ such hat $\gamma$ is $\tau$-exponential. \hspace*{\fill}$\ddagger$
\end{remark}
\noindent
Since $\delta$ is continuously generated at $\ovl{\tau}$,   Lemma \ref{podspods} provides the following statement.
\begin{corollary}
\label{dffssdffds}
There exists a $\delta$-approximation $(\gamma,\rho)\vee \{(g_n,K_n,J_n)\}_{n\in \NN}$ such that \eqref{fdfdgfgf} and  \eqref{ldslsddsoidsoid} hold,   
with $g_n$ positive for each $n\in \NN$.
\end{corollary}
\begin{proof}
We fix a bounded open interval $I'\subseteq \ovl{I}$ with $\ovl{\tau}\in I'$ such that $\gamma:= \delta|_{I'}$ is an analytic embedding (Corollary \ref{dfdsasasasassa}), and set $I'':=I'$ as well as $\rho:= \id_{I'}$.
Then, the claim is immediate from Lemma \ref{podspods}.
\end{proof}
\noindent
During this section, we will perform several modifications to the $\delta$-approximation obtained in Corollary \ref{dffssdffds}, and eventually end up with a $\delta$-approximation as mentioned in Remark \ref{remmmma}.
\begin{remark}
\label{kldlkfdklfd}
Given a $\delta$-approximation $(\gamma,\rho)\vee \{(g_n,K_n,J_n)\}_{n\in \NN}$, by   
  ``passing to a subsequence'' we mean to 
 pass to the $\delta$-approximation $(\gamma,\rho)\vee \{(g_{\iota(n)},K_{\iota(n)},J_{\iota(n)})\}_{n\in \NN}$ for some strictly increasing $\iota\colon \NN\rightarrow \NN$  (see Remark \ref{jsakjsakjs}.\ref{ddiudsdsmncxycxc}). 
Evidently, ``passing to a subsequence'' does not change the properties of the $\delta$-approximation constructed in Corollary \ref{dffssdffds}; and the same is true for the properties of a $\delta$-approximation that we will obtain in Lemma \ref{pofpofdofdop}, Lemma \ref{daaddsd}, and Lemma \ref{popofdpofd} below.\hspace*{\fill}$\ddagger$
\end{remark}

\begin{remark}
\label{podspoddsa}
Let $(\gamma,\rho)\vee \{(g_{n},K_n,J_n)\}_{n\in \NN}$ be a $\delta$-approximation as in Corollary \ref{dffssdffds}, as well as $J\subseteq I''$ an open interval with $\tau\in J$. Then, passing to a subsequence and restricting $\gamma$, we can additionally assume that $\gamma$ is defined on $J$. For this,  
 we just have to fix $n_0\in \NN$ with $K_n\subseteq J$ for each $n\geq n_0$, set $\iota\colon \NN\ni n\mapsto n_0+n\in \NN$, and   consider $(\gamma|_J,\rho|_{I'\he\cap\he \rho^{-1}(J)})\vee \{(g_{\iota(n)},K_{\iota(n)},J_{\iota(n)})\}_{n\in \NN}$.
\hspace*{\fill}$\ddagger$
\end{remark}
\noindent
Combining Corollary \ref{dffssdffds} with Remark \ref{podspoddsa} and Lemma \ref{gdfgfgf}, we obtain from stability of $x\equiv\delta(\ovl{\tau}) = \gamma(\tau)$:
\begin{lemma}
\label{pofpofdofdop}
There exists a $\delta$-approximation $(\gamma,\rho)\vee \{(g_n,K_n,J_n)\}_{n\in \NN}$ such that \eqref{fdfdgfgf} and  \eqref{ldslsddsoidsoid} hold, with $\im[\gamma]\subseteq G\cdot \gamma(\tau)$ as well as $g_n$ positive for each $n\in \NN$.
\end{lemma}
\begin{proof} 
Confer Appendix \ref{appB3}. 
\end{proof}
\noindent
Passing to a subsequence and ``inverting the direction of $\gamma$'' if necessary, we obtain the following statement.
\begin{lemma}
\label{daaddsd}
There exists a $\delta$-approximation $(\gamma,\rho)\vee \{(g_{n},K_n,J_n)\}_{n\in \NN}$ with  \eqref{fdfdgfgf} and    \eqref{ldslsddsoidsoid} and $\im[\gamma]\subseteq G\cdot \gamma(\tau)$, such that each $g_n$ is positive and shifts $\tau$ to the right.
\end{lemma}
\begin{proof}
The proof is elementary but technical, and can be found in Appendix \ref{appB5}.
\end{proof}
\subsubsection{Convergence}
We now finally modify the $\delta$-approximation obtained in Lemma \ref{daaddsd} such that additionally $\lim_n g_n=e$ holds;  and then prove some properties for such $\delta$-approximations that we will need for the proof of Proposition \ref{fdfdsfds}  in Sect.\ \ref{sddssdsd}. 
\begin{lemma}
\label{popofdpofd}
There exists a $\delta$-approximation as in Lemma \ref{daaddsd}, such that additionally $\lim_n g_n=e$ holds.
\end{lemma}
\begin{proof}
Let $(\gamma,\rho)\vee \{(g_{n},K_n,J_n)\}_{n\in \NN}$ be as in Lemma \ref{daaddsd}, and recall that $x\equiv \gamma(\tau)$ holds. 
\begingroup
\setlength{\leftmargini}{14pt}
{
\renewcommand{\theenumi}{\sf\small\alph{enumi})} 
\renewcommand{\labelenumi}{\theenumi}
\begin{enumerate}
\item
\label{fdskjkjsdkjkjdskjdsiuewiuewiuewiuiuew1}
Obviously, we can replace $\{g_n\}_{n\in \NN}$ by $\{h_n\cdot g_n\cdot h_n'\}_{n\in \NN}$ for sequences $\{h_n\}_{n\in \NN}, \{h'_n\}_{n\in \NN}\subseteq G_{\gamma}$, without affecting any of the properties that we have established so far.\footnote{For instance, if \eqref{ldslsddsoidsoid} holds, then also $(h_n\cdot g_n\cdot h_n')\cdot \gamma(K_n)=h_n\cdot (g_n\cdot \gamma(K_n))\subseteq h_n\cdot \gamma(K)=\gamma(K)$.} 
\item
\label{fdskjkjsdkjkjdskjdsiuewiuewiuewiuiuew2}
$G_{[x]}\subseteq G_\gamma$ holds by $\im[\gamma]\subseteq G\cdot x$, and $x$ is stable 
  with $\lim_n g_n\cdot x=x$ by \eqref{fdfdgfgf}. 
  \vspace{3pt}
  
  Hence, we can modify $\{g_n\}_{n\in \NN}$ as described in \ref{fdskjkjsdkjkjdskjdsiuewiuewiuewiuiuew1}, and then pass to a subsequence (recall Remark \ref{kldlkfdklfd}) to achieve that 
  $\lim_n g_n =g\in G_{x}$ exists.
\end{enumerate}}
\endgroup
\noindent  
It suffices to show that $g\in G_\gamma$ holds; because, according to  \ref{fdskjkjsdkjkjdskjdsiuewiuewiuewiuiuew1}, then we can replace  $g_n$ by $g_n\cdot g^{-1}$ for each $n\in \NN$  to achieve $G\backslash G_{\gamma(\tau)}\supseteq \{g_n\}_{n\in \NN}\rightarrow e$.  
\begingroup
\setlength{\leftmargini}{12pt}
\begin{itemize}
\item[$\triangleright$]	
By Corollary \ref{lemma:simpli}.\ref{lemma:simpli2}, there exists a compact neighbourhood 
$L\subseteq K$ of $\tau$ and $n_0\in \NN$, such that $g_n\cdot \gamma(L)\subseteq \gamma(K)$ holds for all $n\geq n_0$. 
For each $n\geq n_0$, Lemma \ref{lemma:BasicAnalyt2} thus yields a unique analytic diffeomorphism 
\vspace{-3pt}
$$\kappa_n\colon L\rightarrow L_n\subseteq K\qquad\quad \text{with}\qquad\quad   
 g_n\cdot \gamma|_L=\gamma\cp \kappa_n.$$
 \vspace{-19pt} 
\item[$\triangleright$]
Let $n'_0\geq n_0$ be such large that $J_n\subseteq L$ holds for each $n\geq n'_0$. 
Then, $\dot\kappa_n>0$ holds for 
each $n\geq n'_0$, because $g_n$ is positive with $\kappa_n|_{J_n}=\rho_n$ by uniqueness. Hence, $\kappa$ is strictly monotonic increasing.
\item[$\triangleright$]
Since $g_n$ shifts $\tau$ to the right, for $L=[\ell',\ell]$, $K=[k',k]$, and $n\geq n'_0$, we  have $g_n\cdot \gamma([\tau,\ell])\subseteq \gamma([\tau,k])$.
\item[$\triangleright$]
For each $t\in [\tau,\ell]$ and $n\geq n_0'$, we thus have $g_n\cdot \gamma(t)\in \gamma([\tau,k])$, hence
\begin{align*} 
	\textstyle g\cdot \gamma(t)=\lim_n g_n\cdot \gamma(t)\in \gamma([\tau,k]).
\end{align*}
	This shows $g\cdot \gamma([\tau,\ell])\subseteq \gamma([\tau,k])$  with $g\in G_{\gamma(\tau)}$; so that Lemma \ref{stabbiii} yields $g\in G_{\gamma}$. \qedhere
\end{itemize}
\endgroup
\end{proof}
\noindent
The next lemma collects some important properties of a $\delta$-approximation as in Lemma \ref{popofdpofd}. 
\begin{lemma}
\label{approx} 
Let $(\gamma,\rho)\vee \{(g_{n},K_n,J_n)\}_{n\in \NN}$ be as in Lemma \ref{popofdpofd}, write $K= K_0=[a,b]$, and let   $K'=[a',b']$ be a compact interval as well as $I'$ an open interval with       
 $K\subseteq I'\subseteq K'\subseteq I=\dom[\gamma]$. For each $n\in \NN$, let $p(n)\in \NN_{>0}\cup \{\infty\}$ be maximal with\he\footnote{Since $\gamma$ is an embedding, this is equivalent to that $(g_n)^p\cdot \gamma(J_n)\subseteq K'$ holds for all $0\leq p\leq p(n)$.}
 \hspace*{\fill}(we write $0\leq p\leq \infty$ \deff $p\in \NN$ holds)
\begin{align*}
	(g_n)^p\cdot \gamma(J_n)= \gamma(I_{n,p})
	\qquad\quad\forall\: 0\leq p\leq p(n)
\end{align*}
for necessarily unique open intervals $I_{n,p}=(i'_{n,p},i_{n,p})\subseteq K'$, hence
\begin{align}
\label{fgopopfgopfg}
\hspace{5pt}(g_n)^p\cdot \gamma(\tau)=\gamma(\tau_{n,p})\quad\:\:\text{for unique}\quad\:\:\tau_{n,p}\in I_{n,p}\subseteq K' \qquad\quad \forall\: 0\leq p\leq p(n).
\end{align}
Then, the following assertions hold:
\begingroup
\setlength{\leftmargini}{16pt}
{
\renewcommand{\theenumi}{\arabic{enumi})} 
\renewcommand{\labelenumi}{\theenumi}
\begin{enumerate}
\item
\label{approx1}
We have $p(n)\geq 1$ for each $n\in \NN$, as well as
\begin{align}
\label{propies}
\tau=\tau_{n,0}< \tau_{n,1}<\tau_{n,2} <{\dots}<\tau_{n,p(n)}\qquad\quad\text{and}\qquad\quad
\tau_{n,p+1}\in I_{n,p}\qquad\forall\:\:0\leq p \leq p(n)-1. 
\end{align}
\vspace{-15pt}
\item
\label{approx2}
We have $p(n)< \infty$ for each $n\in \NN$.
\item
\label{approx3}
There exists $d\in \NN$, with $b< i_{n,p(n)}$ for all $n\geq d$.
\item
\label{approx4}
For each $t\in (\tau,b]$, $n_0\in \NN$, and $\epsilon>0$ with $\tau< t-\epsilon$, there exist some $n\geq n_0$ and $1\leq m\leq p(n)$ with $\tau_{n,m}\in(t-\epsilon, t]$; hence
\begin{align*}
	(g_{n}\cdot h)^k\cdot \gamma(\tau) \in \gamma((\tau,t])\qquad\forall\:k=1,\dots,m\qquad\quad\text{as well as}\qquad\quad (g_{n}\cdot h)^m\cdot \gamma(\tau)\in \gamma((t-\epsilon,t])
\end{align*}
holds for each $h\in G_\gamma$.
\item
\label{approx5}
Let $\{q(n)\}_{n\in \NN}\subseteq \NN$ be such that $q(n)\leq p(n)$ holds for infinitely many $n\in \NN$. Then,
\begin{align}
\label{fhjjghhhg}
\textstyle\lim_n (g_n\cdot h_n)^{q(n)}=g\in G_{\gamma(\tau)}\quad\text{for}\quad \{h_n\}_{n\in \NN}\subseteq G_\gamma\qquad\quad \Longrightarrow\qquad\quad g\in G_\gamma.
\end{align}
\end{enumerate}}
\endgroup
\end{lemma}
\begin{proof}
Confer  
Appendix \ref{appB6}. 
\end{proof}

\subsubsection{Proof of Proposition \ref{fdfdsfds}} 
\label{sddssdsd}
Let $(\gamma,\rho)\vee \{(g_{n},K_n,J_n)\}_{n\in \NN}$ be as in Lemma \ref{popofdpofd}; and recall Remark \ref{remmmma}, i.e., to prove Proposition \ref{fdfdsfds}, it suffices to show that $\gamma$ is $\tau$-exponential.  
For this, we will tacitly use  in the following that Lemma \ref{approx} also applies to each subsequence of $(\gamma,\rho)\vee \{(g_{n},K_n,J_n)\}_{n\in \NN}$ by Remark \ref{kldlkfdklfd}.

\begin{proof}[Proof of Proposition \ref{fdfdsfds}]
Let $H\subseteq G$ denote the closure of the group generated by $O_\gamma:=\{g\in G\:|\: g\cdot \gamma\cpsim \gamma\}\subseteq G$. Then, $G_\gamma\subseteq H$ is a normal subgroup, as $g^{-1}\cdot q \cdot g \in G_\gamma$ holds for all $q\in G_\gamma$ and $g\in O_\gamma$ by Corollary \ref{fsdfsfdfs}. Therefore, $Q:=H\slash G_\gamma$ is a Lie group; and the canonical projection $\pri\colon H\rightarrow Q$ is a Lie group homomorphism with $\ker[\dd_e\pri]=\mg_\gamma$. 
Moreover, there exists a  
smooth local section $s\colon Q\supseteq V\rightarrow U\subseteq H$ defined on an open neighbourhood $V\subseteq Q$ of $[e]$ such that $s([e])=e$ and $\pri\cp s=\id_V$ holds, hence $\dd_e\pri\cp \dd_{[e]} s=\id_\mq$.
\vspace{6pt}

\noindent
Let $\exp$ and $\exp_\mq$ denote the exponential maps of $G$ and $Q$, respectively;\footnote{Also observe that $\exp|_\mh$ is the exponential map of $H$, for $\mh$ the Lie algebra of $H$.} and let $W\subseteq \mq$ be an open neighbourhood of $0_\mq$ such that $\exp_\mq|_W\colon W\rightarrow V'\subseteq V$ is a homeomorphism to an open neighbourhood $V'$ of $[e]\in Q$. We observe that $\{g_n\}_{n\in \NN}\subseteq H\backslash G_{\gamma(\tau)}\subseteq H\backslash G_{\gamma}$ holds by \eqref{conni} (recall $\{g_n\}_{n\in \NN}\subseteq G\backslash G_{\gamma(\tau)}$ by definition of a $(\gamma,\tau)$-approximation), and write $[g_n]\equiv \pri(g_n)\neq [e]$ for each $n\in \NN$. Since $\pri$ is continuous, $\lim_n g_n=e$ implies  
$\lim_n\he [g_n]=[e]$; so that, passing to a subsequence if necessary, we can assume that $\{[g_n]\}_{n\in \NN}\subseteq V'$ holds. We proceed as follows: 
\begingroup
\setlength{\leftmargini}{12pt}
\begin{itemize}
\item
Let $\|\cdot \|$ be a fixed norm on $\mq$. 
Since $\{[g_n]\}_{n\in \NN}\subseteq V'$ holds, we have
\begin{align*}
	[g_n]=\exp_\mq(\lambda_n\cdot \q_n)\quad\:\text{for some}\quad\: \q_n\in \mq\quad\:\text{with}\quad\: \|\q_n\|=1\quad\:\text{and}\quad\: \lambda_n> 0\qquad\quad\forall\: n\in \NN;
\end{align*}
hence $\lim_n \lambda_n=0$, as $\lim_n [g_n]=[e]$ holds and $\exp_\mq|_W$ is a homeomorphism. 
\item
By compactness of the unit sphere, we can 
pass to a subsequence to achieve that $\lim_n \q_n=\q\in \mq\neq 0$ exists (with $\|\q\|=1$). We set \hspace*{\fill}($\mh$ the Lie algebra of $H$)
\begin{align*}
	\g:=\dd_{[e]}s(\q)\in \mh\subseteq \mg\qquad\quad\text{as well as}\qquad\quad \g_n:=\dd_{[e]}s(\q_n)\in \mh\subseteq \mg\qquad\forall\: n\in \NN,
\end{align*}
so that $\lim_n \g_n=\g$ holds by continuity of $\dd_{[e]} s$. 
\item
Since $\pri$ is a Lie group homomorphism, for each $n\in \NN$, we have
\begin{align}
\label{opsddsofyadsfhdfs}
\begin{split}
	\pri(\underbrace{\exp(\lambda_n\cdot \g_n)}_{\displaystyle =: \tilde{g}_n})=\exp_\mq(\lambda_n\cdot \underbrace{\dd_e\pri(\g_n)}_{\displaystyle=\q_n})=[g_n]\qquad\:\: &\Longrightarrow \qquad\:\: \tilde{g}_n= g_n \cdot h_n\:\:\:\text{for some}\:\:\: h_n\in G_\gamma\\[-13pt]
	&\Longrightarrow \qquad\:\: \tilde{g}_n\in G\setminus G_{\gamma(\tau)}\\[5pt]
	&\Longrightarrow \qquad\:\: \g_n\in \mg\setminus \mg_{\gamma(\tau)},
\end{split}
\end{align}
where we have used $\dd_e\pri\cp \dd_{[e]} s=\id_\mq$ on left side, $g_n\in G\setminus G_{\gamma(\tau)}$ in the second implication, as well as $\mg_{\gamma(\tau)}=\{\h\in \mg\:|\: \exp(\RR\cdot \h)\subseteq G_{\gamma(\tau)}\}$ in the third implication.
\end{itemize}
\endgroup
\noindent
Now, Lemma \ref{approx}.\ref{approx4} together with the first line (right side) in \eqref{opsddsofyadsfhdfs} shows \hspace*{\fill}(recall Definition \ref{Def:LC})  
\begin{align*}
\qquad\quad&\gamma(\tau)\barrows \gamma(t)\qquad\forall\: t\in (\tau,b] \subseteq K=[a,b]\subseteq I=\dom[\gamma]
\\[5pt]
\Longrightarrow\qquad\quad &\hspace{6pt}(\tau,b]\subseteq \T=\{\tau< t\in I\:|\: \gamma(\tau)\barrows \gamma(t)\}\\[5pt]
\Longrightarrow\qquad\quad &\hspace{6pt}\:\:\tau  \:\text{ is an  accumulation point of }\: \T.
\end{align*} 
It thus remains to show that $\g\in \mg\setminus\mg_{\gamma(\tau)}$ holds; because then $\{\tilde{g}_n\}_{n\in \NN}$ is faithful, hence $\gamma$ is $\tau$-exponential (so that the claim follows from Remark \ref{remmmma}). 
For this, it suffices to show the implication
\begin{align*}
	\g\in \mg_{\gamma(\tau)}\qquad\quad&\Longrightarrow\qquad\quad g_t:=\exp([0,\ell]\cdot \g)\subseteq  G_\gamma\quad\text{for some}\quad \ell>0,
\end{align*}
because the right side implies $\g\in \mg_\gamma$, which contradicts  $\dd_e\pri(\g)=\q\neq 0_\mq$. 
\vspace{6pt}

\noindent
Assume thus that $\g\in \mg_{\gamma(\tau)}$ holds, and let $O$ be a neighbourhood of $x\equiv \gamma(\tau)$ as well as $\epsilon>0$ such that $O\cap \gamma((b-\epsilon,b])=\emptyset$ holds (observe that $\gamma$ is injective). 
Since $\wm_{x}\cp \exp$ is continuous and since $\lim_n \g_n =\g$ holds, there exist $n_0\in \NN$ and $\ell>0$, such that the image of  $\delta_n:=\gamma_{\g_n}^x|_{[0,\ell]}$ is contained in $O$ for each $n\geq n_0$. We fix $t\in [0,\ell]$, and conclude $g_t\in G_\gamma$ as follows: 
\begingroup
\setlength{\leftmargini}{12pt}
\begin{itemize}
\item[$\triangleright$]	
For each  $n\in \NN$, we choose $q(n)\in \NN$ maximal with $q(n)\cdot \lambda_n \leq t$. Then, $\lim_n\lambda_n=0$ together with $\lim_n\g_n=\g$ implies 
\vspace{-9pt}
\begin{align}
\label{oidsoidslksdlkewlkewnmnmewew98ew98ew989ewewe}
\textstyle\lim_n q(n)\cdot \lambda_n\cdot \g_n= t\cdot \g.
\end{align}
\item[$\triangleright$]	
Since $\g\in \mg_{\gamma(\tau)}$ holds by assumption, we obtain
\vspace{-6pt}
\begin{align*}
	G_{\gamma(\tau)}\ni g_t=\exp(t\cdot \g)\stackrel{\eqref{oidsoidslksdlkewlkewnmnmewew98ew98ew989ewewe}}{=}\textstyle   \lim_n\exp(q(n)\cdot \lambda_n\cdot \g_n)=\lim_n (\tilde{g}_n)^{q(n)}\stackrel{\eqref{opsddsofyadsfhdfs}}{=} \lim_n (g_n\cdot h_n)^{q(n)}.
\end{align*}
Lemma \ref{approx}.\ref{approx5} thus yields $g_t\in G_\gamma$ once we have shown that $q(n)\leq p(n)$ holds  for infinitely many $n\in \NN$:
\begingroup
\setlength{\leftmarginii}{12pt}
\begin{itemize}
\item
By Lemma \ref{approx}.\ref{approx4},  
for infinitely many $n\geq n_0$, there exists $1\leq m(n)\leq p(n)$ with  
\begin{align*}
	\underbrace{(g_{n})^k\cdot \gamma(\tau) \in \gamma((\tau,b])\qquad\forall\:k=1,\dots,m(n)}_{(*)}\qquad\quad\text{as well as}\qquad\quad \underbrace{(g_{n})^{m(n)}\cdot \gamma(\tau)\in \gamma((b-\epsilon,b]).}_{(**)}
\end{align*}
\vspace{-5pt}
\item
For each such $n\in \NN$, we necessarily have $q(n)< m(n)\leq p(n)$; because $m(n)\leq q(n)$ implies
\begin{align*}
 0\leq m(n)\cdot \lambda_n&\leq  q(n)\cdot \lambda_n\leq t\leq \ell\\[3pt]
 \text{h}&\text{ence}\\[-5pt]
\gamma((b-\epsilon,b])\stackrel{(**)}{\ni}(g_n)^{m(n)}\cdot \gamma(\tau)\stackrel{(*)}{=}(\tilde{g}_n)^{m(n)}\cdot \gamma(\tau)&=\exp(m(n)\cdot \lambda_n\cdot \g_n)\cdot \gamma(\tau)=\delta_n(m(n)\cdot \lambda_n)\in  O, 
\end{align*}
which contradicts $O\cap \gamma((b-\epsilon,b])=\emptyset$.
\qedhere
\end{itemize}
\endgroup
\end{itemize}
\endgroup
\end{proof}

\section{Decompositions}
\label{disgencur}
In the previous section, we have shown that an analytic curve is exponential \deff it is not free, provided that $\wm$ is regular and separately analytic. In this section, we show that if $\wm$ is sated and analytic in $M$, then each free immersive $\gamma\colon D\rightarrow M$ is discretely generated by the symmetry group -- Roughly speaking, this means that $\gamma$ can be naturally decomposed into free segments that are mutually (and uniquely) related by the group action. 
\begin{convention}
\label{nmsnmdsnmdsddscxcxcx}
In this section, $\wm\colon G\times M\rightarrow M$ always denotes a left action that is analytic in $M$.
\end{convention}
\noindent
At this point, the reader might recall the notions and statements  in   Sect.\ \ref{repari}.
\subsection{Basic Facts and Definitions}
\label{sfd}
We now first want to make the stated decomposition results  a little bit more clear. For this, we let $\gamma\colon I\rightarrow M$ be free immersive, with $\gamma|_D$ a free segment for  $D\subset I$ properly contained in $I$. We additionally assume that $g\cdot \gamma|_J= \gamma\cp \rho$ holds for some $g\in G\backslash G_\gamma$ and an analytic diffeomorphism $\rho\colon I\supseteq D\supseteq J\rightarrow J'\subseteq I\backslash D$. Then, 
\begin{align*}
g\cdot \gamma|_\CM=\gamma\cp \ovl{\rho}|_\CM\qquad\:\:\text{holds for}\qquad\:\: \CM:=D\cap \ovl{\rho}^{-1}(I)\qquad\:\:\text{with}\qquad\:\:
\ovl{\rho}(\CM)\subseteq I\backslash \innt[D]
\end{align*}
as $\ovl{\rho}(C)\cap \innt[D]\neq \emptyset$ implies $g\cdot\gamma|_D \cpsim \gamma|_D$, hence $g\in G_{\gamma|_D}=G_\gamma$ (Lemma \ref{lemma:stabi}). 
Now, one can ask the question whether the intervals $D$ and $\ovl{\rho}(C)$ can be brought together 
by a suitable choice of $g$, i.e., whether $g\in G\backslash G_\gamma$ can be chosen such that $D$ and $\ovl{\rho}(C)$ share a common boundary point $s\in I$. 
One finds immediately that this cannot happen  if $\gamma|_{I'}$ is a free segment for an open interval $I'\subseteq I$ that contains the closure of $D$ in $I$ -- just because then   $s\in I'$ implies  $g\cdot\gamma|_{I'} \cpsim \gamma|_{I'}$, hence $g\in G_{\gamma|_{\tilde{I}}}=G_\gamma$ (Lemma \ref{lemma:stabi}). 
This motivates the following definitions:
\begin{definition}
\label{asffd}
Let $\gamma\colon D \rightarrow M$ be an analytic immersion. An interval $A\subseteq D$ is said to be
\begingroup
\setlength{\leftmargini}{12pt}
\begin{itemize}
\item
	free \hspace{17.8pt} (w.r.t.\ $\gamma$)\:\: \deff\:\:$\gamma|_A$ is a free segment,
\item
	maximal (w.r.t.\ $\gamma$) \:\:\deff\:\:{}$A$ is free such that there does not exist a free interval $A'\subseteq D$ with $A\subset A'$.
\end{itemize}
\endgroup
\end{definition}
\noindent
Obviously, each subinterval of a free interval is free; and we additionally observe: 
\begingroup
\setlength{\leftmargini}{12pt}
\begin{itemize}
\item
If $A\subseteq D$ is free, then the closure $\ovl{A}\subseteq D$ of $A$ in $D$ is free, because 
\begin{align}
\label{dfdffdfdfdfd}
	 g\cdot\gamma|_{\ovl{A}}\cpsim \gamma|_{\ovl{A}}\qquad\quad\Longrightarrow\qquad\quad g\cdot\gamma|_{A}\cpsim  \gamma|_{A}\qquad\quad\Longrightarrow\qquad\quad g\in G_\gamma.
\end{align}
In particular, each maximal interval $A\subseteq D$ is automatically closed in $D$.
\item
Lemma \ref{lemma:stabi} yields
\begin{align}
\label{ldslkdlkdsd09s09d09ds09ds09ds09dssddsdscxcxcx}
	A\subseteq D'\subseteq D\quad\text{free w.r.t.}\quad \gamma|_{D'}  \qquad\:\:\:\stackrel{ G_{\gamma|_{{D'_{}}}}=\: G_\gamma\:\:}{\Longleftrightarrow}\qquad\:\: A \quad \text{free w.r.t.\ to}\quad \gamma.
\end{align} 
\end{itemize}
\endgroup
\noindent 
Moreover, Zorn's lemma provides the following statement:
\begin{lemma}
\label{maxim}
Let $\gamma\colon D \rightarrow M$ be an analytic immersion, and let  $D'\subseteq D$ be free. Then, there exists $A\subseteq D$ maximal with $D'\subseteq A$.
\end{lemma}
\begin{proof}
	Let $\MK$ denote the set of all free intervals $C\subseteq D$ with $D'\subseteq C$. 
We order $\MK$ by inclusion; and observe that a chain $\ML$ in $\MK$ has the upper bound 
	$B:=\bigcup_{C\in \ML} C$. In fact, $B$ is free because $g\cdot \gamma|_B \cpsim  \gamma|_B$ implies $g\cdot \gamma|_{C} \cpsim  \gamma|_{C}$ for some $C\in \ML$, hence $g\in G_\gamma$ by Lemma \ref{lemma:stabi}. Thus, by Zorn's lemma, the set of maximal elements in $\MK$ is non-empty, which proves the claim.
\end{proof}
\begin{example}
\label{gfddgfikuhjwaq}
Let $\gamma\colon B\rightarrow M$ be immersive and free. 
\begingroup
\setlength{\leftmargini}{16pt}
{
\renewcommand{\theenumi}{\alph{enumi})} 
\renewcommand{\labelenumi}{\theenumi}
\begin{enumerate}
\item
\label{gfddgfikuhjwaq1}
If $D\subset B$ is maximal, then $\gamma|_{D}$ is not self-related.
\vspace{-6pt}

\begin{proof}
Let $\ovl{\gamma}\colon I\rightarrow M$ denote the maximal analytic immersive extension of $\gamma$; and assume that the claim is wrong, i.e., that $D$ is maximal and $\gamma|_{D}$ is self-related.  
\begingroup
\setlength{\leftmarginii}{12pt}
\begin{itemize}
\item[$-$]
	Since $D\subset B$ is closed in $B$ by maximality, there exists 
\begin{align*}
	\tau\in \underbrace{B\cap \{\sup(D),\inf(D)\}}_{\subseteq\: I\:\cap\: \{\sup(D),\:\inf(D)\}}\qquad\text{with}\qquad (\tau-\epsilon,\tau+\epsilon)\subseteq B\qquad\text{for some}\qquad \epsilon >0. 
\end{align*}
	Lemma \ref{sdffsd}.\ref{sdffsd1}  then shows that 
\begin{align*}
	\gamma|_{J} = \gamma\cp\rho_\tau \quad\text{holds for an analytic diffeomorphism}\quad\rho_\tau\colon  D\supseteq J\rightarrow J'\subseteq B\subseteq I \quad\text{with}\quad \tau\in J'. 
\end{align*}
\item[$-$]
It follows that the interval  $D\subset D':= D\cup J'\subseteq B$ is free 
w.r.t.\ $\gamma$, which  
contradicts maximality of $D$.
\vspace{-10pt}

\noindent
In fact, if $g\cdot \gamma|_{D'}\cpsim \gamma|_{D'}$ holds for some $g\in G$, then we necessarily have 
\begin{align*}
	g\cdot \gamma|_{D}\cpsim \gamma|_{D}\qquad\:\:\text{or}\qquad\:\: g\cdot \gamma|_{J'}\cpsim \gamma|_{J'}\qquad\:\:\text{or}\qquad\:\: g\cdot \gamma|_{D}\cpsim \gamma|_{J'}\qquad\:\:\text{or}\qquad\:\: g\cdot \gamma|_{J'}\cpsim \gamma|_{D}.
\end{align*}
In each of theses cases, we have $g\cdot \gamma|_{D}\cpsim \gamma|_{D}$ (use $\gamma|_{J} = \gamma\cp\rho_\tau$ in the last three cases), hence $g\in G_\gamma$ as $D$ is free w.r.t.\ $\gamma$. 
This shows that $D'$ is free w.r.t.\ $\gamma$.
\qedhere
\end{itemize}
\endgroup
\end{proof} 
\item
\label{gfddgfikuhjwaq2}
If $\gamma\colon D\rightarrow M$ is a self-related free segment, then $\ovl{\gamma}\colon I\rightarrow M$ is a self-related free segment.
\vspace{-6pt}

\begin{proof} 
Since $D$ is free w.r.t\ $\ovl{\gamma}$ by \eqref{ldslkdlkdsd09s09d09ds09ds09ds09dssddsdscxcxcx}, Lemma \ref{maxim} provides $A\subseteq I$ maximal w.r.t\ $\ovl{\gamma}$ with $D\subseteq A$. Now, if $\ovl{\gamma}$ is not a free segment, then necessarily $A\subset I$ holds. Since    $\ovl{\gamma}|_A$ is self-related  (as $(\ovl{\gamma}|_A)|_D=\gamma$ is self-related), this contradicts Part \ref{gfddgfikuhjwaq1}.    
\end{proof}
\vspace{-5pt}

{\sf Example:} Let $\wm$ be a  left action on $M= \RR^2\cong\mathbb{C}$. Then, $\gamma\colon \RR\ni t\mapsto \e^{\I t}\in \CC\cong \RR^2$ is a free segment     
\deff
\phantom{{\sf Example:} }\he$\gamma|_{(-\epsilon,2\pi)}$ is a free segment  for each $\epsilon>0$. \hspace*{\fill}$\ddagger$
\end{enumerate}}
\endgroup
\end{example}
\noindent
We now briefly want to outline the results obtained in this section. For this, let $\gamma\colon I\rightarrow M$ be immersive and free, and $A\subseteq I$ maximal. Then, (as we have seen above) $A$ is closed in $I=(i',i)$,  hence  of the form $(i',i)$ or $(i',\tau]$ or $[\tau,i)$ or $[c,c']$ (compact). In the case $A=I$ there is nothing to show,  because then $I$ is the only maximal interval. In the case $A\subset I$ ($\gamma$ is free but not a free segment) we will find the following statements under the additional assumption that $\wm$ is sated:
\begingroup
\setlength{\leftmargini}{12pt}
\begin{itemize}
\item
If $\gamma$ admits no compact maximal interval, then there exists $\tau\in I$ uniquely determined such that $(i',\tau]$ and $[\tau,i)$ are the only maximal intervals, whereby
\begin{align*}
	g\cdot \gamma|_{(i',\tau]}\psim \gamma|_{[\tau,i)}\qquad\text{holds for}\qquad [g]\in G\slash G_\gamma\qquad\text{unique}.
\end{align*}
\item
If $\gamma$ admits a compact maximal interval $A=[a_-,a_+]$, then there exist unique classes $[g_{-1}],[g_{1}]\in G\slash G_\gamma$ as well as unique intervals $A_{-1},A_1$ closed in $I$ (and maximal if compact) such that $g_{\pm 1}\cdot \gamma|_A\psim \gamma|_{A_{\pm 1}}$ and $A\cap A_{\pm 1} =\{a_\pm\}$ holds. Inductively,  we will construct a decomposition of $I$ into free intervals closed in $I$ such that the corresponding subcurves are related to $\gamma|_A$ in an  analogous way. The two different cases that can occur (positive and negative case) are discussed in Sect.\ \ref{jsdjklsdjklsd}.
\end{itemize}
\endgroup
\noindent
Before we can formulate the above statements precisely, we  first need to provide some further facts and definitions; which is the content of Sect.\ \ref{ouisduoisduiosd} and Sect.\ \ref{cnccbcyxmcxyxcycyx}.
\subsubsection{Decompositions}
\label{ouisduoisduiosd}
Let $\CN$ denote the set of all subsets of $\ZZ$ that are of the form\footnote{If $\cn_-=-\infty$ holds, then $\cn_-\leq n$ means $n\in \ZZ$. Analogously, if $\cn_+=\infty$ holds, then $n\leq \cn_+$ means $n\in \ZZ$.} 
\begin{align*}
	\cn=\{n\in \ZZ_{\neq 0}\: | \: \cn_{-} \leq n \leq \cn_+\}\qquad\text{for}\qquad  \cn_-,\cn_+ \in \ZZ_{\neq 0}\cup \{-\infty,\infty\} \qquad\text{with}\qquad \cn_-< 0< \cn_+.
\end{align*}
\begin{definition}
\label{dskskjsdkjdskjdsds98ds98dsdsdsdsds}
A decomposition of an interval $D$ is a family $\{a_n\}_{n\in \cn}\subseteq \innt[D]$ with $\cn\in \CN$, such that $a_m<a_{n}$ holds for all $\cn\ni m < n\in \cn$. We  
set
\begin{align*}
A_n:=[a_{n-1},a_{n}]\quad\text{ for }\quad \cn_-< n\leq -1, \qquad\quad A_0:=&\:[a_{-1},a_1],\qquad\quad A_n:=[a_n,a_{n+1}]\quad\text{ for }\quad 1\leq n<\cn_+\\[2pt]
\text{as }&\text{well as}\\[1pt]
A_{\cn_-}:=D\cap(-\infty,a_{\cn_-}]\:\:\: \text{if}\:\:\: \cn_-\neq -\infty \:\:\: \text{holds} \qquad\quad&\text{and}\qquad\quad 
A_{\cn_+}:=D\cap [a_{\cn_+},\infty)\:\:\: \text{if}\:\:\: \cn_+\neq \infty\:\:\: \text{holds}. 
\end{align*}
\end{definition}
\noindent
Next, given a fixed analytic curve $\gamma\colon D\rightarrow M$, we let $[g]$ denote the class of $g\in G$ in $G\slash G_\gamma$; and observe  that $G\slash G_{\gamma}=G\slash G_{\gamma|_{D'}}$ holds for each interval $D'\subseteq D$ by Lemma \ref{lemma:stabi}. 
The following definitions are central: 
\begin{definition}
\label{fdsafdsdfs}
Let $\gamma\colon I\rightarrow M$ be immersive and free.
\begingroup
\setlength{\leftmargini}{16pt}
{
\renewcommand{\theenumi}{\arabic{enumi})} 
\renewcommand{\labelenumi}{\theenumi}
\begin{enumerate}
\item
\label{fdsafdsdfs1}
Let $\tau\in I=(i',i)$ be given such that $(i',\tau]$ and $[\tau,i)$ are free intervals.
\begingroup
\setlength{\leftmarginii}{12pt}
\begin{itemize}
\item
A {\bf $\boldsymbol{\tau}$-decomposition} of $\gamma$ is  a class $[g]\subseteq G_{\gamma(\tau)}\setminus G_\gamma$, such that\hspace{2pt}\footnote{We note that \eqref{dsoioidsoidslkdslkdsdsds8787ds87ds874324343} already implies  $[g]\subseteq G\setminus G_\gamma$ (as 
$\gamma$ is injective on a neighbourhood of $\tau$) as well as $[g] \subseteq G_{\gamma(\tau)}$ (by our conventions in Sect.\ \ref{repari}), hence $[g]\subseteq  G_{\gamma(\tau)}\setminus G_\gamma$.  
In particular,  we must have $G_{\gamma(\tau)}\neq \{e\}$, as $g\notin G_\gamma \ni e$.}  
\begin{align}
\label{dsoioidsoidslkdslkdsdsds8787ds87ds874324343}
g\cdot \gamma|_{(i',\tau]}\psim \gamma|_{[\tau,i)}
\quad \text{w.r.t.} \quad\mu.
\end{align}
\item
The $\boldsymbol{\tau}$-decomposition $[g]$ is said to be {\bf faithful} \deff\he $g'\cdot \gamma|_{(i',\tau]} \cpsim \gamma$ w.r.t.\ $\rho$ implies that 
\begin{align*}
\text{either}\qquad\:\:\:[g']=[e]\quad\wedge\quad\ovl{\rho}|_{(i',\tau]}=\id_{(i',\tau]}\qquad\:\:\:\text{or}\qquad\:\:\: [g']=[g]\quad\wedge\quad \ovl{\rho}|_{\dom[\mu]}=\mu\qquad\:\:\:\text{holds.}
\end{align*}
\end{itemize}
\endgroup
\item
\label{fdsafdsdfs2}
Let $A\subseteq I$ be compact and free. 
\begingroup
\setlength{\leftmarginii}{12pt}
\begin{itemize}
\item
An {\bf $\boldsymbol{A}$-decomposition} of $\gamma$ is a pair $(\{a_n\}_{n\in \cn},\{[g_n]\}_{n\in \cn})$ with $\{a_n\}_{n\in \cn}$ a decomposition of $I$ and $\{g_n\}_{n\in \cn}\subseteq G$, such that 
$A_0= A$ and\: $[g_{\pm1}]\neq [e]$ holds, as well as 
\begin{align}
\label{sdfdsf}	
g_n\cdot \gamma|_{A}\psim \gamma|_{A_n}\quad\text{w.r.t.}\quad \mu_n\qquad\quad\forall\:  n\in \cn.
\end{align} 
We set $\mu_0:=\id_{A}$ and $g_0:= e$. 
\item
An $A$-decomposition $(\{a_n\}_{n\in \cn},\{g_n\}_{n\in \cn})$ is said to be {\bf faithful} \deff the following implication holds:
\begin{align*}
	g\cdot \gamma|_A \cpsim \gamma\:\:\:\text{w.r.t.}\:\:\: \rho\qquad\Longrightarrow \qquad 
	[g]=[g_n]\:\:\:\text{and}\:\:\: \ovl{\rho}|_{\dom[\mu_n]}=\mu_n 
	\:\:\:\text{for some unique}\:\:\: n\in \cn\cup \{0\}	
\end{align*}
\end{itemize}
\endgroup
Given a fixed $A$-decomposition $(\{a_n\}_{n\in \cn},\{[g_n]\}_{n\in \cn})$  
of $\gamma$, then   
we define
\begin{align}
\label{ffffs}
\begin{split} 
	 h_n:= g_{n}\cdot g_{n+1}^{-1}\quad\text{for}\quad \cn_-\leq n\leq -2,\qquad h_{\pm 1}&:= g_{\pm 1},\qquad
	h_n:= g_{n}\cdot g_{n-1}^{-1}\quad\text{for}\quad 2\leq n\leq \cn_+,
	\\[2pt]
\text{s}&\text{o that}\\[1pt]
	h_n\cdot \gamma|_{A_{n+1}}\psim\gamma|_{A_n}\quad\text{for}\quad \cn_-\leq n\leq -1\qquad\:\:&\:\text{and}\qquad\:\: h_n\cdot \gamma|_{A_{n-1}}\psim\gamma|_{A_n}\quad\text{for}\quad 1\leq n\leq \cn_+.
\end{split}
\end{align}
\end{enumerate}}
\endgroup
\end{definition}

\subsubsection{Elementary Properties}
\label{cnccbcyxmcxyxcycyx}
Before we prove the existence of decompositions, we now first discuss their elementary properties. 

\begin{remark}
\label{odspdspodspodspods}
Assume that we are in the situation of Definition \ref{fdsafdsdfs}.\ref{fdsafdsdfs2}. 
\begingroup
\setlength{\leftmargini}{12pt}
\begin{itemize}
\item
We will see in Lemma \ref{eoder}.\ref{eoder2} below that $[g_{\pm1}]\neq [e]$ and $g_{\pm 1}\cdot \gamma|_A\psim \gamma|_{A_{\pm 1}}$ already imply that $A$ is maximal. 
\item
Conversely, if we would require $A$ to be maximal right from the beginning, then $[g_{\pm1}]\neq [e]$ would follow from $g_{\pm 1}\cdot \gamma|_A\psim \gamma|_{A_{\pm 1}}$.
\vspace{4pt}
 
In fact, assume that $A$ is maximal; and that $[g_1]=[e]$ holds, hence $\gamma|_A\psim \gamma|_{A_{1}}$. Then, if $g\cdot \gamma|_{A\cup A_1}\cpsim \gamma|_{A\cup A_1}$ holds for some $g\in G$, we necessarily have 
\begin{align*}
	g\cdot \gamma|_{A}\cpsim \gamma|_{A}\qquad\:\:\text{or}\qquad\:\: g\cdot \gamma|_{A_1}\cpsim \gamma|_{A_1}\qquad\:\:\text{or}\qquad\:\: g\cdot \gamma|_{A}\cpsim \gamma|_{A_1}\qquad\:\:\text{or}\qquad\:\: g\cdot \gamma|_{A_1}\cpsim \gamma|_{A}.
\end{align*} 
In each of theses cases, we have $g\cdot \gamma|_{A}\cpsim \gamma|_{A}$ (use $\gamma|_A\psim \gamma|_{A_{1}}$ in the last three cases), hence $g\in G_\gamma$. This shows that $A\cup A_1$ is free, which contradicts that $A$ is maximal. The same argumentation yields a contradiction to the assumption $[g_{-1}]= [e]$.\hspace*{\fill}$\ddagger$
\end{itemize}
\endgroup
\end{remark}
\noindent  
Similar arguments as in (the second point in) Remark \ref{odspdspodspodspods} yield the following statements:
\begin{lemma}
\label{eoder}
Let $\gamma\colon I\rightarrow M$ be immersive and free.
\begingroup
\setlength{\leftmargini}{16pt}
{
\renewcommand{\theenumi}{\arabic{enumi})} 
\renewcommand{\labelenumi}{\theenumi}
\begin{enumerate}
\item
\label{eoder1}
If $[g]$ is a $\tau$-decomposition of $\gamma$, then we have the implication:
\begin{align*}
	A\subseteq I \quad\text{free}\qquad\quad\Longrightarrow\qquad\quad A\subseteq (i',\tau]\quad\vee\quad A\subseteq [\tau,i)
\end{align*}
In particular, the intervals $(i',\tau]$ and $[\tau,i)$ are maximal, and the only maximal intervals.
\item
\label{eoder2}
If $(\{a_n\}_{n\in \cn},\{[g_n]\}_{n\in \cn})$ is an $A$-decomposition of $\gamma$, then $A$ is maximal. Moreover, if $\wm$ is sated, then
\begingroup
\setlength{\leftmarginii}{16pt}
\begin{itemize}
\item[a)]
\vspace{-4pt}
$\cn_-=-\infty$ implies that for each $t\in I$, we have $a_n<t$ for some $n\in \cn$.
\item[b)]
\vspace{2pt}
$\cn_+=\phantom{-}\infty$ implies that for each $t\in I$, we have $t< a_n$ for some $n\in \cn$.
 \end{itemize}
\endgroup
\item
\label{eoder3}
If $[g]$ is a $\tau$-decomposition of $\gamma$, then there cannot exist any other decomposition of $\gamma$. 
\end{enumerate}}
\endgroup
\end{lemma}
\begin{proof}
Confer Appendix \ref{appC1}.
\end{proof}
\noindent
Moreover, from Lemma \ref{sdffsd} and Lemma \ref{sshift} we obtain the following two statements.
\begin{lemma}
\label{taufaith}
Each $\tau$-decomposition is faithful. 
\end{lemma}
\begin{proof}
Confer Appendix \ref{appC2}.
\end{proof}
\begin{lemma}
\label{dsfdsffds}
Let $\wm$ be sated; and let $\gamma\colon I\rightarrow M$ be immersive and free. If $A$ is a compact free interval, then  there exists at most one $A$-decomposition of $\gamma$; and this $A$-decomposition is faithful if it exists.\footnote{In other words: If $\wm$ is sated, then an $A$-decomposition $(\{a_n\}_{n\in \cn},\{[g_n]\}_{n\in \cn})$ of $\gamma$ is automatically faithful, as well as  unique in the sense that the index set $\cn\in \CN$, the reals $\{a_n\}_{n\in \cn}$, the classes $\{[g_n]\}_{n\in \cn}$, and the analytic diffeomorphisms $\{[\mu_n]\}_{n\in \cn}$ are already uniquely determined by the compact free interval $A$ (which is necessarily maximal by Lemma \ref{eoder}.\ref{eoder2}).} 
\end{lemma}
\begin{proof}
Confer Appendix \ref{appC3}.
\end{proof}
\subsection{Existence of Decompositions}
\label{pofdpofdpofdpofdpofd}
After we have discussed 
 the elementary properties of decompositions in Sect.\ \ref{sfd}, we now prove their existence. We first treat the situation where a compact maximal interval exists; and then discuss the other (non-compact) case. For this, we will need the following two statements:
\begin{lemma}
\label{freemax}
Let $\gamma\colon I \rightarrow M$ and $\gamma'\colon I' \rightarrow M$ be analytic immersions, such that $g\cdot \gamma|_D = \gamma'\cp \rho$ holds for some $g\in G$ and some analytic diffeomorphism $\rho\colon I\supseteq D\rightarrow D'\subseteq I'$. 
\begingroup
\setlength{\leftmargini}{16pt}
{
\renewcommand{\theenumi}{{\arabic{enumi}})} 
\renewcommand{\labelenumi}{\theenumi}
\begin{enumerate}
\item
\label{freemax1}
	If $D$ is free w.r.t.\ $\gamma$, then $D'$ is free w.r.t.\ $\gamma'$.
\item
\label{freemax2}
	If $D$ is compact and maximal w.r.t.\ $\gamma$, then $D'$ is compact and maximal w.r.t.\ $\gamma'$. 	
\end{enumerate}}
\endgroup
\end{lemma}
\begin{proof}
The proof is elementary, and can be found in Appendix \ref{appC4}.
\end{proof}

\begin{lemma}
\label{lemmfreeneig}
Let $\wm$ be sated, $\delta\colon [t',t]\rightarrow M$ immersive and free, and $[a',a]\subseteq [t',t]$ a free interval. Then, the following implications hold:
\begin{align}
\label{fre1}
	a<t \hspace{4pt}\qquad&\Longrightarrow \qquad [a,k]\quad\:\text{ is free for some}\quad a<k\leq t,\\
\label{fre2}
	t'<a' \qquad&\Longrightarrow \qquad [k',a']\quad\text{is free for some}\quad t'\leq k'<a'.
\end{align}
\end{lemma}
\begin{proof}
Confer Appendix \ref{appC5}.
\end{proof}
\subsubsection{The Compact Case} 
As already mentioned above, we now prove the existence of an $A$-decomposition for the case that a given  free analytic immersive curve admits a compact maximal interval. The particular cases that can occur are discussed in detail in Sect.\ \ref{jsdjklsdjklsd}. 
\begin{proposition}
\label{prop:shifttrans}
Let $\wm$ be sated, $\gamma\colon [t',t]\rightarrow M$ immersive and free, and $[a',a]\subseteq [t',t]$ maximal. 
\begingroup
\setlength{\leftmargini}{15pt}
\begin{enumerate}
\item
\label{prop:shifttrans1}
If $a<t$ holds,  
then there exists $[e]\neq [g]\in G\slash G_\gamma$ uniquely determined by the condition
\begin{align}
\label{dsoioidsoisdoidsds98sd98ds98s98ddsdsdsdsds1}
g\cdot \gamma|_{[a',a]}\cpsim \gamma|_{[a,k]} \qquad\quad\forall\: a<k\leq t,
\end{align}
\vspace{-28pt}

and we have
\vspace{-3pt}
\begin{align}
\label{gl11}
	\text{ either}\:\qquad\quad  g\cdot \gamma|_{[a',s]}\psim_+ \gamma|_{[a,s']}\qquad&\text{for some}\qquad a'<s\leq a<s'\qquad\quad\qquad\quad\:\\[2pt]
\label{gl12}
	\text{ or}\qquad\qquad g\cdot \gamma|_{[s,a]}\hspace{2pt}\psim_- \gamma|_{[a,s']}\qquad&\text{for some}\qquad\qquad\:  s<a<s'.
\end{align}
\item
\label{prop:shifttrans2}
If $t'<a'$ holds, then there exists $[e]\neq [g']\in G\slash G_\gamma$ uniquely determined by the condition
\begin{align}
\label{dsoioidsoisdoidsds98sd98ds98s98ddsdsdsdsds2}
g'\cdot \gamma|_{[a',a]}\cpsim \gamma|_{[k',a']}\qquad\quad\forall\: t'\leq k'<a',
\end{align}
\vspace{-28pt}

and we have
\vspace{-3pt}
\begin{align}
\label{gl21}
	\text{ either}\:\qquad\quad g'\cdot \gamma|_{[s,a]}\:\hspace{1pt}\psim_+ \gamma|_{[s',a']}\quad\:\:\:\:&\text{for some}\qquad s'<a'\leq s<a\qquad\quad\qquad\quad\\[2pt]
\label{gl22}
	\text{ or}\qquad\qquad g'\cdot \gamma|_{[a',s]}\psim_- \gamma|_{[s',a']}\quad\:\:\:\:&\text{for some}\qquad   s'<a'<s.
\end{align}
\end{enumerate}
\endgroup
\end{proposition} 
\begin{remark}
\label{fdfdffddffddffddf}
Figuratively speaking,
\begingroup
\setlength{\leftmargini}{12pt}
\begin{itemize}
\item 
\eqref{gl11} means that $g$\: ``right shifts''  
the segment 
$\gamma|_{[a',s]}$ to $\gamma|_{[a,s']}$.
\item
\eqref{gl12} means that $g$\: ``flips'' the segment $\gamma|_{[s,a]}$ at $\gamma(a)$ to $\gamma|_{[a,s']}$. 
\item 
\eqref{gl21} means that $g'$ ``left shifts'' the segment $\gamma|_{[s,a]}$ to $\gamma|_{[s',a']}$.
\item
\eqref{gl22} means  that $g'$ ``flips'' the segment $\gamma|_{[a',s]}$ at $\gamma(a')$ to $\gamma|_{[s',a']}$. 
\end{itemize}
\endgroup
\end{remark}
\begin{proof}[Proof of Proposition]
It suffices to prove Part \ref{prop:shifttrans1}, because Part \ref{prop:shifttrans2} follows analogously (or alternatively from  Part \ref{prop:shifttrans1} and  
Lemma \ref{freemax}, just by replacing $\gamma$ by $\gamma \cp \inv$ for $\inv\colon  [t',t]\ni s\mapsto t+t'-s\in [t',t]$). 
\vspace{6pt}

\noindent
Assume now first that there exists  $[e]\neq [g]\in G\slash G_\gamma$ with \eqref{dsoioidsoisdoidsds98sd98ds98s98ddsdsdsdsds1},  such that \eqref{gl11} or \eqref{gl12} holds:
\begingroup
\setlength{\leftmargini}{11pt}
\begin{itemize}
\item
Then automatically either \eqref{gl11} or \eqref{gl12} holds; because, if both \eqref{gl11} and \eqref{gl12} hold, then we have $\gamma|_{[a',r']}\psim_-\gamma|_{[r,a]}$ for certain $a'<r'$ and $r<a$, which contradicts Corollary \ref{dfdgttrgf} as $\gamma$ is immersive.
\item 
For $q\in G$, we have the implications
\begin{align*}
	q\cdot \gamma|_{[a',a]}\cpsim \gamma|_{[a,k]}\qquad\forall\: a<k\leq t\qquad\qquad&\Longrightarrow\qquad\qquad q\cdot \gamma|_{[a',a]}\cpsim g\cdot \gamma|_{[a',a]}\\[2pt]
	&\Longrightarrow\qquad\qquad \hspace{29pt}[q]=[g]
\end{align*}
(in the first step we have used \eqref{gl11}\slash\eqref{gl12},  and in the second step we have used that $[a',a]$ is free). 
This proves the uniqueness statement.
\end{itemize}
\endgroup
\noindent
Since both \eqref{gl11} and \eqref{gl12} obviously imply   \eqref{dsoioidsoisdoidsds98sd98ds98s98ddsdsdsdsds1}, it 
remains to show that there exists $[e]\neq [g]\in G\slash G_\gamma$ such that  \eqref{gl11} or \eqref{gl12} holds. We observe the following:
\begingroup
\setlength{\leftmargini}{11pt}
\begin{itemize}
\item
By Lemma \ref{lemmfreeneig}  
 and Corollary \ref{dfdsasasasassa}, 
there exists $a<\lambda\leq t$ such that $\gamma|_{[a,\lambda]}$ is a free segment 
 and an embedding.  
We fix
$\{t_n\}_{n\in \NN}\subseteq (a,\lambda]$ strictly decreasing with $\lim_n t_n=a$.
\item
The interval $[a',t_n]$ is not free for each $n\in \NN$, because $[a',a]$ is maximal. Hence, there exists 
$$
\{g_n\}_{n\in \NN}\subseteq G\backslash G_\gamma\qquad\text{with}\qquad g_n\cdot \gamma|_{[a',t_n]}\cpsim \gamma|_{[a',t_n]}\qquad\text{for all}\qquad n\in \NN.
$$ 
\end{itemize}
\endgroup 
\noindent
We observe the following:
\begingroup
\setlength{\leftmargini}{11pt}
\begin{itemize}
\item[$\triangleright$]
For each $n\in\NN$, we necessarily have
\begin{align*}
	g_n\cdot \gamma|_{[a',a]}\cpsim \gamma|_{[a,t_n]}\qquad\quad\text{or}\qquad\quad g_n\cdot \gamma|_{[a,t_n]}\cpsim \gamma|_{[a',a]},
\end{align*}
because 
 $[a',a]$ and $[a,t_n]$ are free intervals and $g\in G\setminus G_\gamma$ holds. 
\item[$\triangleright$]
Hence, replacing $g_n$ by $g_n^{-1}$ if necessary, we can assume that for each $n\in \NN$, we have 
\begin{align*}
 \gamma|_{[a,t_n]} \cpsim g_n\cdot\gamma|_{[a',a]} \qquad\:\: \stackrel{\text{Lemma }\ref{lemma:BasicAnalyt2}}{\Longrightarrow} \qquad\:\:  (\gamma|_{[a',\lambda]})|_{J_n}=g_n\cdot\gamma|_{[a',a]}\cp \rho_n
\end{align*}
for an analytic diffeomorphism $\rho_n\colon [a',\lambda] \supseteq [a,t_n]\supseteq J_n\rightarrow J'_n\subseteq [a',a]$.	
\item[$\triangleright$]
For each $n\in \NN$, we define\hspace*{\fill} (observe that $C_n$ is compact by Lemma \ref{lem:maxextKomp} )
\begin{align*}
	[c_n',c_n]\equiv\CM_n:= [a',\lambda]\cap \ovl{\rho}_n^{-1}([a',a])\qquad\quad\text{as well as}\qquad\quad L_n:=\ovl{\rho}_n(C_n)\subseteq [a',a].
\end{align*}
Then, Lemma \ref{corgleich} yields
\vspace{-5pt}
\begin{align}
\label{zitiifizu}
	(\gamma|_{[a',\lambda]})|_{C_n}=g_n\cdot \gamma|_{[a',a]}\cp \ovl{\rho}_n|_{C_n}\qquad\quad\stackrel{(*)}{\Longrightarrow}\qquad\quad c_n'\in [a,t_n).\hspace{35pt}
\end{align}
For the implication $(*)$, observe that we necessarily have $c_n'\leq \inf(J_n)< t_n$. Moreover, if $c_n'<a$ holds, then the  left side of \eqref{zitiifizu} implies   
$\gamma|_{[a',a]}\cpsim g_n\cdot\gamma|_{[a',a]}$, hence $g_n\in G_\gamma$ which contradicts the choices. 
\item[$\triangleright$]
The right side of \eqref{zitiifizu} in particular shows\hspace*{\fill}($a'<a$)
\begin{align}
\label{podspodsds}
	a'<c_n'\qquad\forall\: n\in \NN\qquad\quad\stackrel{\eqref{ddddd}}{\Longrightarrow}\qquad\quad \ovl{\rho}_n(c'_n)\in \{a',a\}\qquad\forall\: n\in \NN.
\end{align} 
\end{itemize}
\endgroup
\noindent 
Now, there are two different cases that can occur:
\vspace{8pt}

\noindent
{\bf Case I} 
\vspace{6pt}

\noindent
Assume that $\bigcup_{n\in \NN}\{[g_n]\}$ is finite. Then, passing to a subsequence, we can assume that $[g_n]=[g]$ holds for all $n\in \NN$, for some $[e]\neq [g]\in G\slash G_\gamma$. Then, \eqref{zitiifizu} yields \hspace*{\fill}($a\leq c_n'<c_n\leq \lambda$)
\begin{align}
\label{rtgffg}
(\gamma|_{[a,\lambda]})|_{C_n}=g\cdot \gamma|_{[a',a]}\cp \ovl{\rho}_n|_{C_n}\qquad\quad \forall\:n\in \NN.
\end{align}
By \eqref{podspodsds}, we can pass to a subsequence to achieve that $\ovl{\rho}_n(c'_n)=b$ holds for all $n\in \NN$, for some $b\in \{a',a\}$. Then, since $\gamma|_{[a,\lambda]}$ is injective with $c_n'\in [a,t_n]\subseteq  [a,\lambda]$, we obtain 
\begin{align*}
	\gamma|_{[a,\lambda]}(c_n')\stackrel{\eqref{rtgffg}}{=}g\cdot \gamma(b)\qquad \forall\: n\in \NN\qquad\quad\Longrightarrow\qquad\quad  c_0'=c'_n\qquad \forall\: n\in \NN\qquad\stackrel{\lim_n t_n=a}{\Longrightarrow}\qquad c_0'=a.
\end{align*}
This shows $a=c_0'$ as well as $\ovl{\rho}_0(a)=\ovl{\rho}_0(c_0')=b$. Since $\ovl{\rho}_0(c_0)\in [a',a]$ holds, we obtain:
\begingroup
\setlength{\leftmargini}{11pt}
\begin{itemize}
\item
If $b=a'$ holds, then we have $\ovl{\rho}_0(a)=a'$; hence, $\gamma|_{[a,c_0]}=g\cdot \gamma|_{[a',\ovl{\rho}_0(c_0)]}\cp\ovl{\rho}_0|_{[a,c_0]}$ by \eqref{rtgffg} -- implying  
 \eqref{gl11}. 
 \item
If $b=a$ holds,\hspace{2.8pt} then we have $\ovl{\rho}_0(a)=a$;\hspace{2.8pt} hence, $\gamma|_{[a,c_0]}=g\cdot \gamma|_{[\ovl{\rho}_0(c_0),a]}\:\cp\ovl{\rho}_0|_{[a,c_0]}$ by \eqref{rtgffg} -- implying   
 \eqref{gl12}. 
\end{itemize}
\endgroup
\vspace{8pt}

\noindent
{\bf Case II} 
\vspace{6pt}

\noindent
Assume that $\bigcup_{n\in \NN}\{[g_n]\}$ is infinite. Then,   passing to a subsequence, we can assume that $[g_n]\neq [g_m]$ holds for all $\NN\ni n\neq m\in \NN$. Then, $C_n\cap C_m$ cannot contain an open interval  for $n\neq m$; because, if $J \subseteq C_n\cap C_m$ is an open interval such that $\gamma|_J$ is an embedding (Corollary \ref{dfdsasasasassa}), then we have
\begin{align*}
	(g_n\cdot \gamma)(\ovl{\rho}_n(J))=\gamma(J)=(g_m\cdot \gamma)(\ovl{\rho}_m(J))\qquad&\Longrightarrow\qquad g_m^{-1}\cdot g_n \cdot \gamma|_{[a',a]}\cpsim \gamma|_{[a',a]} \\ \qquad&\Longrightarrow\qquad g_m^{-1}\cdot g_n \in G_\gamma,
\end{align*}
hence $m=n$ by the choice of $\{g_n\}_{n\in \NN}$. For $\NN\ni n\neq m\in \NN$, we thus have  
\begin{align*}
 \text{either}\qquad\quad c'_n<c_n\leq c'_m<c_m\qquad\quad \text{or}\qquad\quad\ c'_m<c_m\leq c'_n<c_n.
\end{align*}
Additionally \eqref{zitiifizu} yields
\begin{align*}
a\leq c'_n<c_n\leq \lambda\qquad\text{with}\qquad c_n'<t_n\qquad\text{for all}\qquad n\in \NN. 
\end{align*}
Since $\lim_n t_n = a$ holds,  we thus can pass to a subsequence  to achieve  
\begin{align*}
	a<c'_{n+1}<c_{n+1}\leq c_n' <c_n<\lambda \qquad\quad\forall\: n\in \NN.
\end{align*}
\begingroup
\setlength{\leftmargini}{18pt}
\begin{itemize}
\item[(i)]
We  have $\lim_n c_n'=a$, as  $\lim_n t_n=a$ holds, and because $a\leq c'_{n}<t_{n}$ holds for each $n\in \NN$.
\item[(ii)]
For each $n\in \NN$, we have  
\begin{align*}
C_{n+1}\subseteq [a,c_{n}']\subseteq (a',\lambda) \qquad\quad&\stackrel{\eqref{ddddd}}{\Longrightarrow}\qquad\quad L_{n+1}= \ovl{\rho}(C_{n+1})=[a',a]\\[3pt]
&\Longrightarrow\qquad\quad 
\hspace{7pt} g_{n+1}\cdot \gamma([a',a])=\gamma(C_{n+1})\subseteq \gamma([a,c'_{n}]).
\end{align*}
\end{itemize}
\endgroup
\noindent
Since $\lim_n c'_{n}=a$ holds by Point (i), Point (ii) contradicts that $\wm$ is sated. \qedhere
\end{proof}
\noindent
Now, let $\wm$ is sated; and let $\gamma\colon I\rightarrow M$ be free immersive, as well as $[a',a]\subseteq (i',i)=I$ maximal: 
\begingroup
\setlength{\leftmargini}{11pt}
\begin{itemize}
\item
Proposition \ref{prop:shifttrans} applied to $\gamma|_{[t',t]}$ for some fixed $i'<t'<a'<a<t<i$ shows:
\begingroup
\setlength{\leftmarginii}{11pt}
\begin{itemize}
\item
We have either \eqref{gl11} or \eqref{gl12}.
\item
We have either \eqref{gl21} or \eqref{gl22}.
\end{itemize}
\endgroup
\item
It is clear from the discussion in Sect.\ \ref{repari} that the following implications hold\footnote{For instance, the implication in the first line follows from Lemma \ref{dgfggfg} with $\bullet\equiv +$, $D\equiv[a,a']$, $D'\equiv [a,i)$, $\gamma\equiv g\cdot \gamma|_{D}$, $\gamma'\equiv \gamma|_{D'}$, some fixed $a'<t<s$, and $t'=\rho(t)$ for $\rho\colon K\equiv [a',s]\rightarrow [a,s']\equiv K'$ the unique (positive) analytic diffeomorphism that corresponds to the relation \eqref{gl11}.}
\begin{align*}
\begin{array}{lllllll}
\eqref{gl11}&\quad\Longrightarrow\quad &\text{we either have}\qquad \hspace{3pt}g\cdot \gamma|_{[a',j')}&\hspace{-6pt}\psim_+ \gamma|_{[a,i)} \quad &\text{or}\qquad  \hspace{3pt}g\cdot \gamma|_{[a',a]}&\hspace{-6pt}\psim_+ \gamma|_{[a,b]},\quad&\text{}\\[4pt]
\eqref{gl12}&\quad\Longrightarrow\quad &\text{we either have}\qquad\hspace{3pt} g\cdot \gamma|_{(j,a]}&\hspace{-6pt}\psim_- \gamma|_{[a,i)} \quad &\text{or}\qquad  \hspace{3pt}g\cdot \gamma|_{[a',a]}&\hspace{-6pt}\psim_- \gamma|_{[a,b]},\quad&\text{}\\[4pt]
\eqref{gl21}&\quad\Longrightarrow\quad &\text{we either have}\qquad g'\cdot \gamma|_{(j,a]}&\hspace{-6pt}\psim_+ \gamma|_{(i',a']} \quad &\text{or}\qquad  g'\cdot \gamma|_{[a',a]}&\hspace{-6pt}\psim_+ \gamma|_{[b',a']},\quad&\text{}\\[4pt]
\eqref{gl22}&\quad\Longrightarrow\quad &\text{we either have}\qquad g'\cdot \gamma|_{[a',j')}&\hspace{-6pt}\psim_- \gamma|_{(i',a']} \quad &\text{or}\qquad  g'\cdot \gamma|_{[a',a]}&\hspace{-6pt}\psim_- \gamma|_{[b',a']},\quad&\text{}
\end{array}(\natural)
\end{align*}
for certain unique $a'\leq j',j\leq a$ and $i'<b'<a'<a<b<i$.
\end{itemize}
\endgroup
\noindent
Now, in the cases on the right side of $(\natural)$, the intervals $[a,b]$ and $[b',a']$ are maximal by Lemma \ref{freemax}.\ref{freemax2}. Hence, we can apply the above arguments inductively to conclude the following: 
\begin{corollary}
\label{sdffdfdexi}
Let $\wm$ be sated; and let $\gamma\colon I\rightarrow M$ be immersive and free. Then, each compact maximal interval $A\subseteq I$ admits an  $A$-decomposition of $\gamma$.
\end{corollary}
\begin{proof}
Applying the above arguments inductively, we obtain a decomposition $\{a_n\}_{n\in \cn}$ of $I$ with $A_0= A$  together with a family $\{h_n\}_{\cn}\subseteq G\backslash G_\gamma$, such that 
\begin{align*}
h_n\cdot \gamma|_{A_{n+1}}\psim \gamma|_{A_{n}}\quad\text{for all}\quad \cn_-\leq n\leq -1\qquad\quad\text{and}\qquad\quad h_n\cdot \gamma|_{A_{n-1}}\psim \gamma|_{A_{n}}\quad\text{for all}\quad 1\leq n\leq \cn_+
\end{align*}
holds. 
The desired $A$-decomposition of $\gamma$ is then given by $(\{a_n\}_{n\in \cn},\{[g_n]\}_{n\in \NN})$, with
\begingroup
\setlength{\leftmargini}{12pt}
\begin{itemize}
\item
$g_n:=h_n\cdot {\dots}\cdot h_{-1}$\:\: for all\:\: $\cn_-\leq n\leq -1$,
\item
$g_n:=h_n\cdot {\dots}\cdot h_{1}$\:\:\hspace{6.2pt} for all\:\:\hspace{7.3pt} $1\leq n\leq \cn_+$.\qedhere
\end{itemize}
\endgroup
\end{proof}
\subsubsection{The Non-Compact Case}
In the previous subsection, we have shown the existence of decompositions for the case that a given free immersive curve admits a compact maximal interval. In this  subsection, we use Lemma \ref{lemmfreeneig} and Proposition \ref{prop:shifttrans} to prove  that a free immersive curve that is not a free segment by itself either admits a compact maximal interval or a (necessarily unique) $\tau$-decomposition. For this, we need the following (technical) statement:
\begin{lemma}
\label{asaqwqw}
Let $\wm$ be sated, and $\delta\colon I\rightarrow M$ immersive and free. 
\begingroup
\setlength{\leftmargini}{16pt}
\begin{enumerate}
\item
\label{asaqwqw1}
Let $[a',a]$ be maximal w.r.t.\ $\gamma:= \delta|_{[t',t]}$ with $t'= a'$ and $a<t\in I$. Additionally assume that the condition \eqref{gl11} holds for $\gamma$  with $[e]\neq [g]\in G\slash G_\gamma$. Then, $[a',a]$ is maximal w.r.t.\ $\delta$.
\item
\label{asaqwqw2}
Let $[a',a]$ be maximal w.r.t.\ $\gamma:= \delta|_{[t',t]}$ with  $I\ni t'<a'$ and $t=a$. Additionally assume that the condition \eqref{gl21} holds for $\gamma$ with $[e]\neq [g']\in G\slash G_\gamma$. Then, $[a',a]$ is maximal w.r.t.\ $\delta$. 
\end{enumerate}
\endgroup
\end{lemma}
\begin{proof}
Confer Appendix \ref{appC6}. 
\end{proof}
\begin{proposition}
\label{shifty}
Let $\wm$ be sated, and $\delta\colon I\rightarrow M$ immersive and free. If $\delta$ is not a free segment by itself, then $\delta$ either admits a (necessarily unique) $\tau$-decomposition or a compact maximal interval.
\end{proposition}
\begin{proof}
First, it is clear from Lemma \ref{eoder}.\ref{eoder3} that only one of the mentioned situations can occur. Hence, we only have to prove   
  that the non-existence of a compact maximal interval implies the existence of a $\tau$-decomposition of $\delta$. Assume thus now that $\delta$ admits no compact maximal interval. Then, there necessarily exists some $\tau_-\in I$ such that $(i',\tau_-]$ is maximal or some $\tau_+\in I$ such that $[\tau_+,i)$ is maximal: 
\begingroup
\setlength{\leftmargini}{11pt}
\begin{itemize}
\item
Since $\delta$ is free, Lemma  \ref{maxim} provides a maximal interval $A\subseteq I$. Then $A$ is necessarily closed in $I$   by our discussions in the beginning of Sect.\ \ref{sfd}  (see \eqref{dfdffdfdfdfd}).  
\item
Since $I$ is not free by assumption, we must have $A\subset I$. Since $A$ is closed in $I$ as well as non-compact by assumption, we  
necessarily have $A=(i',\tau_-]$ for some $\tau_-\in I$ or $A=[\tau_+,i)$ for some $\tau_+\in I$.
\end{itemize}
\endgroup  
\noindent
We claim that there exists some $\tau\in I$, such that both $(i',\tau]$ and $[\tau,i)$ are free intervals:
\vspace{6pt}

\noindent
\textit{Proof of the Claim.}
We only discuss the case $A=[\tau_+,i)$ (the case $A=(i',\tau_-]$ is treated analogously):
\vspace{-4pt}
\begingroup
\setlength{\leftmargini}{12pt}
\begin{itemize}
\item
Let $i'<t'<\tau_+<t<i$ be fixed: 
\begingroup
\setlength{\leftmarginii}{12pt}
\begin{itemize}
\item[$\triangleright$]
Then, $[\tau_+,t]$ is free w.r.t.\ $\delta$, as $A$ is free and $[\tau_+,t]\subseteq A$. 
Hence, $[\tau_+,t]$ is free w.r.t\ $\gamma:= \delta|_{[t',t]}$ by \eqref{ldslkdlkdsd09s09d09ds09ds09ds09dssddsdscxcxcx}. 
\item[$\triangleright$]
Lemma \ref{lemmfreeneig}  yields some $t'<s<\tau_+$ such that $[s,\tau_+]$ is  free w.r.t.\ $\gamma$. Then, $[s,\tau_+]$ is  free w.r.t.\ $\delta$ by \eqref{ldslkdlkdsd09s09d09ds09ds09ds09dssddsdscxcxcx}. 
\item[$\triangleright$]
Lemma \ref{maxim} provides some (w.r.t\ $\delta$) maximal 
 $B\subseteq I$  with $[s,\tau_+]\subseteq B$.
\end{itemize}
\endgroup 
\item
Then, $B$ is necessarily of the form $(i',\tau']$ for some $\tau_+\leq \tau'$, because: 
\begingroup
\setlength{\leftmarginii}{12pt}
\begin{itemize}
\item 
$I$ is not free, and $B$ is closed in $I$ as maximal (see \eqref{dfdffdfdfdfd}),
\item
 $\delta$ admits no compact maximal interval (by assumption), 
\item
$B$ cannot be of the form $[\tau'',i)$ for $\tau''\leq s<\tau_+$, because $A=[\tau_+,i)$ is maximal.
\end{itemize}
\endgroup
\item
We set $\tau:=\tau_+$. Then, $[\tau,i)=A$ is free by construction, and $(i',\tau]$ is free because $(i',\tau_+]\subseteq B$ holds.
\hspace*{\fill} 
$\mathsmaller{\square}$
\end{itemize}
\endgroup
\noindent
It remains to show that $g\cdot \gamma|_{(i',\tau]}\psim \gamma|_{[\tau,i)}$ holds for some $g\in G\backslash G_\delta$. Since $(i',\tau]$, $[\tau,i)$ are free, we have:
\begin{align}
\label{kjdskjdskjdskjds98ds98ds98ds98sd9898dsdsdssd}
 i'<a'<\tau<t<i\qquad\quad\Longrightarrow\qquad\quad [a',\tau]\quad\text{is free}\qquad\wedge\qquad
 [\tau,t]\quad\text{is free.}
\end{align} 
We fix $\tau<t<i$, and proceed as follows: 
\begingroup
\setlength{\leftmargini}{12pt}
\begin{itemize}
\item
Assume that there exists  $i'<a'<\tau$, such that $[a',\tau]$ is maximal w.r.t.\ $\gamma:= \delta|_{[t',t]}$ for $t'=a'$: 
\begingroup
\setlength{\leftmarginii}{12pt}
\begin{itemize}
\item[$\triangleright$]
Since $[a',\tau]$ is compact,  $[a',\tau]$ cannot be maximal w.r.t.\ $\delta$ ($\delta$ admits no compact maximal interval). 
\item[$\triangleright$]
Hence, Proposition \ref{prop:shifttrans} together with  
Lemma \ref{asaqwqw} implies that    
\eqref{gl12} holds with $a\equiv\tau$ there, i.e.,   
there exist $g\in  G\backslash G_\gamma=G\backslash G_\delta$ and $s< \tau<s'$, with 
\begin{align*}
g\cdot \gamma|_{[s,\tau]}\psim_{\tau,\tau}\gamma|_{[\tau,s']}\qquad\Longrightarrow\qquad
	g\cdot \delta|_{[s,\tau]}\psim_{\tau,\tau}\delta|_{[\tau,s']}\quad\:\: \stackrel{\text{Sect.\ }\ref{repari}}{\Longrightarrow}\quad\:\:  g\cdot \delta|_{(i',\tau]}\psim\delta|_{[\tau,i)}.
\end{align*} 
\end{itemize}
\endgroup
\item
Assume that the assumption in the previous point is not fulfilled, i.e., 
 that $[a',\tau]$ is not maximal w.r.t.\ $\delta|_{[a',t]}$ for all $i'<a'<\tau$. We fix $\{a'_n\}_{n\in \NN}\subseteq (i',\tau)$ decreasing with $\lim_n a'_n=i'$, and proceed as follows:  
\begingroup
\setlength{\leftmarginii}{12pt}
\begin{itemize}
\item[$\triangleright$]
For each $n\in \NN$, we fix $A_n\subseteq [a'_n,t]$ maximal w.r.t.\ $\delta|_{[a'_n,t]}$ such that $[a'_n,\tau]\subseteq A_n$ holds (Lemma \ref{maxim} and the left side of \eqref{kjdskjdskjdskjds98ds98ds98ds98sd9898dsdsdssd}). 
\item[$\triangleright$]
Then, for each $n\in \NN$, we have  $A_n=[a'_n,t_n]$ for some $\tau<t_n\leq t$, and $A_n$ is free w.r.t\ $\delta$ by \eqref{ldslkdlkdsd09s09d09ds09ds09ds09dssddsdscxcxcx}. 
\item[$\triangleright$]
If $t_n=t$ holds for all $n\in \NN$, then $(i',t]=\bigcup_{n\in \NN}A_n$ is free (same argument as in the proof of Lemma \ref{maxim}); which contradicts maximality of $(i',\tau]$, as $\tau<t$ holds.
\item[$\triangleright$]
Consequently, there exist $a'<\tau<a<t$, such that $[a',a]$ is maximal w.r.t.\ $\gamma\equiv\delta|_{[a',t]}$. Since $[a',a]$ cannot be maximal w.r.t.\ $\delta$ ($\delta$ admits no compact maximal interval),  \eqref{gl12} holds by Proposition \ref{prop:shifttrans} and 
Lemma \ref{asaqwqw}.  
 Since $a'<\tau<a<t$ holds, we obtain
\begin{align*}
\eqref{gl12}\qquad
\Longrightarrow\qquad g\cdot \gamma|_{[\tau,a]}\cpsim \gamma|_{[a,t]}
\qquad\Longrightarrow\qquad
	g\cdot \delta|_{[\tau,t]} =g\cdot \gamma|_{[\tau,t]}\cpsim \gamma|_{[\tau,t]}=\delta|_{[\tau,t]} 
\end{align*}
with $g\in G\backslash G_\gamma= G\backslash G_\delta$; which 
contradicts that $[\tau,t]$ is free, by the right side of \eqref{kjdskjdskjdskjds98ds98ds98ds98sd9898dsdsdssd}.\qedhere
\end{itemize}
\endgroup
\end{itemize}
\endgroup
\end{proof}
\begin{example}
\label{dfggf}
Let $G=\SO(2)$ act via rotations on $M= \RR^2$. This  action is sated by Remark \ref{remmmmmi}.\ref{regp3} (even regular as proper by compactness of $\SO(2)$). 
Let $g_\pi\in \SO(2)$ denote the rotation by the angle $\pi$. Then, 
\begin{align*}
	\gamma\colon (-\infty,\infty)\rightarrow \RR^2,\qquad 
	t \mapsto (t,t^3)
\end{align*}  
admits the $0$-decomposition $[g_\pi]$: \hspace*{\fill} (\:$\tau\equiv 0$, $(i',\tau]\equiv (-\infty,0]$, $[\tau,i)\equiv [0,\infty)$) 
\begingroup
\setlength{\leftmargini}{12pt}
\begin{itemize}
\item 
$
g_\pi\cdot \gamma|_{(-\infty,0]}\psim \gamma|_{[0,\infty)}\:\:  \text{ holds\:\: w.r.t.\:\: the\:\: analytic\:\: diffeomorphism } \:\: \mu\colon (-\infty,0]\ni t\mapsto -t\in  [0,\infty).
$ 
\item 
The intervals $(-\infty,0]$, $[0,\infty)$ are free, because we have the equivalence:
\begin{align*}
g\cdot \gamma(t)=\gamma(t')\qquad\text{for}\qquad  t,&\: t'\neq 0\qquad\text{and}\qquad g\in \SO(2)\\[3pt]
&\Longleftrightarrow\\[2pt]
\text{either}\qquad g=g_\pi\:\:\:\wedge\:\:\: t'=-&t\quad\:\:\:\:\text{or}\qquad  g=\id_{\RR^2}\:\:\:\wedge\:\:\: t=t'\hspace{40pt}
\end{align*}
Notably,  maximality of these intervals is clear from:
\begin{align*}
\forall\:\epsilon>0\colon\qquad g_\pi\cdot \gamma|_{(-\infty,\epsilon)}\cpsim \gamma|_{(-\infty,\epsilon)}\qquad\wedge\qquad g_\pi\cdot \gamma|_{(-\epsilon,\infty)}\cpsim \gamma|_{(-\epsilon,\infty)}
\end{align*}
\vspace{-35pt}

\hspace*{\fill}$\ddagger$
\end{itemize}
\endgroup
\end{example}
\subsection{Compact Decompositions}
\label{jsdjklsdjklsd}
This subsection is dedicated to a detailed analysis of the situation where an immersive free curve  
admits a compact maximal interval. 
For this, we assume that  $\wm$ is sated (and analytic in $M$) in the following. 

\subsubsection{Basic Properties}
We start our investigations with the following elementary observations.
\begin{remark}
\label{fdnmfnmdnmfdfdkjkjfdkjfdiuiufdiufdfd98fd98fd98fdfdfd}
Let $\wm$ be sated, $\gamma\colon I\rightarrow M$ immersive and free, and $B=[b',b]$ a compact maximal interval. Let $B_-,B_+\subseteq I$ be intervals that are closed in $I$ with $\sup(B_-)=b'$ and $\inf(B_+)=b$, such that  
\begin{align}
\label{oidsoidsoidsoidsoidsoidsds0ds09ds0909ds}
g_-\cdot \gamma|_B \psim \gamma|_{B_-}\qquad\quad\text{as well as}\qquad\quad g_+\cdot \gamma|_B \psim \gamma|_{B_+}
\end{align} 
holds for certain $[g_\pm]\neq [e]$.\footnote{By faithfulness, $B_\pm$ are the intervals $B_{\pm 1}$ that correspond to the unique $B$-decomposition of $\gamma$, and we have $[g_\pm]=[g_{\pm 1}]$.} Then the following two equivalences hold:
\begin{align}
\label{wassadpp}
\begin{split}
	g_+\cdot \gamma|_B\psim_+ \gamma|_{B_+} \qquad\quad&\Longleftrightarrow\qquad\quad g_-\cdot \gamma|_B\psim_+ \gamma|_{B_-}\\
	g_+\cdot \gamma|_B\psim_- \gamma|_{B_+} \qquad\quad&\Longleftrightarrow\qquad\quad g_-\cdot \gamma|_B\psim_- \gamma|_{B_-}
\end{split}
\end{align}
In fact, by faithfulness of the $B$-decomposition of $\gamma$, the first line is equivalent to the second line (under the assumption that \eqref{oidsoidsoidsoidsoidsoidsds0ds09ds0909ds} holds). Moreover, 
the first line is immediate from the following two (more general)  implications:
\begin{align}
\label{wassad1}
	g_+\cdot \gamma|_B\psim_+ \gamma|_{B_+} \qquad\quad&\Longrightarrow\qquad\quad g_-\cdot \gamma|_B\psim_+ \gamma|_{B_-}\quad\wedge\quad [g_-]=[g_+^{-1}]\\
\label{wassad2}
	g_-\cdot \gamma|_B\psim_+ \gamma|_{B_-} \qquad\quad&\Longrightarrow\qquad\quad g_+\cdot \gamma|_B\psim_+ \hspace{0.5pt}\gamma|_{B_+}\quad\wedge\quad [g_+]=[g_-^{-1}]
\end{align} 
\begin{proof}[Proof of the Implications \eqref{wassad1} and \eqref{wassad2}]
We only prove the implication \eqref{wassad1}, because the implication \eqref{wassad2} follows analogously. We have\he\footnote{For the second implication observe that Lemma \ref{lemma:BasicAnalyt1} implies $g_+^{-1}\cdot \gamma|_{J} \psim_{b,b'} \gamma|_{J'}$ w.r.t.\ a positive analytic diffeomorphism, for certain open intervals   $J,J'\subseteq I$ with $b\in J$ and $b'\in J'$.} 
\begin{align*}
g_+\cdot \gamma|_B\psim_+ \gamma|_{B_+}\qquad&\stackrel{\phantom{\text{Lemma } \ref{lemma:BasicAnalyt1}}}{\Longrightarrow}\qquad 
	g_+^{-1}\cdot \gamma|_{[b,b+\delta]} \psim_{b,b'} \gamma|_{[b',b'+\delta']}\:\text{ for certain }\: \delta,\delta'>0
		\\[2pt]
	&\stackrel{\text{Lemma } \ref{lemma:BasicAnalyt1}}{\Longrightarrow}\qquad \he g_+^{-1}\cdot \gamma|_{[b-\epsilon,b]} \psim_{b,b'} \gamma|_{[b'-\epsilon',b']}\:\text{ for certain }\: \epsilon,\epsilon'>0,
\end{align*}
whereby the second line implies 
\eqref{wassad1} by faithfulness of the (unique) $B$-decomposition of $\gamma$.    
\end{proof}
\end{remark}
\noindent
Let $\wm$ be sated; and let $\gamma\colon I\rightarrow M$ be  immersive and free with  $A$-decomposition $(\{a_n\}_{n\in \cn},\{[g_n]\}_{n\in \cn})$. According to \eqref{wassadpp} in Remark \ref{fdnmfnmdnmfdfdkjkjfdkjfdiuiufdiufdfd98fd98fd98fdfdfd},   we have
\begin{align*}
\text{either}\qquad\:\: &g_{\pm 1}\cdot \gamma|_A\psim_+ \gamma|_{A_{\pm 1}}\qquad\stackrel{\text{\rm def.}}{\Longleftrightarrow}\qquad A\:\:\text{is}\:\:\text{\bf positive}
\\
\text{or}\qquad\:\: &g_{\pm 1}\cdot \gamma|_A\psim_- \gamma|_{A_{\pm 1}}\qquad\stackrel{\text{\rm def.}}{\Longleftrightarrow}\qquad A\:\:\text{is}\:\:\text{\bf negative}.
\end{align*}
We will furthermore use the notation
\begin{align*}
	\psim_{-^n}\equiv \psim_-\quad \deff\quad |n|\quad \text{is odd}\qquad\qquad\text{as well as}\qquad\qquad \psim_{-^n}\equiv \psim_+\quad \deff \quad|n|\quad \text{ is even}.
\end{align*}
\begin{lemma}
\label{fdggf}
Let $\wm$ be sated; and let $\gamma\colon I\rightarrow M$ be free immersive with $A$-decomposition $(\{a_n\}_{n\in \cn},\{[g_n]\}_{n\in \cn})$. 
Then, the following assertions hold: 
\begingroup
{
\renewcommand{\theenumi}{\arabic{enumi})} 
\renewcommand{\labelenumi}{\theenumi}
\setlength{\leftmargini}{16pt}
\begin{enumerate}
\item
\label{fdggf1}
If $A$ is positive\slash negative, then each compact $A_n$ is positive\slash  negative, i.e., we have  
(recall \eqref{ffffs})
\begin{align*}
	 h_n\cdot \gamma|_{A_{n+1}}\psim_{+\slash -} \gamma|_{A_{n}}\quad\text{for}\quad \cn_-\leq n\leq -1\qquad \text{and}\qquad  h_n\cdot \gamma|_{A_{n-1}}\psim_{+\slash -} \gamma|_{A_{n}}\quad\text{for}\quad 1\leq n \leq\cn_+. 
\end{align*}
Consequently, for each $n\in \cn$, we have 
\begin{align*}
\hspace{11pt}g_n\cdot \gamma|_A\:\psim_+\: \gamma|_{A_n}\:\:\:\text{if}\:\:\: A\:\:\: \text{is positive}\qquad\quad\text{and}\qquad\quad  g_n\cdot \gamma|_A\psim_{-^n} \gamma|_{A_n}\:\:\text{if}\:\:\: A\:\:\:\text{is negative}.
\end{align*}
\item
\label{fdggf2}
If $A$ is negative and $D\subseteq I$ free, then  
 $D\subseteq A_n$ holds for some (necessarily unique) $\cn_-\leq n\leq\cn_+$. In particular, the intervals $\{A_n\}_{\cn_-\leq n\leq\cn_+}$ are maximal, and the only maximal intervals. 
\item
\label{fdggf3}
If $A$ is positive\slash negative, then each compact maximal interval is $positive\slash negative$.
\end{enumerate}}
\endgroup
\end{lemma}
\begin{proof}
\begingroup
\setlength{\leftmargini}{16pt}
\begin{enumerate}
\item
By Lemma \ref{freemax}, each $A_n$ is free, and each compact $A_n$ is maximal. Now, if $B,B_\pm\subseteq I$ are compact intervals and $g_\pm\in G$, then we evidently have the implications: 
\begin{align*}
g_{\pm}\cdot \gamma|_B\psim_+ \gamma|_{B_\pm}\quad\text{w.r.t}\quad \mu_{\pm}
\qquad\:\:&\Longrightarrow\qquad\quad g^{-1}_\pm\cdot \gamma|_{B_\pm}\psim_+ \gamma|_{B}\quad\text{w.r.t}\quad \mu_{\pm}^{-1}\\
g_{\pm}\cdot \gamma|_B\psim_- \gamma|_{B_\pm}\quad\text{w.r.t}\quad \mu_{\pm}
\qquad\:\:&\Longrightarrow\qquad\quad g^{-1}_\pm\cdot \gamma|_{B_\pm}\psim_- \gamma|_{B}\quad\text{w.r.t}\quad \mu_{\pm}^{-1}
\end{align*}
The first statement thus follows inductively from faithfulness of each $A_n$-decomposition for $\cn_-<n<\cn_+$ and the equivalences in  \eqref{wassadpp}.
The second statement is then just clear from (recall \eqref{ffffs})
\begin{align*}
g_n=h_n\cdot {\dots}\cdot h_{-1} \qquad\forall\: \cn_-\leq n\leq -1\qquad\quad\text{as well as}\qquad\quad g_n=h_n\cdot {\dots}\cdot h_{1}\qquad\forall\: 1\leq n\leq \cn_+.
\end{align*} 
\item
Part \ref{fdggf1} shows that  
 $A_n$ is negative for all $\cn-<n<\cn_+$. Hence, for each $n\in \cn$, we have\footnote{If $1\leq n<\cn_+$ holds, choose $h\equiv h_n$; and if $\cn_-<n\leq -1$ holds, choose $h\equiv h_n^{-1}$.}
\begin{align*} 
	h\cdot \gamma|_{[a_n-\epsilon,a_n]}\psim_{a_n,a_n} \gamma|_{[a_n,a_n+\epsilon']}\quad\text{for certain}\quad h\in  G\backslash G_\gamma\quad\text{and}\quad \epsilon,\epsilon'>0.
\end{align*} 
Consequently, $a_n\notin \innt[D]$ holds for all 
$n\in \cn$, as $D$ is free; so that  
the claim is clear from Lemma \ref{eoder}.\ref{eoder2}.
\item
If $A$ is negative, then each compact maximal interval necessarily  equals some compact $A_n$ by Part \ref{fdggf2},  is thus negative by Part \ref{fdggf1}. 
Hence, there cannot exist any negative interval if $A$ is positive and vice versa.\qedhere
\end{enumerate}
\endgroup
\end{proof}
\noindent
Lemma \ref{fdggf} motivates the following definition.
\begin{definition}
\label{pofdpodfpofdopdf}
Let $\wm$ be sated, and $\gamma\colon I\rightarrow M$  immersive and free. 
We say that $\gamma$ is {\bf positive\slash negative} \deff $\gamma$ admits a positive\slash negative interval; with what each other compact maximal interval is positive\slash negative by Lemma \ref{fdggf}.\ref{fdggf3}. 
\end{definition}
\subsubsection{The Positive Case}
In this subsection, we discuss the situation where $\gamma\colon I\rightarrow M$ is positive. 
\begin{proposition}
\label{sdfdfdlla}
Let $\wm$ be sated; and let $\gamma\colon I\rightarrow M$ be  positive. 
\begingroup
{
\renewcommand{\theenumi}{\arabic{enumi})} 
\renewcommand{\labelenumi}{\theenumi}
\setlength{\leftmargini}{16pt}
\begin{enumerate}
\item
\label{sdfdfdlla1}
There exists $[e]\neq [h]\in G\slash G_\gamma$ unique, such that 
for each  
$A$-decomposition $(\{a_n\}_{n\in \cn},\{[g_n]\}_{n\in \cn})$ of $\gamma$ ($A\subseteq I$ compact maximal, hence positive by Lemma \ref{fdggf}.\ref{fdggf3}) we have  
\begin{align}
\label{dslkdslkdsoidspodspodsopsd09ds09ds0dsdsdssdds}
[h_n]=[h^{\sign(n)}]\qquad\quad\forall\: n\in \cn.
\end{align}
\vspace{-25pt}

\noindent
In particular, we have
\begin{align}
\label{dslfofodpfd}
	 [g_n]=[h^n]\qquad\forall\: n\in \cn\qquad\quad\text{hence}\qquad\quad
	  h^n \cdot \gamma|_A\psim_+ \gamma|_{A_n}    \qquad\forall\: n\in \cn. 
\end{align}
\vspace{-16pt}
\item
\label{sdfdfdlla2}  
The following assertions hold:
\vspace{-4pt}
\begingroup
\setlength{\leftmarginii}{12pt}
\begin{itemize}
\item
Given $\cn_-<n \leq 0$ and $b'\in \innt[A_{n-1}]$, then  there exists $b\in \innt[A_{n}]$\hspace{2.7pt} such that $[b',b]$ is positive.
\item
Given $0\leq n<\cn_+$ and $b\in \innt[A_{n+1}]$,\hspace{3pt} then there exists $b'\in \innt[A_{n}]$ such that $[b',b]$ is positive.
\end{itemize}
\endgroup
\noindent
In particular, for each $t\in I$, there exists a positive interval $A_t$ with $t\in \innt[A_t]$.
\end{enumerate}}
\endgroup
\end{proposition}
\begin{proof}
\begingroup
{
\renewcommand{\theenumi}{\arabic{enumi})} 
\renewcommand{\labelenumi}{\theenumi}
\setlength{\leftmargini}{16pt}
\begin{enumerate}
\item
If such a class $[h]$ exists, then it is necessarily unique by faithfulness. Moreover, the right side of \eqref{dslfofodpfd} is immediate from the left side of \eqref{dslfofodpfd} as well as Lemma \ref{fdggf}.\ref{fdggf1} (second statement). 

Let now $A\subseteq I$ be a fixed positive interval with corresponding $A$-decomposition $(\{a_n\}_{n\in \cn},\{[g_n]\}_{n\in \cn})$. 
We define $h:=g_1$, and observe the following:
\vspace{-6pt} 
\begingroup
\setlength{\leftmarginii}{12pt}
\begin{itemize}
\item[$*$]
We can assume that $[h_n]=[h^{\sign(n)}]$ holds for some $1\leq n< \cn_+$, as we have $h_1\stackrel{\eqref{ffffs}}{=}g_1\stackrel{\rm def.}{=}h=h^{\sign(1)}$. 
 Now,  $h_n\cdot \gamma|_{A_{n-1}}\psim_+ \gamma|_{A_n}$ holds by Lemma \ref{fdggf}.\ref{fdggf1}; hence, we have
\begin{align*}
\\[-19pt]
h^{-1}_{n}\cdot \gamma|_{A_{n}} \psim_+ \gamma|_{A_{n-1}} \qquad\stackrel{\eqref{wassad2}}{\Longrightarrow}\qquad 
[h_{n+1}]=[(h_n^{-1})^{-1}]=[h_n]=[h^{\sign(n)}].
\end{align*} 
Thus,   
\eqref{dslkdslkdsoidspodspodsopsd09ds09ds0dsdsdssdds} follows inductively for all $1\leq n\leq \cn_+$.
\vspace{4pt}
\item[$*$]
We can assume that $[h_n]=[h^{\sign(n)}]$ holds for some $\cn_-<n\leq -1$, as we have  
\vspace{-4pt}
$$[h_{-1}]\stackrel{\eqref{ffffs}}{=}[g_{-1}]\stackrel{\eqref{wassad1}}{=}[g_1^{-1}]\stackrel{\rm def.}{=}[h^{-1}]=[h^{\sign(-1)}].$$ 
Now,  $h_n\cdot \gamma|_{A_{n+1}}\psim_+ \gamma|_{A_n}$ holds by Lemma \ref{fdggf}.\ref{fdggf1}; hence, we have
\begin{align*}
\\[-19pt]
h^{-1}_{n}\cdot \gamma|_{A_{n}} \psim_+ \gamma|_{A_{n+1}} \qquad
\stackrel{\eqref{wassad1}}{\Longrightarrow}\qquad
 [h_{n-1}]=[(h_n^{-1})^{-1}]=[h_n]=[h^{\sign(n)}].
\end{align*}  
Thus,   
\eqref{dslkdslkdsoidspodspodsopsd09ds09ds0dsdsdssdds} follows inductively for all $\cn_-\leq n\leq 1$.
\end{itemize}
\endgroup
\noindent
The left side of \eqref{dslfofodpfd} now follows by induction:
\vspace{-6pt} 
\begingroup
\setlength{\leftmarginii}{12pt}
\begin{itemize}
\item[$*$] 
We can assume that $[g_n]=[h^n]$ holds for some $1\leq n< \cn_+$, because we have $[g_1]\stackrel{\eqref{ffffs}}{=}[h_1]\stackrel{\eqref{dslkdslkdsoidspodspodsopsd09ds09ds0dsdsdssdds}}{=}[h]$. Now, $h_{n+1}=h\cdot q$ holds for some $q\in G_\gamma$ by \eqref{dslkdslkdsoidspodspodsopsd09ds09ds0dsdsdssdds}, and $g_n=h^n\cdot q'$ holds for some $q'\in G_\gamma$ by assumption. We  obtain \hspace*{\fill}($h^n\cdot \gamma=(h^n\cdot q')\cdot \gamma=g_n\cdot \gamma\cpsim \gamma$ in the last step)
\vspace{-5pt}
\begin{align*}
	[g_{n+1}]\stackrel{\eqref{ffffs}}{=}[h_{n+1}\cdot g_{n}]=[(h\cdot q) \cdot (h^{n}\cdot q')]=[h^{n+1} \cdot  (h^{-n}\cdot  q \cdot h^{n})]\stackrel{\eqref{stabiconji}}{=}[h^{n+1}].
\end{align*} 
Thus, \eqref{dslfofodpfd} follows inductively for all $1\leq n\leq \cn_+$. 
\item[$*$]
We can assume that $[g_n]=[h^n]$ holds for some $\cn_-<n\leq 1$, because we have $[g_{-1}]\stackrel{\eqref{ffffs}}{=}[h_{-1}]\stackrel{\eqref{dslkdslkdsoidspodspodsopsd09ds09ds0dsdsdssdds}}{=}[h^{-1}]$. Now,  
$h_{n-1}=h^{-1}\cdot q$ holds for some $q\in G_\gamma$ by \eqref{dslkdslkdsoidspodspodsopsd09ds09ds0dsdsdssdds}, and $g_n=h^n\cdot q'$ holds for some $q'\in G_\gamma$ by assumption. We  obtain
\hspace*{\fill}($h^n\cdot \gamma=(h^n\cdot q')\cdot \gamma=g_n\cdot \gamma\cpsim \gamma$ in the last step)

\vspace{-15pt}
\begin{align*}
	[g_{n-1}]\stackrel{\eqref{ffffs}}{=}[h_{n-1}\cdot g_{n}]=[(h^{-1}\cdot q) \cdot (h^{n}\cdot q')]=[h^{n-1} \cdot  (h^{-n}\cdot  q \cdot h^{n})]\stackrel{\eqref{stabiconji}}{=}[h^{n-1}].
\end{align*}
Thus, \eqref{dslfofodpfd} follows inductively for all $\cn_-\leq n\leq -1$. 
\end{itemize}
\endgroup
\noindent
Finally, let $B=[b',b]$ be positive with $B$-decomposition $(\{b_n\}_{n\in \cn'},\{[g'_n]\}_{n\in \cn'})$. It remains to  show that then necessarily  $[h']=[h]$ holds for $h':=g_1'$. For this, we first observe the following:
\begingroup
\setlength{\leftmarginii}{15pt}
\begin{itemize}
\item[i)]
$B$ can neither be contained in $A_{\cn_-}$ (if $\cn_-\neq \infty$ holds) nor be contained in $A_{\cn_+}$ (if $\cn_+\neq \infty$ holds); just because (in the respective cases) these intervals are non-compact as well as free by Lemma \ref{freemax}.\ref{freemax1}, and  $B$ is compact maximal.
\item[ii)]
If $B\subseteq A_n$ or $A_n\subseteq B$ holds for some $\cn_-<n<\cn_+$, then we automatically have $B=A_n$ by maximality of $B$ and $A_n$ (see Lemma \ref{freemax}.\ref{freemax2}). 
\end{itemize}
\endgroup
\noindent
Now, if $B=A_n$ holds for some $\cn_-<n<\cn_+$, then $[h']=[h]$ is clear from faithfulness of the $B$-decomposition, as well as from \eqref{dslkdslkdsoidspodspodsopsd09ds09ds0dsdsdssdds} (i.e.\ $h\cdot \gamma|_B = h\cdot\gamma|_{A_n} \psim_+ \gamma|_{A_{n+1}}$). 

In the other case, 
i) and ii) imply that there exist some compact $A_n\equiv [a',a]$ such that 
\begin{align*}
\text{either}\qquad\quad b'<a'< b<a\qquad\quad\text{or}\qquad\quad a'< b'<a<b\qquad\quad\text{holds.}
\end{align*}
Now, we have the following implications:
\begin{align*}
b'<a'< b<a\qquad&\stackrel{\text{ii)}}{\Longrightarrow} \qquad \hspace{11pt}B_1=[b,i)\qquad\hspace{1pt}\vee\hspace{11pt}\qquad B_1=[b,\wt{b}]\hspace{2pt}\quad\text{for}\quad b<a<\wt{b}\\[-4pt]
a'< b'<a<b\qquad&\stackrel{\text{ii)}}{\Longrightarrow} \qquad A_{n+1}=[a,i)\qquad\vee\qquad A_{n+1}=[a,\wt{a}]\quad\text{for}\quad a<b<\wt{a}.
\end{align*} 
\begingroup
\setlength{\leftmarginii}{11pt}
\begin{itemize}
\item[$\bullet$]
In the first case, $h\cdot \gamma|_{A_n}\psim \gamma|_{A_{n+1}}$ implies $h\cdot \gamma|_B \cpsim \gamma|_{B_1}$, hence \hspace{3.8pt}$[h]=[g_1']=[h']$ by faithfulness of the $B$-decomposition.
\item[$\bullet$]
In the second case, $h'\cdot \gamma|_B\psim \gamma|_{B_1}$ implies $h'\cdot \gamma|_{A_n} \cpsim \gamma|_{A_{n+1}}$. Then, $[h']=[h]$ is immediate from faithfulness of the $A_n$-decomposition, as well as from \eqref{dslkdslkdsoidspodspodsopsd09ds09ds0dsdsdssdds} (i.e.\ $h\cdot\gamma|_{A_n} \psim_+ \gamma|_{A_{n+1}}$).
\end{itemize}
\endgroup
\item
We only prove the 
second statement, because the first statement follows analogously. 
For this, let $0\leq n<\cn_+$ and $b\in \innt[A_{n+1}]$ be given; and let $\mu\colon A_n\supseteq \dom[\mu]\rightarrow A_{n+1}$ denote the unique analytic diffeomorphism that corresponds to the relation $h\cdot \gamma|_{A_n}\psim_+ \gamma|_{A_{n+1}}$ (hence $\mu$ is positive).\footnote{By faithfulness, $\mu$ is the diffeomorphism $\tilde{\mu}_1$ that correspond to the unique $A_n$-decomposition $(\{\tilde{a}_n\}_{n\in \tilde{\cn}},\{[\tilde{g}_n]\}_{n\in \tilde{\cn}})$ of $\gamma$, and we have $[\tilde{g}_1]=[h]$.} 
Let $b< c \in A_{n+1}$. Then, we have
\begin{align}
\label{sdsd}
\begin{split}
h\cdot \gamma|_{[a_n,b']}\psim_+ \gamma|_{[a_{n+1},b]}&\quad \text{w.r.t.}\quad \mu|_{[a_n,b']}\quad\text{for}\quad b':=\mu^{-1}(b)\in \innt[A_n]\\
h\cdot \gamma|_{[b',c']}\psim_+ \gamma|_{[b,c]}\hspace{14pt}&\quad \text{w.r.t.}\quad \mu|_{[b,\tilde{b}]}\hspace{8.6pt}\quad\text{for}\quad c':=\mu^{-1}(\tilde{c})\in A_n.
\end{split}
\end{align}
Together with $a_n<b'<a_{n+1}<b$, this implies that the compact interval  $B:=[b',b]$ is maximal if it is  free. Assume thus that $B$ is not free, i.e., that 
$g\cdot \gamma|_B\cpsim \gamma|_B$ holds for some $g\in G\backslash G_\gamma$: 
\vspace{4pt}

\noindent
The intervals  $[b', a_{n+1}]\subseteq A_n$ and $[a_{n+1},b]\subseteq A_{n+1}$ are free, because $A_n,A_{n+1}$ are free by Lemma \ref{freemax}. 
Hence, $g\cdot \gamma|_B\cpsim \gamma|_B$ implies that there exists $p\in \{-1,1\}$ with
\begin{align*}
g^{p}\cdot \gamma|_{[b',a_{n+1}]}\cpsim \gamma|_{[a_{n+1},b]}\quad\hspace{2pt} \text{w.r.t.}\quad \rho
\qquad\:&\Longrightarrow \qquad\: g^p\cdot \gamma|_{A_n}\cpsim \gamma|_{A_{n+1}}\quad\hspace{4pt} \text{w.r.t.}\quad \rho\\
&\stackrel{(*)}{\Longrightarrow} \qquad\: \hspace{3.8pt} h\cdot \gamma|_{A_n}\psim_+ \gamma|_{A_{n+1}}\quad \text{w.r.t.}\quad \ovl{\rho}|_{\dom[\mu]}=\mu\\[3pt]
&\Longrightarrow \qquad\: \hspace{3.8pt} \ovl{\rho}\:\text{ is\: positive};
\end{align*}
whereby in $(*)$ we have used  
faithfulness of the $A_n$-decomposition of $\gamma$.  
Since $\dom[\rho]\subseteq [b',a_{n+1}]$ and $\im[\rho]\subseteq [a_{n+1},b]$ holds, there exists  
some $t\in (b', a_{n+1})$ with $\ovl{\rho}(t)=\rho(t)\in (a_{n+1},b)$. This, however, contradicts that $\ovl{\rho}$ is positive with $\ovl{\rho}(b')=\mu(b')=b$.
\qedhere
\end{enumerate}}
\endgroup 
\end{proof}
\begin{example}
\label{dsajsu}
We let $G\equiv \RR$ act on $M\equiv\RR^2$ via 
$$
\wm\colon G\times M\rightarrow M,\qquad (t,(x,y))\mapsto(t+x,y)
$$ 
and consider the analytic immersion $\gamma\colon \RR\ni t\mapsto (t,\sin(t))\in  M$:  
\begingroup
\setlength{\leftmargini}{12pt}
\begin{itemize}
\item
The action $\wm$ is regular as pointwise proper, hence sated (alternatively apply Remark \ref{remmmmmi}.\ref{regp3}). Then, $[t,t+2\pi]$ is positive for each $t\in \RR$; and the class $[h]$ is given by $[2\pi]$. 
\item
We consider the restriction $\tilde{\gamma}:=\gamma|_{(0,\infty)}$, and observe the following: 
\begingroup
\setlength{\leftmarginii}{12pt}
\begin{itemize}
\item
$\tilde{\gamma}$  
is not a free segment, and admits the non-compact maximal interval $(0,2\pi]$.
\item 
$\tilde{\gamma}$ admits no $\tau$-decomposition, just because $G_x=\{e\}$ holds for each $x\in M$. 
\end{itemize}
\endgroup
 Hence, Proposition \ref{shifty}  shows that there must exist a compact maximal interval; which is indeed the case for the intervals $[t,t+2\pi]$ with $t>0$. \hspace*{\fill}$\ddagger$
\end{itemize}
\endgroup
\end{example}
\subsubsection{The Negative Case}
We now finally discuss the case where $\gamma\colon I\rightarrow M$ is negative. Also in this situation, the classes $[g_{-1}]$ and $[g_1]$ can be related to each other (see Example \ref{sdjksdkjsd}); but, in general this need not be the case (see Example \ref{sepocvofgvifg}).  

\begin{example}
\label{sdjksdkjsd}
We let $G\equiv \{-1,1\}$ act on $M\equiv \RR^2$ via 
$$
\wm\colon G\times M\rightarrow M,\qquad (\pm 1,(x_1,x_2))\mapsto (x_1,- x_2),
$$ 	 
and consider the analytic immersion $\gamma\colon \RR\ni t\rightarrow (\cos(t),\sin(t))$.
\begingroup
\setlength{\leftmargini}{12pt}
\begin{itemize}
\item
The action $\wm$ is regular as pointwise proper, hence sated (alternatively apply Remark \ref{remmmmmi}.\ref{regp3}).  
\item
	The curve $\gamma$  
	is negative, and admits the 
	$A\equiv [0,\pi]$-decomposition $(\{a_n\}_{n\in \cn},\{[g_n]\}_{n\in \cn})$ with $g_n=(-1)^n$ for all $n\in \cn=\ZZ_{\neq 0}$  as well as
\begin{align*}
a_n=(n+1)\cdot \pi\qquad \forall\: n\leq -1 \qquad\qquad&\text{and}\qquad\qquad 
		a_n=n\cdot \pi\qquad\forall\: 1\leq n\\[1pt]
\text{h}&\text{ence}\\[1pt]
		A_n=[n\cdot \pi,(n+1)\cdot\pi]\qquad \forall\: n\leq -1, \qquad\quad A_0&=[0,\pi],\qquad\quad 
		A_n=[n\cdot \pi,(n+1)\cdot \pi]\qquad\forall\: 1\leq n.\end{align*}
		In particular, we have $[g_{-1}]=[-1]=[g_1]$.
\end{itemize}
\endgroup
\noindent
Notably, the restriction $\gamma|_{(-\pi,\pi)}$ admits the $\tau\equiv 0$-decomposition $[g_{-1}]$.
\end{example}
\begin{example}
\label{sepocvofgvifg}
\noindent

\vspace{-6pt}
\begingroup
{
\renewcommand{\theenumi}{\sf\alph{enumi})} 
\renewcommand{\labelenumi}{\theenumi}
\setlength{\leftmargini}{15pt}
\begin{enumerate}
\item
\label{sepocvofgvifg1}
Let the euclidean group $G\equiv \RR^2 \rtimes \SO(2)$ act canonically on $M\equiv \RR^2$; and let $\gamma\colon \RR\ni t\mapsto (t,\sin(t))\in M$. 
\begingroup
\setlength{\leftmarginii}{12pt}
\begin{itemize}
\item
\label{sepocvofgvifg2}
The action $\wm$ is regular as it is pointwise proper, hence sated (alternatively apply Remark \ref{remmmmmi}.\ref{regp3}). 
\item
The interval $A\equiv [0,\pi]$ is negative; and $[g_{-1}]$ and $[g_{1}]$ are the classes of the rotations by the angle $\pi$ around $(0,0)$ and $(\pi,0)$, respectively. 
\end{itemize}
\endgroup
\noindent
Notably, the restriction $\gamma|_{(-\pi,\pi)}$ admits the $\tau\equiv 0$-decomposition $[g_{-1}]$. \hspace*{\fill}$\ddagger$
\item
\label{sepocvofgvifgdsdsds}
Let the euclidean group $G\equiv \RR^3 \times \SO(3)$ act canonically on $M\equiv \RR^3$; and let $\gamma\colon \RR\ni t\mapsto (t,\sin(t),0)\in M$.
\begingroup
\setlength{\leftmarginii}{12pt}
\begin{itemize}
\item
The action $\wm$ is regular as it is pointwise proper, hence sated (alternatively apply Remark \ref{remmmmmi}.\ref{regp3}). 
\item
The interval $A\equiv [0,\frac{\pi}{2}]$ is negative; and $[g_{-1}]$ and $[g_{1}]$ are the classes of the rotations by the angle $\pi$ around the axis 
$\{0\}\times \{0\}\times \RR$ and the axis $\{\frac{\pi}{2}\}\times\RR\times \{0\}$, respectively.      
\end{itemize}
\endgroup
\noindent
Notably, the restriction $\gamma|_{(-\frac{\pi}{2},\frac{\pi}{2})}$ admits the $\tau\equiv 0$-decomposition $[g_{-1}]$.
\end{enumerate}}
\endgroup
\end{example}
\noindent
In any case, each class $[g_n]$ can be expressed in terms of the classes $[g_{-1}]$ and $[g_1]$. In fact, let 
\begin{align}
\label{sdfgd}
\sigma\colon \ZZ_{\neq 0}\rightarrow \{-1,1\},\qquad n\mapsto 
\begin{cases} 
	(-1)^{n-1} &\mbox{for }\: n > 0 \\ 
	(-1)^n & \mbox{for }\: n < 0. 
\end{cases} 
\end{align} 
Then, we have the following statement:
\begin{proposition}
\label{trhdhg}
Let $\wm$ be sated; and let $\gamma\colon I\rightarrow M$ be  negative with $A$-decomposition $(\{a_n\}_{n\in \cn},\{[g_n]\}_{n\in \cn})$ ($A\subseteq I$ compact maximal, hence negative by Lemma \ref{fdggf}.\ref{fdggf3}).  
\begingroup
{
\renewcommand{\theenumi}{\arabic{enumi})} 
\renewcommand{\labelenumi}{\theenumi}
\setlength{\leftmargini}{16pt}
\begin{enumerate}
\item
\label{trhdhg1}
The intervals $\{A_n\}_{\cn_- \leq n\leq \cn_+}$ are maximal, and the only maximal intervals (Lemma \ref{fdggf}.\ref{fdggf2}).
\item
\label{trhdhg2}
For each $n\in \cn$, we have
\hspace*{\fill}(left side by Lemma \ref{fdggf}.\ref{fdggf1})
\begin{align}
\label{sdsdffghhh}
g_n\cdot \gamma|_A\psim_{-^n}\gamma|_{A_n}\qquad\text{with}\qquad 
[g_n]=\underbrace{[g_{\sigma(1\he\cdot\: \sign(n))}\cdot g_{\sigma(2\he\cdot\: \sign(n))}\cdot  {\dots}\cdot g_{\sigma(|n|\he\cdot\:  \sign(n))}]}_{\displaystyle [g_{\sigma(\sign(n))}\cdot {\dots}\cdot g_{\sigma(n)}]}.
\end{align} 
\vspace{-19pt}
\item
\label{trhdhg3}
We have\hspace{1pt}  
 $[g_{\pm 1}]\subseteq  G_{\gamma(a_{\pm 1})}\setminus G_\gamma$\hspace{1pt} as well as\hspace{1pt} $[g^{-1}_{\pm 1}]=[g_{\pm1}]$.
\end{enumerate}}
\endgroup
\end{proposition}
\noindent
It remains to show the right side of \eqref{sdsdffghhh} as well as Part \ref{trhdhg3}. 
We start with the following observations:
\begin{remark}
\label{fdfdsasasaasxafdfd}
Let $\gamma\colon I\rightarrow M$ be free and immersive; and set $O_\gamma:=  \{g\in G\:|\: g\cdot \gamma\cpsim \gamma\}$. 
\begingroup
\setlength{\leftmargini}{16pt}
{
\renewcommand{\theenumi}{{\rm\alph{enumi})}} 
\renewcommand{\labelenumi}{\theenumi}
\begin{enumerate}
\item
\label{fdfdsasasaasxafdfda}
Let  $h\in G\backslash G_\gamma$ and $a\in I$ be given, with  
$h\cdot \gamma|_{[a-\epsilon,a]}\psim_{a,a}\gamma|_{[a,a+\epsilon']}$.   
Then, the following implication holds:
\begin{align*}
[a,a+\epsilon']\:\:\: \text{is\: free}	
	\qquad\stackrel{{\rm Appendix\:\:\ref{appC7}}}{\Longrightarrow}\qquad [h]=[h^{-1}]\quad\hspace{8pt}
\end{align*} 
\vspace{-17pt}
\item
\label{fdfdsasasaasxafdfdb}
Let $\wm$ be sated; and assume that the following holds:
\begingroup
\setlength{\leftmarginii}{11pt}
\begin{itemize}
\item
$A_-=[a_-,a]$ and $A_+=[a,a_+]$ are negative, and we are given  $h\in G\backslash G_\gamma$ with \hspace*{\fill}(\:$[h]\stackrel{\ref{fdfdsasasaasxafdfda}}{=}[h^{-1}]$\he)
\vspace{-4pt} 
\begin{align}
\label{dskjdskjdsiudsiudsiudsiudsiuds87878787iuiu1}
	h\cdot \gamma|_{A_-}\psim_- \gamma|_{A_+}\qquad\quad\stackrel{\ref{fdfdsasasaasxafdfda}}{\Longleftrightarrow}\qquad\quad h\cdot \gamma|_{A_+}\psim_- \gamma|_{A_-}.
\end{align}
\item
$A_{--},A_{++}\subseteq I$ are closed in $I$ with $A_{--}\cap A_-=\{a_-\}$ and $A_{+}\cap A_{++}=\{a_+\}$, such that
\begin{align}
\label{dskjdskjdsiudsiudsiudsiudsiuds87878787iuiu2}
h_-\cdot \gamma|_{A_-}\psim_- \gamma|_{A_{--}}\qquad\text{and}\qquad h_+\cdot \gamma|_{A_+}\psim_- \gamma|_{A_{++}}\qquad\text{holds for certain}\qquad h_\pm\in G\backslash G_\gamma.
\end{align}
\end{itemize}
\endgroup
\noindent
Then, we have (Appendix \ref{appC7})\: $[h_\pm]=[h\cdot h_\mp\cdot h]$.     
\hspace*{\fill}$\ddagger$
\end{enumerate}}
\endgroup
\end{remark}
\noindent
Assume now that we are in the situation of Proposition \ref{trhdhg}. We observe the following:
\begingroup
\setlength{\leftmargini}{12pt}
\begin{itemize}
\item
Since each compact $A_n$ is negative, Remark \ref{fdfdsasasaasxafdfd}.\ref{fdfdsasasaasxafdfda} shows 
\begin{align}
\label{safgrfgtr}
\:[h_n]=[h_n^{-1}]\qquad\forall\: n\in \cn\qquad\quad \text{hence}\qquad\quad [g_{\pm1}]=[h_{\pm1}]=[h^{-1}_{\pm1}]=[g^{-1}_{\pm1}].
\end{align}
In particular, then $g^{-1}_{\pm 1}\in O_\gamma$ implies the following:
\begin{align}
\label{qforrme2}
	\:[g_{\pm 1}\cdot q\cdot g_{\pm 1}]\stackrel{\eqref{safgrfgtr}}{=}[g_{\pm1}\cdot q\cdot g_{\pm 1}^{-1}]\stackrel{\eqref{stabiconji}}{=}[e]\qquad\quad&\forall\: q\in G_\gamma\\[2pt]
\label{qforrmel}
\stackrel{\text{Induction}}{\Longrightarrow}\qquad\quad	 q^\pm_{n}:=(g_{\mp 1}\cdot g_{\pm 1})^{n}\cdot (g_{\pm 1}\cdot g_{\mp1})^{n}\in G_\gamma\qquad\quad&\forall\: n\in \NN\hspace{40pt}
\end{align}
\item
Remark \ref{fdfdsasasaasxafdfd}.\ref{fdfdsasasaasxafdfdb} shows the following:
\begingroup
\setlength{\leftmarginii}{12pt}
\begin{itemize}
\item
If $\cn_-\leq -2$ holds, then we have
\begin{align}
\label{twerfdff1}
	[h_{-2}]=[g_{-1}\cdot g_{1}\cdot g_{-1}]\qquad\quad\text{as well as}\qquad\quad [h_{n-1}]=[h_{n}\cdot h_{n+1}\cdot h_{n}]\qquad\forall\: \cn_-<n\leq -2       .
\end{align}
\item
If $\cn_+\geq 2$ holds, then we have
\begin{align}
\label{twerfdff2}
	\phantom{-}[h_2]=[g_1\cdot g_{-1}\cdot g_1]\phantom{-} \qquad\quad\text{as well as}\qquad\quad [h_{n+1}]=[h_{n}\cdot h_{n-1}\cdot h_{n}]\qquad\forall\: 2\leq n< \cn_+.\phantom{-}
\end{align}
\end{itemize}
\endgroup
\noindent
We furthermore have the  implications (cf.\ Appendix \ref{appC8})
\begin{align}
\label{kcfdjksdfkjds}
\cn_-\leq -2\qquad\Longrightarrow\qquad g_{1}\cdot g_{-1} \in O_\gamma\qquad\quad\text{and}\qquad\quad \cn_+\geq 2\qquad\Longrightarrow\qquad g_{-1}\cdot g_{1}\in O_\gamma.
\end{align}
\end{itemize}
\endgroup
\begin{proof}[Proof of Proposition \ref{trhdhg}]
\begingroup
{
\renewcommand{\theenumi}{\arabic{enumi})} 
\renewcommand{\labelenumi}{\theenumi}
\setlength{\leftmargini}{16pt}
\begin{enumerate}
\item
Clear from Lemma \ref{fdggf}.\ref{fdggf2}.
\setcounter{enumi}{2}
\item
By definition, we have $[g_{\pm 1}]\subseteq  G\setminus G_\gamma$. Moreover, since $A$ is negative, we have $g_{\pm 1}\cdot \gamma|_A\psim_{a_\pm,a_\pm}\gamma|_{A_{\pm 1}}$, hence $[g_{\pm 1}]\subseteq G_{\gamma(a_{\pm 1})}$. 
\setcounter{enumi}{1}
\item 
The left side of \eqref{sdsdffghhh} is clear from Lemma \ref{fdggf}.\ref{fdggf1}, so that it remains to show the right side of \eqref{sdsdffghhh}.  
\begingroup
\setlength{\leftmarginii}{12pt}
\begin{itemize}
\item[$\triangleright$]
We now first prove the identities
\begin{align}
\label{awee1}
	\:[h_n]&=[g_{-1}\cdot (g_{1}\cdot g_{-1})^{|n|-1}]	\qquad\quad \forall\: \cn_-\leq n\leq -1,\\
\label{awee2}
	[h_n]&=[g_1\cdot (g_{-1}\cdot g_1)^{n-1}]			\hspace{11pt}\qquad\quad \forall\: 1\leq n\leq \cn_+.
\end{align}
These identities are evident for $n=\pm 1$. Moreover, if $\cn_-\leq -2=n$ holds, then \eqref{awee1} is clear from \eqref{twerfdff1}; and, if $n=2 \leq \cn_+$ holds, then \eqref{awee2} is clear from \eqref{twerfdff2}. 
Hence, \eqref{awee1} is clear for $2\leq \cn_-\leq n\leq 1$; and \eqref{awee2} is clear for $1 \leq n\leq \cn_+\leq 2$.   
\vspace{4pt}

To prove \eqref{awee2} (the formula \eqref{awee1} follows  analogously), we thus can assume that $\cn_+\geq 3$ holds; and furthermore that  
there exists $2\leq m< \cn_+$ such that \eqref{awee2} holds for all $1\leq n\leq m$.  
We now show that then \eqref{awee2} also holds $n=m+1$, with what \eqref{awee2} follows by induction: 
\vspace{4pt}

Let  $\alpha:= g_1\cdot (g_{-1}\cdot g_1)^{m-2}$  and $\beta:= g_1\cdot (g_{-1}\cdot g_1)^{m-1}$.  
By induction hypothesis, $h_{m-1}=\alpha\cdot q$ and  $h_m=\beta\cdot q'$ holds for certain $q,q'\in G_\gamma$. Hence, we have 
\begin{align}
\label{dzuzufdssdufds}
\begin{split}
\:[h_{m-1}\cdot h_{m}]&\stackrel{\phantom{\eqref{twerfdff2}}}{=} [\alpha\cdot q\cdot h_m]=[(\alpha\cdot h_m)\cdot (h_m^{-1}\cdot q \cdot h_m)]\\
&\stackrel{\eqref{stabiconji}}{=}[\alpha\cdot h_m]=[\alpha\cdot \beta]=[(g_1\cdot q_{m-2}^+)\cdot (g_1\cdot g_{-1}\cdot g_1)]\\
&\stackrel{\eqref{twerfdff2}}{=} [(g_1\cdot q_{m-2}^+)\cdot h_2]
=[g_1\cdot h_2\cdot (h_2^{-1}\cdot q^+_{m-2}\cdot h_2)]
\stackrel{\eqref{stabiconji},\he\eqref{qforrmel}}{=}[g_1\cdot h_2]\qquad\\
&\stackrel{\eqref{twerfdff2}}{=}[g_1^2\cdot (g_{-1}\cdot g_1)]= [(g_{-1}\cdot g_1) \cdot ((g_{-1}\cdot g_1)^{-1}\cdot g_1^2\cdot (g_{-1}\cdot g_1))]\\
&\stackrel{\eqref{stabiconji}}{=}[g_{-1}\cdot g_1]
\end{split}
\end{align}
whereby in the last step we have used that $(g_{-1}\cdot g_1)\in O_\gamma$ holds by \eqref{kcfdjksdfkjds}, and  that $g_1^2\in G_\gamma$ holds by \eqref{qforrme2}.  
We obtain
\begin{align*}
	[h_{m+1}]&\stackrel{\eqref{twerfdff2}}{=}[h_m\cdot (h_{m-1}\cdot h_m)]\stackrel{\eqref{dzuzufdssdufds}}{=}[h_m\cdot (g_{-1}\cdot g_1)]=[\beta\cdot q'\cdot (g_{-1}\cdot g_1)]\\
&	\stackrel{\phantom{\eqref{twerfdff2}}}{=}[g_1\cdot (g_{-1}\cdot g_1)^m \cdot ((g_{-1}\cdot g_1)^{-1}\cdot q'\cdot (g_{-1}\cdot g_1))]
	\stackrel{\eqref{kcfdjksdfkjds},\he\eqref{stabiconji}}{=}[g_1\cdot (g_{-1}\cdot g_1)^m].
\end{align*}
\item[$\triangleright$]
The right side of \eqref{sdsdffghhh} now follows by induction  from the identities \eqref{awee1} and \eqref{awee2}. 
 In fact, the right side of \eqref{sdsdffghhh} obviously holds for $n=1$. Hence, we can assume that there exists $1\leq m < \cn_+$  such that the right side of \eqref{sdsdffghhh} holds  for all $1\leq n\leq m$: 
\vspace{4pt}

\noindent
Let $\alpha:= g_1\cdot (g_{-1}\cdot g_1)^{m}$. Then,     
$h_{m+1}= \alpha\cdot q$ holds for some $q\in G_\gamma$  by \eqref{awee2}. Hence, we have
\begin{align}
\label{lkdskldslkdslklkdslkdslkdslkdsids98d09ds09ds0909ds09dsds}
[g_{m+1}]\stackrel{\eqref{ffffs}}{=}[h_{m+1}\cdot g_m]
=[\alpha\cdot q \cdot g_m ]=[\alpha\cdot g_m \cdot (g_m^{-1}\cdot q\cdot g_m)]\stackrel{\eqref{stabiconji},\: g_m\he\in\he O_\gamma}{=}	[\alpha\cdot g_m].
\end{align} 
\begingroup
\setlength{\leftmarginiii}{12pt}
\begin{itemize}
\item
If $m=2\cdot k$ is even, then $g_m\stackrel{\eqref{sdsdffghhh}}{=}(g_1\cdot g_{-1})^k$ holds by the induction hypothesis; hence,
\begin{align*}
	[g_{m+1}]&\stackrel{\eqref{lkdskldslkdslklkdslkdslkdslkdsids98d09ds09ds0909ds09dsds}}{=}[\alpha\cdot g_m]=
	[g_1\cdot (g_{-1}\cdot g_1)^{2k}\cdot (g_1\cdot g_{-1})^k]
	=[g_1\cdot (g_{-1}\cdot g_1)^{k}\cdot q_k^+]\stackrel{\eqref{qforrmel}}{=}[g_{\sigma(1)}\cdot {\dots}\cdot g_{\sigma(m+1)}].
\end{align*} 
\item
If $m=2k+1$ is odd, then $g_m\stackrel{\eqref{sdsdffghhh}}{=}(g_1\cdot g_{-1})^k\cdot g_1$ holds by the induction hypothesis; hence,
\begin{align*}
[g_{m+1}]&	\stackrel{\eqref{lkdskldslkdslklkdslkdslkdslkdsids98d09ds09ds0909ds09dsds}}{=}[\alpha\cdot g_m]
=
[g_1\cdot(g_{-1}\cdot g_{1})^{2k+1}\cdot (g_{1}\cdot g_{-1})^k \cdot g_1]\\
&	\stackrel{\phantom{\eqref{lkdskldslkdslklkdslkdslkdslkdsids98d09ds09ds0909ds09dsds}}}{=}[(g_1\cdot g_{-1})^{k+1}\cdot g_1\cdot q_k^+\cdot g_1]
	\stackrel{\eqref{qforrmel},\he \eqref{qforrme2}}{=}[(g_{1}\cdot g_{-1})^{k+1}]=[g_{\sigma(1)}\cdot {\dots}\cdot g_{\sigma(m+1)}].
\end{align*}
\end{itemize}
\endgroup
\noindent
An analogous induction shows that the right side of \eqref{sdsdffghhh} also holds for all $\cn_-\leq n\leq -1$.\qedhere 
\end{itemize}
\endgroup
\end{enumerate}}
\endgroup
\end{proof}
\subsection{Synopsis of the Results}
Combining the results obtained in this section so far, we obtain the following statement. 
\begin{theorem}
\label{sfdknfdhujfd}
Let $\wm$ be sated, and $\gamma\colon I\rightarrow M$ immersive and free. If $\gamma$ is not a free segment by itself, then it either admits a unique $\tau$-decomposition or a compact maximal interval. In the second case, $\gamma$ is either positive or negative, with what the statements in Proposition \ref{sdfdfdlla} or Proposition \ref{trhdhg} hold, respectively. 
\end{theorem}

\begin{corollary}
\label{sdfggf}
If $\wm$ is sated and free, then a free immersive curve $\gamma\colon I\rightarrow M$ is either a free segment or positive. In particular, for each $t\in I$, there exists $J\subseteq I$ free and open with $t\in J$ such that $g\cdot \gamma(J)\cap \gamma(J)$ is finite for each $g\neq e$.
\end{corollary}
\begin{proof}
Since $\wm$ is free, it only admits trivial stabilizers; in particular $G_\gamma=\{e\}$ holds. Thus, $\gamma$ can neither admit a  $\tau$-decomposition nor a negative interval, i.e., $\gamma$ is either a free segment or positive by Theorem \ref{sfdknfdhujfd}. 
Let now $t\in I$ fixed. Then, there exists $J'\subseteq I$ free and open with $t\in J'$. This is clear if $\gamma$ is a free segment, and follows from Proposition \ref{sdfdfdlla}.\ref{sdfdfdlla2} if $\gamma$ is positive. 
By Corollary \ref{dfdsasasasassa}, we can shrink $J'$ around $t$ such that $\gamma|_{J'}$ is an embedding; and we fix a compact interval $K\subseteq J'$ with $t\in J:=\innt[K]$. Now, if $g\cdot \gamma(J)\cap \gamma(J)$ is infinite for some $g\in G$, then $\im[g\cdot \gamma|_{J'}]\cap \im[\gamma|_{J'}]$ admits an accumulation point (in $g\cdot \gamma(K) \cap \gamma(K)$ by compactness). Then, Lemma \ref{lemma:BasicAnalyt1} shows $g\cdot \gamma|_{J'}\cpsim \gamma|_{J'}$, hence $g\in G_{\gamma|_{J'}}=G_\gamma=\{e\}$. 
\end{proof}
\begin{example}
\label{pofdpodfoppofdpofd}
Corollary \ref{sdfggf} in particular applies to the situation where a Lie group acts via left multiplication on itself (the corresponding action is sated by  Remark \ref{remmmmmi}.\ref{dslkjkjdskjdskjdskjsdkjkjdskjds87ds87dsa}).\hspace*{\fill}$\ddagger$
\end{example}

\subsection{Arbitrary Domains}
\label{rzerfgd}
To this point, we have only discussed decompositions of free immersive curves $\gamma\colon D\rightarrow M$, with $D$ an open interval. 
This was mainly for technical reasons; because, if $D$ is open, then there are no difficulties concerning the conventions that we have fixed in the end of Sect.\ \ref{repari}, when defining decompositions of free curves.  
However, the general case ($D$ is an arbitrary interval) now follows easily from the statements that already proven -- just by considering 
the maximal immersive analytic extension $\og\colon \oI\rightarrow M$ of $\gamma$, as well as the restriction $\ug := \gamma|_I$ of $\gamma$ to $(i',i)\equiv \uI:=\innt[D]$. Indeed, the key observations are that 
$\og$ and $\ug$ are free if $\gamma$ is free, and that  the following statement holds:
\begin{lemma}
\label{oidsoidsoidsds98ds98ds98dsdsdsdsds}
Let $\gamma\colon D\rightarrow M$ be a free immersive curve. Then, $\gamma$ is a free segment \deff $\ug$ is a free segment. 
\end{lemma}
\begin{proof}
It is clear that $\ug$ is a free segment if $\gamma$ is a free segment. Conversely, if $\ug$ is a free segment, then $\uI$ is free w.r.t.\ $\gamma$ by \eqref{ldslkdlkdsd09s09d09ds09ds09ds09dssddsdscxcxcx}, so that $D$ is free w.r.t.\ $\gamma$ by \eqref{dfdffdfdfdfd} (as $D$ is the closure of $I$ in $D$).
\end{proof}
\noindent
Theorem \ref{sfdknfdhujfd} and Lemma \ref{oidsoidsoidsds98ds98ds98dsdsdsdsds} yield:
\begin{corollary}
\label{ddsf}
Let $\wm$ be sated; and let  $\gamma\colon D\rightarrow M$ be  immersive and free. If $\gamma$ is not a free segment by itself, then $\ug$ either admits a unique $\tau$-decomposition or a compact maximal interval.
\end{corollary}
\subsubsection{$\boldsymbol{\tau}$-Decompositions}
Let $\gamma\colon D\rightarrow M$ be analytic immersive; and let 
 $\tau\in \innt[D]$ be fixed. 
\begingroup
\setlength{\leftmargini}{12pt}
\begin{itemize}
\item
 We define\qquad $D_-:=D\cap(-\infty,\tau]$\qquad as well as\qquad  $D_+:=D\cap [\tau,\infty)$.
 \item
 We write $g\cdot \gamma|_{D_-}\rightarrow \gamma|_{D_+}$ for $g\in G$ \deff there exists a (necessarily unique) analytic diffeomorphism $\mu\colon D_-\supseteq B_-\rightarrow B_+\subseteq D_+$ with $\mu(\tau)=\tau\in B_\pm$, such that 
 $$g\cdot \gamma|_{B_-}= \gamma|_{B_+}\cp \mu\qquad\quad\text{holds with}\qquad\quad B_-=D_-\quad\vee\quad  B_+=D_+.$$ 
If it helps to clarify the argumentation, we also say that $g\cdot \gamma|_{D_-}\rightarrow \gamma|_{D_+}$  holds w.r.t.\ $\mu$.
\end{itemize}
\endgroup 
\noindent
Assume now that $D_\pm$ are free intervals: 
\begingroup
\setlength{\leftmargini}{12pt}
\begin{itemize}
\item
A {\bf $\boldsymbol{\tau}$-decomposition} of $\gamma$ is a class 
$[g]\subseteq G_{\gamma(\tau)}\setminus G_\gamma$,    
such that $g\cdot \gamma|_{D_-}\rightarrow \gamma|_{D_+}$ holds w.r.t.\ $\mu$.

(The same arguments as in Lemma \ref{eoder}.\ref{eoder1} show that then $D_\pm$ are automatically maximal, and the only maximal intervals.)
\item
We say that $[g]$ is {\bf faithful} \deff $g'\cdot \gamma|_{D_-} \cpsim \gamma$  w.r.t.\ $\rho$ implies that  
\begin{align*}
\text{either}\qquad\:\:\: [g']=[e]\quad\wedge\quad\ovl{\rho}|_{D_-}=\id_{D_-}\qquad\:\:\text{or}\qquad\:\:\: [g']=[g]\quad\wedge\quad \ovl{\rho}|_{\dom[\mu]}=\mu\qquad\:\:\:\text{holds.}
\end{align*} 
\end{itemize}
\endgroup
\noindent 
By \eqref{dfdffdfdfdfd} and  \eqref{ldslkdlkdsd09s09d09ds09ds09ds09dssddsdscxcxcx}, the intervals $D_-$ and $D_+$ are free \deff the intervals $(i',\tau]$ and $[\tau,i)$ are free. Moreover, from our discussions in Sect.\ \ref{repari}, we easily obtain the equivalence:
\begin{align}
\label{sdpodspopodspod9898454545kj45kj4545sds09ds09ds}
	g\cdot \gamma|_{D_-}\rightarrow \gamma|_{D_+}\qquad\quad\Longleftrightarrow \qquad\quad g\cdot \gamma|_{(i',\tau]}\psim \gamma|_{[\tau,i)}
\end{align} 
\vspace{-25pt}

\noindent
In fact:
\begingroup
\setlength{\leftmargini}{12pt}
\begin{itemize}
\item[$\triangleright$]
If the left side of \eqref{sdpodspopodspod9898454545kj45kj4545sds09ds09ds} holds w.r.t.\ $\mu$, then the right  side holds w.r.t.\
\begingroup
\setlength{\leftmarginii}{12pt}
\begin{itemize}
\item
$\um:=\mu|_{(i',\tau]}$\:\: if \:\:$B_-=D_-$ holds,
\item
$\um:=\mu|_{(b',\tau]}$\:\: if\:\: $B_-\subset D_-$ is of the form $(b',\tau]$ or $[b',\tau]$.
\end{itemize}
\endgroup
\noindent 
\item[$\triangleright$]
If the right side of \eqref{sdpodspopodspod9898454545kj45kj4545sds09ds09ds} holds w.r.t.\ $\um$, then the left  side holds w.r.t.\
\begingroup
\setlength{\leftmarginii}{12pt}
\begin{itemize}
\item
$\mu:=\ovl{\um}|_{D_-}$ \hspace{6pt}if\: $\im[\um]\subset[\tau,i)$ holds,
\item
$\mu:=\um$\hspace{16pt} \hspace{6pt}if\: $\im[\um]= [\tau,i)=D_+$ holds,
\item
$\mu:=\ovl{\um}|_C$\hspace{7pt} \hspace{6pt}if\: $\im[\um]= [\tau,i)$\: and\: $D_+=[\tau,i]$ holds; where $C$ denotes the closure of $\dom[\um]$ in $D_-$,
\end{itemize}
\endgroup
\noindent 
whereby $\ovl{\um}$ denotes the maximal analytic immersive extension of $\um$.
\end{itemize}
\endgroup  
\noindent
Consequently, we have the equivalence
\begin{align}
\label{oifdoifdoidf98fd09fd0909fd09fdfdfdfdfdffd}
	[g]\:\:\text{is a}\:\: \tau\text{-decomposition of}\:\: \gamma\:\: \text{w.r.t.}\:\:\mu\qquad\:\:\Longleftrightarrow\qquad\:\:
	[g]\:\:\text{is a}\:\: \tau\text{-decomposition of}\:\: \ug\:\: \text{w.r.t.}\:\:\um,
\end{align}
whereby $\mu$ and $\um$ are related to each other as described above. 
 Hence, if $[g]$ is a $\tau$-decomposition of $\gamma$, then  
\begingroup
\setlength{\leftmargini}{12pt}
\begin{itemize}
\item
$[g]$ is unique  
by Lemma \ref{eoder}.\ref{eoder3},
\item
$[g]$ is faithful w.r.t.\ $\gamma$, because $[g]$ is faithful w.r.t.\ $\ug$ by  Lemma \ref{taufaith} and we have the implication:
\begin{align*}
	g'\cdot \gamma|_{D_-} \cpsim \gamma\quad\text{w.r.t.}\quad \rho\qquad\quad\Longrightarrow\qquad\quad g'\cdot \ug|_{(i',\tau]} \cpsim \ug \quad\text{w.r.t.}\quad \rho\qquad\:\:
\end{align*}     
\end{itemize}
\endgroup
\noindent
Finally, \eqref{ldslkdlkdsd09s09d09ds09ds09ds09dssddsdscxcxcx} yields the equivalence: 
\begin{align}
\label{dsdslkdsoidsoidsds09ds09ds09ds09ds0909dsdsds}
 A\:\:\text{compact maximal w.r.t.}\:\:\gamma \qquad\quad\Longleftrightarrow\qquad\quad A\:\:\text{compact maximal w.r.t.}\:\:\ug\qquad 
\end{align}
\begin{corollary}
\label{ddsfdf}
Let $\wm$ be sated; and let $\gamma\colon D\rightarrow M$ be  immersive and free. If $\gamma$ is not a free segment by itself, then it either admits a unique $\tau$-decomposition or a compact maximal interval $A\subseteq \innt[D]$.
\end{corollary} 
\begin{proof}
This is clear from Corollary \ref{ddsf} as well as the equivalences \eqref{oifdoifdoidf98fd09fd0909fd09fdfdfdfdfdffd} and \eqref{dsdslkdsoidsoidsds09ds09ds09ds09ds0909dsdsds}.
\end{proof}
\subsubsection{$\boldsymbol{A}$-Decompositions}
According to Corollary \ref{ddsfdf}, it remains to discuss the situation where the analytic immersion  $\gamma\colon D\rightarrow M$ admits a compact maximal interval $A\subseteq \innt[D]$:
\vspace{6pt}

\noindent
Let $\delta \colon [b',b]\equiv B \rightarrow M$ and $\delta'\colon C\rightarrow M$ be analytic immersions, with $C$ of the form $[c',c]$ or $[c',c)$ or $(c',c]$. We write 
 $\delta \ppsim \delta'$\hspace{0.5pt} \deff\hspace{1.5pt}$\delta|_{B'}=\delta'\cp \rho$ holds for an analytic diffeomorphism $\rho\colon B\supseteq B' \rightarrow C$ with 
\begin{align*}
	B'\cap \{b',b\}\neq \emptyset\qquad\quad\text{as well as}\qquad\quad \rho(B'\cap \{b',b\})\cap (C\backslash \innt[C])\neq \emptyset.
\end{align*} 
In analogy to Definition \ref{fdsafdsdfs}.\ref{fdsafdsdfs2},  we  define:\he\footnote{Deviating from Definition \ref{fdsafdsdfs}.\ref{fdsafdsdfs2}, we have assumed that $A$ is maximal (not only free). However, according to Lemma \ref{eoder}.\ref{eoder2}, this makes no difference.}
\begingroup
\setlength{\leftmargini}{12pt}
\begin{itemize}
\item
An {\bf $\boldsymbol{A}$-decomposition} of $\gamma$ is a pair $(\{a_n\}_{n\in \cn},\{[g_n]\}_{n\in \cn})$ with $\{a_n\}_{n\in \cn}$ a decomposition of $D$ and $\{g_n\}_{n\in \cn}\subseteq G$, such that 
$A_0= A$ and\: $[g_{\pm1}]\neq [e]$ holds, as well as
\begin{align*}
\begin{split}
g_n\cdot \gamma|_{A}\psim \gamma|_{A_n}\:\:\:\text{w.r.t.}&\:\:\: \mu_n\qquad\quad\forall\:  n\in \cn\setminus\{\cn_-,\cn_+\}\\[2pt]
&\text{and}\\
g_{\cn_\pm}\cdot \gamma|_A\ppsim \gamma|_{A_{\cn_\pm}}\:\:\:\text{w.r.t.}&\:\:\:\mu_{\cn_\pm}\:\:\:\text{if}\:\:\:\cn_\pm\neq \pm\infty\:\:\:\text{holds}. 
\end{split}
\end{align*}
We set $\mu_0:=\id_{A}$ and $g_0:= e$. 
\item
An $A$-decomposition $(\{a_n\}_{n\in \cn},\{g_n\}_{n\in \cn})$ is said to be {\bf faithful} \deff the following implication holds:
\begin{align*}
	g\cdot \gamma|_A \cpsim \gamma\:\:\:\text{w.r.t.}\:\:\: \rho\qquad\Longrightarrow \qquad 
	[g]=[g_n]\:\:\:\text{and}\:\:\: \ovl{\rho}|_{\dom[\mu_n]}=\mu_n 
	\:\:\:\text{for some unique}\:\:\: n\in \cn\cup \{0\}.	
\end{align*}
\end{itemize}
\endgroup
\begin{remark}
\label{lkdslkdsdsmnmdsnmdsnmdsnmdsdsdsds9999aaaxxx}
Let $\wm$ be sated, $\gamma\colon D\rightarrow M$ analytic immersive, and $A\subseteq \innt[D]$  a compact interval.
\begingroup
\setlength{\leftmargini}{16pt}
{
\renewcommand{\theenumi}{\sf\alph{enumi})} 
\renewcommand{\labelenumi}{\theenumi}
\begin{enumerate}
\item
\label{lkdslkdsdsmnmdsnmdsnmdsnmdsdsdsds9999aaaxxx1}
We have the equivalences:
\vspace{-2pt}
\begin{align*}
	\underbrace{A\:\:\text{maximal w.r.t.}\:\: \og}_{(*)} \qquad\:\:\Longleftrightarrow\qquad\:\:
	\underbrace{A\:\:\text{maximal w.r.t.}\:\: \gamma}_{(**)}
	\qquad\:\:\Longleftrightarrow\qquad\:\: 
	\underbrace{A\:\:\text{maximal w.r.t.}\:\: \ug}_{(***)}
\end{align*}
\vspace{-12pt}

\noindent
In fact, the equivalence \eqref{ldslkdlkdsd09s09d09ds09ds09ds09dssddsdscxcxcx} yields the implications $(*)\:\Rightarrow\: (**)\:\Rightarrow\: (***)$, whereby    
$(*)\:\Leftrightarrow\: (***)$ holds by 
Lemma \ref{freemax}.\ref{freemax2}  with $g\equiv e$, $\rho\equiv \id_A$ applied to $\gamma\equiv \og$ and $\gamma'\equiv \ug$ as well as to $\gamma\equiv \ug$ and $\gamma'\equiv \og$. 
\item
\label{lkdslkdsdsmnmdsnmdsnmdsnmdsdsdsds9999aaaxxx2}
The (unique faithful) $A$-decomposition $\underline{\alpha}$ of $\ug$ arises from the (unique faithful) $A$-decomposition $\overline{\alpha}$ of $\og$ via restriction; i.e., 
if\: $\ovl{\alpha}\equiv (\{\ovl{a}_n\}_{n\in \cn},\{[\ovl{g}_n]\}_{n\in \overline{\cn}})$\: with corresponding diffeomorphisms $\{\ovl{\mu}_n\}_{n\in \ovl{\cn}}$, then 
\begin{align*}
\underline{\alpha}= (\{\ovl{a}_n\}_{n\in \underline{\cn}},\{[\ovl{g}_n]\}_{n\in \underline{\cn}})
	\qquad\quad\text{ho}&\text{lds for}\qquad\quad \underline{\cn}:= \{n\in \overline{\cn}\:|\: \ovl{a}_n\in \innt[D]\}\in \CN\\[1pt]
	&\text{with}\\
	\underline{\mu}_n=\ovl{\mu}_n\:\: \text{ for }\:\: \underline{\cn}_-<n<\underline{\cn}_+\qquad &\he\text{and}\qquad \underline{\mu}_{\cn_{\pm}}\!=\overline{\mu}|_{A_{\cn_\pm}}\:\: \text{if}\quad \underline{\cn}_{\pm} \neq \pm \infty\: \text{ holds.}
\end{align*} 
\end{enumerate}}
\endgroup
\end{remark}
\noindent
A restriction argument analogous to that in Remark \ref{lkdslkdsdsmnmdsnmdsnmdsnmdsdsdsds9999aaaxxx}.\ref{lkdslkdsdsmnmdsnmdsnmdsnmdsdsdsds9999aaaxxx2}, yields the following statement.
\begin{lemma}
\label{dskjdsjskjkdsjsd9898ds98ds98ds98dsdsdsdsds}
Let $\wm$ be sated, $\gamma\colon D\rightarrow M$ analytic immersive, and $A\subseteq \innt[D]$ compact maximal. Then,  there exists a (necessarily unique) faithful $A$-decomposition of $\gamma$, and the points a) and b) in Lemma \ref{eoder}.\ref{eoder2} hold with $I\equiv \innt[D]$ there.
\end{lemma}
\begin{proof}
\begingroup
\setlength{\leftmargini}{12pt}
\begin{itemize}
\item
{\sf Existence:} 
\!\!$A$ is compact maximal w.r.t.\ $\og$ by Remark \ref{lkdslkdsdsmnmdsnmdsnmdsnmdsdsdsds9999aaaxxx}.\ref{lkdslkdsdsmnmdsnmdsnmdsnmdsdsdsds9999aaaxxx1}. Hence, $\og$ admits a unique faithful $A$-decomposition $\ovl{\alpha}\equiv (\{\ovl{a}_n\}_{n\in \ovl{\cn}},\{[\ovl{g}_n]\}_{n\in \overline{\cn}})$ by Lemma \ref{dsfdsffds} and Corollary \ref{sdffdfdexi}.  
Then, the restriction
\begin{align*}
\alpha\equiv (\{\ovl{a}_n\}_{n\in \cn},\{[\ovl{g}_n]\}_{n\in \cn})
	\qquad\quad\text{for}\qquad\quad \cn:= \{n\in \overline{\cn}\:|\: \ovl{a}_n\in \innt[D]\}\in \CN
\end{align*}
is a faithful $A$-decomposition of $\gamma$ (with $\mu_n=\ovl{\mu}_n$ for  $\cn_-<n<\cn_+$ and $\mu_{\cn_{\pm}}=\overline{\mu}|_{A_{\cn_\pm}}$ if ${\cn}_{\pm} \neq \pm \infty$ holds) such that the points a) and b) in Lemma \ref{eoder}.\ref{eoder2} hold with $I\equiv \innt[D]$ there.
\item
{\sf Uniqueness:} Since $\alpha$ is faithful, the same arguments as in Lemma \ref{dsfdsffds} show the uniqueness statement. \qedhere  
\qedhere  
\end{itemize}
\endgroup
\end{proof}
\begin{remark}
The uniqueness statement in Lemma \ref{dskjdsjskjkdsjsd9898ds98ds98ds98dsdsdsdsds} also follows via restriction. Specifically, 
an $A$-decomposition $\alpha$ of $\gamma$ yields an $A$-decomposition $\underline{\alpha}$ of $\ug$, just by restricting the diffeomorphisms $\mu_{\cn_\pm}$ to $\innt[D]\cap A_{\cn_\pm}$ for the case $\cn_\pm \neq \pm \infty$.\footnote{The index set $\cn$ remains the same; because, according to our definition of a decomposition of an arbitrary interval $D$ (see Definition \ref{dskskjsdkjdskjdsds98ds98dsdsdsdsds}), we have $\{a_n\}_{n\in \cn}\subseteq\innt[D]$.}  
Since the diffeomorphisms $\mu_{\cn_\pm}$ are continuous, uniqueness of $\alpha$ is then clear from uniqueness of the $A$-decomposition of $\ug$. 
Notably, $\underline{\alpha}$ coincides with the $A$-decomposition of $\ug$ that we have obtained in Remark \ref{lkdslkdsdsmnmdsnmdsnmdsnmdsdsdsds9999aaaxxx} by restricting the $A$-decomposition $\ovl{\alpha}$ of $\og$ to $\ug$ (this is clear from uniqueness; alternatively from the explicit  construction).\hspace*{\fill}$\ddagger$
\end{remark}
\noindent
Summing up, we have that $\overline{\alpha}$ restricts to $\alpha,\underline{\alpha}$, whereby $\alpha$ restricts to $\underline{\alpha}$.   
\vspace{6pt}

\noindent
Let now $\wm$ be sated, $\gamma\colon D\rightarrow M$  analytic immersive, and $A\subseteq \innt[D]$ compact maximal w.r.t.\ $\gamma$. Then, $A$ is compact maximal w.r.t.\ $\ug$ by Remark \ref{lkdslkdsdsmnmdsnmdsnmdsnmdsdsdsds9999aaaxxx}.\ref{lkdslkdsdsmnmdsnmdsnmdsnmdsdsdsds9999aaaxxx1}, and we define:  
\vspace{-1pt}
$$\gamma\:\:\:\text{{\bf positive\slash negative}}\qquad\quad\stackrel{\rm def.}{\Longleftrightarrow} \qquad\quad \ug\:\:\:\text{positive\slash negative}
$$
We observe the following:
\begingroup
\setlength{\leftmargini}{12pt}
\begin{itemize}
\item
Theorem \ref{sfdknfdhujfd} (or Lemma \ref{fdggf}.\ref{fdggf3}) shows that \ref{sfdknfdhujfd} $\gamma$ is either positive or negative.
\item
Lemma \ref{fdggf}.\ref{fdggf3} implies that
\vspace{-2pt}
\begingroup
\setlength{\leftmarginii}{12pt}
\begin{itemize}
\item
each compact maximal $B\subseteq \innt[D]$ is positive/negative w.r.t.\ $\ug$.
\vspace{2pt} 
\item
each compact maximal $B\subseteq \ovl{I}$ is positive/negative w.r.t.\ $\og$,   
because $B$ is compact maximal w.r.t.\ $\og$ by Remark \ref{lkdslkdsdsmnmdsnmdsnmdsnmdsdsdsds9999aaaxxx}.\ref{lkdslkdsdsmnmdsnmdsnmdsnmdsdsdsds9999aaaxxx1} as well as positive/negative by faithfulness.
\end{itemize}
\endgroup  
\end{itemize}
\endgroup
\noindent
Together with Remark \ref{lkdslkdsdsmnmdsnmdsnmdsnmdsdsdsds9999aaaxxx}.\ref{lkdslkdsdsmnmdsnmdsnmdsnmdsdsdsds9999aaaxxx1},  Proposition \ref{sdfdfdlla} and Proposition \ref{trhdhg} applied to $\og$ and $\ug$ yield the following:
\begin{remark}
\label{oidsoidsoisdoids9dsds98ds9898dsdsdsdsdsds}
Let $\wm$ be sated;  
and let    
 $\gamma\colon D\rightarrow M$ be analytic immersive.
\begingroup
\setlength{\leftmargini}{16pt}
{
\renewcommand{\theenumi}{\sf\alph{enumi})} 
\renewcommand{\labelenumi}{\theenumi}
\begin{enumerate}
\item
\label{aaafdpfdpopofd2}
	If $\gamma$ is negative, then Proposition \ref{trhdhg}   holds (with $D$ instead of $I$ and $A\subseteq \innt[D]$)  
	as follows: 
\begingroup
\setlength{\leftmarginii}{12pt}
\begin{itemize}
\item
Part \ref{trhdhg1} holds in the same form (same proof as in Lemma \ref{fdggf}.\ref{fdggf2}).\footnote{Since $\overline{\alpha}$ restricts to $\alpha$, this alternatively  follows from \eqref{ldslkdlkdsd09s09d09ds09ds09ds09dssddsdscxcxcx} and Lemma \ref{fdggf}.\ref{fdggf2} applied to $\og$.} 
\item	
Part \ref{trhdhg2}  holds in the same form (as $\overline{\alpha}$ restricts to $\alpha$), whereby the left side of \eqref{sdsdffghhh} reads
\begin{align*}	
	g_{\cn_\pm}\cdot \gamma|_A \ppsim\gamma|_{A_{\cn_\pm}}\:\: \:\text{if}\quad \cn_{\pm}\neq \pm\infty\quad\text{holds, with}\quad \mu_{\cn_\pm}\:\:\:\text{positive\slash negative}\:\:\:\: \text{if}\quad \cn_\pm\:\:\:\text{is even\slash odd.}
\end{align*}
\item
Part \ref{trhdhg3} holds in the same form (as $\overline{\alpha}$ restricts to $\alpha$).
\end{itemize}
\endgroup
\item
If $\gamma$ is positive, then Proposition \ref{sdfdfdlla} holds (with $D$ instead of $I$ and $A\subseteq \innt[D]$) as follows:
\begingroup
\setlength{\leftmarginii}{12pt}
\begin{itemize}
\item
Part \ref{sdfdfdlla1} holds in the same form (as $\overline{\alpha}$ restricts to $\alpha$), whereby the right side of 
\eqref{dslfofodpfd} reads 
\begin{align*}
	h^{\cn_\pm}\cdot \gamma|_A \ppsim\gamma|_{A_{\cn_{\pm}}}\:\:\text{with}\quad \dot\mu_{\cn_\pm}>0\quad \text{if}\quad \cn_{\pm}\neq \pm\infty\quad\text{holds.}
\end{align*}
\vspace{-15pt}
\item
Part \ref{sdfdfdlla2} holds in the same form (as $\alpha$ restricts to $\underline{\alpha}$) if the term  ``positive'' is replaced by ``compact maximal (and positive w.r.t.\ $\ug$)'', whereby the last statement holds for each 
$t\in \innt[D]$.
\end{itemize}
\endgroup
Additionally,  
for each $t\in D$, there exists $B_t\subseteq D$ compact maximal (and positive w.r.t.\ $\og$) with $t\in B_t$.
\vspace{-6pt}
\begin{proof}
For $t\in \innt[D]$, the claim is clear from the second point above (and Remark \ref{lkdslkdsdsmnmdsnmdsnmdsnmdsdsdsds9999aaaxxx}.\ref{lkdslkdsdsmnmdsnmdsnmdsnmdsdsdsds9999aaaxxx1}). Assume thus  
$t=\sup(D)\in D$ (the case  $t= \inf(D)\in D$ is treated analogously); and let $A\subseteq \innt[D]$ be positive w.r.t.\ $\ug$, hence positive w.r.t.\ $\og$. Let  $\ovl{\alpha}\equiv (\{\ovl{a}_n\}_{n\in \cn},\{[\ovl{g}_n]\}_{n\in \overline{\cn}})$ denote the $A$-decomposition of $\og$. Since $A\subseteq \innt[D]$ holds, we have 
\begin{align*}
	\text{either}\qquad\quad t=\ovl{a}_n\quad\text{for some}\quad n\geq 2
	\qquad\quad\text{or}\qquad\quad  t\in \innt[A_{n+1}]\quad\text{for some}\quad n\geq 0. 
\end{align*} 
\begingroup
\setlength{\leftmarginii}{12pt}
\begin{itemize}
\item
	In the first case, $A_t:=A_2$ is compact maximal w.r.t.\ $\gamma$ by \eqref{ldslkdlkdsd09s09d09ds09ds09ds09dssddsdscxcxcx}.
\item
	In the second case, the second point in Proposition \ref{sdfdfdlla}.\ref{sdfdfdlla2} (applied to $\og$ with $b\equiv t$) yields $t'\in A_n$ such that $A_t=[t',t]$ is positive w.r.t.\ $\og$, hence compact maximal w.r.t.\ $\gamma$ by \eqref{ldslkdlkdsd09s09d09ds09ds09ds09dssddsdscxcxcx}.
	\qedhere
\end{itemize}
\endgroup 
 \end{proof} 
\end{enumerate}}
\endgroup
\end{remark}

\subsubsection{The Classification}
The above discussions together with Theorem \ref{sfdknfdhujfd}  show the following statement: 
\begin{theorem}
\label{sfdknfdhujfdd}
Let $\wm$ be sated, and $\gamma\colon D\rightarrow M$ immersive and free. If $\gamma$ is not a free segment by itself, then it either admits a unique $\tau$-decomposition or a compact maximal interval contained in the interior of $D$. In the second case, $\gamma$ is either positive or negative, with what the corresponding statements in Remark \ref{oidsoidsoisdoids9dsds98ds9898dsdsdsdsdsds} hold.
\end{theorem}
\noindent
We finally want to state the following variation of Corollary \ref{sdfggf}:
\begin{corollary}
\label{sdfggfsda}
If $\wm$ is sated and free, then a free immersive curve $\gamma\colon I\rightarrow M$ is either a free segment or positive. Moreover, for each $t\in D$, there exists an open interval $J\subseteq \RR$ containing $t$ such that $g\cdot \gamma(J\cap D)\cap \gamma(J\cap D)$ is finite for each $g\neq e$.
\end{corollary}
\begin{proof}
\begingroup
\setlength{\leftmargini}{12pt}
\begin{itemize}
\item
The first statement is clear from Lemma \ref{oidsoidsoidsds98ds98ds98dsdsdsdsds} as well as Corollary \ref{sdfggf} applied to $\ug$.
\item 
The second statement is clear from      
Corollary \ref{sdfggf} applied to $\og$.\qedhere 
\end{itemize}
\endgroup
\end{proof}

\section{Extension: Analytic 1-Manifolds}
\label{ghdhgg} 
Let $(S,\iota)$ be  an {\bf analytic 1-submanifold} of $M$, i.e., $S$ is a  
 connected,  Hausdorff, second countable $1$-dimensional  analytic manifold with boundary,   
and $\iota\colon S\rightarrow M$ is an injective analytic immersion.\footnote{We note that each  connected, Hausdorff, second countable 1-dimensional manifold with boundary  is either homeomorphic to $\UE$ or to an interval \cite{Gale}.} 

Given a chart $(U,\psi)$ of $S$, we always assume that $U\subseteq S$ is open and connected, and that $\iota(U)\subseteq M$ is contained in the domain of a chart of $M$. We furthermore use the convention that $\im[\psi]$ is an open subset of $(-\infty,0]$. 
 Then,
 $(U,\psi)$   
 defines the analytic immersive curve
\begin{align*}
\gamma_\psi\colon D_\psi\equiv \psi(U)\rightarrow M,\qquad t\mapsto \iota\cp \psi^{-1}(t).
\end{align*} 
This raises the question, whether 
the results obtained so far carry over to analytic 1-submanifolds:
\begin{definition}
\label{dslklkdssdodsoisdoisiudzudsds87s87d87dsdssd}
An analytic 1-submanifold $(S,\iota)$ of $M$ is said to be
\begingroup
\setlength{\leftmargini}{12pt}
\begin{itemize}
\item
{\bf exponential}\qquad\:\:$\stackrel{\rm def.}{\Longleftrightarrow}$\qquad\:\:$\wm$ is analytic in $G$\quad\:\:\hspace{3pt}$\wedge$\quad\:\:$\gamma_\psi$ is exponential for some chart $(U,\psi)$.
\item
\hspace{35.9pt} {\bf free}\qquad\:\:$\stackrel{\rm def.}{\Longleftrightarrow}$\qquad\:\:$\wm$ is analytic in $M$\quad\:\:$\wedge$\quad\:\:$\gamma_\psi$ is free for some chart $(U,\psi)$.
\end{itemize}
\endgroup
\end{definition}
\noindent 
In Sect.\ \ref{dffdfddfdf}, we prove the following theorem:
\begin{theorem}
\label{ghfgh}
If $\wm$ is regular and separately analytic, then an analytic 1-submanifold $(S,\iota)$ of $M$ is either exponential or free; whereby the following assertions hold: 
\vspace{-4pt}
\begingroup
\setlength{\leftmargini}{12pt}
\begin{itemize}
\item
$(S,\iota)$ is exponential \deff $\gamma_\psi$ is exponential (and not free) for each chart $(U,\psi)$ of $S$.
\item
$(S,\iota)$ is free \deff $\gamma_\psi$ is free (and not exponential) for each chart $(U,\psi)$ of $S$.
 \end{itemize}
 \endgroup
\end{theorem}
\subsection{The Classification}
\label{dffdfddfdf}
Let $(S,\iota)$ be an analytic 1-submanifold of $M$. In the following, we write $G_x\equiv  G_{\iota(x)}$ for each $x\in S$.   
\begin{lemma}
\label{dffdfdfd}
If $\wm$ is analytic in $M$, then    
the stabilizer 
$\textstyle G_S:= \bigcap_{z\in S}G_{z}$ 
of $S$    
coincides with the stabilizer of  $\gamma_\psi$ for each chart $(U,\psi)$.
\end{lemma}
\begin{proof}
The proof is elementary, and can be found in Appendix \ref{appD1}.
\end{proof}
\begin{lemma}
\label{oisdkdslkdslkdskjdskjsdkjds98ds98ds98dsdsdsds}
If $(U_0,\psi_0)$ is a chart of $S$, then there exist charts $\{(U_n,\psi_n)\}_{n\geq 1}$   of $S$ with $S=\bigcup_{n\in \NN}U_n$ such that $U_n\cap U_{n+1}\neq \emptyset$ holds for all $n\in \NN$.
\end{lemma}
\begin{proof}
The proof is elementary, and can be found in Appendix \ref{appD2}.
\end{proof}
\begin{proposition}
 \label{dfgffhfh}
Let $\wm$ be sated and separately analytic; and let $(S,\iota)$ be  exponential.
\begingroup
\setlength{\leftmargini}{16pt}
{
\renewcommand{\theenumi}{\sf\arabic{enumi})} 
\renewcommand{\labelenumi}{\theenumi}
\begin{enumerate}
\item
 \label{dfgffhfh1}
Each $\gamma_\psi$ is exponential with respect to the same $(x,\g)$, i.e., the following assertions hold:
\vspace{-4pt}
\begingroup
\setlength{\leftmarginii}{12pt}
\begin{itemize}
\item 
There exist $x\in \im[\iota]$ and $\g\in \mg\setminus \mg_{x}$,  such that $\gamma_\psi$ is exponential w.r.t.\ $(x,\mg)$ for each chart $(U,\psi)$.
\item
If $(U,\psi)$ is a chart of $S$, then we have the equivalence: \hspace*{\fill}($\mg_S$ denotes the Lie algebra of $G_S$)
$$
\gamma_\psi\:\: \text{is exponential w.r.t.}\:\: (y,\q)\qquad\quad\Longleftrightarrow\qquad\quad y\in \exp(\RR\cdot \g)\cdot x\quad\wedge\quad \q \in \RR_{\neq 0}\cdot \g +\mg_S
$$
\end{itemize}
\endgroup
\item
 \label{dfgffhfh2}
Exactly one of the following two situations hold:\he\footnote{Both $D$ and $\UE$ are assumed to carry their standard analytic structures.}
\begingroup
\setlength{\leftmarginii}{12pt}
\begin{itemize}
\item
$(S,\iota)$ is analytically diffeomorphic to $\UE$ via
\begin{align}
\label{ghhgghqqqq1}
 \Omega\colon \UE\rightarrow S,\qquad \e^{\I \phi}\mapsto \iota^{-1}(\exp(\phi\cdot \g)\cdot x)
 \end{align}
 with $\g$ rescaled such that $\pii(x,\g)=2\pi$ holds.
 \item
$(S,\iota)$ is analytically diffeomorphic to an interval $D\subseteq \RR$ via
\begin{align}
\label{ghhgghqqqq2}
  \:\:\Omega\colon D\rightarrow S,\qquad  t\mapsto \iota^{-1}(\exp(t\cdot \g)\cdot x).
\end{align}
\end{itemize}
\endgroup
\noindent
In both situations, $\g$ is unique up to addition of elements in $\mg_S$ (with $D$ fixed in the second case). 
\end{enumerate}}
\endgroup
\end{proposition}
\begin{proof}
Confer Appendix \ref{appD3}. 
\end{proof}
\begin{corollary}
\label{lkdslksdlkdslkdsoidsoioidsds78987ds98ds98dsdsdsdsdsdsdss}
Let $\wm$ be sated and separately analytic. An analytic 1-submanifold $(S,\iota)$ of $M$ is exponential \deff it is either 
analytically diffeomorphic to $\UE$ or to some interval $D\subseteq \RR$ via 
\begin{align*}
	\UE\ni \e^{\I \phi}\mapsto \iota^{-1}(\exp(\phi\cdot \g)\cdot x)\in S\qquad\text{or}\qquad D\ni t\mapsto \iota^{-1}(\exp(t\cdot \g)\cdot x)\in S\qquad\text{respectively},
\end{align*}
with $x\in im[\iota]$ as well as $\g\in \mg\backslash \mg_x$ 
unique up to addition of elements in $\mg_S$ (with $D$ fixed in the second case). 
\end{corollary}
\begin{proof}
Confer Appendix \ref{appD4}.
\end{proof}

\begin{proof}[Proof of Theorem \ref{ghfgh}]
The claim is clear from  Proposition \ref{dfgffhfh}.\ref{dfgffhfh1} and Theorem \ref{classi}. 
\end{proof}
\noindent
If $\wm$ is sated and analytic in $M$, then the  last point in Theorem \ref{ghfgh} also holds in the following form: 
\begin{lemma}
\label{dslklkdslkdssdoioidsoidsds09ds09ds09ddsdsds09}
If $\wm$ is sated and analytic in $M$, then $(S,\iota)$ is free \deff $\gamma_\psi$ is free for each chart $(U,\psi)$. 
\end{lemma}
\begin{proof}
Confer Appendix \ref{appD4b}.
\end{proof}
\subsection{Mirror Points}
\label{dskjkjdsdsnmnmdsnmdsds98sd9898dsds}
In this brief subsection, we discuss the concept of Mirror points. 
\begin{definition}
\label{989898oioioihjhjhghg75767676uzhgffddsdadsa}
Assume that $\wm$ is analytic in $M$, and let   
$(S,\iota)$ be an analytic 1-submanifold of $M$.  
Then, $\FP\subseteq S$ denotes the set of all {\bf mirror points} in $S$, i.e., all $z\in S$ such that there exist $g\in G\backslash G_S$, a chart $(U,\psi)$ of $S$, and $J\subseteq \dom[\ovl{\gamma}_\psi]$ open with $\psi(z)\in J$ such that
\vspace{-2pt}
\begin{align}
\label{sdfsdfds}
g\cdot \ovl{\gamma}_\psi|_{(-\infty,\psi(z)]\:\cap\: J}\psim_{\psi(z),\psi(z)} \ovl{\gamma}_\psi|_{J\:\cap\: [\psi(z),\infty)}
\end{align}
holds. 
Evidently, \eqref{sdfsdfds} implies $g\in G_{z}$ (\he{}i.e., $g\in G_z\setminus G_S$, hence $g\neq e$).
\end{definition} 
\begin{lemma}
\label{ghhjgghas}
Let $\wm$ be analytic in $M$, and let $z\in S$ be given such that $\psi'(z)$ is contained in a free open interval $J'\subseteq \dom[\ovl{\gamma}_{\psi'}]$ for some chart $(U',\psi')$ of $S$. Then,  $z\in S\backslash \FP$ holds. 
\end{lemma}
\begin{proof}
Assume that the claim is wrong, i.e., there exist $g\in G\backslash G_S$, $(U,\psi)$, $J\subseteq \dom[\ovl{\gamma}_\psi]$  such that \eqref{sdfsdfds} holds. 
\begingroup
\setlength{\leftmargini}{12pt}
\begin{itemize}
\item
Since $z\in U\cap U'$, Lemma \ref{lemma:BasicAnalyt2} provides   an analytic diffeomorphism $\rho\colon \dom[\ovl{\gamma}_\psi]\supseteq I\rightarrow I'\subseteq \dom[\ovl{\gamma}_{\psi'}]$, with $\psi(z)\in I$, $\psi'(z)\in I'$, $\rho(\psi(z))=\psi'(z)$, and $\ovl{\gamma}_\psi|_I=\ovl{\gamma}_{\psi'}\cp \rho$.
\item
Together with the previous point, \eqref{sdfsdfds} 
 implies $g\cdot \ovl{\gamma}_{\psi'}|_{J'} \cpsim  \ovl{\gamma}_{\psi'}|_{J'}$. This, however, contradicts that $J'$ is free, as $g\notin G_S=G_{\gamma_{\psi'}}$ holds by Lemma \ref{dffdfdfd}.
\qedhere
\end{itemize}
\endgroup   
\end{proof}
\begin{lemma}
\label{ghhjggh}
Let $\wm$ be sated and analytic in $M$. Assume that $(S,\iota)$ is free; and let $(U,\psi)$ be a chart of $S$. 
Then, the following assertions hold:
\begingroup
\setlength{\leftmargini}{16pt}
{
\renewcommand{\theenumi}{\alph{enumi})} 
\renewcommand{\labelenumi}{\theenumi}
\begin{enumerate} 
\item
\label{ghhjggh1}
$\FP\cap U$ 
is (at most) countable. 
\item
\label{ghhjggh2}
For   
each $x\in U \setminus \FP$, there exists a free open interval $\tilde{J}\subseteq \dom[\ovl{\gamma}_\psi]=:I\equiv (i',i)$ with $\psi(x)\in \tilde{J}$.  
\end{enumerate}}
\endgroup
\noindent
Specifically, (exactly) one of the  following situations hold: 
\vspace{-2pt}
\begingroup
\setlength{\leftmargini}{12pt}
\begin{itemize}
\item
$\ovl{\gamma}_\psi$ is a free segment, and $\FP\cap U=\emptyset$. \hspace*{\fill}(\he$\tilde{J}=I$) 
\item
$\ovl{\gamma}_\psi$ admits a $\tau$-decomposition, and $\FP\cap U=  \psi^{-1}(\{\tau\})$.\hspace*{\fill}(\he$\tilde{J}=(i',\tau)$\: or\: $\tilde{J}=(\tau,i)$\he)
\item
$\ovl{\gamma}_\psi$  admits an $A$-decomposition $(\{a_n\}_{n\in \cn},\{[g_n]\}_{n\in \cn})$, and (exactly) one of the following situations hold:
\begingroup
\setlength{\leftmarginii}{17pt}
\begin{itemize}
\item[$(*)$]
$\ovl{\gamma}_\psi$ is positive, and $\FP\cap U=\emptyset$. \hspace*{\fill}(\he$\tilde{J}=\innt[A_t]$\: for some compact maximal\: $A_t\subseteq I$)
\vspace{1pt}
\item[$(*)$]
$\ovl{\gamma}_\psi$ is negative,\he and $\FP\cap U=\{\psi^{-1}(a_n) \:|\: n\in \cn\:\:\text{with}\:\: a_n\in \psi(U)\}$.
\hspace*{\fill}(\he$\tilde{J}=\innt[A_n]$\: for some\: $n\in \cn$)
\end{itemize}
\endgroup
\end{itemize}
\endgroup
\end{lemma}
\begin{proof}
$\ovl{\gamma}_\psi$ is free, because $\gamma_\psi$ is free (Lemma \ref{dslklkdslkdssdoioidsoidsds09ds09ds09ddsdsds09}). Hence, by  Theorem \ref{sfdknfdhujfd}, $\ovl{\gamma}_\psi$ is either a free segment or admits a $\tau$-decomposition or admits an $A$-decomposition: 
\begingroup
\setlength{\leftmargini}{12pt}
\begin{itemize}
\item
If $\ovl{\gamma}_\psi$ is a free segment, then 
$U\subseteq  S\setminus \FP$ holds by Lemma \ref{ghhjgghas}; 
and we can choose  $\tilde{J}= I$ for each $x\in U$.
\item
Assume that $\ovl{\gamma}_\psi$  admits a $\tau$-decomposition:
\vspace{-2pt}
\begingroup
\setlength{\leftmarginii}{12pt}
\begin{itemize}
\item
If $\tau\in \psi(U)$, then 
\eqref{sdfsdfds} holds for $z:=\psi^{-1}(\tau)$, hence $\psi^{-1}(\tau)\in  \FP\cap U$.
\vspace{2pt}
\item
If $z\in \FP\cap U$,  
then Lemma \ref{ghhjgghas} shows $(i',\tau)\not\ni \psi(z)\notin (\tau,i)$ (and $J\subseteq I$ suitably), hence $\psi(z)=\tau$.
\end{itemize}
\endgroup
Thus, $\FP\cap U=  \{\psi^{-1}(\tau)\}$ holds; and, for $x\in U\setminus \FP$, we either have $\psi(x)\in (i',\tau)$ or $\psi(x)\in (\tau,i)$.
\item
$\ovl{\gamma}_\psi$  admits an $A$-decomposition $(\{a_n\}_{n\in \cn},\{[g_n]\}_{n\in \cn})$:
\begingroup
\setlength{\leftmarginii}{17pt}
\begin{itemize}
\item[$(*)$]
If $\ovl{\gamma}_\psi$ is positive, then Proposition \ref{sdfdfdlla}.\ref{sdfdfdlla2} shows that each $t\in I$ is contained in the interior of a compact maximal interval $A_t\subseteq I$. Hence, $\FP\cap U=\emptyset$ holds by Lemma \ref{ghhjgghas}.
\item[$(*)$]
Assume that $\ovl{\gamma}_\psi$ is negative:
\begingroup
\setlength{\leftmarginiii}{12pt}
\begin{itemize}
\item[$-$]
If $a_n\in \psi(U)$ holds for $n\in \cn$, then \eqref{sdfsdfds} holds for $z:=\psi^{-1}(a_n)$ (and $J\subseteq I$ suitably), hence $\psi^{-1}(a_n)\in \FP\cap U$. 
\vspace{2pt}
\item[$-$]
If $z\in \FP\cap U$ holds, then 
Lemma \ref{ghhjgghas} shows $\psi(z)\notin \innt[A_n]$ for all $n\in \cn$. Hence, we necessarily have $\psi(z)=a_n$ for some $n\in \cn$.
\end{itemize}
\endgroup
Thus,  $\FP\cap U=\{\psi^{-1}(a_n) \:|\: n\in \cn\:\:\text{with}\:\: a_n\in \psi(U)\}$ holds; and, for $\in U\setminus \FP$, we necessarily have $\psi(x)\in \innt[A_n]$ for some $n\in \cn$.\qedhere
\end{itemize}
\endgroup
\end{itemize}
\endgroup
\end{proof}
\begin{lemma}
\label{jgjhjh}
Let $\wm$ be sated and analytic in $M$; and let $(S,\iota)$ be free. Then, the following assertions hold:
\begingroup
\setlength{\leftmargini}{16pt}
{
\renewcommand{\theenumi}{\arabic{enumi})} 
\renewcommand{\labelenumi}{\theenumi}
\begin{enumerate} 
\item
\label{jgjhjh1} 
$\FP$ is (at most) countable (even empty if $\wm$ is free).
\item
\label{jgjhjh2} 
For each $x\in S\backslash \FP$, there exists a  neighbourhood $V\subseteq S$ of $x$ such that $g\cdot \iota(V)\cap \iota(V)$ is finite for all $g\in G\backslash G_S$. 
\end{enumerate}}
\endgroup
\end{lemma}
\begin{proof}
\begingroup
\setlength{\leftmargini}{16pt}
{
\renewcommand{\theenumi}{\arabic{enumi})} 
\renewcommand{\labelenumi}{\theenumi}
\begin{enumerate}
\item
If $\wm$ is free, then we necessarily have $\FP=\emptyset$,  because $G_{z}=\{e\}$ holds for each $z\in S$. In general, we can cover $S$ by countably many charts,   
and conclude from Lemma \ref{ghhjggh}.\ref{ghhjggh1} that $\FP$ is countable. 
\item 
Let $(U,\psi)$ be a chart around $x$ (hence, $x\in U\setminus \FP$).  
\vspace{-3pt}
\begingroup
\setlength{\leftmarginii}{12pt}
\begin{itemize}
\item
Lemma \ref{ghhjggh}.\ref{ghhjggh2} yields  a free open interval $\tilde{J}\subseteq \dom[\ovl{\gamma}_\psi]$ with $\psi(x)\in \tilde{J}$. 
\item 
Shrinking $\tilde{J}$ around $\psi(x)$ if necessary, we can assume that $\ovl{\gamma}_\psi|_{\tilde{J}}$ is an embedding. 
\item
We fix a compact neighbourhood $K\subseteq \tilde{J}$ of $\psi(x)$, and set $J:=\innt[K]$ as well as $V:=\psi^{-1}(J\cap D_\psi)$. 
\end{itemize}
\endgroup
\vspace{-5pt}

\noindent
Assume now that $g\cdot \iota(V)\cap \iota(V)$ is infinite. Then, $g\cdot \ovl{\gamma}_\psi(J)\cap \ovl{\gamma}_\psi(J)$ is infinite; hence, $\im[g\cdot \ovl{\gamma}_\psi|_{\tilde{J}}]\cap \im[\ovl{\gamma}_\psi|_{\tilde{J}}]$ admits an accumulation point (in $g\cdot \ovl{\gamma}_\psi(K) \cap \ovl{\gamma}_\psi(K)$ by compactness). Lemma \ref{lemma:BasicAnalyt1} shows 
\vspace{-4pt}
\[
g\cdot \ovl{\gamma}_\psi|_{\tilde{J}}\cpsim \ovl{\gamma}_\psi|_{\tilde{J}} \qquad\:\:\stackrel{\tilde{J}\text{ free}}{\Longrightarrow}\qquad\:\: g\in G_{\ovl{\gamma}_\psi}\hspace{-3pt}\stackrel{\text{Lemma } \ref{lemma:stabi}}{=}G_{\gamma_\psi}\hspace{-3pt}\stackrel{\text{Lemma } \ref{dffdfdfd}}{=}G_S.\qedhere
\] 
\end{enumerate}}
\endgroup
\end{proof}
\noindent
Notably, Lemma \ref{jgjhjh}.\ref{jgjhjh2}   
holds in an analogous form if $x$ is a boundary point of $S$ (also if contained in $\FP$): 

\begin{lemma}
\label{fdfxvcx}
Let $\wm$ be sated and analytic in $M$; and let $(S,\iota)$ be free. If  $z\in S$  is a boundary point of $S$, then   
there exists a  neighbourhood $V\subseteq S$ of $z$ such that $g\cdot \iota(V)\cap \iota(V)$ is finite for each $g\in G\backslash G_S$.
\end{lemma} 
\begin{proof}
Confer Appendix \ref{appD5}.
\end{proof}

\subsection{Outlook -- Decompositions}
To obtain global decomposition results also for analytic 1-submanifolds $(S,\iota)$ of $M$, it seems reasonable to make the definitions more similar to that ones that we have used for analytic immersive curves. Let $\wm$ be analytic in $M$:
\begingroup
\setlength{\leftmargini}{12pt}
\begin{itemize}
\item
We write $g\cdot \iota|_\Sigma\cpsim \iota|_{\Sigma'}$ for segments (connected subsets of $S$ with non-empty interior) $\Sigma,\Sigma'\subseteq S$ \deff\he $g\cdot \iota(\OO)= \iota(\OO')$ holds for open segments $\OO\subseteq \Sigma$ and $\OO'\subseteq \Sigma'$ on which $\iota$ is an embedding.  
\item
We say that $(S,\iota)$ is {\bf free} \deff\he it admits a {\bf free segment}, i.e., a 
  connected subset $\Sigma\subseteq S$ with non-empty interior,   such that the following implication holds:
\begin{align*}
	g\cdot \iota|_\Sigma\cpsim \iota|_\Sigma\quad\text{for}\quad g\in G\qquad\quad\Longrightarrow\qquad\quad g\in G_S
\end{align*}
\begin{customlem}{I}\label{ssI98ds98ds98ds98ds89ds89ds}
$(S,\iota)$ is free in the sense stated above \deff $(S,\iota)$ is free in the sense of Definition \ref{dslklkdssdodsoisdoisiudzudsds87s87d87dsdssd}. 
\end{customlem} 
\vspace{-12pt}
\begin{proof}
Let $(U,\psi)$ be a chart of $S$, and let $I\subseteq D_\psi$ be an open interval such that $\gamma_\psi|_I$ is an (analytic) embedding (Corollary \ref{dfdsasasasassa}). Then, $C:=\psi^{-1}(I)$ is a segment such that $\iota|_C=\gamma_\psi\cp\psi|_C$ is an embedding. Hence, for  $g\in G$ fixed, we have the equivalence
\begin{align}
\label{sdkjkjdskjsdnmdsnmdsdsdssd998ds98ds98dsdss233233232d}
\begin{split}
	g\cdot \gamma_\psi|_I\cpsim \gamma_\psi|_I\qquad\quad&\Longleftrightarrow\qquad\quad  g\cdot \iota|_C\cpsim \iota|_C\\[3pt]
	\text{as} &\text{ well as}\\[2pt]
	&\hspace{-26pt} G_{\gamma_\psi}\stackrel{\rm Lemma\:\ref{dffdfdfd}}{=}G_S.
\end{split}
\end{align}
\vspace{-17pt}

We conclude as follows:
\vspace{-4pt}

\begingroup
\setlength{\leftmarginii}{12pt}
\begin{itemize}
\item
Let $(U,\psi)$ be a chart of $S$ such that $\gamma_\psi$ is free. Corollary \ref{dfdsasasasassa} yields a free open interval $I\subseteq D_\psi$ such that $\gamma_\psi|_{I}$ is an embedding. Then, $\Sigma:=\psi^{-1}(I)\subseteq S$ is a free segment by \eqref{sdkjkjdskjsdnmdsnmdsdsdssd998ds98ds98dsdss233233232d}.
\item
Let $\Sigma\subseteq S$ be a free segment. We fix a chart $(U,\psi)$ of $S$ with $U\subseteq \Sigma$, as well as an open interval $I\subseteq D_\psi$ such that $\gamma_\psi|_{I}$ is an  embedding (Corollary \ref{dfdsasasasassa}). Then, $C:=\psi^{-1}(I)\subseteq \Sigma$ is a free segment, because $\Sigma$ is a free segment. Hence, $\gamma_\psi|_I$ is a free segment by \eqref{sdkjkjdskjsdnmdsnmdsdsdssd998ds98ds98dsdss233233232d}, so  that $\gamma_\psi$ is free. \qedhere
\end{itemize}
\endgroup
\end{proof}
\item
A free segment $\Sigma$ is said to be maximal \deff\he the following implication holds:
$$
\Sigma'\subseteq S\quad\text{free segment with}\quad\Sigma\subseteq \Sigma'
\qquad\Longrightarrow\qquad
\Sigma'=\Sigma$$  
It is not hard to see that $\Sigma$ is closed in $S$, and that the following analogue to Lemma \ref{maxim} holds:
\begin{customlem}{II}\label{I98ds98ds98ds98ds89ds89ds}
If $\Sigma\subseteq S$ is a free segment, then there exists $\Sigma'\subseteq S$ maximal with $\Sigma\subseteq \Sigma'$. 
\end{customlem} 
\item
A boundary segment of a segment $\Sigma$ is a segment $\Sigma'\subseteq \Sigma$ such that $\Sigma\setminus\Sigma'$ is connected.
\end{itemize}
\endgroup
\noindent
Now, one possible strategy is first to consider such $(S,\iota)$ without boundary, and then to carry over the obtained results to the boundary case by considering the interior of $S$.  
More specifically, 
one first  has 
to adapt Proposition \ref{prop:shifttrans} to the  analytic 1-submanifold case (without boundary) by means of the classes $[g]\in G\slash G_S$; and, then has to go through the argumentation in Sect.\ \ref{disgencur}, keeping in mind that  $S$ can be compact now  (see  Case B below). 
Indeed, if $\wm$ is sated and analytic in $M$ and  $(S,\iota)$ is free without boundary, then the following results are to be expected:\footnote{Since a connected, Hausdorff, second countable 1-dimensional manifold with boundary  is either homeomorphic to $\UE$ or to an interval \cite{Gale}, $S$ is compact/non-compact without boundary \deff $S$ is homeomorphic to $\UE$/an open  interval.}
\vspace{6pt}

\noindent
{\bf Case A:}\:\: Let $S$ be non-compact and not a free segment by itself:\footnote{Observe that analytic embedded curves are both analytic immersive curves and analytic 1-submanifolds. In this case, the following  two situations just encode the $\tau$-decomposition and the $A$-decomposition case that we have discussed in this paper.}  
\begingroup
\setlength{\leftmargini}{14pt}
\begin{itemize}
\item[$\blacktriangleright$]
If $S$ admits no compact maximal segment, then $S$ admits only two maximal segments $\Sigma,\Sigma'$, and we have $S=\Sigma\cup \Sigma'$ as well as $\Sigma\cap \Sigma'=\{z\}$ for $z\in S$ unique. In addition to that, either $g\cdot \iota(\Sigma) \subseteq \iota(\Sigma')$ or $ g\cdot \iota(\Sigma)\supset \iota(\Sigma')$ holds for some $g\in G_{z}$, whereby the class of $g$  is uniquely determined by the condition  $g\cdot \iota|_{\Sigma}\cpsim \iota$ with $[g]\neq [e]$.
\item[$\blacktriangleright$]
If $S$ admits a compact maximal segment $\Sigma_0$, then there exists a $\Sigma_0$-decomposition $\mathcal{S}$ of $S$, i.e., a family $\{(\Sigma_n,[g_n])\}_{n\in \cn}$ consisting of free segments $\{\Sigma_n\}_{n\in \cn}$ on which $\iota$ is an embedding,  
as well as classes $\{[g_n]\}_{n\in \cn}$ such that the following holds:
\vspace{-2pt}
\begingroup
\setlength{\leftmarginii}{12pt}
\begin{itemize}
\item[$\bullet$]
$\Sigma_p\cap \Sigma_{q}$ is singleton if $|p-q|= 1$, and empty if $|p-q|\geq 2$,
\vspace{2pt}
\item[$\bullet$]
 $\Sigma_n$ is compact for $\cn_-<n<\cn_+$, and $\Sigma_{\cn_\pm}$ is a boundary segment of $S$ if $\cn_{\pm}\neq \pm\infty$ holds.
 \vspace{1pt}
\item[$\bullet$]
$g_n\cdot \iota(\Sigma_0)=\iota(\Sigma_n)$ holds for all $\cn_-<n<\cn_+$, and we have
\begin{align*} 
	g_{\cn_{-}}\cdot \iota(\Sigma_{-})=\iota(\Sigma_{\cn_-})\quad\text{if}\quad \cn_-\neq -\infty\qquad\quad\text{as well as}\qquad\quad g_{\cn_{+}}\cdot \iota(\Sigma_{+})=\iota(\Sigma_{\cn_+})\quad\text{if}\quad \cn_+\neq \infty
\end{align*}
for certain boundary segments $\Sigma_\pm$ of $\Sigma_0$. 
\end{itemize}
\endgroup  
\noindent
The only other $\Sigma_0$-decomposition $\ovl{\mathcal{S}}$ of $S$ is given by \hspace*{\fill}($\ocn_\pm=-\cn_\mp$)
\begin{align*}
(\{\ovl{\Sigma}_n\}_{n\in \ocn},\{[\ovl{g}_n]\}_{n\in \ocn})\quad\text{with}\quad \ovl{\Sigma}_n=\Sigma_{-n}\quad\text{and}\quad [\ovl{g}_n]=[g_{-n}]
\quad\text{for each}\quad n\in \ocn=\{-n \: | \: n\in \cn\}.
\end{align*}
Moreover, for  the boundary points $z_\pm$ shared by $\Sigma_{\pm 1}$ with $\Sigma_0$, we either have $g_{\pm 1}\notin G_{z_\pm}$ or $g_{\pm 1}\in G_{z_\pm}$.   
In the first/second case, we say that $\Sigma_0$ is {\bf positive/negative}.   
Let $\kappa\colon S\rightarrow I$ be a fixed homeomorphism, with $I$ an open interval.   
We say that $\mathcal{S}$ is $\kappa$-oriented \deff $\kappa(z_-)<\kappa(z_+)$ holds: 
\begingroup
\setlength{\leftmarginii}{12pt}
\begin{itemize}
\item[$\bullet$]
If $\Sigma_0$ is positive, then each other compact maximal segment is positive, and each point in $S$ is contained in the interior of such a  positive segment. Moreover, $[g_n]=[h^n]$ holds for all $n\in \cn$, for some unique class $[h]\in G\slash G_S$. This class is independent of the maximal segment $\Sigma_0$ provided that the corresponding $\kappa$-oriented decomposition of $S$ is chosen.\footnote{For the only other $\Sigma_0$-decomposition $\ovl{S}$ of $S$, the mentioned formula holds with $\ovl{h}:=h^{-1}$ instead of $h$.}   
\item[$\bullet$]
If $\Sigma_0$ is negative, then the segments $\{\Sigma_n\}_{n\he\in\he \cn\he\cup\he\{0\}}$ are maximal, and the only maximal segments of $S$.
Each compact $\Sigma_n$ is negative, and $[g_n]=[g_{\sigma(\sign(n))}\cdot {\dots}\cdot g_{\sigma(n)}]$ holds for all $n\in \cn$ with $\sigma$ as in \eqref{sdfgd}.\hspace*{\fill}$\ddagger$
\end{itemize}
\endgroup  
\noindent
\end{itemize}
\endgroup
\noindent
{\bf Case B:}\:\: If $S$ is compact, then $S$ necessarily admits a compact maximal segment $\Sigma_0$, just by existence (Lemma \ref{I98ds98ds98ds98ds89ds89ds}) and closedness of such segments. If $S$ is not a free segment by itself, then it is to be expected that the following assertions hold: 
\vspace{6pt}

\noindent
There exists a finite $\Sigma_0$-decomposition $\mathcal{S}$ of $S$ -- i.e.\ compact maximal segments $\Sigma_1,\dots,\Sigma_n$ with $S=\Sigma_0\cup{\dots}\cup \Sigma_n$ for $n\geq 1$ as well as classes $[g_1],\dots,[g_n]\in G\slash G_S$ -- such that 
$$g_k\cdot \iota(\Sigma_0) =\iota(\Sigma_k)\quad 
\text{holds for}\quad
k=1,\dots,n,$$ 
and additionally the following:
\vspace{-2pt} 
\begingroup
\setlength{\leftmargini}{12pt}
\begin{itemize}
\item
If $n=1$,\: then $\Sigma_0\cap \Sigma_1$ consists of two elements. Moreover, there exists only one $\Sigma_0$-decomposition. 
\item
If $n\geq 2$, then $\Sigma_p\cap \Sigma_{q}$ is singleton for $|p-q|\in\{1,n\}$, and empty for $2\leq |p-q|\leq n-1$. Moreover, the only other $\Sigma_0$-decomposition $\ovl{\mathcal{S}}$ of $S$ is given by
\begin{align}
\label{dfddsf}
	\ovl{\Sigma}_{k}=\Sigma_{\zeta(k)}\qquad\text{and}\qquad [\ovl{g}_k]=[g_{\zeta(k)}]\qquad\quad\forall\: 1\leq k\leq n,
\end{align}
with $\zeta\in \mathrm{S}_n$ defined by $\zeta(k):=n-(k-1)$ for $k=1,\dots,n$.  
\end{itemize}
\endgroup  
\noindent
Finally, let $z_\pm$ denote the boundary points of $\Sigma_0$,   chosen such that in the case $n\geq 2$ we have $\{z_+\}=\Sigma_0\cap \Sigma_1$ as well as $\{z_-\}=\Sigma_0\cap \Sigma_n$.  
Let furthermore $\kappa\colon \UE\rightarrow S$ be a fixed  
homeomorphism, and set $\alpha\colon \RR\ni \phi\mapsto \e^{\I\phi}\in \UE$. We say that $\mathcal{S}$ is $\kappa$-oriented \deff 
$$\Sigma_0=(\kappa\cp\alpha)([\lambda_-,\lambda_+])\qquad\text{holds with}\qquad \kappa(\lambda_\pm)=z_\pm\qquad\text{for some}\qquad \lambda_-<\lambda_+.$$  
Then, $\Sigma_0$ is either {\bf positive} or {\bf negative} (same definition as in Case A), and the following assertions hold:
\begingroup
\setlength{\leftmargini}{12pt}
\begin{itemize}
\item[$\bullet$]
If $\Sigma_0$ is positive, then each other maximal (necessarily compact) segment  is positive; and each point in $S$ is contained in the interior of some positive segment. 
Moreover, there exists a unique class $[h]\in G\slash G_S$ such that 
the following holds for each positive segment $\Sigma_0'$: 

If $\mathcal{S}'$ is the  
$\kappa$-oriented  $\Sigma'_0$-decomposition  of $S$, with corresponding maximal segments $\Sigma_1',\dots,\Sigma'_{n'}$ and classes $[g'_1],\dots,[g'_{n'}]\in G\slash G_S$, then  $n'=n$ holds as well as $[g_k']=[h^k]$ for $k=1,\dots,n$. 

For instance, let $S\equiv \UE$, and let $G$ be the discrete subgroup of $\UE$ that is generated by $h:=\e^{\I 2\pi \slash n}$, just acting via left multiplication on $S$. Then, $\Sigma=\e^{\I K}$ is positive for each $K=[t,t+2\pi \slash n]$.
\item[$\bullet$]
If $\Sigma_0$ is negative, then the segments $\Sigma_0,\dots,\Sigma_n$ are   negative, and the only maximal segments of $S$. The integer $n$ is necessarily odd; and, for $n\geq 1$ and $g_{-1}:=g_n$, we have $[g_k]=[g_{\sigma(1)}\cdot {\dots}\cdot g_{\sigma(k)}]$ for $k=1,\dots,n$. 
For instance,
\vspace{-4pt} 
\begingroup
\setlength{\leftmarginii}{12pt}
\begin{itemize}
\item
Let $S\equiv \UE\subseteq \RR^2$, and let $G$ be the discrete group that is generated by the reflection at the $x_2$-axis. Then, $\Sigma_0=\e^{\I K_0}$ and $\Sigma_1=\e^{\I K_1}$ are negative for $K_0=[-\pi/2,\pi/2]$ and $K_1=[\pi/2, 3\pi/2]$. 
\item
Similarly, if $G$ is the discrete group that is generated by the reflections at the $x_1$- and the $x_2$-axis, then $\Sigma_p=\e^{\I K_p}$ is negative for $K_p=[p\cdot \pi/2, (p+1)\cdot \pi/2]$ for $p=0,\dots,3$.
\end{itemize}
\endgroup
In both cases, the mentioned formula for the classes $[g_p]$ is easily verified.\hspace*{\fill}$\ddagger$
\end{itemize}
\endgroup  
\section*{Acknowledgements}
The author thanks Christopher Beetle, Rory Conboye, Jonathan Engle and Bernhard Kr\"otz for general remarks on drafts of the present article. 
This work was supported in part by the Alexander von Humboldt foundation of Germany, and NSF Grants PHY-1205968 and PHY-1505490.

\addtocontents{toc}{\protect\setcounter{tocdepth}{0}}
\appendix

\section*{APPENDIX}

\section{Appendix to Sect.\ \ref{gffggf}}

\subsection{}
\label{appA7}
\begin{proof}[Proof of Lemma  \ref{lemma:expeig}]
Analyticity of $\gamma$ is clear from the assumptions; and $\gamma$ is immersive, because
\begin{align*}
	\textstyle\dot\gamma(t)=0\quad\text{for}\quad t\in \RR\qquad\quad&\textstyle\Longleftrightarrow\qquad\quad \forall\he\tau\in \RR\colon \quad 0=\frac{\dd}{\dd s}\big|_{s=0}\: \exp(\tau\cdot \g)\cdot \gamma(t+s)\\
	&\textstyle\quad\hspace{97.6pt}=\frac{\dd}{\dd s}\big|_{s=0}\: \exp(((t+\tau)+s)\cdot \g)\cdot x\\[1pt]
	&\textstyle\quad\hspace{97.6pt}=\dot\gamma(t+\tau)\\[2pt]
	& \Longleftrightarrow\qquad\quad \im[\gamma]=\{\gamma(0)\}=\{x\}\\[2pt]
	& \Longleftrightarrow\qquad\quad \exp(\RR\cdot \g)\subseteq  G_x\\[-2pt]
	& \stackrel{\text{def.}}{\Longleftrightarrow}\qquad\quad \g \in \mg_x. 
\end{align*}  
Assume now that $\gamma$ is not injective. Since 
\begin{align}
\label{eq:cycl}
	\gamma(t')=\gamma(t)\qquad\quad \Longleftrightarrow\qquad\quad \gamma(t'+s)=\gamma(t+s)\qquad\forall\: s\in \RR
\end{align}
holds, $Z:=\{t\in \RR_{>0}\:|\: \gamma(t)=\gamma(0)\}\neq \emptyset$ is non-empty, hence $\pii(x,\g):=\inf(Z)\geq 0$ is defined. Moreover: 
\begingroup
\setlength{\leftmargini}{15pt}
{
\renewcommand{\theenumi}{\alph{enumi})} 
\renewcommand{\labelenumi}{\theenumi}
\begin{enumerate}
\item
\label{nmdsnmsdnmdsnmdsds81}
$\pii(x,\g)>0$\hspace{27.1pt}\quad  as $\gamma$ is locally injective,  because $\gamma$ is immersive.
\item
\label{nmdsnmsdnmdsnmdsds82}
$\gamma(\pii(x,\g))=\gamma(0)$\quad as $\gamma$ is continuous; thus, $\gamma(n\cdot \pii(x,\g))=\gamma(0)$ follows for $n\in \ZZ$ inductively from \eqref{eq:cycl}.
\end{enumerate}}
\endgroup
\noindent
Point \ref{nmdsnmsdnmdsnmdsds82} together with \eqref{eq:cycl} shows:
$$
\gamma(\pii(x,\g))\stackrel{\ref{nmdsnmsdnmdsnmdsds82}}{=}\gamma(0) \qquad\quad\stackrel{\eqref{eq:cycl}}{\Longrightarrow}\qquad\quad \gamma(t + n\cdot \pii(x,\g))=\gamma(t)\qquad\forall\: t\in\RR,\: n\in \ZZ.
$$ 
Finally assume that $\gamma(t)=\gamma(t')$ holds for $t,t'\in \RR$, hence   $\gamma(t-t')=\gamma(0)$ by \eqref{eq:cycl}. We choose $n\in \mathbb{Z}$ with  
      $0\leq (t-t')+n \cdot\pii(x,\g) \leq \pii(x,\g)$, and obtain:   
      \begin{align*}
      \gamma(t-t')=\gamma(0)\qquad\quad\stackrel{\eqref{eq:cycl}}{\Longrightarrow}\qquad\quad \gamma((t-t')+n \cdot\pii(x,\g))=\gamma(n \cdot\pii(x,\g))\stackrel{\ref{nmdsnmsdnmdsnmdsds82}}{=}\gamma(0).
	\end{align*} 
	Then, minimality of $\pii(x,\g)$ implies $(t-t') + n\cdot \pii(x,\g) \in \{0,\pii(x,\g)\}$, so that $(t-t')=n'\cdot \pii(x,\g)$ holds for some   $n'\in \mathbb{Z}$ (necessarily unique as $\pii(x,\g)>0$).
\end{proof}

\subsection{}
\label{appA77}

\begin{proof}[Proof of Lemma \ref{dfdfdfdfdf}]
For $\ell=0$, the claim is just clear from Theorem \ref{ofdpofdpfdpof}. Let thus $\ell\geq 1$; and  
 choose $L\colon \RR^\ell\rightarrow \RR^{m+\ell}$ injective and linear, such that 
$$ \phi\colon \RR^{m}\times \RR^\ell \rightarrow \RR^{m}\times \RR^\ell,\qquad (v,w)\mapsto \dd_x f(v) + L(w)
$$ 
is an isomorphism. Then, the assumptions of Theorem \ref{ofdpofdpfdpof} hold for 
  $p:=(x,0)$ and 
\[
	F\colon U\times \RR^\ell  \rightarrow \RR^{m}\times \RR^{\ell}, \qquad (v,w)\mapsto f(v)+L(w),
\]
as we have $\dd_{p} F=\phi$. There thus exist $O,W\subseteq \RR^{m}\times \RR^{\ell}$ open with 
$$(x,0)=p\in O\subseteq U\times\RR^\ell\qquad\text{and}\qquad (f(x),0)=F(p)\in W\subseteq \RR^{m}\times \RR^{\ell}$$ 
such that $(F|_O)|^W$ is an analytic diffeomorphism; hence, $\alpha\cp F|_{O}=\id_{O}$ holds for 
$\alpha:= ((F|_O)|^W)^{-1}\colon W\rightarrow O$. Since $O$ is open with $(x,0)\in O$, there exists $V\subseteq U$ open with $x\in V$, such that $V\times\{0\}\subseteq O$ holds; hence,
\[
\alpha(f(v))=\alpha(F(v,0))=(v,0)\qquad\quad\forall\: v\in V.\qedhere
\]
\end{proof}

\subsection{}
\label{appA2}
\begin{proof}[Proof of Lemma \ref{lemma:BasicAnalyt0}]
Let $n\equiv\dim[M]$. 
By Corollary \ref{dfdsasasasassa}, there exists an open interval $J\subseteq I$ with $t\in J$ and an analytic chart $(O,\psi)$ of $M$ with $\gamma(J)\subseteq O$ such that
\begin{align}
\label{cnmcxmdsfbsjfdmssdds}
	(\psi\cp \gamma)(s)=(s,0)\in  \RR\times \RR^{n-1}\qquad\quad\forall\: s\in J
\end{align} 
holds. Hence, $\gamma|_J$ is an embedding, thus $\gamma_J:= (\gamma|_J)|^{\gamma(J)}$ is bijective. We write $\psi=(\psi_1,\dots,\psi_n)$: 
\begingroup
 \setlength{\leftmargini}{12pt}
\begin{itemize}
\item[$\triangleright$]
By continuity of $\gamma,\gamma'$, we have $\gamma(t)=\gamma'(t')$, hence
\begin{align}
\label{lkfdlkfd}
	\quad(\psi\cp\gamma')(t')=(\psi\cp\gamma)(t) =(t,0)\in J\times \{0\}\subseteq \RR\times \RR^{n-1}.
\end{align}
\item[$\triangleright$]
By continuity of $\gamma'$ and $\gamma(t')=\gamma'(t)\in O$, there exists an open interval $J'\subseteq I'$ with $t'\in J'$ and $\gamma'(J')\subseteq O$.

Then, \eqref{cnmcxmdsfbsjfdmssdds} shows that $t'$ is an accumulation point of zeroes of the analytic function $f_{\zd}:=\psi_{\zd}\cp\gamma'|_{J'}$, for $k=2,\dots,n$; 
so that analyticity implies 
\begin{align}
\label{lkfdlkfd1}
	(\psi\cp \gamma')(J')\subseteq \RR\times \{0\} \subseteq \RR\times \RR^{n-1}.
\end{align}
\item[$\triangleright$]
By continuity of $\gamma'$, \eqref{lkfdlkfd} and \eqref{lkfdlkfd1}, we can achieve $(\psi\cp\gamma')(J')\subseteq J\times \{0\}$, just by shrinking $J'$ if necessary. 
\end{itemize}
\endgroup
\noindent	  
The claim now holds for the analytic function
\begin{equation*}
	\rho:= \psi_1\cp \gamma'|_{J'}\stackrel{\eqref{cnmcxmdsfbsjfdmssdds}}{=} ((\psi\cp\gamma_J)|^{J\times \{0\}})^{-1} \cp (\psi \cp \gamma'|_{J'}) =(\gamma_J^{-1}\cp\psi^{-1}|_{J\times \{0\}})\cp (\psi \cp \gamma'|_{J'})=\gamma_J^{-1}\cp \gamma'|_{J'}
\end{equation*}
and the connected ($\rho$ is continuous) subset $B:=\rho(J')$. In fact, clearly $B$ is singleton if $\gamma'|_{J'}$ is constant; and,  if $B$ is singleton, then $\gamma'|_{J'}$ must be constant by injectivity of $\gamma_J^{-1}$.   
\end{proof}

\subsection{}
\label{appA1}
\begin{proof}[Proof of Lemma \ref{gdfgfgf}]
Passing to a subsequence, we can achieve $\{g_n\}_{n\in \NN}\subseteq G_{\gamma(t)}$ or $\{g_n\}_{n\in \NN}\subseteq G\slash G_{\gamma(t)}$:
\begingroup
\setlength{\leftmargini}{12pt}
\begin{itemize}
\item[$\bullet$]
$\{g_n\}_{n\in \NN}\subseteq G_{\gamma(t)}$:\:\:     
We have $\gamma(t_n)=g_n\cdot \gamma(t)=\gamma(t)$ for all $n\in \NN$. 
We fix an open Interval $J\subseteq I$ with $t\in J$, and an analytic chart $(O,\psi)$ of $M$ with $\gamma(J)\subseteq O$ and $\psi(\gamma(t))=0$. We write $\psi=(\psi_1,\dots,\psi_{\dim[M]})$. Then, $t$ is an accumulation point of zeros of the analytic function $f_k:=\psi_k\cp\gamma|_J$ for $1\leq k\leq \dim[M]$; so that analyticity implies 
$\gamma|_J=\gamma(t)$, hence $\gamma(J)\subseteq G\cdot \gamma(t)$.
\item
$\{g_n\}_{n\in \NN}\subseteq G\slash G_{\gamma(t)}$:\:\: 
We write $\mg=\mc\oplus \mg_z$ for $z:=\gamma(t)$ and a suitable  linear subspace $\mc\subseteq \mg$; and set $m:= \dim[\mc]\geq 1$ and $k:= \dim[\mg_z]$.\footnote{The property $\dim[\mc]\geq 1$ follows from the condition $G\backslash G_z\supseteq \{g_n\}_{n\in \NN}\rightarrow e$:\: In fact, since $\exp$ is a local diffeomorphism at $0\in \mg$, there exists $O\subseteq \mg$ open with $0\in O$, such that  
$\widetilde{O}:=\exp(O)$ is an open neighbourhood of $e$ in $G$. Assume now $\dim[\mc]=0$, hence $\mg=\mg_z$. Then, $\widetilde{O}\subseteq G_z$ holds as $\mg_z=\{\g\in \mg\:|\: \exp(\RR\cdot \g)\subseteq G_z\}$, which contradicts $G\backslash G_z\supseteq \{g_n\}_{n\in \NN}\rightarrow e\in \widetilde{O}$.} We furthermore fix  isomorphisms $\varsigma_m\colon \RR^m\rightarrow \mc$ and $\varsigma_k\colon \RR^k\rightarrow \mg_z$.    
\begingroup
 \setlength{\leftmarginii}{12pt}
\begin{itemize}
\item[$\triangleright$]
By Theorem \ref{ofdpofdpfdpof}, there exist open neighbourhoods $U\subseteq \RR^m$ and $U'\subseteq \RR^k$ of zero, such that 
\begin{align*}
	\Lambda\colon U\times U'\rightarrow W,\qquad (u,u')\mapsto \exp(\varsigma_m(u))\cdot \exp(\varsigma_k(u'))
\end{align*}	  
is a homeomorphism (analytic diffeomorphism) to an open neighbourhood $W\subseteq G$ of $e$.  
\item[$\triangleright$]
The differential $\dd_0 \iota=\dd_e\wm_z\cp\varsigma_m$ at $0\in \RR^m$ of the analytic map
\begin{align*}
	\iota:=\wm_z\cp\exp\cp \:\varsigma_m\colon \RR^m\rightarrow M
\end{align*}
\vspace{-25pt}

\noindent
 is injective, because we have: 
\begin{align*}
	\textstyle\dd_e\wm_z(\g)=0\quad\text{for}\quad \g\in \mg\qquad\quad&\textstyle\Longleftrightarrow\qquad\quad \forall\: t\in \RR\colon\quad 0= \dd_{z}\wm_{\exp(t\cdot g)}(\dd_e\wm_z(\g) )\\[2pt]
	&\hspace{106pt}\textstyle= \frac{\dd}{\dd s}\big|_{s=0}\:
	\wm( \exp(t\cdot \g), \wm( \exp(s\cdot \g),z))\\[2pt]
	&\hspace{106pt}\textstyle= \frac{\dd}{\dd s}\big|_{s=0}\:\underbrace{\wm( \exp((t+s)\cdot \g), z)}_{\displaystyle\exp((t+s)\cdot \g) \cdot z } \\[-9pt]
	& \Longleftrightarrow\qquad\quad \exp(\RR\cdot \g)\subseteq  G_z\\[-2pt]
	& \stackrel{\text{def.}}{\Longleftrightarrow}\qquad\quad \g \in \mg_z. 
\end{align*} 
In particular, we have $\ell:= \dim[M]-m\geq 0$;  and we 
 choose $V\subseteq U$ and $(O,\psi)$ as in Corollary \ref{dfdsasasasassa}, for $x\equiv 0$ there (hence $O\ni \iota(x)=z=\gamma(t)$), hence
\begin{align}
\label{opasopasopsaopasopas}
	(\psi\cp\iota)(V)=V\times \{0\}\subseteq \RR^m\times \RR^\ell. 
\end{align}
\vspace{-20pt}

\noindent
We write $\psi=(\psi_1,\dots,\psi_{\dim[M]})$.
\item[$\triangleright$] 
Since $\lim_n g_n=e$ holds, we have for $q\in \NN$ suitably large:\hspace*{\fill}($\tilde{W}$ is open, as $\Lambda$ is a homeomorphism) 
\begin{align*}
\{g_n\cdot z\}_{n\geq q}\subseteq O\qquad\quad\text{as well as}\qquad\quad		\{g_n\}_{n\geq q}\subseteq \tilde{W}:=\Lambda(V\times U')\subseteq W. 
\end{align*}
Let $n\geq q$. Then,   
$g_n=\Lambda(v_n,u'_n)$ holds for some (necessarily unique) $(v_n,u'_n)\in V\times U'$, hence      
\begin{align}
\label{dfsjhdfhfdaaa}	  
	g_n\cdot z=\Lambda(v_n,u'_n)\cdot z=\exp(\varsigma_m(v_n))\cdot \exp(\varsigma_\ell(u'_n))\cdot z=\exp(\varsigma_m(v_n))\cdot z=\iota(v_n)\in \iota(V).
\end{align}
This shows $\{g_n\cdot z\}_{n\geq q}\subseteq \iota(V)\cap O$. 
\end{itemize}
\endgroup
\noindent	  
Let now $J\subseteq I$ be an open interval with $t \in J$ and $\gamma(J)\subseteq O$, and let $m+1\leq k \leq  m+\ell$. Then, $t$ is an accumulation point of zeroes of the analytic function $f_k:=\psi_k \cp \gamma|_{J}$, as we have
\begin{align*}
	f_k(t_n)=\psi_k(g_n\cdot \gamma(t))=\psi_k(g_n\cdot z)\stackrel{\eqref{dfsjhdfhfdaaa}}{=}(\psi_k\cp \iota)(v_n)\stackrel{\eqref{opasopasopsaopasopas}}{=}0\qquad\quad\forall\: n\geq q.
\end{align*}
Hence, analyticity implies $(\psi\cp\gamma)(J)\subseteq \RR^m\times \{0\}$. Since $(\psi\cp\gamma)(t)=(\psi\cp\iota)(0)=(0,0)\in V\times \{0\}$ holds, we can shrink $J$ around $t$ in such a way that 
\[
 (\psi\cp\gamma)(J)\subseteq V \times \{0\} \stackrel{\eqref{opasopasopsaopasopas}}{=} (\psi\cp\iota)(V)\quad\text{holds, hence}\quad \gamma(J)\subseteq \iota(V)\subseteq G\cdot \gamma(t).\qedhere
\]
\end{itemize}
\endgroup
\end{proof}

\subsection{}
\label{appA3}
\begin{proof}[Proof of Lemma \ref{lemma:BasicAnalyt2}]
Assume first that $D\equiv I$ and $D'\equiv I'$ are open. Due to our assumptions
$$\rho:=(\gamma'|^{\im[\gamma']})^{-1}\cp \gamma\colon D \rightarrow D'$$
is a homeomorphism, with $\rho^{-1}= (\gamma|^{\im[\gamma]})^{-1}\cp \gamma'$. 
For $t\in I$ fixed, Lemma \ref{lemma:BasicAnalyt0} provides  an analytic map $\rho_t\colon D\supseteq  J_t\rightarrow B'_t \subseteq D'$, with $t\in J_t \subseteq D$ ($J_t$ an open interval and $B'_t$ an interval) and $\gamma|_{J_t}=\gamma' \cp\rho_t$; hence, 
$$\rho|_{J_t}=(\gamma'|^{\im[\gamma']})^{-1}\cp \gamma|_{J_t}=\rho_t$$ 
is analytic. Since $t\in I$ was arbitrary, it follows that $\rho$ is analytic.  
Interchanging the roles of $\gamma$ and $\gamma'$ in the above argumentation, we see that also $\rho^{-1}$ is analytic, i.e., that $\rho$ is an analytic diffeomorphism with the desired properties. 

 Assume now that $D,D'$ are arbitrary intervals; and let $\wt{\gamma}\colon I\rightarrow M$ and $\wt{\gamma}'\colon I'\rightarrow M$ be analytic embeddings extending $\gamma$ and $\gamma'$, respectively. For each $t\in I$, Lemma \ref{lemma:BasicAnalyt0} provides $\rho_t\colon D\supseteq  J_t\rightarrow B'_t \subseteq D'$ analytic with $t\in J_t \subseteq D$ ($J_t$ an open interval and $B'_t$ an interval), 
 such that $\wt{\gamma}|_{J_t}=\wt{\gamma}'\cp \rho_t$ holds. We set 
\begin{align*}
	\textstyle I\supseteq J:= \bigcup_{t\in D}J_t\qquad\quad\text{as well as}\qquad\quad I'\supseteq B':= \bigcup_{t\in D}B'_t;
\end{align*}
 and observe that $(\wt{\gamma}'|^{\im[\tilde{\gamma}']})^{-1}\cp\wt{\gamma}|_{J}\colon J\rightarrow B'$ is a homeomorphism as $\wt{\gamma}$ and $\wt{\gamma}'$ are embeddings. Since this implies that $B'$ is an open interval, the claim follows from the first part.    
\end{proof}

\subsection{}
\label{appA4}
\begin{proof}[Proof of Lemma \ref{maximalextension}]
    Let $\gamma\colon D\rightarrow M$ be a fixed analytic (immersive) curve; and denote the set of all analytic (immersive) extensions $\delta\colon J_\delta\rightarrow M$ of $\gamma$ by $\EE$ (observe $\EE\neq \emptyset$, due to our conventions). We define
    	\begin{align*}
    		\ovl{\gamma}\colon I:= \textstyle\bigcup_{\delta\in \EE}J_\delta\rightarrow M\qquad\text{by}\qquad 
    		\ovl{\gamma}(t):=\delta(t)\qquad\text{for some}\qquad\delta\in \EE\qquad\text{with}\qquad t\in J_\delta.
    	\end{align*}
This is well defined; because, if $\delta'\in \EE$ is another extension with $t\in J_{\delta'}$, then 
\begin{align*}
	D\subseteq J_\delta\cap J_{\delta'}=: J\ni t \qquad\quad\stackrel{\text{Lemma }\ref{corgleich}}{\Longrightarrow}\qquad\quad \delta|_J=\delta'|_J \qquad\quad\Longrightarrow\qquad\quad \delta(t)=\delta'(t).
\end{align*} 
Evidently, $\ovl{\gamma}$ is an (immersive) analytic extension of $\gamma$. Moreover, we obtain from the definition of $I$:
\begingroup
\setlength{\leftmargini}{12pt}
\begin{itemize}
\item
$\ovl{\gamma}$ is maximal, 
because for $\delta\in \EE$ we have $J_\delta\subseteq I$, hence $\delta=\ovl{\gamma}|_{J_\delta}$ by Lemma \ref{corgleich} ($D'\equiv D$, $\gamma\equiv \delta$, $\gamma'\equiv \ovl{\gamma}|_{J_\delta}$).
\item
If $\delta\in \EE$ is maximal, then $\dom[\delta]= I$ holds, hence $\delta=\ovl{\gamma}$ by 
Lemma \ref{corgleich} ($D'\equiv D$, $\gamma\equiv \delta$, $\gamma'\equiv \ovl{\gamma}$).  
\qedhere 
\end{itemize}
\endgroup
\end{proof}

\subsection{}
\label{appA5}
\begin{proof}[Proof of Corollary \ref{cgdcfgd}]
Let $[c',c]\equiv \CM:=[a,b]\cap \ovl{\rho}^{-1}(K')$ be as in Lemma \ref{lem:maxextKomp}, and assume $\im[\gamma']\nsubseteq \im[\gamma]$. Then, we cannot have $\ovl{\rho}([c',c])=K'\equiv [a',b']$; hence, 
$c'=a$ or $c=b$ must hold by \eqref{ddddd}. Since $t\in J\subseteq C$ holds, we have the implications:
\begin{equation*}
\begin{array}{lllllllll}
&c'&\hspace{-8pt}=a\quad\:\:&\Longrightarrow\quad\:\: &[a,t]&\hspace{-8pt}\subseteq \CM\quad\:\:&\Longrightarrow\qquad\:\: \gamma([a,t])&\hspace{-8pt}=\gamma'(\ovl{\rho}([a,t]))&\hspace{-8pt}\subseteq \im[\gamma']\\[3.5pt]
 &c&\hspace{-8pt}=b\quad\:\:&\Longrightarrow\quad\:\: &[t,b]&\hspace{-8pt}\subseteq \CM\quad\:\:&\Longrightarrow\qquad\:\: \gamma([t,b])&\hspace{-8pt}=\gamma'(\ovl{\rho}([t,b]))&\hspace{-8pt}\subseteq \im[\gamma'],
\end{array}
\end{equation*}
which proves the claim.
\end{proof}

\subsection{}
\label{appA6}
\begin{proof}[Proof of Corollary \ref{dfdgttrgf}]
Assume that $\rho$ is such a negative diffeomorphism. Then, $[a',a]\cap \ovl{\rho}^{-1}([a',a])=[a',c]$ holds for some $c\in [r,a]$:
\begingroup
\setlength{\leftmargini}{12pt}
\begin{itemize}
\item[$\triangleright$]
Assume $c=a$, hence $\gamma=\gamma\cp \ovl{\rho}|_{[a',a]}$. Since $\rho$ (thus $\ovl{\rho}$) is negative, we have $\ovl{\rho}(a')=a$ and  $\ovl{\rho}(a)<a$. The intermediate value theorem applied to $\alpha:=\ovl{\rho}-\id_{[a',a]}$ yields $\tau\in (a',a)$ with $\alpha(\tau)=0$, hence $\ovl{\rho}(\tau)=\tau$. 

Then, for $\epsilon >0$ suitably small, we have 
\begin{align*}
	\gamma(\tau +\epsilon)=\gamma( \ovl{\rho}(\tau +\epsilon))\quad\text{with}\quad\ovl{\rho}(\tau+\epsilon)\leq\tau<\tau+\epsilon
\quad\text{by negativity of}\quad \ovl{\rho}, 	
\end{align*}
which contradicts that $\gamma$ is injective on a neighbourhood of $\tau$. 
\item[$\triangleright$] 
Assume $c<a$; hence, $\gamma|_{[a',c]}=\gamma\cp \ovl{\rho}|_{[a',c]}$ holds,  with $\ovl{\rho}(a')=a$ and $\ovl{\rho}(c)<a$  as $\rho$ (thus $\ovl{\rho}$) is negative.  
Then, $\ovl{\rho}(c)=a'$ holds by \eqref{ddddd} in Lemma \ref{lem:maxextKomp}. The intermediate value theorem applied to $\alpha:=\ovl{\rho}|_{[a',c]}-\id_{[a',c]}$ yields $\tau\in (a',c)$ with $\alpha(\tau)=0$, hence $\ovl{\rho}(\tau)=\tau$. 

Then, for $\epsilon >0$ suitably small, we have 
\begin{align*}
	\gamma(\tau +\epsilon)=\gamma( \ovl{\rho}(\tau +\epsilon))\quad\text{with}\quad\ovl{\rho}(\tau+\epsilon)\leq\tau<\tau+\epsilon
\quad\text{by negativity of}\quad \ovl{\rho}, 	
\end{align*}
which contradicts that $\gamma$ is injective on a neighbourhood of $\tau$. 
\qedhere
\end{itemize}
\endgroup
\end{proof}
\subsection{}
\label{appA799}

\begin{proof}[Proof of the Equivalence used in Remark \ref{remmmmmi}.\ref{regp1}] 
First observe that if $C_1,\dots,C_n$ are components of $G$, then $C_1\cup{\dots}\cup C_n$ is open and second (hence first) countable (the components of $G$ are open and second countable as the identity component of $G$ is open and first countable).
\begingroup
\setlength{\leftmargini}{12pt}
\begin{itemize}
\item
Let $\wm_x$ be proper, and assume that the left side of \eqref{lkdslkdslklkdslkdslklkdsdsds09ds9009dsdsdsdsdsdsds} holds. We fix $K\subseteq M$ compact with $\{g_n\cdot x\}_{n\in \NN}\subseteq K\ni y$, and observe that $L:=\wm_x^{-1}(K)$ is covered by finitely many components $C_1,\dots,C_n$ of $G$. Since 
$C:=C_1\cup{\dots}\cup C_n$ is second (hence first) countable,  
$L\subseteq C$ is sequentially compact. Hence, $\{g_n\}_{n\in \NN}\subseteq L$ admits a convergent subsequence.    
\item
Assume that the implication \eqref{lkdslkdslklkdslkdslklkdsdsds09ds9009dsdsdsdsdsdsds} holds,  let $K\subseteq M$ be compact, and set $L:=\wm_x^{-1}(K)$: 
\vspace{-4pt}
\begingroup
\setlength{\leftmarginii}{15pt}
\begin{itemize}
\item [a)]
$K$ is sequentially compact as $M$ is first countable (as locally homeomorphic to $\RR^{\dim(M)}$).
\vspace{2pt}
\item [b)]
$L$ is compact if sequentially compact:
\vspace{4pt}

{\it Proof.}
It suffices to show that $L$ is covered by finitely many components of $G$ (as then second countable). Assume thus that $G$ cannot be covered by finitely many components of $G$. Then, (since $G$ is covered by its components) there exist (mutually disjoint) components $\{C_n\}_{n\in \NN}$ of $G$  
with $C_n\cap L\neq \emptyset$ for all $n\in \NN$. We fix some  $g_n\in C_n\cap L$ for each $n\in \NN$. 
Since $K$ is sequentially compact by a), $\{g_n\cdot x\}_{n\in \NN}\subseteq K$ admits a convergent subsequence $\{g_{\iota(n)}\cdot x\}_{n\in \NN}\rightarrow y\in K\subseteq  M$. Hence, \eqref{lkdslkdslklkdslkdslklkdsdsds09ds9009dsdsdsdsdsdsds} 
yields a convergent subsequence  $\{g_{\kappa(\iota(n))}\}_{n\in \NN}\rightarrow g\in G$. Let $C$ denote the component of $G$ that contains $g$. Since $C$ is open, there exists $m\in \NN$ with  $g_{\kappa(\iota(n))}\in C\cap C_{\kappa(\iota(n))}$ for all $n\geq m$; which contradicts that the components of $G$ are mutually disjoint.

\vspace{4pt}
\item[c)]
$L$ is sequentially compact, hence compact by b):
\vspace{2pt}

{\it Proof.}
Let $\{g_n\}_{n\in \NN}\subseteq L$. Since $K$ is sequentially compact by a), $\{g_n\cdot x\}_{n\in \NN}\subseteq K$ admits a convergent subsequence $\{g_{\iota(n)}\cdot x\}_{n\in \NN}\rightarrow y\in K\subseteq  M$. Hence, \eqref{lkdslkdslklkdslkdslklkdsdsds09ds9009dsdsdsdsdsdsds} shows that $\{g_{\iota(n)}\}_{n\in \NN}$ admits a convergent subsequence; which proves the claim. \qedhere
\end{itemize}
\endgroup 
\end{itemize}
\endgroup
\end{proof}

\subsection{}
\label{appA8}
\begin{proof}[Proof of Lemma \ref{podspodsapopo}]
	If $\tau:= \rho|_{\wt{D}}$ is an analytic diffeomorphism and $\g\in \mg\backslash \mg_x$, then Lemma \ref{lemma:expeig} shows that $\gag \cp \ovl{\tau}$ is an analytic immersive extension of $\gamma|_{\wt{D}}$; which proves the one direction. 
	For the other direction, we assume that $\gamma|_{\wt{D}}$ is an analytic immersion, i.e., there exists an analytic immersion 
	\begin{align*}
		\delta\colon I\rightarrow M\qquad\text{with}\qquad \wt{D}\subseteq I\qquad\text{and}\qquad\delta|_{\wt{D}}=\gamma|_{\wt{D}}.
	\end{align*}		
Since $\dom[\ovl{\rho}]$ is open with $\wt{D}\subseteq \dom[\ovl{\rho}]$, Lemma \ref{corgleich} implies that  
	\begin{align*}
		\delta|_{J}=\gag\cp\ovl{\rho}|_{J}\qquad\text{holds for the open interval}\qquad J:=I\cap \dom[\ovl{\rho}]\supseteq \tilde{D}.
	\end{align*} 
	Since $\delta$ is immersive, we have $\g\in \mg\backslash\mg_x$ (elsewise $\gag$ is constant, hence $\delta|_J$). Since $\delta$ is immersive, the product rule implies $\dot{\ovl{\rho}}|_J\neq 0$, i.e., that $\ovl{\rho}|_J$ is a diffeomorphism. Since  $\delta|_{J}$ and $\gag\cp\ovl{\rho}|_{J}$ are locally embeddings (Corollary \ref{dfdsasasasassa}), Lemma \ref{lemma:BasicAnalyt2} implies that $\ovl{\rho}|_J$ is analytic.
\end{proof}

\subsection{}
\label{appA9}
\begin{proof}[Proof of Lemma \ref{vd}]
The claim is clear if $\gamma$ is constant. Assume thus that $\gamma$ is non-constant with $\gamma|_{\dom[\kappa]}= \gag\cp \kappa$. Set $I':=\dom[\ovl{\kappa}]$, and  
 let $\ovl{\gamma}\colon \ovl{I}\rightarrow M$  denote the maximal analytic extension of $\gamma$. Since $\gag\cp\ovl{\kappa}$ is an analytic extension of $\gamma|_{\dom[\kappa]}$,  Lemma \ref{corgleich} and maximality of $\ovl{\gamma}$ yield 
\begin{align}
\label{opsdopsdopsd}
	\ovl{\gamma}|_{I'}= \gag \cp \ovl{\kappa} \qquad\text{hence}\qquad I'\subseteq \ovl{I}.
\end{align} 
The claim thus follows if we show $D\subseteq I'$.  
For this, we assume that $D\not\subseteq I'$ holds; and observe the following:
\begingroup
\setlength{\leftmargini}{15pt}
{
\renewcommand{\theenumi}{\sf\alph{enumi})} 
\renewcommand{\labelenumi}{\theenumi}
\begin{enumerate}
\item
\label{oidskjdskjdskjdskjdsjkdsds9ds98ds98ds98ds98ds98ds98ds0}
We have $\g\in \mg\setminus\mg_z$ for each $z\in \im[\gag]=\exp(\RR\cdot \mg)\cdot x$, since $\gamma$ is non-constant.

In fact, let $z=\exp(\tau\cdot \g)\cdot x$ for $\tau\in \RR$. Then, the following implication holds:
\begin{align*}
\g\in \mg_z\qquad\Longrightarrow\qquad \gag(t)&=\exp((t-\tau)\cdot \g)\cdot z= z\qquad\forall\: t\in \RR\\
\Longrightarrow\qquad\hspace{6.5pt} \ovl{\gamma}|_{I'}&=z
\quad\:\:\stackrel{\text{Lemma }\ref{corgleich}}{\Longrightarrow}\:\:\quad\ovl{\gamma}=z\qquad\Longrightarrow\qquad\gamma=z.
\end{align*}
In particular,  $\pii(x,\g)>0$ holds, and $\gag$ is (analytic) immersive by Lemma \ref{lemma:expeig}. Hence, there exists an open interval $J\subseteq \RR$ with $t\in J$ and $J\cap [\ppi(x,\g)+J]=J$, such that $\gag\big|_J$ is an embedding (Corollary \ref{dfdsasasasassa}).
\item
\label{oidskjdskjdskjdskjdsjkdsds9ds98ds98ds98ds98ds98ds98ds1}
If we are given $I'\supseteq \{\tau_n\}_{n\in \mathbb{N}} \rightarrow t'\in D\backslash I' \subseteq \ovl{I}$, then $\{\ovl{\kappa}(\tau_n)\}_{n\in \NN}\subseteq \RR$ is bounded.

In fact, if the claim is wrong, we can pass to a subsequence to achieve  $\lim_n \ovl{\kappa}(\tau_n) =\pm \infty$ (we now treat both cases simultaneously). We set $z:=\gamma(\tau_0)=\gag(\ovl{\kappa}(\tau_0))$ and fix $0<d<\pii(z,\g)$ (observe $\pii(z,\g)>0$ by \ref{oidskjdskjdskjdskjdsjkdsds9ds98ds98ds98ds98ds98ds98ds0} and Lemma \ref{lemma:expeig}). By the intermediate value theorem, we can redefine $\{\tau_n\}_{n\geq 1}$ such that
\begin{align*}
	\ovl{\kappa}(\tau_n)=\ovl{\kappa}(\tau_0)\pm n\cdot d\qquad\quad \forall\: n\in \NN.
\end{align*} 
We set $g:=\exp(d\cdot \g)$, and obtain:
\begin{align}
\label{podpoadodspoaads}
	\textstyle  \lim_n g^n\cdot (g\cdot z) = \lim_n g^n\cdot z =\lim_n \gag(\ovl{\kappa}(\tau_0)\pm n\cdot d)=\lim_n \gag(\ovl{\kappa}(\tau_n))=\lim_n \gamma(\tau_n)=\gamma(t').
\end{align}	 
This contradicts that $\gamma(t')$ is sated, as $\gamma(t')\neq g\cdot z\neq z\neq \gamma(t')$ holds:
\vspace{-2pt}
\begingroup
\setlength{\leftmarginii}{12pt}
\begin{itemize}
\item
We have $g\cdot z\neq z$ by the choice of $d$.
\item
We obtain $z\neq \gamma(t')\neq g\cdot z$ from $g\cdot z\neq z$, since 
 \eqref{podpoadodspoaads} implies 
	$g\cdot \gamma(t')=\gamma(t')=g^{-1}\cdot \gamma(t')$.
\end{itemize}
\endgroup
\end{enumerate}}
\endgroup
  	\noindent
  	Let now $I'\supseteq \{t'_n\}_{n\in \mathbb{N}} \rightarrow t'\in D\backslash I' \subseteq \ovl{I}$ be fixed, and choose $J$ as in \ref{oidskjdskjdskjdskjdsjkdsds9ds98ds98ds98ds98ds98ds98ds0}. 
  By \ref{oidskjdskjdskjdskjdsjkdsds9ds98ds98ds98ds98ds98ds98ds1}, we can assume that $\lim_n \ovl{\kappa}(t'_n)=t\in \RR$ exists (pass to a subsequence if necessary). Then, continuity yields    
	\begin{align}
	\label{opisdopdsopsd}
	\textstyle\ovl{\gamma}(t')=\lim_n \ovl{\gamma}(t'_n)\stackrel{\eqref{opsdopsdopsd}}{=}\lim_n(\gag\cp \ovl{\kappa})(t'_n)=\gag(t).
	\end{align}	
\vspace{-18pt}

\noindent
We conclude the following:
\begingroup
\setlength{\leftmargini}{12pt}
\begin{itemize}
\item[$\triangleright$]
For each sequence $I'\supseteq \{\tau_n\}_{n\in \mathbb{N}} \rightarrow t'\in D\backslash I' \subseteq \ovl{I}$, we automatically  have
\begin{align}
	\label{opisdopdsopsd2}
	\textstyle\lim_n(\gag\cp \ovl{\kappa})(\tau_n)\stackrel{\eqref{opsdopsdopsd}}{=}\lim_n \ovl{\gamma}(\tau_n)=\ovl{\gamma}(t') \stackrel{\eqref{opisdopdsopsd}}{=}\gag(t).
\end{align}	
\item[$\triangleright$]
Lemma \ref{lemma:BasicAnalyt0} (applied to $\gamma\equiv \gag$, $I\equiv \RR$, $\gamma'\equiv \ovl{\gamma}$, $I'\equiv \ovl{I}$) shows that 
\begin{align}
\label{sdoioisdoidsoidsds98s98ds98ds988ds998ds98sd98ds89ds8ds98ddsds}
	\ovl{\gamma}|_{J'}=\gag\cp \rho
\end{align}	
	 holds for an analytic map $\rho\colon \ovl{I}\supseteq J'\rightarrow B\subseteq \RR$, with $\rho(t')=t\in B$, $B$ connected, and $J'$ an open interval with $t'\in J'$. By continuity, we can additionally assume $\rho(J')\subseteq J$, just by shrinking $J'$ around $t'$ if necessary.
\end{itemize}
\endgroup
\noindent
The claim now follows from \eqref{sdoioisdoidsoidsds98s98ds98ds988ds998ds98sd98ds89ds8ds98ddsds}  once we have shown that  $\ovl{\kappa}(J'\cap I')\subseteq J$ can be achieved by shrinking $J'$ around $t'$. In fact, then we obtain (for the first implication observe that $\gag\big|_J$ is an embedding)
\vspace{-3pt}
	\begin{align*}
		\gag\cp \ovl{\kappa}|_{J'\cap I'}\stackrel{\eqref{opsdopsdopsd}}{=}\ovl{\gamma}|_{J'\cap I'}\stackrel{\eqref{sdoioisdoidsoidsds98s98ds98ds988ds998ds98sd98ds89ds8ds98ddsds}}{=}\gag\cp\rho|_{J'\cap I'}\qquad\Longrightarrow\qquad \ovl{\kappa}|_{J'\cap I'}=\rho|_{J'\cap I'}\quad\:\stackrel{\text{Corollary } \ref{sdoppsdoods}}{\Longrightarrow}\quad\: t'\in J'\subseteq I', 	
	\end{align*}	  
	which contradicts $t'\in D\backslash I'$. 
To prove the claim, it thus remains to show that $\ovl{\kappa}$ extends continuously to $I'\cup \{t'\}$ via $\ovl{\kappa}(t'):=t$, i.e., that the following statement holds:
\begingroup
\setlength{\leftmargini}{15pt}
{
\renewcommand{\theenumi}{\sf\alph{enumi})} 
\renewcommand{\labelenumi}{\theenumi}
\begin{enumerate}
\setcounter{enumi}{2}
\item
\label{oidskjdskjdskjdskjdsjkdsds9ds98ds98ds98ds98ds98ds98ds2}
For each fixed $I'\supseteq \{s_n\}_{n\in \mathbb{N}} \rightarrow t'\in D\backslash I' \subseteq \ovl{I}$, we have $\lim_n \ovl{\kappa}(s_n)= t$.

In fact,  $\{\ovl{\kappa}(s_n)\}_{n\in \NN}\cup\{t\}$ is contained in a compact interval $K\equiv [a,b]\subseteq \RR$ by \ref{oidskjdskjdskjdskjdsjkdsds9ds98ds98ds98ds98ds98ds98ds1}; and we conclude:  
\begingroup
\setlength{\leftmarginii}{12pt}
\begin{itemize}
\item
If $\gag|_K$ is injective, then $\gag|_K$ is an embedding; so that 
\begin{align*}
	\textstyle \lim_n\gag|_K(\ovl{\kappa}(s_n))\stackrel{\eqref{opisdopdsopsd2}}{=} \gag|_K(t)\qquad\quad\Longrightarrow\qquad\quad \lim_n \ovl{\kappa}(s_n)=t.
\end{align*}  
\item
If $\gag|_K$ is not injective, then we necessarily have $\pii(x,\g)<\infty$ (by Lemma \ref{lemma:expeig}), and 
\begin{align*}
	\Omega\colon \UE\rightarrow \im[\gag]\subseteq M,\qquad  \e^{\I\cdot \alpha}\mapsto \gag(\alpha\cdot \pii(x,\g)/(2\pi))
\end{align*}
is an embedding. Then, it is clear from \eqref{opisdopdsopsd2} (applied to $\{\tau_n\}_{n\in \mathbb{N}}\equiv \{s_n\}_{n\in \mathbb{N}}$) and the definition of the period,  that for each open neighbourhood $O\subseteq J$ of $t$ there exists  $m_O\in \NN$ with $\{\ovl{\kappa}(s_n)\}_{n\geq m_O}\subseteq \ZZ\cdot \pii(x,\g) + O$.
\begingroup
\setlength{\leftmarginiii}{13pt}
\begin{itemize}
\item[$\triangleright$]
If $\{\ovl{\kappa}(s_n)\}_{n\geq m_O}\subseteq O$ holds for a suitable choice of $O$, then we necessarily have $\lim_n\ovl{\kappa}(s_n)=t$. 
\item[$\triangleright$]
	In the other case, we can  pass to subsequences of $\{s_n\}_{n\in \NN}$ and $\{t'_n\}_{n\in \NN}$, such that $s_n<t'_n$ as well as  $\ovl{\kappa}(s_n)\in \ZZ_{\neq 0}\cdot \pii(x,\g) + J$ and $\ovl{\kappa}(t'_n)\in J$ holds for each $n\in \NN$. Since $J\cap [\ppi(x,\g)+J]=J$ holds, the intermediate value theorem provides  
\begin{align*}
	s_n<\tau_n <t'_n\qquad\text{with}\qquad\ovl{\kappa}(\tau_n)\in \RR\backslash (\ppi(x,\g) + J)\qquad\quad\forall\: n\in \NN.
\end{align*}	
This, however contradicts \eqref{opisdopdsopsd2}, as $\lim_n \tau_n=t'$ holds by the squeeze lemma. \qedhere
\end{itemize}
\endgroup
\end{itemize}
\endgroup 
\end{enumerate}}
\endgroup
\end{proof}

\subsection{}
\label{appA10}
\begin{proof}[Proof of Lemma \ref{kljklvjkvjklvjklxcv}]
Let $H\subseteq G$ denote the closure of the group generated by $O_\gamma:=\{g\in G\:|\: g\cdot \gamma\cpsim \gamma\}\subseteq G$. Then, $G_\gamma\subseteq H$ is a normal subgroup, as $g^{-1}\cdot q \cdot g \in G_\gamma$ holds for all $q\in G_\gamma$ and $g\in O_\gamma$ by Corollary \ref{fsdfsfdfs}. Therefore, $Q:=H\slash G_\gamma$ is a Lie group; and the canonical projection $\pri\colon H\rightarrow Q$ is a Lie group homomorphism with $\ker[\dd_e\pri]=\mg_\gamma$. We now argue as follows: 
\begingroup
\setlength{\leftmargini}{12pt}
\begin{itemize}
\item
We have $\g\in \mh\subseteq \mg$ ($\mh$ the Lie algebra of $H$), as
\begin{align*}
\exp(s\cdot \g)\cdot \gamma(t)=\gamma(s+t)\qquad \forall\: s,t\in \RR\qquad\quad\Longrightarrow\qquad\quad \exp(\RR\cdot \g)\subseteq  O_\gamma\subseteq H.
\end{align*}
\item
For  $t,\lambda \in \RR$ and $\ccc\in \mg_\gamma\subseteq \mh$, we  have $t\cdot (\lambda\cdot \g+\ccc)\in \mh$ by the previous point, and we obtain
\begin{align*}
&\pri(\exp(t\cdot (\lambda\cdot \g+\ccc)))=   
\exp_\mq(\dd_e\pri(t\cdot (\lambda\cdot \g+\ccc) ))
=\exp_\mq(\dd_e\pri(t\cdot \lambda\cdot \g ))
=\pri(\exp(t\cdot \lambda\cdot \g))\qquad\qquad\\[2pt]
&\hspace{170pt}\text{hence}\\[2pt]
&\qquad\:\:\:\exp(t\cdot (\lambda\cdot \g+\ccc))= \exp(t\cdot \lambda\cdot \g)\cdot \mathrm{h}(t,\lambda,\ccc)\qquad\text{for some}\qquad \mathrm{h}(t,\lambda,\ccc)\in G_\gamma.
\end{align*}
\end{itemize}
\endgroup
\noindent
Since $G_\gamma\subseteq G_{\gamma(\tau)}$ holds, the second point yields 
\[
\textstyle\exp(t\cdot (\lambda\cdot \g+\ccc))\cdot \gamma(\tau) = \exp(t\cdot\lambda\cdot \g)\cdot \gamma(\tau)\qquad\quad\forall\: t,\tau,\lambda \in \RR. \qedhere
\] 
\end{proof}
\subsection{}
\label{appA11}
\begin{proof}[Proof of Lemma \ref{ldldsldldsalldsakdshfdsf}]
	The uniqueness statement is clear from $\g\notin \mg_x$ and $\mg_\gag\subseteq \mg_x$. Now, since $\wm$ is sated, Lemma \ref{vd} shows  $\gagg=\gag\cp\ovl{\rho}$.  
Since $\gagg$ is immersive by Lemma \ref{lemma:expeig}, $\ovl{\rho}$ is an analytic diffeomorphism (as immersive by the chain rule). Hence, $\gag=\gagg\cp\kappa$ holds for the analytic diffeomorphism $\kappa:=\ovl{\rho}^{-1}$:
\begingroup
\setlength{\leftmargini}{12pt}
\begin{itemize}
\item
Replacing $\q$ by $-\q$ and $\kappa$ by $-\kappa$ if necessary, we can assume $\gag=\gagg\cp\kappa$ with $\dot\kappa>0$. 
\item
Replacing $y$ by $\exp(\kappa(0)\cdot \q)\cdot y$ and $\kappa$ by $\kappa-\kappa(0)$, we can assume $\gag=\gagg\cp\kappa$ with $\dot\kappa>0$ and $\kappa(0)=0$.
\end{itemize} 
\endgroup
\noindent
The claim now follows once we have shown 
\begin{align*} 
	g_t:=\exp(-\kappa(t)\cdot \q)\cdot \exp(t\cdot \g)\in G_{\gag}\qquad\quad\forall\: t\in \RR,
\end{align*}
just by considering the derivative of $\alpha\colon \RR\ni t\mapsto g_t\in G_{\gag}$ at $t=0$. Let thus $t\in \RR$ be fixed. We consider the homeomorphisms 
\begin{align*}
	\Delta:=\kappa(t+ \cdot) -\kappa(t)\colon [0,\infty)\rightarrow [0,\infty)\qquad \text{and}\qquad\Delta':=\kappa^{-1}\cp\Delta\qquad \text{with}\qquad \Delta(0),\Delta'(0)=0.
\end{align*}
For each $s\geq 0$, we obtain (with $\gag=\gagg\cp\kappa$ in the second\slash fourth step)
 \begin{align*}
	g_t\cdot \gag(s)&=\exp(-\kappa(t)\cdot \q)\cdot\gag(t+s)
	= \exp(-\kappa(t)\cdot \q)\cdot\gagg(\kappa(t+s))\\
	&=\exp(-\kappa(t)\cdot \q)\cdot\gagg(\kappa(t)+\Delta(s))=\gagg(\Delta(s))
	=\gag(\kappa^{-1}(\Delta(s)))=\gag(\Delta'(s)). 
\end{align*}
We obtain $g_t\in G_{\gag(0)}$ if we choose $s\equiv 0$. Moreover, the above identity shows that there exist $0<s,s'<\pii(x,\g)$ 
with $g_t\cdot \gag([0,s])=\gag([0,s'])$: 
\begingroup
\setlength{\leftmargini}{12pt}
\begin{itemize}
\item[$\triangleright$]
If $s\leq s'$, then Lemma \ref{stabbiii} with $\gamma\equiv\gag$, $g\equiv g_t$, $\tau\equiv 0$, $\ell\equiv s$, $k\equiv s'$ shows $g_t\in G_\gag$. 
\item[$\triangleright$]
If $s>s'$, then Lemma \ref{stabbiii} with $\gamma\equiv\gag$, $g\equiv g^{-1}_t$, $\tau\equiv 0$, $\ell\equiv s'$, $k\equiv s$ shows $g^{-1}_t\in G_\gag$, hence $g_t\in G_\gag$.
\qedhere
\end{itemize}
\endgroup
\end{proof}

\section{Appendix to Sect.\ \ref{skjsdghfsd}}

\subsection{}
\label{appB1}
\begin{proof}[Proof of Lemma \ref{lemma:simpliiiiii}]
Assume that the claim is wrong; and write $A=[a',a]$, as well as $K_n=[k'_n,k_n]$ for each $n\in \NN$. Then, there exists $\iota\colon \NN\rightarrow \NN$ strictly increasing with $g_{\iota(n)}\cdot \gamma(K_{n})\nsubseteq\gamma(A)$ for each $n\in \NN$. We obtain from \eqref{conni} that
\vspace{-3pt}
\begin{align*}
	 \gamma|_{K_{\iota(n)}}\cpsim g_{\iota(n)}\cdot \gamma|_{K_{\iota(n)}}\quad\quad\:\:&\stackrel{\iota(n)\geq n}{\Longrightarrow}\quad\quad\:\: \gamma|_{K_{n}}\cpsim g_{\iota(n)}\cdot \gamma|_{K_{n}}
	\qquad\quad\Longrightarrow\qquad\quad  \gamma|_{\tilde{J}_n}= (g_{\iota(n)}\cdot \gamma)\cp \rho_n
\end{align*}
for an analytic diffeomorphism $\rho_n\colon K_n\supseteq \tilde{J}_n \rightarrow \tilde{J}'_n\subseteq K_n$. Let now $m \in\NN$ be such large that $K_m\subseteq A$ holds, and fix some $t_n\in \tilde{J}_n$ for each $n\geq m$. Then, we have
\begin{align*}
a'<k'_m\leq k_n'<t_n<k_n\leq k_m<a\qquad\quad\forall\: n\geq m,
\end{align*}
so that Corollary \ref{cgdcfgd} (for $\gamma\equiv \gamma|_A$, $\gamma'\equiv (g_{\iota(n)}\cdot \gamma)|_{K_n}$, and $t\equiv t_n$ there) shows that 
\begin{align*}
	\gamma([a',t_n])\subseteq g_{\iota(n)}\cdot \gamma(K_n)\qquad\qquad&\text{or}\qquad\qquad \gamma([t_n,a])\subseteq g_{\iota(n)}\cdot \gamma(K_n)\qquad\qquad\:\text{holds},\\
	\text{hence}\qquad\qquad \gamma([a',k'_m])\subseteq g_{\iota(n)}\cdot \gamma(K_n)\qquad\qquad&\text{or}\qquad\qquad \gamma([k_m,a\hspace{1pt}])\hspace{1pt}\subseteq g_{\iota(n)}\cdot \gamma(K_n).
\end{align*}  
We thus have 
\begin{align*}
	\hspace{-8.6pt}g_n^{-1}\cdot \gamma([a',k'_m])\subseteq \gamma(K_n)\qquad\qquad\text{or}\qquad\qquad g_n^{-1}\cdot \gamma([k_m,a])\subseteq \gamma(K_n)
\end{align*}
for infinitely many $n\in \NN$; which contradicts that $\gamma(\tau)$ is sated, as $\bigcap_n K_n=\{\tau\}$ holds. 
\end{proof}

\subsection{}
\label{appB2}
\begin{proof}[Proof of Lemma \ref{dds}]
It is clear from Definition \ref{contgen} that $\gamma$ admits a $(\gamma,\tau)$-approximation $\{(g_n,K_n,J_n)\}_{n\in \NN}$. 
\begingroup
\setlength{\leftmargini}{12pt}
\begin{itemize}
\item[$\triangleright$]
By Corollary \ref{lemma:simpli}.\ref{lemma:simpli1} applied to $L\equiv K_{n+1}$ for $n\in \NN$, there exists $p(n)\in \NN$ with 
\begin{align*}
	g_p\cdot \gamma(K_{p})\subseteq \gamma(K_{n+1}) \qquad\quad \forall\: p\geq p(n).
\end{align*}
We thus inductively obtain $\iota\colon \NN\rightarrow \NN$  strictly increasing with
\begin{align*}
	g_{\iota(n)}\cdot \gamma(K_{\iota(n)})\subseteq \gamma(K_{n+1})\subseteq \gamma(J_n)\qquad \text{hence}\qquad g_{\iota(n)}\cdot\gamma(\tau)\in \gamma(J_{n})\qquad\text{for all}\:\: n\in \NN.
\end{align*}
By Remark \ref{jsakjsakjs}.\ref{jsakjsakjsb}, we thus can assume that \eqref{fdfdgfgf} holds, just by passing to $\{(g_{\iota(n)},K_n,J_n)\}_{n\in \NN}$ if necessary. 
\item[$\triangleright$]	
By Corollary \ref{lemma:simpli}.\ref{lemma:simpli1} applied to $L\equiv K_{n+1}$ for $n\in \NN$, there exists some  
$p(n)\geq n+1$ with
\begin{align*}
	g_{p(n)}\cdot \gamma(K_{p(n)})\subseteq \gamma(K_{n+1})\subseteq \gamma(J_n). 
\end{align*}
We define $\iota\colon\NN\rightarrow \NN$ inductively by $\iota(0):=0$ as well as $\iota(n):=p(\iota(n-1))$ for each $n\geq 1$. Then, $ \{(g_{\iota(n)},K_{\iota(n)},J_{\iota(n)})\}_{n\in \NN}$ fulfills \eqref{eq:schachtelung} by construction; so that by Remark \ref{jsakjsakjs}.\ref{ddiudsdsmncxycxc}, we  can achieve that \eqref{eq:schachtelung} holds, just by passing to $ \{(g_{\iota(n)},K_{\iota(n)},J_{\iota(n)})\}_{n\in \NN}$ if necessary.
\item[$\triangleright$]
By Corollary \ref{lemma:simpli}.\ref{lemma:simpli1} applied to $L\equiv K$, there exists $n_0\in \NN$ with 
	$g_n\cdot \gamma(K_{n})\subseteq \gamma(K)$ for all  
	for all $n\geq n_0$. 
By Remark \ref{jsakjsakjs}.\ref{ddiudsdsmncxycxc}, we thus can assume that \eqref{ldslsddsoidsoid} holds, just by passing to $ \{(g_{\iota(n)},K_{\iota(n)},J_{\iota(n)})\}_{n\in \NN}$ with $\iota\colon \NN\ni n\mapsto n_0+n\in \NN$.  \qedhere
\end{itemize}
\endgroup
\end{proof}

\subsection{}
\label{appB4}
\begin{proof}[Proof of Lemma \ref{podspods}]
Let $\{(g_{n},K_n,J_n)\}_{n\in \NN}$ be as in Lemma \ref{dds}. If infinitely many $g_n$ are positive, we can just pass to a subsequence to achieve that all of them are positive (Remark \ref{jsakjsakjs}.\ref{ddiudsdsmncxycxc}). 
 
In the other case, infinitely many $g_n$ are negative:
\begingroup
{
\setlength{\leftmargini}{15pt}
\renewcommand{\theenumi}{{\small\sf\arabic{enumi}})} 
\renewcommand{\labelenumi}{\theenumi}
\begin{enumerate}
\item
 By Remark \ref{jsakjsakjs}.\ref{ddiudsdsmncxycxc}, we can assume that each $g_n$ is negative, just by  passing to a subsequence if necessary.
\item
\label{ksdhjdshjsewiuiuewiuewewnnmewmewmew1}
By \eqref{eq:schachtelung} we have
\begin{align}
\label{odsioidsoioisd98dss98ddssddsewcxxccx}
 g_{n+1}\cdot g_{n+2}\cdot \gamma(K_{n+2})\subseteq g_{n+1}\cdot \gamma(K_{n+1})\subseteq \gamma(J_n)\qquad\quad \forall\: n\in \NN.
\end{align}
We set $K':=K$ as well as
\begin{align*}
	 g_n':=g_{n+1}\cdot g_{n+2}, \qquad K_n':=K_{n},\qquad J'_n:= J_{n}\qquad\quad\hspace{6.8pt} \forall\: n\in \NN.
\end{align*} 
Then, \eqref{odsioidsoioisd98dss98ddssddsewcxxccx} implies that \eqref{conni} and \eqref{fdfdgfgf} hold with $g_n\equiv g'_n$ and $J_n\equiv J'_n$ there. Hence, $\{(g'_n,K'_n,J'_n)\}_{n\in \NN}$ is a $(\gamma,\tau)$-collection such that \eqref{fdfdgfgf} holds.
\item
\label{ksdhjdshjsewiuiuewiuewewnnmewmewmew2}
By Corollary \ref{lemma:simpli}.\ref{lemma:simpli1}, there exists $n_0 \in \NN$ with  
$g'_n\cdot \gamma(K'_n)\subseteq \gamma(K')$ for all $n\geq n_0$. Then, for each $n\geq n_0$, we have
	$g'_n\cdot \gamma|_{J'_n} = \gamma\cp \rho'_n$ for the (unique) analytic diffeomorphism (Lemma \ref{lemma:BasicAnalyt2})
\begin{align*}
	\rho'_n\equiv (\gamma|^{\im[\gamma]})^{-1}\cp (g'_n\cdot \gamma|_{J_n'})\colon J'_n\rightarrow I'_n\subseteq K'. 
\end{align*}
By Remark  \ref{opdspodspodspodspodsds0909ds09dsdsdsdsds}, we furthermore have $g_n\cdot \gamma|_{J_n}=\gamma\cp\rho_n$  for the analytic diffeomorphism 
\begin{align*}
	\rho_n= (\gamma|^{\im[\gamma]})^{-1}\cp (g_n\cdot \gamma|_{J_n})\colon J_n\rightarrow I_n\subseteq J_{n-1}\subseteq I\qquad
	\text{for each}\qquad n\geq 1.
\end{align*}
	Hence, $J_{n+2}\subseteq J_n\equiv J'_n$ and $I_{n+2}\subseteq J_{n+1}$ holds for each $n\in \NN$; and we obtain for $n\geq n_0$ that
\begin{align*}
	\rho'_{n}|_{J_{n+2}}&=(\gamma|^{\im[\gamma]})^{-1}\cp(g_{n+1}\cdot g_{n+2}\cdot \gamma|_{J_{n+2}})=(\gamma|^{\im[\gamma]})^{-1}\cp(g_{n+1}\cdot (\gamma\cp \rho_{n+2}))\\
	&=(\gamma|^{\im[\gamma]})^{-1}\cp((g_{n+1}\cdot \gamma|_{I_{n+2}})\cp \rho_{n+2})=(\gamma|^{\im[\gamma]})^{-1}\cp((g_{n+1}\cdot \gamma|_{J_{n+1}})\cp \rho_{n+2})\\
	&=(\gamma|^{\im[\gamma]})^{-1}\cp((\gamma\cp \rho_{n+1})\cp \rho_{n+2})=\rho_{n+1}\cp\rho_{n+2}
\end{align*}
holds, which implies that $g'_n$ is positive as both $\rho_{n+1}$ and $\rho_{n+2}$ are negative. 	
\item	
\label{ksdhjdshjsewiuiuewiuewewnnmewmewmew3}
\begingroup
\setlength{\leftmarginii}{10pt}
\begin{itemize}
\item[\:{\sf a)}]
It remains to show that $g'_n\in G\backslash G_{\gamma(\tau)}$ holds for infinitely many $n\geq n_0$.\vspace{3pt}

\textit{Proof of the Claim.}
If $g'_n\in G\backslash G_{\gamma(\tau)}$ holds for infinitely many $n\geq n_0$, then there exists $\iota\colon \NN\rightarrow \NN$ strictly increasing with $\iota(0)\geq n_0$ as well as $\{g_{\iota(n)}'\}_{n\in \NN}\subseteq G\backslash G_{\gamma(\tau)}$. According to Remark \ref{jsakjsakjs}.\ref{ddiudsdsmncxycxc} $\{(g'_{\iota(n)},K'_{\iota(n)},J'_{\iota(n)})\}_{n\in \NN}$ is a $(\gamma,\tau)$-approximation, such that 
   \eqref{fdfdgfgf} holds by Step  \ref{ksdhjdshjsewiuiuewiuewewnnmewmewmew1} as well as \eqref{ldslsddsoidsoid} 
by our choice of $n_0$ in Step \ref{ksdhjdshjsewiuiuewiuewewnnmewmewmew2}, with $g_{\iota(n)}'$ positive for each $n\in \NN$ by Step \ref{ksdhjdshjsewiuiuewiuewewnnmewmewmew2}.
\hspace*{\fill}$\mathsmaller{\square}$
\item[{\sf b)}]
To prove the statement in {\sf a)} (i.e.\ that $g'_n\in G\backslash G_{\gamma(\tau)}$ holds for infinitely many $n\geq n_0$), it suffices to show that      
$g_n'\in G\setminus G_\gamma$ holds for infinitely many $n\geq n_0$.
\vspace{3pt}

\textit{Proof of the Claim.}
Assume that $g_n'\in G\setminus G_\gamma$ holds for infinitely many $n\geq n_0$, hence that there exists $\iota\colon \NN\rightarrow \NN$ strictly increasing with $\iota(0)\geq n_0$ and $\{g'_{\iota(n)}\}_{n\in \NN}\subseteq G\setminus G_\gamma$. 
Then, Lemma \ref{seque}.\ref{seque1} implies that $g'_{\iota(n)}\in G_{\gamma(\tau)}$ can only hold for  finitely many $n\in \NN$ (which proves the claim):
\vspace{2pt}
\begingroup
\setlength{\leftmarginiii}{12pt}
\begin{itemize} 
\item
Assume that $g'_{\iota(n)}\in G_{\gamma(\tau)}$ holds for infinitely many $n\in \NN$, hence that there exists $\tilde{\iota}\colon \NN\rightarrow \NN$ strictly increasing with $\{g'_{\tilde{\iota}(\iota(n))}\}_{n\in \NN}\subseteq  G_{\gamma(\tau)}$.
\vspace{2pt}
\item
Then, for each $n\in \NN$ we have 
\begin{align*}
g_{\tilde{\iota}(\iota(n))}'\cdot \gamma|_{J_{\tilde{\iota}(\iota(n))}'}=\gamma\cp \rho'_{\tilde{\iota}(\iota(n))}
\qquad\:\:\Longrightarrow\qquad\:\: 
g_{\tilde{\iota}(\iota(n))}'\cdot \gamma|_{K_{\tilde{\iota}(\iota(n))}'\cap [\tau,\infty)}\cpsim \gamma|_{K_{\tilde{\iota}(\iota(n))}'\cap [\tau,\infty)},
\end{align*}
because  $\rho'_{\tilde{\iota}(\iota(n))}$ is positive with $\rho'_{\tilde{\iota}(\iota(n))}(\tau)=\tau$ as  $g'_{\tilde{\iota}(\iota(n))}\in G_{\gamma(\tau)}$ holds.
\vspace{3pt}
\item
Lemma \ref{seque}.\ref{seque1} applied to $\{g'_{\tilde{\iota}(\iota(n))}\}_{n\in \NN}$, yields $m\in \NN$ with $\{g'_{\tilde{\iota}(\iota(n))}\}_{n\geq m}\subseteq G\setminus G_{\gamma(\tau)}$, which contradicts the assumptions. 
  \hspace*{\fill}$\mathsmaller{\square}$
\end{itemize}
\endgroup
\item[\:{\sf c)}]
Finally assume that the condition in {\sf b)} is not fulfilled, i.e., that 
$g_n'\in G\setminus G_\gamma$ only holds for finitely many $n\geq n_0$. Then, there exists $m\geq n_0+1$ with $\{g'_n\}_{n\geq m-1}\subseteq G_\gamma$.  For each $n\geq m$, we have 
$g_{n+1}=g^{-1}_n\cdot h_n$ for $h_n:=g'_{n-1}\in G_\gamma$, hence 
\vspace{-17pt}
\begin{align*}
	g_{n+2}=g_{n+1}^{-1}\cdot h_{n+1}=h_n^{-1}\cdot g_n\cdot  h_{n+1}\qquad\quad &\Longrightarrow\qquad\quad   g_{n+2}\cdot \gamma(\tau)=\overbrace{h_n^{-1}}^{\in\: G_\gamma}\cdot \overbrace{g_n\cdot \gamma(\tau)}^{\stackrel{\eqref{ldslsddsoidsoid}}{\in} \im[\gamma]}=g_n\cdot \gamma(\tau).
\end{align*}
We obtain  
$\gamma(\tau)\stackrel{\eqref{fdfdgfgf}}{=}\lim_n g_{m+2n}\cdot \gamma(\tau)=g_m\cdot \gamma(\tau)$, which contradicts $g_m\in G\setminus G_{\gamma(\tau)}$.
\qedhere
\end{itemize}
\endgroup
\end{enumerate}}
\endgroup
\end{proof}

\subsection{}
\label{appB3}
\begin{proof}[Proof of Lemma \ref{pofpofdofdop}]
Let $\{(g_n,K_n,J_n)\}_{n\in \NN}$ be as in Corollary \ref{dffssdffds}. 
By Remark \ref{podspoddsa}, it suffices to show that there exists an open interval $J\subseteq I\equiv\dom[\gamma]$ with $\tau\in J$, such that $\gamma(J)\subseteq G\cdot \gamma(\tau)$ holds. 
Now, by \eqref{fdfdgfgf}, we have $\{g_n\cdot \gamma(\tau)\}_{n\in \NN}\subseteq \im[\gamma]$ with $\lim_n g_n\cdot \gamma(\tau)=\gamma(\tau)$.  
Since $\gamma$ is an embedding, there exists
$$I\backslash \{\tau\}\supseteq \{s_n\}_{n\in \NN}\rightarrow \tau\qquad\quad\:\:\text{with}\qquad\quad\:\: g_n\cdot \gamma(\tau)= \gamma(s_n)\qquad\forall\: n\in \NN.$$ 
Since $x\equiv\gamma(\tau)$ is stable with $\{g_n\}_{n\in \NN}\subseteq G\backslash G_{x}$, there exist sequences $\{h_n\}_{n\in \NN}, \{h'_n\}_{n\in \NN}\subseteq G_{[x]}$, as well as $\iota\colon \NN\rightarrow \NN$ strictly increasing, such that $\{h_{\iota(n)}\cdot g_{\iota(n)}\cdot h'_{\iota(n)}\}_{n\in \NN}\rightarrow g\in G_x$ converges. We define 
\begin{align*}
	g'_n:= h_{\iota(n)}\cdot g_{\iota(n)}\cdot h'_{\iota(n)}\cdot g^{-1} \qquad\quad\forall\: n\in \NN;
\end{align*}	
and observe that $\lim_n g'_n=e$ holds, as well as 
\begin{align*}
	g'_n\cdot \gamma(\tau)=g_{\iota(n)}\cdot \gamma(\tau)= \gamma(s_{\iota(n)})\qquad\quad\forall\: n\in \NN.
\end{align*}		  
The claim now follows from Lemma \ref{gdfgfgf}. 
\end{proof}

\subsection{}
\label{appB5}
\begin{proof}[Proof of Lemma \ref{daaddsd}]
Let $(\gamma,\rho)\vee \{(g_{n},K_n,J_n)\}_{n\in \NN}$ be as in Lemma \ref{pofpofdofdop}. If $\Delta_n>0$ holds for infinitely many $n\in \NN$, the claim just follows by passing to a  subsequence (cf.\ Remark \ref{kldlkfdklfd}).  
In the other case, passing to a subsequence, we can achieve that $\Delta_n<0$ holds for all $n\in \NN$. We set  
$$\inv\colon \dom[\gamma]\equiv I=(i',i)\ni t\mapsto i'+i- t\in I,$$ 
observe that $\inv^{-1}=\inv$ holds, and define
\begin{align*}
	\tilde{\gamma}:=\gamma\cp \inv,\qquad \tilde{\tau}:=\inv^{-1}(\tau),\qquad \tilde{g}_n:= g_n,\qquad \tilde{K}_n:=\inv^{-1}(K_n),\qquad \tilde{J}_n:=\inv^{-1}(J_n),\qquad \tilde{\rho}_n:=\inv^{-1}\cp \rho_n\cp \inv|_{\tilde{J}_n}
\end{align*}
as well as $\tilde{K}:= \tilde{K}_0$ and $\tilde{I}:=\inv(I)=I$. Then, $\delta|_{I'}=\tilde{\gamma}\cp \tilde{\rho}$ with $\tilde{\rho}(\ovl{\tau})=\tilde{\tau}$ holds for $\tilde{\rho}:= \inv^{-1}\cp \rho$ 
(i.e.\ $\tilde{\gamma}$ 
 fulfills the ``diffeomorphism-condition'' in Definition \ref{ddskjdsjkdsjd}); and we evidently have: 
\begin{align*}
\textstyle\bigcap_{n\in \NN}\tilde{K}_n =\{\tilde{\tau}\}\qquad\quad\:\: &\!\text{with}\qquad\quad\tilde{I} \supset \tilde{K} = \tilde{K}_0\supset \tilde{J}_0 \supset \tilde{K}_1\supset \tilde{J}_1 \supset \tilde{K}_2\supset \tilde{J}_2\supset  \dots\\[4pt]
\{\tilde{g}_n\}_{n\in \NN}&\subseteq G\backslash G_{\tilde{\gamma}(\tilde{\tau})}\\[4pt]
\tilde{\gamma}(\tilde{I})=\gamma(I)&\subseteq G\cdot \gamma(\tau)=G\cdot \tilde{\gamma}(\tilde{\tau})\\[4pt] 
\text{as } &\text{well as}\\[4pt]
 \tilde{g}_n\cdot \tilde{\gamma}(\tilde{\tau})=g_n\cdot \gamma(\tau)&\in \gamma(J_n)=\tilde{\gamma}(\tilde{J}_n)
\\[4pt]
	\underbrace{\tilde{g}_n\cdot \tilde{\gamma}(\tilde{K}_n)}_{g_n\cdot\:\gamma(K_n)}\subseteq\underbrace{\tilde{\gamma}(\tilde{K})}_{\gamma(K)}\subseteq \tilde{\gamma}(\tilde{I})\qquad\quad&\text{and}\qquad\quad \tilde{\rho}_n\colon \tilde{J}_n\rightarrow \tilde{I}_n:= \inv^{-1}(I_n)\subseteq \tilde{I}\\[-12pt] 
&\!\text{with}\\[4pt]
	\tilde{g}_n\cdot \tilde{\gamma}|_{\tilde{J}_n}= \gamma\cp \rho_n \cp\inv|_{\tilde{J}_n}=\tilde{\gamma}\cp \tilde{\rho}_n
\qquad\quad\:\:&\!\!\text{hence}\qquad\quad\:\: \tilde{g}_n\cdot \gamma|_{\tilde{J}_n}\cpsim \gamma|_{\tilde{J}_n} 	
	\\[3pt]    
\text{for e} &\text{ach }n\in \NN.
\end{align*}
Obviously,  $\tilde{\rho}_n$ (hence $\tilde{g}_n$) is  positive for each $n\in \NN$; and $\tilde{g}_n$ shifts $\tilde{\tau}$ to the right, because
\[
	\tilde{g}_n\cdot \tilde{\gamma}(\tilde{\tau})=g_n\cdot \gamma(\tau)=\gamma(\tau+\Delta_n)=\tilde{\gamma}(\inv^{-1}(\tau+\Delta_n))=\tilde{\gamma}(\tilde{\tau}-\Delta_n).\qedhere
\]
\end{proof}

\subsection{}
\label{appB6}
\begin{proof}[Proof of Lemma \ref{approx}]
\begingroup
\setlength{\leftmargini}{16pt}
\begin{enumerate}
\item
We have $p(n)\geq 1$, as $g_n\cdot \gamma(J_n)\subseteq \gamma(K)\subseteq \gamma(K')$ holds by \eqref{ldslsddsoidsoid}; and   
the right side of \eqref{propies} is clear from 
\vspace{-4pt}
\begin{align*}
\gamma(\tau_{n,p+1})=(g_n)^p\cdot (g_n\cdot\gamma(\tau))\stackrel{\eqref{fdfdgfgf}}{\in} (g_n)^p\cdot \gamma(J_n)=\gamma(I_{n,p}) \qquad\quad\forall\: 0\leq p \leq p(n)-1. 
\end{align*}
To prove the left side of \eqref{propies}, for each $0\leq p \leq p(n)$, we let  
$\rho_{n,p}\colon J_n\rightarrow I_{n,p}$ denote the unique analytic diffeomorphism with 
	$(g_n)^p\cdot \gamma|_{J_n}=\gamma\cp \rho_{n,p}$ (Lemma \ref{lemma:BasicAnalyt2}).\footnote{In anticipation of the proof of Part \ref{approx3}, we define $\rho_{n,p}$ also for $p=p(n)$ at this point (although we only need the diffeomorphisms $\rho_{n,p}$ for $0\leq p<p(n)$ for the argumentation in this part).} 
	\vspace{4pt}
	
	\noindent
	These diffeomorphisms are positive, hence strictly increasing:
	\vspace{-3pt}
\begingroup
\setlength{\leftmarginii}{12pt}
\begin{itemize}
\item[$\triangleright$]	
We have $\dot\rho_{n,0}>0$, as necessarily $I_{n,0}=J_n$ and $\rho_{n,0}=\id_{J_n}$ holds. In addition to that,    
we have $\dot\rho_{n,1}>0$ by positivity of $g_n$, as $\rho_{n,1}=\rho_n\colon J_n \rightarrow I_{n,1}=I_n$ holds by uniqueness. 
\item[$\triangleright$]
Positivity thus follows inductively, if we prove the following implication:
\begin{align*}
	\dot\rho_{n,p}>0\quad\text{for some}\quad 1\leq p \leq p(n)-1\qquad\quad\Longrightarrow\qquad\quad\dot\rho_{n,p+1}>0. 
\end{align*}
Let thus $1\leq p \leq p(n)-1$. 
Since $g_n\cdot \gamma|_{J_n}\cpsim \gamma|_{J_n}$ holds by \eqref{conni}, there exists $J\subseteq J_n$ with $\rho_n(J)\subseteq  J_n$. We obtain 
\begin{align*}
\gamma\cp \rho_{n,p+1}|_{J}=(g_n)^{p+1}\cdot \gamma|_{J}= (g_n)^{p}\cdot (g_n\cdot \gamma|_{J}) = (g_n)^{p}\cdot (\gamma\cp\rho_n|_J)=\gamma\cp\rho_{n,p}\cp\rho_n|_J,
\end{align*}
hence $\rho_{n,p+1}|_{J}=\rho_{n,p}\cp\rho_n|_J$ by injectivity of $\gamma$, which shows $\dot\rho_{n,p+1}>0$.
\end{itemize}
\endgroup
\vspace{-5pt}

\noindent	 
Since $g_n$ shifts $\tau$ to the right, we have $\tau_{n,0}\equiv\tau<\tau_{n,1}= \rho_{n}(\tau)=\tau +\Delta_n\stackrel{\eqref{fdfdgfgf}}{\in} J_n$ for some $\Delta_n>0$. For $0\leq p \leq p(n)-1$, we obtain
\begin{align}
\label{sdfsdffsd}
\begin{split}
	 \gamma(\tau_{n,p+1})&= (g_n)^{p+1}\cdot \gamma(\tau)=(g_n)^{p}\cdot (g_n \cdot \gamma(\tau))=(g_n)^{p}\cdot \gamma(\tau_{n,1})\\
	 &=  
	 \gamma(\rho_{n,p}(\tau_{n,1}))=\gamma(\rho_{n,p}(\tau+\Delta_n)) 
	 =\gamma(\rho_{n,p}(\tau)+\Delta'_n)=\gamma(\tau_{n,p}+\Delta_n')
\end{split}
\end{align}
for some $\Delta'_n>0$ as $\rho_{n,p}$ is positive. The left side of \eqref{propies} now follows inductively. 
\item
If $p(n)=\infty$ holds, then $\{\tau_{n,p}\}_{p\in \NN}\subseteq K'$ is strictly increasing by \eqref{propies}, hence $\lim_p \tau_{n,p} =t\in K'$ exists with $t\geq \tau_{n,1}>\tau$. 
We thus have 
\begin{align*}
\textstyle\gamma(t)=\lim_p\gamma(\tau_{n,p})=\lim_p\he(g_n)^p\he\cdot \gamma(\tau)\qquad\quad\Longrightarrow \qquad\quad	 \gamma(t)&=\textstyle\lim_p\he (g_n)^p\cdot (g_n\cdot \gamma(\tau))\\
&=\textstyle\lim_p \he(g_n)^p \cdot\gamma(\tau_{n,1});
\end{align*} 
which contradicts that $\gamma(t)$ is sated, as $\tau, \tau_{n,1}, t$ are mutually different and $\gamma$ is  injective. 
\item
By Part \ref{approx2}, $h_n:=(g_n)^{p(n)+1}$ is defined for each $n\in \NN$, with 
\begin{align}
\label{pocspodspodspods}
	h_n\cdot \gamma(K_n)\nsubseteq \gamma(K')\qquad\quad\forall\: n\in \NN 
\end{align}
by the definition of $p(n)$. Assume now that the claim is wrong, hence that $i_{n,p(n)}\leq b$ holds for infinitely many $n\in \NN$. 
\begingroup
\setlength{\leftmarginii}{12pt}
\begin{itemize}
\item
It suffices to show that for each $n\in \NN$ with $i_{n,p(n)}\leq b$, there exists an open interval $J[n]\subseteq K'$ with $J[n]\cap K\neq \emptyset$ as well as an analytic diffeomorphism $\rho[n]\colon K'\supseteq J[n]\rightarrow J'[n]\subseteq K_n$ with
\begin{align*}
	(\gamma|_{K'})|_{J[n]}=(h_n\cdot \gamma)|_{K_n}\cp \rho[n].
\end{align*}
In fact, then we obtain a contradiction to satedness of $\gamma(\tau)$ as follows: 
\begingroup
\setlength{\leftmarginiii}{12pt}
\begin{itemize}
\item[$\triangleright$]	
By \eqref{pocspodspodspods}, Corollary \ref{cgdcfgd} applied to $t\in J[n]\cap K$ and $\gamma\equiv \gamma|_{K'}$, $\gamma'\equiv (h_n\cdot \gamma)|_{K_n}$  shows that  
\begin{align*}
	\gamma([a',t])\subseteq h_n\cdot \gamma(K_n)\qquad\quad&\text{or}\qquad\quad \gamma([t,b'])\subseteq h_n\cdot \gamma(K_n)\qquad\quad\:\text{holds};\\
	\text{hence,}\qquad\quad \gamma([a',a])\subseteq h_n\cdot \gamma(K_n)\qquad\quad&\text{or}\qquad\hspace{7.33pt} \hspace{2pt}\gamma([b,b'])\subseteq h_n\cdot \gamma(K_n).
\end{align*}
\item[$\triangleright$]	
We thus have $h_n^{-1}\cdot \gamma([a',a])\subseteq \gamma(K_n)$ or $h_n^{-1}\cdot \gamma([b,b'])\subseteq \gamma(K_n)$ for infinitely many $n\in \NN$, which contradicts that $\gamma(\tau)$ is sated as $\bigcap_{n\in \NN}K_n =\{\tau\}$. 
\end{itemize}
\endgroup
\item
To show the existence of intervals and diffeomorphisms as in the previous point, let now $n\in \NN$ with $i_{n,p(n)}\leq b$ be fixed. We argue as follows:  
\begingroup
\setlength{\leftmarginiii}{12pt}
\begin{itemize}
\item[$\triangleright$]	
Since $g_n\cdot \gamma(\tau)\in \gamma(J_n)$ holds, there exists an open interval $J'[n]\subseteq J_n\subseteq K_n$ containing $\tau$ with $g_n\cdot \gamma(J'[n])\subseteq \gamma(J_n)$. We obtain
\begin{align}
\label{dffdf}
	h_n\cdot \gamma(J'[n])=(g_n)^{p(n)}\cdot (g_n\cdot \gamma(J'[n]))\subseteq (g_n)^{p(n)}\cdot \gamma(J_n)=\gamma(I_{n,p(n)})\subseteq \gamma(K').
\end{align}
Hence, $h_n\cdot \gamma(J'[n])=\gamma(J[n])$ holds for an open interval $J[n]\subseteq K'$, because $\gamma$ is an embedding; specifically, 
 $(\gamma|_{K'})|_{J[n]}=h_n\cdot \gamma\cp \rho[n]$ holds  for the analytic diffeomorphism (Lemma \ref{lemma:BasicAnalyt2})
\begin{align*}
	\rho[n]= (h_n\cdot\gamma|^{\im[\gamma]})^{-1}\cp \gamma|_{J[n]}\colon K'\supseteq J[n]\rightarrow J'[n]\subseteq K_n.
\end{align*}
\item[$\triangleright$]
Then, $h_n\cdot \gamma(\tau)=\gamma(\tau_{n,p(n)+1})$ holds for some unique
\vspace{-7pt}
\begin{align*}
	\qquad\qquad\qquad 
J[n]\ni \tau_{n,p(n)+1}
	\stackrel{\eqref{dffdf}}{<}i_{n,p(n)}\leq b.
\end{align*} 
Copying the proof of Part \ref{approx1} with $p= p(n)$ in \eqref{sdfsdffsd}, we additionally obtain 
$a<\tau_{n,p(n)}<\tau_{n,p(n)+1}$; which shows     
$\tau_{n,p(n)+1}\in J[n]\cap K\neq \emptyset$.
\end{itemize}
\endgroup	
\end{itemize}
\endgroup 
\item 
The second statement is clear from the first statement as well as from \eqref{fgopopfgopfg} and \eqref{propies}. 
Assume now that the first statement is wrong:
\begingroup
\setlength{\leftmarginii}{12pt}
\begin{itemize}
\item
There exist $t\in (\tau,b]$, $n_0\in \NN$ and $\epsilon>0$ with $\tau< t-\epsilon$, as well as $\iota\colon \NN\rightarrow \NN$ strictly increasing with $\iota(0)\geq \max(n_0,d)$, such that  
$\tau_{\iota(n),m}\notin (t-\epsilon,t]$ holds for all  $1\leq m\leq p(\iota(n))$ and $n\in \NN$.  
\item
For each $n\in \NN$, the left side of \eqref{propies} implies that one of the following two situations hold:
\begingroup
\setlength{\leftmarginiii}{12pt}
\begin{itemize}
\item[$*$]
$(t-\epsilon,t]\subseteq (\tau_{\iota(n),q(n)},\tau_{\iota(n),q(n)+1})$ holds for some $0\leq q(n)\leq p(\iota(n))-1$, hence   
$(t-\epsilon,t]\subseteq I_{\iota(n),q(n)}$ by the right side of \eqref{propies}.
\vspace{2pt}
 \item[$*$]
 $(t-\epsilon,t]\subseteq   (\tau_{\iota(n),p(\iota(n))},b]$ holds, hence $(t-\epsilon,t]\subseteq I_{\iota(n),q(n)}$ for  $q(n):= p(\iota(n))$ as  $b<i_{\iota(n),p(\iota(n))}$ holds by Part \ref{approx3}  (observe $\iota(n)\geq \iota(0) \geq d$). 
\end{itemize}
\endgroup
\noindent
We thus have
\begin{align*}
	\gamma((t-\epsilon,t])\subseteq \gamma(I_{\iota(n),q(n)})=(g_{\iota(n)})^{q(n)}\cdot \gamma(J_{\iota(n)})\qquad\:\:\Longrightarrow\qquad\:\:
	(g_{\iota(n)})^{-q(n)}\cdot \gamma((t-\epsilon,t])\subseteq \gamma(J_{\iota(n)})
\end{align*}
for each $n\in \NN$, 
which contradicts that $\gamma(\tau)$ is sated.
\end{itemize}
\endgroup
\item
If $q(n)=0$ holds for infinitely many $n\in \NN$, the left side of \eqref{fhjjghhhg} implies $g=e$, which shows the right side of \eqref{fhjjghhhg}. We thus can assume that   
there exists $\iota\colon \NN\rightarrow \NN$ strictly increasing with $1\leq q(\iota(n))\leq p(\iota(n))$ for all $n\in \NN$, and set 
\begin{align*}
	g_{\iota(n)}':=(g_{\iota(n)}\cdot h_{\iota(n)})^{q(\iota(n))}\qquad\quad\forall\: n\in \NN.
\end{align*}
Since $(g_{\iota(n)})^{p}\cdot \gamma(J_{\iota(n)})\subseteq \im[\gamma]$ holds for each $0\leq p\leq q(\iota(n))$, we have 
\begin{align}
\label{odspdspdspds}
	g_{\iota(n)}'\cdot \gamma|_{J_{\iota(n)}}=(g_{\iota(n)})^{q(\iota(n))}\cdot \gamma|_{J_{\iota(n)}}=\gamma\cp \rho_{\iota(n),q(\iota(n))}\qquad\quad\forall\: n\in \NN,
\end{align}
for $\rho_{\iota(n),q(\iota(n))}\colon J_{\iota(n)}\rightarrow I_{\iota(n),q(\iota(n))}\subseteq K'$ defined as in Part \ref{approx1}; hence 
\begingroup
\setlength{\leftmarginii}{18pt}
\begin{itemize}
\item[$\rm(a)$]
	$\rho_{\iota(n),q(\iota(n))}$\quad\hspace{78.3pt} is positive,
	\vspace{2pt}
\item[$\rm(b)$]
	$g_{\iota(n)}'\cdot \gamma(\tau)=\gamma(\tau_{\iota(n),q(\iota(n))})$\qquad with $\tau_{\iota(n),q(\iota(n))}=\rho_{\iota(n),q(\iota(n))}(\tau)$ holds for each $n\in \NN$,
\item[$\rm(c)$]
	$g_{\iota(n)}'\cdot \gamma|_{J'}\cpsim \gamma|_{J'}$\quad\hspace{46.7pt} holds for each open interval $J'$ containing $\tau, \tau_{\iota(n),q(\iota(n))}$.
\end{itemize}
\endgroup
\noindent
Combining (b) with the left side of \eqref{fhjjghhhg}, we obtain 
\begin{align*}
	\textstyle\lim_n \gamma(\tau_{\iota(n),q(\iota(n))})=\gamma(\tau)\qquad\text{hence}\qquad \lim_n \tau_{\iota(n),q(\iota(n))}=\tau\qquad\text{as }\gamma\text{ is an embedding}.
\end{align*}
Passing to a subsequence (modifying $\iota$), we thus can assume that $\{\tau_{\iota(n),q(\iota(n))}\}_{n\in \NN}$ is  strictly decreasing (observe that $\tau_{\iota(n),q(\iota(n))}>\tau$ holds for each $n\in \NN$, as $q(\iota(n))\geq 1$ holds for each $n\in \NN$). 
  Then, we find compact intervals $\{K_n\}_{n\in \NN}$ with $K_0'=K'$ 
  as well as 
\begin{align*}
K_n'\subseteq K',\qquad\tau\in J'_n:= \innt[K_n'],\qquad\tau_{\iota(n),q(\iota(n))}\in J_n',\qquad K'_{n+1}\subset J'_n\qquad\qquad\forall\: n\in \NN
\end{align*}
such that $\bigcap_{n\in \NN} K_n'=\{\tau\}$ holds. 
Then, \eqref{conni} holds for $\{(g_{\iota(n)}',K_n',J_n')\}_{n\in \NN}$ by (c); so that Corollary \ref{lemma:simpli}.\ref{lemma:simpli2} yields  $n_0\in \NN$ and a compact neighbourhood 
$L\subseteq K'$ of $\tau$ with  $g'_{\iota(n)}\cdot \gamma(L)\subseteq \gamma(K')$ for all $n\geq n_0$. We proceed as follows:\he\footnote{The following argumentation is  the same as in the end of Lemma \ref{popofdpofd}.} 
\begingroup
\setlength{\leftmarginii}{11pt}
\begin{itemize}
 \item[$\triangleright$]
For each $n\geq n_0$, we have $g'_{\iota(n)}\cdot \gamma|_L=\gamma\cp \kappa_n$ for a unique analytic diffeomorphism $\kappa_n\colon L\rightarrow L_n\subseteq K'$. 
\item[$\triangleright$]
Let $n'_0\geq n_0$ be such large that $J_{\iota(n)}\subseteq L$ holds for each $n\geq n'_0$. 
Then, $\kappa_n$ is positive for 
each $n\geq n'_0$ by (a), because $\kappa_n|_{J_{\iota(n)}}=\rho_{\iota(n),q(\iota(n))}|_{J_{\iota(n)}}$ holds by uniqueness and \eqref{odspdspdspds}.
 \item[$\triangleright$]
Since (b) shows that $g'_{\iota(n)}$ shifts $\tau$ to the right, for $L=[\ell',\ell]$, $K'=[k',k]$, $n\geq n'_0$, and we have
\begin{align*}
	g'_{\iota(n)}\cdot \gamma([\tau,\ell])\subseteq \gamma([\tau,k]).
\end{align*} 
\item[$\triangleright$]
For each $t\in [\tau,\ell]$ and $n\geq n_0'$, we thus have $g'_{\iota(n)}\cdot \gamma(t)\in \gamma([\tau,k])$, hence 
\begin{align*} 
	\textstyle g\cdot \gamma(t)=\lim_n g'_{\iota(n)}\cdot \gamma(t)\in \gamma([\tau,k]).
\end{align*}
	This shows $g\cdot \gamma([\tau,\ell])\subseteq \gamma([\tau,k])$ with $g\in G_{\gamma(\tau)}$; so that Lemma \ref{stabbiii} yields $g\in G_{\gamma}$. \qedhere
\end{itemize}
\endgroup
\end{enumerate}
\endgroup
\end{proof}

\section{Appendix to Sect.\ \ref{disgencur}}

\subsection{}
\label{appC1}
\begin{proof}[Proof of Lemma \ref{eoder}]
\begingroup
\setlength{\leftmargini}{16pt}
\begin{enumerate}
\item
If $A\subseteq I$ is free, then $A$ must be contained in $(i',\tau]$ or $[\tau,i)$. In fact, elsewise $\tau$ is contained in the interior of $A$;  and we obtain:
\vspace{-2pt}
\begin{align*}
g\cdot \gamma|_{(i',\tau]}\psim \gamma|_{[\tau,i)}\qquad
\stackrel{\tau\:\in\: \innt[A]}{\Longrightarrow}\qquad g\cdot \gamma|_{A}\cpsim \gamma|_{A}\qquad
\stackrel{A\:\text{free}}{\Longrightarrow}\qquad g\in G_\gamma\qquad
\Longrightarrow\qquad [g]=[e],
\end{align*}
which contradicts $[g]\neq [e]$. 
\item
As in Part \ref{eoder1},\footnote{If $B\subseteq I$ is free with $A\subset B$, then $\{a_{-1},a_1\}\cap \innt[B]\neq \emptyset$ holds. Then, $g_{\pm 1}\cdot \gamma|_A\psim \gamma|_{A_{\pm 1}}$ implies $\{g_{-1},g_1\}\cap G_\gamma\neq \emptyset$, which contradicts $[g_{\pm 1}]\neq [e]$.} it follows from $g_{\pm 1}\cdot \gamma|_A\psim \gamma|_{A_{\pm 1}}$ and $[g_{\pm 1}]\neq [e]$  that an interval $B\subseteq I$ cannot be free if it properly contains $A$. Hence, $A$ is maximal. 

If a) is wrong (the condition b) is treated analogously), then $\lim_{n\rightarrow -\infty} a_{n} =t$ holds for some $i'<t$. Then, for each $\epsilon>0$, there exists some $N_\epsilon\in \NN$ with $A_n\subseteq  [t,t+\epsilon)$ for all $n\geq N_\epsilon$, hence 
\begin{align*}
	g_n\cdot \gamma(A)=\gamma(A_n)\subseteq \gamma([t,t+\epsilon))\qquad\quad \forall\: n\geq N_\epsilon.
\end{align*}
There thus exists $\iota\colon \NN\rightarrow \NN$ strictly  increasing with $g_{\iota(n)}\cdot \gamma(A)\subseteq \gamma([t,t+1/n))$ for all $n\in \NN$, which contradicts that $\gamma(t)$ is sated (as $\gamma$ is continuous). 
\item
By the maximality statements in Part \ref{eoder1} and Part \ref{eoder2}, each further decomposition of $I$ is necessarily a $\tau$-decomposition $[g']\in G\slash G_\gamma$. But, then 
\begin{align*}
	g\cdot \gamma|_{(i',\tau]}\psim \gamma|_{[\tau,i)}\quad\wedge\quad g'\cdot \gamma|_{(i',\tau]}\psim \gamma|_{[\tau,i)}\qquad\quad&\Longrightarrow\qquad\quad g'\cdot \gamma|_{(i',\tau]}\cpsim g\cdot \gamma|_{(i',\tau]}\\[2pt]
	 &\Longrightarrow\qquad\quad (g^{-1}\cdot g')\cdot \gamma|_{(i',\tau]}\cpsim \gamma|_{(i',\tau]}\\[2pt]
	 &\Longrightarrow\qquad\quad (g^{-1}\cdot g')\in G_\gamma\\[2pt]
	 &\Longrightarrow\qquad\quad [g']=[g]. \qedhere 
\end{align*}	
\end{enumerate}
\endgroup 
\end{proof}

\subsection{}
\label{appC2}
\begin{proof}[Proof of Lemma \ref{taufaith}]
Let $\delta\colon I\rightarrow M$ be immersive and free with $\tau$-decomposition $[g]$; and assume  that 
\begin{align*}
	g'\cdot \delta|_{(i',\tau]} \cpsim \delta\quad\:\text{holds w.r.t.}\quad\: \rho\colon (i',\tau]\supseteq J\rightarrow J'\subseteq I\quad\:\text{for some}\quad\: g'\in G.
\end{align*}
Then, we must have
\begin{align}
\label{popodspodspods11}
g'\cdot \delta|_{(i',\tau]} \cpsim \delta|_{(i',\tau]}\quad\text{w.r.t.}\quad\rho\qquad\quad&\Longrightarrow\qquad\quad [g']=[e]\\
\label{popodspodspods2}
\text{or}\qquad\quad g'\cdot \delta|_{(i',\tau]} \cpsim \delta|_{[\tau,i)}\quad\hspace{3pt}\text{w.r.t.}\quad\rho\qquad\quad&\Longrightarrow\qquad\quad [g']=[g].
\end{align}
The implication in the first line holds, because $(i',\tau]$ is free; and the implication in the second line follows from $g\cdot \delta|_{(i',\tau]} \psim \delta|_{[\tau,i)}$ w.r.t.\ $\mu$:
\begingroup
\setlength{\leftmargini}{12pt}
\begin{itemize}
\item[$\triangleright$]
If $g\cdot \delta|_{(j',\tau]} \psim_{\tau,\tau} \delta|_{[\tau,i)}$ holds for $i'\leq j'<\tau$, then the left side of \eqref{popodspodspods2} implies 
\begin{align*}
	g'\cdot \delta|_{(i',\tau]}\cpsim g\cdot \delta|_{(j',\tau]}\qquad\Longrightarrow\qquad g'\cdot \delta|_{(i',\tau]}\cpsim g\cdot \delta|_{(i',\tau]}\qquad\Longrightarrow\qquad [g']=[g].
\end{align*} 
\item[$\triangleright$]
If $g\cdot \delta|_{(i',\tau]} \psim_{\tau,\tau}  \delta|_{[\tau,j)}$ holds for $\tau<j<i$, then the left side of \eqref{popodspodspods2} implies
\begin{align*}
g'^{-1}\cdot \delta|_{[\tau,i)}\cpsim g^{-1}\cdot \delta|_{[\tau,j)}\qquad&\Longrightarrow\qquad
	g'^{-1}\cdot \delta|_{[\tau,i)}\cpsim g^{-1}\cdot \delta|_{[\tau,i)}\\[-20pt]
	 &\Longrightarrow\qquad g'\cdot g^{-1}=:h\in G_\delta
	 \qquad\Longrightarrow\qquad g'=g\cdot \overbrace{(g^{-1}\cdot h\cdot g)}^{\displaystyle\stackrel{\eqref{stabiconji}}{\in}G_\delta},
\end{align*}
hence $[g']=[g]$.\hspace*{\fill}(for the right side observe $g\cdot \delta\cpsim \delta$)
\end{itemize}
\endgroup
\noindent
Now, since $[g]\neq [e]$ holds by definition, only one of the above cases can occur, i.e., we have either \eqref{popodspodspods11}  (hence $[g']=[e]$) or \eqref{popodspodspods2} (hence $[g']=[g]$):
\begingroup
\setlength{\leftmargini}{12pt}
\begin{itemize}
\item[$\bullet$]
Assume that \eqref{popodspodspods11} holds, hence $[g']=[e]$ and  $\delta|_{(i',\tau]} \cpsim \delta|_{(i',\tau]}$ w.r.t.\ $\rho$. It suffices to show $\rho=\id_J$, because then $\ovl{\rho}|_{(i',\tau]}=\id_{(i',\tau]}$ follows from  Corollary \ref{sdoppsdoods}. 
Assume thus that $\rho\neq \id_J$ holds: 
\begingroup
\setlength{\leftmarginii}{12pt}
\begin{itemize}
\item[$*$]
Lemma \ref{sshift} applied to $\gamma\equiv \delta|_{(i',\tau]}\equiv\gamma'$, $B\equiv (i',\tau]\equiv D'$, $\phi\equiv \id_{(i',\tau]}$, $\psi\equiv\rho$ shows that $\delta|_{(i',\tau]}$ is self-related, because we have  $J\subseteq B$ and $\phi|_J\neq \psi$. 
\item[$*$]
Then, for each $\epsilon >0$ with $(i',\tau+\epsilon)\subseteq  I$, Lemma \ref{sdffsd}.\ref{sdffsd2} applied to $\gamma\equiv \delta|_{(i',\tau+\epsilon)}$ and $D\equiv (i',\tau]$ yields $\delta|_{(i',\tau]}\cpsim \delta|_{(\tau,\tau+\epsilon)}$. Consequently, 
\begin{align*} 
g\cdot \delta|_{(i',\tau]}\psim \delta|_{[\tau,i)}\quad\text{implies}\quad g\cdot \delta|_{(i',\tau]}\cpsim   \delta|_{(i',\tau]} \quad\text{hence}\quad [g]=[e]\quad\text{which contradicts}\quad [g]\neq [e].
\end{align*} 
\end{itemize}
\endgroup
\item[$\bullet$]
Assume that \eqref{popodspodspods2} holds, hence $[g']=[g]$ and  $g\cdot \delta|_{(i',\tau]} \cpsim \delta|_{[\tau,i)}$ w.r.t.\ $\rho$. It suffices to show $\rho=\mu|_J$, because then $\ovl{\rho}|_{\dom[\mu]}=\mu$ follows from Corollary \ref{sdoppsdoods}.  
\begingroup
\setlength{\leftmarginii}{12pt}
\begin{itemize}
\item
	If $\dom[\mu]=(j',\tau]$ holds for some $i'<j'<\tau$, then we have $\im[\mu]=[\tau,i)$. 
\begingroup
\setlength{\leftmarginiii}{10pt}
\begin{itemize}	
\item	
	Lemma \ref{sshift} applied to $\gamma\equiv  \delta|_{(i',\tau]}$, $\gamma'\equiv g^{-1}\cdot\delta|_{[\tau,i)}$, $B\equiv\dom[\mu]$, $D\equiv (i',\tau]$, $D'\equiv [\tau,i)$, $\phi\equiv \mu$, $\psi\equiv \rho$ shows that $\delta|_{(i',\tau]}$ is self-related if 	
	  $J\nsubseteq \dom[\mu]$ holds, or if $J\subseteq \dom[\mu]$ and $\rho\neq \mu|_J$ holds.
	  \item 
	  Now, the same arguments as in the previous point show that   $\delta|_{(i',\tau]}$ cannot be self-related. Hence, we must have 
	  $J\subseteq \dom[\mu]$ as well as $\rho= \mu|_J$. 
	  \end{itemize}
	  \endgroup
\item
	If $\dom[\mu]=(i',\tau]$ holds, then $\dom[\mu^{-1}]=\im[\mu]=[\tau,j)$ holds for some $\tau<j\leq i$. 
\begingroup
\setlength{\leftmarginiii}{10pt}
\begin{itemize}	
\item
Lemma \ref{sshift} applied to $\gamma\equiv \delta|_{[\tau,i)}$, $\gamma'\equiv g\cdot \delta|_{(i',\tau]}$, $B\equiv \dom[\mu^{-1}]=[\tau,j)$, $D\equiv [\tau,i)$, $D'=(i',\tau]$, $\phi\equiv \mu^{-1}$, $\psi\equiv \rho^{-1}$ shows that $\delta|_{[\tau,i)}$ is self-related if 
	$J'\nsubseteq \dom[\mu^{-1}]$ holds, or if $J'\subseteq \dom[\mu^{-1}]$ and $\rho^{-1}\neq \mu^{-1}|_{J'}$ holds. 
\vspace{3pt}

	It thus suffices to show that $\delta|_{[\tau,i)}$ is not self-related, because then we necessarily have $J'\subseteq \dom[\mu^{-1}]$ and $\rho^{-1}=\mu^{-1}|_{J'}$, hence $\rho=\mu|_J$. 
\item	
	Assume now that $\delta|_{[\tau,i)}$ is self-related, and fix  	
 $\epsilon >0$ with $(\tau-\epsilon,i)\subseteq I$. Lemma \ref{sdffsd}.\ref{sdffsd2} applied to $\gamma\equiv \delta|_{(\tau-\epsilon,i)}$  and $D\equiv [\tau,i)$ shows   	
	$\delta|_{[\tau,i)} \cpsim \delta|_{(\tau-\epsilon,\tau)}$, hence $g\cdot \delta|_{[\tau,i)}\cpsim g\cdot \delta|_{(i',\tau]}$. Then, since $\dom[\mu]=(i',\tau]$ holds, we have that 
	\[
g\cdot \delta|_{(i',\tau]}\psim \delta|_{[\tau,i)}\:\:\:\text{w.r.t.}\:\:\:\mu\qquad\text{implies}\qquad g\cdot \delta|_{[\tau,i)}\cpsim   \delta|_{[\tau,i)} \qquad\text{hence}\qquad [g]=[e],
\]
which contradicts $[g]\neq [e]$.\qedhere
\end{itemize}
\endgroup
\end{itemize}
\endgroup
\end{itemize}
\endgroup
\end{proof}

\subsection{}
\label{appC3}
\begin{proof}[Proof of Lemma \ref{dsfdsffds}]
Let $(\{a_n\}_{n\in \cn},\{[g_n]\}_{n\in \cn})$ be an $A$-decomposition of $\gamma$. 
\begingroup
\setlength{\leftmargini}{12pt}
\begin{itemize}
\item
{\sf Faithfulness:} Assume that $g\cdot \gamma|_A\cpsim \gamma$ holds w.r.t.\ the analytic diffeomorphism $\rho\colon A\supseteq J\rightarrow J'\subseteq I$. By a) and b) in Lemma \ref{eoder}.\ref{eoder2}, we have $J'\cap \innt[A_n]\neq \emptyset$ for some $\cn_-\leq n\leq \cn_+$. Then, shrinking $J$ (hence $J'$) if necessary, we can assume that $J'\subseteq A_n$ holds, hence $g\cdot \gamma|_A\cpsim \gamma|_{A_n}$. 
\begingroup
\setlength{\leftmarginii}{12pt}
\begin{itemize}
\item
Then, $\im[\mu_n]=A_n$ implies $g_n\cdot \gamma|_A\cpsim g\cdot \gamma|_{A}$, hence $[g]=[g_n]$. 
\item
Assume that $J\nsubseteq \dom[\mu_n]$ holds, or that $J\subseteq \dom[\mu_n]$ and $\rho\neq \mu_n|_J$ holds: 
\begingroup
\setlength{\leftmarginiii}{10pt}
\begin{itemize}	
\item
Lemma \ref{sshift} applied to $\gamma\equiv  \gamma|_A$, $\gamma'\equiv g_n^{-1}\cdot\gamma|_{A_n}$, $B\equiv \dom[\mu]$, $D\equiv A$, $D'\equiv A_n$, $\phi\equiv \mu$, $\psi\equiv \rho$ shows that $\gamma|_A$ is self-related.
\vspace{3pt}
\item
Set $\tilde{D}:=(a_{-1},a_1]$ and fix $\epsilon>0$ with $\tilde{I}:=(a_{-1},a_1+\epsilon)\subseteq A\cup A_1$. The previous point implies that $\gamma|_{\tilde{D}}$ is self-related. Hence, Lemma \ref{sdffsd}.\ref{sdffsd2} applied to $\gamma\equiv \gamma|_{\tilde{I}}$, $I\equiv \tilde{I}$, $D\equiv \tilde{D}$  yields $\gamma|_{A}\cpsim  \gamma|_{A_1}$.      
\vspace{3pt}
\item
Then, $\im[\mu_1]=A_1$ implies $g_1\cdot \gamma|_A\cpsim  \gamma|_A$, hence $[g_1]=[e]$ as $A$ is free. This contradicts that $[g_1]\neq [e]$ holds by definition.
\end{itemize}
\endgroup
We thus have $J\subseteq \dom[\mu_n]$ and $\rho= \mu_n|_J$, hence $\ovl{\rho}|_{\dom[\mu_n]}=\mu_n$ by Corollary \ref{sdoppsdoods}. Then, the index $n\in \cn\cup\{0\}$ is necessarily unique, because $\mu_n\neq \mu_m$ holds for all $n\neq m\in \cn\cup\{0\}$.  
\end{itemize}
\endgroup
\item
{\sf Uniqueness:} Let $(\{a'_n\}_{n\in \cn'},\{[g'_n]\}_{n\in \cn'})$ be an $A$-decomposition of $\gamma$. 
\begingroup
\setlength{\leftmarginii}{10pt}
\begin{itemize}	
\item
Since $A_0=A=A_0'$ holds, we have $a_{\pm 1}=a'_{\pm 1}$. Hence, 
 $g_{\pm 1}'\cdot \gamma|_A\psim \gamma|_{A'_{\pm 1}}$ implies $g_{\pm 1}'\cdot \gamma|_A\cpsim \gamma|_{A_{\pm 1}}$.
\item 
Then, $[g_{\pm 1}']=[g_{\pm 1}]$ and $\mu'_{\pm 1}=\mu_{\pm 1}$ holds, because $(\{a_n\}_{n\in \cn},\{[g_n]\}_{n\in \cn})$ is faithful.
\item
Then, $A_{\pm 1}=A'_{\pm 1}$ holds by $\mu'_{\pm 1}=\mu_{\pm 1}$; so that  
we have the following implications: 
\begin{align*}
	\cn_-\leq -2\qquad\quad\Longrightarrow \qquad\quad &a_{-2}=a'_{-2}\qquad\stackrel{g_{-2}'\he\cdot\he \gamma|_A\he\psim\he \gamma_{A'_{-2}}}{\Longrightarrow}\qquad g_{-2}'\cdot \gamma|_A\cpsim \gamma|_{A_{-2}}\\
	\cn_+\geq \phantom{-}2\qquad\quad\Longrightarrow\qquad\quad &\hspace{6.5pt}a_{2}=a'_{2}\qquad\hspace{6.3pt}\hspace{5.5pt}\stackrel{g_{2}'\he\cdot\he \gamma|_A\he\psim\he \gamma_{A'_{2}}}{\Longrightarrow}\qquad\hspace{6.1pt}\hspace{5.3pt} g_{2}'\cdot \gamma|_A\cpsim \gamma|_{A_{2}}.
\end{align*}
\end{itemize}
\endgroup
We can now apply the above arguments inductively, to conclude that both decompositions as well as the corresponding diffeomorphisms coincide.\qedhere
\end{itemize}
\endgroup
\end{proof}

\subsection{}
\label{appC4}
\begin{proof}[Proof of Lemma \ref{freemax}]
\begingroup
\setlength{\leftmargini}{12pt}
\begin{enumerate}
\item
For each $g'\in G$, we have the implication
$$
\begin{array}{lll}
g' \cdot \gamma'|_{D'}\cpsim \gamma'|_{D'}\qquad &\Longrightarrow\qquad  g'\cdot \gamma'\cp \rho\cpsim  \gamma'\cp \rho 
\qquad
&\Longrightarrow\qquad  g'\cdot (g\cdot \gamma|_D)\cpsim  g\cdot \gamma|_D\\[4pt]
&\Longrightarrow\qquad(g^{-1}\cdot g'\cdot g)\in G_{\gamma|_D}
= G_\gamma\qquad
&\Longrightarrow\qquad g'\cdot (g\cdot \gamma)\cp \rho^{-1}=g\cdot \gamma\cp \rho^{-1}\\[4pt]
&&\Longrightarrow\qquad g'\cdot \gamma'|_{D'} = \gamma'|_{D'}\\[4pt]
&&\Longrightarrow\qquad g'\in G_{\gamma'|_{D'}}=G_{\gamma'}.
\end{array}
$$
\vspace{-10pt}
\item
Let $J\subseteq I\cap\dom[\ovl{\rho}]$ be an open interval with $D\subseteq J$.   
Since $I'\supseteq D'=\ovl{\rho}(D)=\im[\rho]$ is compact, shrinking $J$ if necessary, we can assume that $J':=\ovl{\rho}(J)\subseteq I'$ holds. Lemma \ref{corgleich} shows $\gamma|_J=\gamma'\cp \ovl{\rho}|_J$, hence $\gamma'|_{J'}=\gamma \cp \rho'$ for $\rho':=\ovl{\rho}^{-1}|_{J'}$. 
Now, if $D'$ is not maximal, then there exists a free (w.r.t.\ $\gamma'$) interval $D''\subseteq I'$ with $D'\subset D''\subseteq J'$. Then, $\rho'(D'')\subset D$ is free (w.r.t.\ $\gamma$) by Part \ref{eoder1}, which contradicts that $D$  is maximal.\qedhere
\end{enumerate}
\endgroup
\end{proof}

\subsection{}
\label{appC5}
\begin{proof}[Proof of Lemma \ref{lemmfreeneig}]
It suffices to prove \eqref{fre1}, because \eqref{fre2}  
then follows from \eqref{fre1} and Lemma \ref{freemax}.\ref{freemax1} (just by replacing $\gamma$ by $\gamma \cp \inv$  for $\inv\colon [t',t]\ni s\mapsto t+t'-s\in [t',t]$).   
Let $I\subseteq (a',t)$ be an open interval with $\tau:= a\in I$   such that $\gamma\equiv \delta|_I$ is an embedding; and 
assume that \eqref{fre1} is wrong, i.e., that 
$[\tau,k]$ is not free for each $\tau< k\leq d:=t$.  Then, Lemma \ref{seque}.\ref{seque2} shows that $\gamma$ is continuously generated at $\tau$. We  choose $\{(g_n,K_n,J_n)\}_{n\in \NN}$ as in Lemma \ref{podspods}, and argue as follows: 
\begingroup
\setlength{\leftmargini}{11pt}
\begin{itemize}
\item[$\triangleright$]
Assume there exists some $n\in \NN$, such that $g_n$ shifts $\tau= a$ to the left. Then, $g_n\cdot \gamma(J)\subseteq \gamma((a',a))$ holds for some $J\subseteq J_n\cap (-\infty,a]$; hence, we have $g_n\cdot \gamma|_{[a',a]}\cpsim \gamma|_{[a',a]}$. Since $[a',a]$ is free, this implies $g_n\in G_\gamma\subseteq G_{\gamma(\tau)}$; which contradicts $g_n\notin G_{\gamma(\tau)}$.  
Consequently, each $g_n$ shifts $\tau$ to the right.
\item[$\triangleright$]
We apply Corollary \ref{lemma:simpli}.\ref{lemma:simpli2} to fix a compact neighbourhood $[\ell',\ell]\equiv L\subseteq I$ of $\tau$  as well as some $n_0\in \NN$, with $g_n\cdot \gamma(L)=\gamma(L_n)\subseteq \gamma(I)$ for all $n\geq n_0$. 
We write $L_n\equiv [\ell'_n,\ell_n]\subseteq I$ for each $n\geq n_0$, and conclude from positivity of $g_n$ that
\begin{align*}
	g_n\cdot \gamma([\ell',\tau])=\gamma([\ell'_n,\tau_n])\qquad\text{holds with}\qquad \tau<\tau_n,\quad g_n\cdot \gamma(\tau)=\gamma(\tau_n),\quad g_n\cdot \gamma(\ell')=\gamma(\ell'_n).
\end{align*}
\begingroup
\setlength{\leftmarginii}{12pt}
\begin{itemize}
\item
If $\ell_n'<\tau$ holds for some $n\geq n_0$, then we have $g_n\cdot \gamma|_{[a',a]}\cpsim \gamma|_{[a',a]}$. Hence, we obtain a contradiction as in the previous point. 
\item
If $\ell_n'\geq \tau$ holds for all $n\geq n_0$, then   
$g_n\cdot \gamma([\ell',\tau])\subseteq \gamma([\tau,\tau_n])$ holds for all $n\geq n_0$. This contradicts that $\wm$ is sated, as \eqref{fdfdgfgf} implies $\lim_n \tau_n=\tau$.\footnote{Observe that $\gamma$ is an embedding, and that $\lim_n \gamma(\tau_n)=\lim_n g_n\cdot \gamma(\tau)=\gamma(\tau)$ holds by \eqref{fdfdgfgf}.}\qedhere
\end{itemize}
\endgroup
\end{itemize}
\endgroup
\end{proof}

\subsection{}
\label{appC6}
\begin{proof}[Proof of Lemma \ref{asaqwqw}]
We only prove Part \ref{asaqwqw1} (Part \ref{asaqwqw2} follows analogously). 
Let thus $[a',a]$ be maximal w.r.t.\ $\gamma= \delta|_{[t',t]}$ with $t=a'$ and $a<t\in I$ and assume that \eqref{gl11} holds for $\gamma$ and $[e]\neq [g]\in G\slash G_\gamma = G\slash G_\delta$. We conclude the following:
\begingroup
\setlength{\leftmargini}{12pt}
\begin{itemize}
\item 
For each $0<\epsilon\leq s'$, we have by \eqref{gl11}
\begin{align*}
	g\cdot \delta|_{[a',a]}\cpsim \delta|_{[a,a+\epsilon]}\qquad\text{hence}\qquad
	g\cdot \delta|_{[a',a+\epsilon]}\cpsim \delta|_{[a',a+\epsilon]}
	\qquad\text{with}\qquad g\in G\setminus G_\delta;
\end{align*}
so that $[a',a+\epsilon]$ is not free w.r.t.\ $\delta$.
\item
By Corollary \ref{dfdsasasasassa}, both $g\cdot \delta$ and $\delta$ are locally embeddings (around $a'$ and $a$, respectively). Thus, Lemma \ref{lemma:BasicAnalyt1} together with \eqref{gl11}  implies\footnote{Apply Lemma \ref{lemma:BasicAnalyt1} to $\gamma\equiv g\cdot\delta|_{(a'-r',a'+r')}$ and $\gamma'\equiv  \delta|_{(a-r,a+r)}$, for $r',r>0$ suitably small.} that for $\epsilon >0$ suitably small,  we have 
$$
g\cdot \delta|_{[a'-\epsilon,a']}\cpsim \delta|_{[a-\epsilon,a]}\qquad\text{hence}\qquad
g\cdot \delta|_{[a'-\epsilon,a]}\cpsim \delta|_{[a'-\epsilon,a]}
\qquad\text{with}\qquad g\in G\setminus G_\delta;
$$
so that $[a'-\epsilon,a]$ is not free w.r.t.\ $\delta$.
\end{itemize}
\endgroup
\noindent
This shows that $[a,a']$ is maximal w.r.t.\ $\delta$.
\end{proof}

\subsection{}
\label{appC7}
\begin{proof}[Proof of the statements made in Remark \ref{fdfdsasasaasxafdfd}]

\noindent
\begingroup
\setlength{\leftmargini}{16pt}
{
\renewcommand{\theenumi}{{\rm\alph{enumi}})} 
\renewcommand{\labelenumi}{\theenumi}
\begin{enumerate}
\item
\label{appC7a}
Since $h\cdot \gamma|_{[a-\epsilon,a]}\psim_{a,a}\gamma|_{[a,a+\epsilon']}$ holds, 
Lemma \ref{lemma:BasicAnalyt1}  yields 
$$h\cdot \gamma|_{[a,a+\delta]}\psim_{a,a}\gamma|_{[a-\delta',a]}\qquad\text{for certain}\qquad  0<\delta\leq \epsilon',\: 0<\delta'\leq \epsilon.$$ 
Consequently, we have
\begin{align*}
 	h\cdot \gamma|_{[a,a+\delta]}\psim_{a,a}\gamma|_{[a-\delta',a]}\psim_{a,a}h^{-1}\cdot \gamma|_{[a,a+\delta'']}\qquad\text{for certain}\qquad 0<\delta''\leq\epsilon'.
\end{align*}
From this, we obtain
\vspace{-4pt}
$$h^2\cdot \gamma|_{[a,a+\epsilon']} \cpsim \gamma|_{[a,a+\epsilon']}\quad\stackrel{[a,a+\epsilon']_{\phantom{O}}\text{free}}{\Longrightarrow}\quad h^2=:q\in G_\gamma
\quad\:\:\Longrightarrow\quad\:\: h=h^{-1}\cdot q \quad\:\:\Longrightarrow\quad\:\:
[h]=[h^{-1}].
$$
\item
\label{appC7b}
Since $h\cdot \gamma|_{A_-} \psim_{a,a}\gamma|_{A_+}$ holds, Lemma \ref{lemma:BasicAnalyt1} yields 
\begin{align}
\label{dskjdslkjlkdslkdslkdsds98ds98ds98dsdsdsdsds}
h\cdot \gamma|_{[a_- -\epsilon,a_-]} \psim_{a_-,a_+}\gamma|_{[a_+,a_++\epsilon']}\qquad\text{for certain}\qquad  \epsilon,\epsilon'>0.
\end{align}
Moreover, \eqref{dskjdskjdsiudsiudsiudsiudsiuds87878787iuiu1} together with \eqref{dskjdskjdsiudsiudsiudsiudsiuds87878787iuiu2} yields 
\begin{align}
\label{kjdskjdsdsnmsdnmdnmds98d98ds98sd98d98ds98ds98dsds}
	(h_-\cdot h)\cdot \gamma|_{A_+} \psim_{a_+,a_-} \gamma|_{A_{--}}\qquad\:\:\text{as well as}\qquad\:\: h_+ \cdot \gamma|_{A_+}\psim_{a_+,a_+} \gamma|_{A_{++}}.
\end{align}
Combining \eqref{dskjdslkjlkdslkdslkdsds98ds98ds98dsdsdsdsds} with \eqref{kjdskjdsdsnmsdnmdnmds98d98ds98sd98d98ds98ds98dsds}, we obtain
\begin{align}
\label{kjdkjdfkjdfkjkjdfdfdfdfdfd45545}
	h_+ \cdot \gamma|_{A_{+}}\cpsim (h\cdot h_-\cdot h)\cdot \gamma|_{A_{+}}
\qquad\:\:\stackrel{A_+\:\text{free}}{\Longrightarrow}\qquad\:\: 
	[h_+]=[h\cdot h_-\cdot h],
\end{align}
which proves the one part of the claim. The other part (i.e.\ $[h_-]=[h\cdot h_+\cdot h]$) follows analogously; and alternatively as follows: 
\begingroup
\setlength{\leftmarginii}{12pt}
\begin{itemize}
\item[$-$]
We have $h_+\cdot q= h\cdot h_-\cdot h$ for some $q\in G_\gamma$  by \eqref{kjdkjdfkjdfkjkjdfdfdfdfdfd45545}, hence $h_-=h^{-1}\cdot (h_+\cdot q)\cdot h^{-1}$.
\item[$-$]
We have $h^{-1}=h\cdot q'$ for some $q'\in G_\gamma$, as $[h]=[h^{-1}]$ holds by \ref{fdfdsasasaasxafdfda}. 
\item[$-$]
We have $(h_+\cdot h) \in O_\gamma$, as $(h_+\cdot h) \cdot \gamma|_{A_-}\psim_{a_-,a_+} \gamma|_{A_{++}}$ holds by \eqref{dskjdskjdsiudsiudsiudsiudsiuds87878787iuiu1} and  \eqref{dskjdskjdsiudsiudsiudsiudsiuds87878787iuiu2}. 
\end{itemize}
\endgroup
We obtain:
\vspace{-23pt}
\begin{align*}
	[h_-]&\stackrel{\phantom{\eqref{stabiconji}}}{=} [(h\cdot q')\cdot (h_+\cdot q)\cdot (h\cdot q')]
	=[(h\cdot q')\cdot (h_+\cdot h)\cdot (h^{-1}\cdot  q\cdot h)]\\
	&\stackrel{\eqref{stabiconji}}{=}[(h\cdot q')\cdot (h_+\cdot h)]=[h\cdot(h_+\cdot h) \cdot ((h_+\cdot h)^{-1} \cdot q'\cdot (h_+\cdot h))]\hspace{20pt} \\
	&\stackrel{\eqref{stabiconji}}{=}[h\cdot h_+\cdot h].\qedhere 
\end{align*}
\qedhere
\end{enumerate}}
\endgroup
\end{proof}

\subsection{}
\label{appC8}
\begin{proof}[Proof of Equation \eqref{kcfdjksdfkjds}]
If $\cn_+\geq 2$ holds (the case $\cn_-\leq -2$ follows analogously), then we have 
\begin{align*}
	\hspace{-25pt}g_1\cdot \gamma|_{A_0}\psim_- \gamma|_{A_1}\quad&\stackrel{\text{Lemma } \ref{lemma:BasicAnalyt1}}{\Longrightarrow}\quad \hspace{10.1pt}g_1\cdot \gamma|_{A_{-1}}\cpsim \gamma|_{A_2} \hspace{32.6pt}\qquad\stackrel{\eqref{safgrfgtr}}{\Longrightarrow}\qquad  \gamma|_{A_{-1}}\cpsim g_1\cdot\gamma|_{A_2}\\[-4pt]
	&\hspace{8pt}\Longrightarrow\qquad\hspace{2pt}  g_{-1}\cdot \gamma|_{A_{-1}} \cpsim (g_{-1}\cdot g_1)\cdot \gamma\qquad\hspace{-1.2pt}\stackrel{(*)}{\Longrightarrow}\qquad\hspace{18.7pt} \gamma\cpsim (g_{-1}\cdot g_1)\cdot \gamma.
\end{align*}
For the implication $(*)$ observe that $g_{-1}\cdot \gamma|_A\psim_- \gamma|_{A_{-1}}$ together with \eqref{safgrfgtr}  implies $\gamma|_A\psim_- g_{-1}\cdot\gamma|_{A_{-1}}$.
\end{proof}

\section{Appendix to Sect.\ \ref{ghdhgg}}

\subsection{}
\label{appD1}
\begin{proof}[Proof of Lemma \ref{dffdfdfd}]
Let $(U_0,\psi_0)$ be a fixed chart of $S$. Since $S$ is connected, for each fixed $y\in S$, there exist charts $(U_1,\psi_1),\dots,(U_n,\psi_n)$ with $y\in U_n$ such that $U_{i}\cap U_{i+1}\neq \emptyset$ holds for $i=0,\dots,n-1$.\footnote{Since domains of charts ore open by convention, the  set $C$ of all $y\in S$, for which the claim holds, is non-empty  open. Now, $C$ is additionally closed; because, for $y'\in S\backslash C$, the domain of each chart around $y'$ must completely be contained in $S\backslash C$.} 
\begingroup
\setlength{\leftmargini}{12pt}
\begin{itemize}
\item[$\triangleright$]
For a fixed chart $(U,\psi)$, there thus exist charts $(U_1,\psi_1),\dots,(U_{n-1},\psi_{n-1})$ such that $U_{i}\cap U_{i+1}\neq \emptyset$ holds for $i=0,\dots,n-1$, with $(U_n,\psi_n)= (U,\psi)$.
\item[$\triangleright$]
Then, for each $0 \leq i\leq n-1$, we have
\vspace{-4pt}
\begin{align*}
U_{i}\cap U_{i+1}\neq \emptyset\qquad\quad\Longrightarrow \qquad\quad \gamma_{\psi_{i}}\cpsim \gamma_{\psi_{i+1}} \qquad\stackrel{\text{Corollary } \ref{fdsfds7}}{\Longrightarrow} \qquad  G_{\gamma_{\psi_{i}}}=G_{\gamma_{\psi_{i+1}}}.
\end{align*}
This implies $G_{\gamma_\psi}=G_{\gamma_{\psi_0}}$, from which the claim is clear.\qedhere
\end{itemize}
\endgroup
\end{proof}

\subsection{}
\label{appD2}
\begin{proof}[Proof of Lemma \ref{oisdkdslkdslkdskjdskjsdkjds98ds98ds98dsdsdsds}]

\noindent
\begingroup
\setlength{\leftmargini}{12pt}
\begin{itemize}
\item
Since $S$ is second countable, there exist charts $(U_n,\psi_n)$ for $n\geq 1$, with $S=\bigcup_{n\in \NN}U_n$.\footnote{Let $\mathsf{Q}\equiv \{Q_n\}_{n \in \NN}$ be a countable base of the topology of $S$, let $\{(U_\iota,\psi_\iota)\}_{\iota\in I}$  be a collection of charts with $S=\bigcup_{\iota\in I}U_\iota$, and set $Z:=\{n\in \NN \:|\: \exists\: \iota(n)\in I\colon Q_n\subseteq U_{\iota(n)}\}$. 
Since each $U_\iota$ is the union of Elements of $\mathsf{Q}$, we    have  
 $S=\bigcup_{\iota\in I}U_\iota=\bigcup_{n\in Z} Q_n\subseteq \bigcup_{n\in Z} U_{\iota(n)}$. 
The claim now follows if we fix $\eta\colon \NN_{\geq 1}\rightarrow Z$ surjective,   
and set $(U_n,\psi_n):= (U_{\iota(\eta(n))},\psi_{\iota(\eta(n))})$ for  all $n\geq 1$.}   
\item  
  Since $S$ is connected, for each $n\in \NN$, there exist charts (see proof of Lemma \ref{dffdfdfd})  \hspace*{\fill}($m(n)\geq 2$)
\begin{align*}  
  (U[n]_1,\psi[n]_1),\dots,(U[n]_{m(n)},\psi[n]_{m(n)})\quad\:\:&\text{with}\quad\:\: U[n]_i\cap U[n]_{i+1} \neq \emptyset\quad\:\:\text{for}\quad\:\: i=1,\dots,m(n)-1 
  \\[1pt]
  \text{as}&\text{ well as}\\[1pt]
   (U[n]_1,\psi[n]_1)= (U_n,\psi_n)\quad\:&\:\:\:\wedge \qquad (U[n]_{m(n)},\psi[n]_{m(n)})= (U_{n+1},\psi_{n+1}).
\end{align*} 
Then, the claim  holds for  $\{(\tilde{U}_{\ell},\tilde{\psi}_{\ell})\}_{\ell\geq 1}$ defined by  
\[
	\textstyle (\tilde{U}_{\ell},\tilde{\psi}_\ell):=(U[n]_i,\psi[n]_i)\quad\text{for}\quad\ell\geq 1\quad\text{if}\quad \ell= i+ \sum_{p=0}^{n-1} m(p)\quad\text{holds for}\quad n\in \NN,\: 1\leq i \leq m(n).\qedhere
\]
\end{itemize}
\endgroup

\end{proof}

\subsection{}
\label{appD3}
\begin{proof}[Proof of Proposition \ref{dfgffhfh}]
We observe the following:
\begingroup
\setlength{\leftmargini}{12pt}
\begin{itemize}
\item
The second point in \ref{dfgffhfh1} is clear from the first point in \ref{dfgffhfh1} as well as from Lemma \ref{dffdfdfd} and Proposition \ref{classsiilie}.
\item
Since $\UE$ is not homeomorphic to an interval, 
only one of the situations in \ref{dfgffhfh2} can hold.
\item
The one direction in the uniqueness (last) statement in \ref{dfgffhfh2} is clear from Lemma \ref{kljklvjkvjklvjklxcv} (with $\tau\equiv0$ there). For the other direction, let $\q\in \mg$ be such that 
	$\tilde{\Omega}\colon \UE\ni \e^{\I \phi}\mapsto \iota^{-1}(\exp(\phi\cdot \q)\cdot x)\in S$ 
is an analytic diffeomorphism (the case \eqref{ghhgghqqqq2} is treated analogously). Then,    
\begin{align}
\label{skjdskjdskjdsnmdsnmdsnmds08ds98ds98ds}
	 \gamma_{\q}^x=\iota\cp\tilde{\Omega}\cp \alpha\qquad\text{holds for}\qquad \alpha\colon \RR\ni \phi \mapsto \e^{\I\phi}\in \UE,
\end{align} 
hence $\gamma_\q^x$ is immersive (chain rule) 
with $\pii(x,\q)=2\pi$. 
\vspace{-5pt}
\begingroup
\setlength{\leftmarginii}{12pt}
\begin{itemize}
\item
Let $(U,\psi)$ be a fixed chart centred at $z:=\iota^{-1}(x)$, hence 
$(\tilde{\Omega}\cp\alpha)(0)=z=\psi^{-1}(0)$.
\item
We fix an open interval $0\in \tilde{J}\subseteq (-\pi,\pi)$ with $z\in W:=(\tilde{\Omega}\cp\alpha)(\tilde{J})\subseteq U$, and set  
 $J:=\psi(W)\subseteq D_\psi$.  
\end{itemize}
\endgroup 
By Corollary \ref{dfdsasasasassa}, we can 
shrink $\tilde{J}$ around $0$ such that $\gamma_\q^x|_{\tilde{J}}$ and $\gamma_\psi|_{J}$ are embeddings (Corollary \ref{dfdsasasasassa}). Then,
\begin{align}
\label{lkdslkdslkdskjdsdsjksjshjdszusdzusd76ds87ds87d}
\begin{split}
\gamma_\q^x(\tilde{J})\stackrel{\eqref{skjdskjdskjdsnmdsnmdsnmds08ds98ds98ds}}{=}\iota(W)=\gamma_\psi(J)
\qquad\Longrightarrow\qquad\gamma_\q^x\cpsim \gamma_\psi
&\quad\:\:\stackrel{\text{Corollary }\ref{fdsfds7}}{\Longrightarrow}\quad\:\: G_{\gamma_\q^x}=G_{\gamma_\psi}\stackrel{\text{Lemma }\ref{dffdfdfd}}{=}G_S\hspace{20pt}\\[-6pt]
&\quad\:\:\stackrel{\phantom{\text{Corollary }\ref{fdsfds7}}}{\Longrightarrow}\quad\:\:\:\he \mg_{\gamma_\q^x}=\mg_S.
\end{split}   
\end{align}
\vspace{-12pt}

\noindent
Since  $\gamma_\psi$ is exponential w.r.t.\ $(x,\g)$ by  \ref{dfgffhfh1}, the relation $\gamma_\q^x\cpsim \gamma_\psi$ additionally implies
\begin{align*}
\gamma_\q^x\cpsim \gag\quad\:\:&\stackrel{\text{Lemma }\ref{ldldsldldsalldsakdshfdsf}}{\Longrightarrow}\quad\:\: 
\q\in \RR_{\neq 0}\cdot \g + \mg_{\gag} 
\quad\:\: \stackrel{\pii(x,\q)\:=\: 2\pi}{\Longrightarrow}\quad\:\: \q\in \g + \mg_{\gamma_\g^x} \quad\:\:\: \stackrel{\eqref{lkdslkdslkdskjdsdsjksjshjdszusdzusd76ds87ds87d}}{\Longrightarrow}\quad\:\:\: \q\in \g + \mg_{S}
\end{align*}
whereby we have additionally used Lemma \ref{kljklvjkvjklvjklxcv} (with $\tau\equiv 0$) and injectivity of $\iota$ in the second  implication.
\end{itemize}
\endgroup
\noindent 
It thus remains to show the first point in \ref{dfgffhfh1},  as well as the first statement in \ref{dfgffhfh2}:  
\vspace{6pt}

\noindent
Since $(S,\iota)$ is exponential, there exists a chart $(U_0,\psi_0)$ of $S$ with  
 $\gamma_{\psi_0}=\gag\cp \rho_{\psi_0}$ for some $x\in M$, $\g\in \mg\backslash \mg_x$ (observe that $\gamma_\psi$ is immersive), and an analytic diffeomorphism $\rho_{\psi_0}\colon D_{\psi_0}\rightarrow D'_0\subseteq \RR$.  
\begingroup
\setlength{\leftmargini}{12pt}
\begin{itemize}
\item
If $\pii(x,\g)<\infty$ holds, we can assume $\pii(x,\g)=2\pi$, just by rescaling $\g$ and $\rho_0$ if necessary.
\item
We can assume $x\in \im[\iota]$; just by fixing $t\in D_{\psi_0}$, replacing $x$ by $\exp(\rho_{\psi_0}(t))\cdot x$, and then replacing $\rho_{\psi_0}$ by $\rho_{\psi_0}-\rho_{\psi_0}(t)$. 
\end{itemize}
\endgroup
\noindent
It suffices to prove the following statement:\hspace*{\fill}(proof below)
\vspace{6pt}

\noindent
{\bf Claim 1:}     
There exist charts $(U_n,\psi_n)$ for $n\in \NN$ with $S= \bigcup_{n\in \NN} U_n$, an interval $D\subseteq \RR$ with $\gag(D)=\iota(S)$, and analytic diffeomorphisms $\rho_n\colon  D_{\psi_n} \rightarrow D'_n \subseteq D$ for $n\in \NN$, such that:  
\begingroup
\setlength{\leftmargini}{16pt}
{
\renewcommand{\theenumi}{\sf\alph{enumi})} 
\renewcommand{\labelenumi}{\theenumi}
\begin{enumerate}
\item
\label{lkdslkds1}
\hspace{10.2pt}\:$\gamma_{\psi_n}=\gag\cp \rho_n$\hspace{28.4pt}\: holds for each $n\in \NN$,
\item
\label{lkdslkds3}
\hspace{17.5pt}\:$D=\bigcup_{n\in \NN} D'_n$\hspace{19.6pt}\: holds,
\item
\label{lkdslkds2}
\hspace{0pt}\:$\innt[D]=\bigcup_{n\in \NN} \innt[D'_n]$\:\: holds.
\end{enumerate}}
\endgroup
\noindent  
In fact, assume that  
Claim 1 holds true:
\begingroup
\setlength{\leftmargini}{12pt}
\begin{itemize}
\item 
Proof of the first point in  \ref{dfgffhfh1}:
   
Let $(U,\psi)$ be a chart of $S$. Since $S= \bigcup_{n\in \NN} U_n$ holds, there exists some $n\in \NN$ with 
\begin{align*}
	U\cap U_n\neq \emptyset\quad\:\: \Longrightarrow\quad\:\: \gamma_\psi\cpsim \gamma_{\psi_n} \quad\:\: \stackrel{\ref{lkdslkds1}}{\Longrightarrow}\quad\:\: \gamma_\psi\cpsim \gag
	\quad\:\: \stackrel{\text{Lemma } \ref{vd}}{\Longrightarrow}\quad\:\:
	\gamma_\psi\:\:\:\text{is exponential w.r.t.}\:\:\: (x,\g).
\end{align*}
\item
Proof of the first statement in \ref{dfgffhfh2}:
\vspace{-3pt}
\begingroup
\setlength{\leftmarginii}{12pt}
\begin{itemize}
\item[$*$] 
If $\gag|_D$ is injective, then $\Omega\colon D\rightarrow S$ defined by \eqref{ghhgghqqqq2} is bijective, as $\gag(D)=\iota(S)$ holds. Moreover, $\Omega$ is an analytic diffeomorphism, because 
\vspace{-3pt}
\begin{align*}
\Omega^{-1}\cp \psi_n^{-1}=\big((\gag|_D)^{-1}\cp \iota\big)\cp \psi_n^{-1}= (\gag|_D)^{-1}\cp\gamma_{\psi_n}\stackrel{\ref{lkdslkds1}}{=}\rho_n\qquad\quad\forall\: n\in \NN
\end{align*}
holds with $\bigcup_{n\in \NN}\im[\rho_n]=D$ by \ref{lkdslkds3}. 
\vspace{2pt}
\item[$*$]
If $\gag|_D$ is not injective, then $\pii(x,\g)=2\pi$ holds due to our above modifications (Lemma \ref{lemma:expeig}). Hence, $\Omega$ given by  \eqref{ghhgghqqqq1} is well defined and bijective. 
\begingroup
\setlength{\leftmarginiii}{12pt}
\begin{itemize}
\item[$\triangleright$]
There exist $\tau\in \RR$ and $\epsilon>0$ with $\tau+[-2\epsilon ,2\pi +2\epsilon]\subseteq D$:
\vspace{4pt}

\noindent
	In fact, since $\gag|_D$ is not injective, there exists $\tau \in D$ with $\tau+[0,2\pi]\subseteq D$. 
	Hence, if the claim is wrong, we have $D=\tau+[0,2\pi]$. Now, \ref{lkdslkds3} provides $m,n\in \NN$, with
	 $$\tau_-:=\tau\in D'_{m}\qquad\text{and}\qquad \tau_+:=\tau + 2\pi\in D'_{n}\qquad\text{hence}\qquad D'_{m},D'_n\subseteq D=[\tau_-,\tau_+].$$
	 Since $\im[\psi_m]=D_{\psi_m}=\rho_m^{-1}(D'_m)$ and $\im[\psi_n]=D_{\psi_n}=\rho_n^{-1}(D'_n)$ are open in $(-\infty,0]$, we must have 
	 	 \begin{align*}
	  D'_{m}=[\tau_-,r_-)\quad \text{for some} \quad \tau_-<r_-\in D\qquad\quad&\wedge\qquad\quad 
	  D'_{n}=(r_+,\tau_+]\quad\text{for some}\quad D\ni r_+<\tau_+\\[2pt]
	  &\text{i.e.}\\[-14pt]
	 D_{\psi_m}=(\rho_m^{-1}(r_-),\he\overbrace{\rho_m^{-1}(\tau_-)}^{=\:0}]\qquad\quad&\wedge\qquad\quad
	  D_{\psi_n}=(\rho_n^{-1}(r_+),\he\overbrace{\rho_n^{-1}(\tau_+)}^{=\: 0}].
	 \end{align*}
	Thus,  
	$(U_m,\psi_m)$ and $(U_n,\psi_n)$ are boundary charts around
	\vspace{-3pt}
	\begin{align}
	\label{lkdslkskdsds9898ds98dsdsdsdscxcxcx}
	\begin{split}
		z:=\psi_m^{-1}\big(\rho_{m}^{-1}(\tau_-)\big)=(\iota^{-1}\cp \gamma_{\psi_m})\big(\rho_{m}^{-1}(\tau_-)\big)
		&\stackrel{\ref{lkdslkds1}}{=}
		(\iota^{-1}\cp \gag)(\tau_-)\stackrel{\pii(x,\g)\:=\:2\pi}{=}(\iota^{-1}\cp \gag)(\tau_+)\\
		&\stackrel{\ref{lkdslkds1}}{=}
		(\iota^{-1}\cp \gamma_{\psi_n})\big(\rho_{n}^{-1}(\tau_+)\big)
		=\psi_n^{-1}\big(\rho_{n}^{-1}(\tau_+)\big).
	\end{split}	
	\end{align}
	Now, shrinking both charts around $z$ (decreasing $r_-$ and increasing $r_+$) if necessary, we can assume $D'_m\cap D'_n =\emptyset$. 
Since $D'_m, D'_n\subseteq D$ holds, and since we have 
\vspace{-4pt}
	$$\qquad\gag(x)=\gag(y)\:\:\:\text{for}\:\:\: x,y\in D\qquad\stackrel{\pii(x,\g)\:=\:2\pi}{\Longleftrightarrow}\qquad x,y\in \{\tau_-,\tau_+\}
	$$
	with $\iota^{-1}\cp \gag(\tau_\pm)=z$ by  \eqref{lkdslkskdsds9898ds98dsdsdsdscxcxcx},  
	the intersection
	\vspace{-4pt}
	$$
U_m\cap U_n=	\psi_m^{-1}(D_{\psi_m})\cap \psi_n^{-1}(D_{\psi_n})\stackrel{\ref{lkdslkds1}}{=}\big(\big(\iota^{-1}\cp \gag\big)(D'_m)\big)\cap \big(\big(\iota^{-1}\cp \gag\big)(D'_n)\big)=\{z\}
	$$ 
	is singleton; which contradicts that $(U_m,\psi_m)$ and $(U_n,\psi_n)$ are boundary charts around $z$.
\item[$\triangleright$]
The previous point together with \ref{lkdslkds2} shows that $\tau + [-\epsilon,2\pi+\epsilon] \subseteq \innt[D]$ is covered by the open intervals $J_n':= \innt[D_n']$ for $n\in \NN$. In addition to that, we have  
$$
S=\iota^{-1}\big(\gag(D)\big)\stackrel{\pii(x,\g)\:=\:2\pi}{=}\iota^{-1}(\gag(\tau + [0,2\pi))\subseteq  \iota^{-1}(\gag(\tau + [-\epsilon,2\pi+\epsilon])).
$$ 
Hence, $S= \bigcup_{n\in \NN}U'_n$ holds for
	\vspace{-4pt}
\begin{align*}
		U'_n:= \psi_n^{-1}(J_{\psi_n})\stackrel{\ref{lkdslkds1}}{=} (\iota^{-1}\cp\gag)(J_n')\quad\:\:\text{with}\quad\:\: J_{\psi_n}:= \rho_n^{-1}(J_n')\subseteq D_{\psi_n}\qquad\quad\forall\: n\in \NN.
\end{align*}
(In particular, $S$ is a manifold without boundary, as each $J_{\psi_n}$ is an open interval.)
\end{itemize}
\endgroup 
Then, $\Omega$ is an analytic diffeomorphism, because 
\vspace{-6pt}
\begin{align*}
	\Omega^{-1}\cp \psi_n^{-1}=  \Omega^{-1}\cp (\iota^{-1}\cp  \gamma_{\psi_n})\stackrel{\ref{lkdslkds1}}{=}  \Omega^{-1}\cp ((\iota^{-1} \cp \gag)\cp \rho_n)=\e^{\I\rho_n}
\end{align*}  
holds with\footnote{Hence, if $(V,\chi)$ is a chart around $z\in \UE$, then there exists a chart $(U,\psi)$ of $S$, with $D_\psi$ an open interval, such that $\Omega^{-1}\cp\psi^{-1}(U)\subseteq V$ holds.} $\tau+[-\epsilon, 2\pi+\epsilon]\subseteq \bigcup_{n\in \NN} J_n'$.
\end{itemize}
\endgroup  
\end{itemize}
\endgroup
\vspace{-4pt}

\noindent
It remains to prove Claim 1:
\vspace{4pt}

\noindent 
By Lemma \ref{oisdkdslkdslkdskjdskjsdkjds98ds98ds98dsdsdsds},  
there exist charts $(U_n,\psi_n)$ for $n\geq 1$ with $S=\bigcup_{n\in \NN}U_n$, such that $U_n\cap U_{n+1}\neq \emptyset$ holds for all $n\in \NN$. We now prove by induction that there exist intervals 
\begin{align}
\label{lovcxpoivcxpiiovcxvxc}
A_0\subseteq A_1\subseteq A_2\subseteq A_3\subseteq  {\dots}
\end{align}
as well as analytic diffeomorphisms $\rho_n\colon  D_{\psi_n} \rightarrow D'_n \subseteq A_n$ for $n\in \NN$, such that  
\begingroup
\setlength{\leftmargini}{18pt}
{
\renewcommand{\theenumi}{\sf\roman{enumi})} 
\renewcommand{\labelenumi}{\theenumi}
\begin{enumerate}
\item
\label{hgghhgh1}
\hspace{16.7pt} $\gamma_{\psi_n}=\gag\cp \rho_n$,
\item
\label{hgghhgh2}
\hspace{19.6pt} $A_n=\bigcup_{m\leq n} D'_m$,
\item
\label{hgghhgh3}
\hspace{0pt}\: $\innt[A_n]=\bigcup_{m\leq n} \innt[D'_m]$
\end{enumerate}}
\endgroup
\noindent  
holds for each $n\in \NN$.  
Then, Claim 1 holds for $D:=\bigcup_{n\in \NN} A_n$. In fact, 
\begingroup
\setlength{\leftmargini}{12pt}
\begin{itemize}
\item 
	$D$ is an interval by \eqref{lovcxpoivcxpiiovcxvxc}, with 
	$$
	\textstyle\gag(D)\stackrel{\ref{hgghhgh2}}{\subseteq} \bigcup_{n\in \NN}\gag(D_n')\stackrel{\ref{hgghhgh1}}{=}\bigcup_{n\in \NN}\gamma_{\psi_n}(D_{\psi_n})=\bigcup_{n\in \NN}\iota(U_n)=\iota(S).
	$$
\item
 \ref{lkdslkds1} and \ref{lkdslkds3} are clear from \ref{hgghhgh1} and \ref{hgghhgh2}, respectively.
\item 
  \ref{lkdslkds2} is clear from \eqref{lovcxpoivcxpiiovcxvxc} and \ref{hgghhgh3}.\footnote{If $t'\in \innt[D]$, then there exists $n\in \NN$ with $t'\in A_n$. Then, $t'\in \innt[D]$ together with $D=\bigcup_{n\in \NN} A_n$ and \eqref{lovcxpoivcxpiiovcxvxc} implies that $t'\in \innt[A_m]$ holds for some $m\geq n$. Then, \ref{hgghhgh3} shows that $t'\in \innt[D'_p]$ holds for some $p\in \NN$, which establishes \ref{lkdslkds2}.}  
\end{itemize}
\endgroup
\noindent
Now, to perform the induction, observe that \ref{hgghhgh1}, \ref{hgghhgh2}, \ref{hgghhgh3} evidently hold for $n= 0$, with $A_0= D_0'$ and $\rho_0= \rho_{\psi_0}$.  
We thus can assume that there exist $m\in \NN$ and  diffeomorphisms $\rho_n\colon  D_{\psi_n} \rightarrow D'_n \subseteq A_n$ for $0\leq n\leq m$, such that \ref{hgghhgh1}, \ref{hgghhgh2}, \ref{hgghhgh3} hold. Then, $U_{n}\cap U_{n+1}\neq \emptyset$ implies
\begin{align*} 
	 \gamma_{\psi_{n+1}}\cpsim \gamma_{\psi_{n}}\qquad\:\:\stackrel{\ref{hgghhgh1}}{\Longrightarrow}\qquad\:\: \gamma_{\psi_{n+1}}\cpsim\gag|_{D'_{n}} \qquad\stackrel{\text{Lemma }\ref{vd}}{\Longrightarrow}\qquad \gamma_{\psi_{n+1}}=\gag\cp \rho_{n+1}
\end{align*}
for an analytic diffeomorphism 
\begin{align*}
	\rho_{n+1}\colon D_{\psi_{n+1}}\rightarrow D_{n+1}'\subseteq A_{n+1}:=A_{n}\cup D'_{n+1}\qquad\quad\text{with}\qquad\quad   \innt[D'_{{n}}]\cap \innt[D'_{n+1}]\neq \emptyset. 
\end{align*}
Since $D'_{n}\subseteq A_{n}$ holds, the right side shows $ \innt[A_n]\cap \innt[D_{n+1}']\neq \emptyset$; and we conclude: 
\begingroup
\setlength{\leftmargini}{12pt}
\begin{itemize}
\item
$A_{n+1}$ is an interval, with $A_{n}\subseteq A_{n+1}$ by construction.
\item
If $t'$ is contained in $\innt[A_{n+1}]$, then $t'$ is contained in $\innt[D'_{n+1}]$ or in $\innt[A_{n}]$. Hence, \ref{hgghhgh3} (for $n+1$) is clear from the induction hypothesis.
\end{itemize}
\endgroup
\noindent
Since  \ref{hgghhgh1} and \ref{hgghhgh2} (for $n+1$) are just clear from the construction, the claim now follows by induction.
\end{proof}

\subsection{}
\label{appD4}
\begin{proof}[Proof of Corollary \ref{lkdslksdlkdslkdsoidsoioidsds78987ds98ds98dsdsdsdsdsdsdss}]
The one direction is covered by Proposition \ref{dfgffhfh}.\ref{dfgffhfh2}. Assume thus $S\cong \UE$ via \eqref{ghhgghqqqq1} (the case $S\cong D$ via \eqref{ghhgghqqqq2} is treated analogously).  Then,    
\begin{align}
\label{skjdskjdskjdsnmdsnmdsnmds08ds98zzds98ds33233232}
	 \gamma_{\g}^x=\iota\cp\Omega\cp \alpha\qquad\text{holds for}\qquad \alpha\colon \RR\ni \phi \mapsto \e^{\I\phi}\in \UE,
\end{align}
hence $\gamma_\g^x$ is immersive (chain rule) 
with $\pii(x,\g)=2\pi$. 
\vspace{-5pt}
\begingroup
\setlength{\leftmargini}{12pt}
\begin{itemize}
\item
Let $(U,\psi)$ be a fixed chart centred at $z:=\iota^{-1}(x)$, hence 
$(\Omega\cp\alpha)(0)=z=\psi^{-1}(0)$.
\item
We fix an open interval $0\in J\subseteq (-\pi,\pi)$ with $z\in W:=(\Omega\cp\alpha)(J)\subseteq U$, and set  
 $J':=\psi(W)$.  
\end{itemize}
\endgroup 
\noindent
By Corollary \ref{dfdsasasasassa}, we can 
shrink $J$ around $0$, such that $\gamma_\g^x|_{J}$ and $\gamma_\psi|_{J'}$ are embeddings. Then,
\vspace{-4pt}
\[
\gamma_\g^x(J)\stackrel{\eqref{skjdskjdskjdsnmdsnmdsnmds08ds98zzds98ds33233232}}{=}\iota(W)=\gamma_\psi(J')
\quad\:\:\:\Longrightarrow\quad\:\:\:\gamma_\g^x\cpsim \gamma_\psi\quad\:\: \stackrel{\text{Lemma } \ref{vd}}{\Longrightarrow}\quad\:\:
	\gamma_\psi\:\:\:\text{is exponential w.r.t.}\:\:\: (x,\g).\qedhere
\]
\end{proof}
\subsection{}
\label{appD4b}
\begin{proof}[Proof of Lemma \ref{dslklkdslkdssdoioidsoidsds09ds09ds09ddsdsds09}]
The one direction is evident. For the other direction, it suffices to prove the following:
\vspace{6pt}

\noindent
{\bf Claim 2:} Let $(U,\psi)$ and $(U',\psi')$ be charts of $S$. Then, the following implication holds:   
\begin{align}
\label{lkdslkdsdsnmnmdsnmdsnmdsds8sd98s98d98dsdsds3343}
\gamma_{\psi}\cpsim \gamma_{\psi'}\qquad\quad\Longrightarrow\qquad\quad \gamma_{\psi}\:\text{ is free}\quad\:\:\Longleftrightarrow\quad\:\: \gamma_{\psi'}\:\text{ is free.}
\end{align} 
In fact, assume that Claim 2 holds true, and let $(U_0,\psi_0)$ be a chart of $S$ such that $\gamma_{\psi_0}$ is free: 
\begingroup
\setlength{\leftmargini}{12pt}
\begin{itemize}
\item
Lemma \ref{oisdkdslkdslkdskjdskjsdkjds98ds98ds98dsdsdsds} provides  charts  $\{(U_n,\psi_n)\}_{n\geq 1}$ with $S=\bigcup_{n\in \NN}U_n$, such that  $U_n\cap U_{n+1}\neq \emptyset$ (hence $\gamma_{\psi_n}\cpsim \gamma_{\psi_{n+1}}$) holds for each $n\in \NN$. It now follows by induction from Claim 2 that $\gamma_{\psi_n}$ is free for each $n\in \NN$. 
\item
Let $(U,\psi)$ be a chart of $S$. Since $S=\bigcup_{n\in \NN}U_n$ holds, there exists some $n\in \NN$ with $U\cap U_n\neq \emptyset$, hence $\gamma_{\psi}\cpsim \gamma_{\psi_n}$. Claim 2 thus shows that $\gamma_\psi$ is free.
\end{itemize}
\endgroup
\noindent
{Proof of Claim 2.} Assume that the left side of \eqref{lkdslkdsdsnmnmdsnmdsnmdsds8sd98s98d98dsdsds3343} holds, and  that $\gamma_\psi$ is free (the other implication follows analogously). 
By Theorem \ref{sfdknfdhujfdd}, $\gamma_\psi$ is either a free segment or admits a $\tau$-decomposition or admits an $A$-decomposition. 
In any case, $\gamma_{\psi}\cpsim \gamma_{\psi'}$ implies that $\gamma_{\psi}(J)=\gamma_{\psi'}(J')$ holds for a free open interval $J\subseteq D_{\gamma_\psi}$ and an open interval $J'\subseteq D_{\gamma_{\psi'}}$, such that $\gamma\equiv \gamma_{\psi}|_J$ and $\gamma'\equiv\gamma_{\psi'}|_{J'}$ are embeddings.\footnote{Alternatively, one obtains such intervals form Theorem \ref{sfdknfdhujfd} and  \eqref{ldslkdlkdsd09s09d09ds09ds09ds09dssddsdscxcxcx}.} 
Lemma \ref{freemax}.\ref{freemax1} together with  Lemma \ref{lemma:BasicAnalyt2} shows that $J'$ is free w.r.t.\ $\gamma'$, hence   free w.r.t.\ $\gamma_{\psi'}$ by \eqref{ldslkdlkdsd09s09d09ds09ds09ds09dssddsdscxcxcx}.     
\end{proof}

\subsection{}
\label{appD5}
\begin{proof}[Proof of Lemma \ref{fdfxvcx}]
If $z\in S\backslash \FP$ holds, then the claim is clear from Lemma \ref{jgjhjh}.\ref{jgjhjh2}. Assume thus that $z\in \FP$ holds, i.e., that       
\eqref{sdfsdfds} holds for a (boundary) chart $(U,\psi)$ around $z$ (hence, $\psi(U)\subseteq (-\infty,0]$ and $\psi(z)=0$).   
\vspace{-3pt}
\begingroup
\setlength{\leftmargini}{12pt}
\begin{itemize}
\item
$\gamma_\psi$ is free by Lemma \ref{dslklkdslkdssdoioidsoidsds09ds09ds09ddsdsds09}; hence, $\ovl{\gamma}_\psi$ is free by \eqref{ldslkdlkdsd09s09d09ds09ds09ds09dssddsdscxcxcx}.  
\item
Since $z\in \FP\cap U\neq \emptyset$ holds with $\psi(z)=0$, 
Lemma \ref{ghhjggh} 
shows\footnote{Lemma \ref{ghhjggh} shows that $\ovl{\gamma}_\psi$ either admits a $0$-decomposition or is negative with $A$-decomposition $(\{a_n\}_{n\in \cn},\{[g_n]\}_{n\in \cn})$ such that $a_n=0$ holds for some $n\in \cn$.} 
that  there exist open intervals $(j',j)\equiv J\subset I\equiv(i',i)\subseteq \dom[\ovl{\gamma}_\psi]$   
with $i'<j'<0<j<i$  
such that $(i',0]$ and  $[0,i)$ are free. 
\item
By  Corollary \ref{dfdsasasasassa} and \eqref{sdfsdfds}, we can shrink $I$ around $0$ such that $\ovl{\gamma}_\psi|_{I}$ is an embedding, with 
$g\cdot \ovl{\gamma}_\psi((i',0])=\ovl{\gamma}_\psi([0,i))$. 
Then, $[g]$ is a $0$-decomposition of $\gamma:= \ovl{\gamma}_\psi|_I$ with
\begin{align}
\label{dskkjdsnmdsnmdsdszudszuds87ds87ds87ds87ds87dsdsds}
g\cdot \gamma((i',0])=\gamma([0,i)).
\end{align}
\end{itemize}
\endgroup
\noindent
The claim holds for the neighbourhood $V:=\psi^{-1}((j',0])$ of $z$. 
In fact, assume that $g'\cdot\iota(V)\cap \iota(V)$ is infinite for some $g'\in G\backslash G_S$.  Then, $g'\cdot \gamma((j',0])\cap \gamma((j',0])$ is infinite; hence, (by compactness of $[j',0]$) there exist converging sequences 
\begin{align*}
I\setminus\{t\}\supseteq [j',0]\backslash\{t\} \supseteq  \{t_n\}_{n\in \NN}\rightarrow t\in [j',0]\qquad\:\: &\text{ and}\:\qquad\:\: I'\setminus\{t'\}\supseteq [j',0]\backslash\{t'\} \supseteq \{t'_n\}_{n\in \NN}\rightarrow t'\in [j',0]\\[1pt]
&\text{ with}\\[1pt]
g'\cdot \gamma(t_n)=\:&\gamma(t_n')\quad \text{for all}\quad n\in \NN.
\end{align*}
Lemma \ref{lemma:BasicAnalyt0} 
implies that      
$g'\cdot \gamma|_{(i',0]}\cpsim \gamma|_I$  holds; so that, by faithfulness, we have 
\vspace{-4pt} 
\begin{align*}
	\text{either}\qquad\quad [g']=[e] \qquad\quad\text{or}\qquad\quad [g']=[g]\quad\:\wedge\quad\: g'\cdot \gamma((i',0])\stackrel{\eqref{dskkjdsnmdsnmdsdszudszuds87ds87ds87ds87ds87dsdsds}}{=}\gamma([0,i)).
\end{align*}
\begingroup
\setlength{\leftmargini}{12pt}
\begin{itemize}
\item
In the first case, we obtain  
$$
g'\in G_{\gamma} \stackrel{\eqref{ldslkdlkdsd09s09d09ds09ds09ds09dssddsdscxcxcx}}{=} G_{\ovl{\gamma}_\psi}\stackrel{\eqref{ldslkdlkdsd09s09d09ds09ds09ds09dssddsdscxcxcx}}{=}G_{\gamma_\psi}\stackrel{\text{Lemma }\ref{dffdfdfd}}{=}G_S
\qquad\text{which contradicts the assumption}\qquad g'\in G\setminus G_S. 
$$
\item
In the second case, we obtain
\begin{align*}
	g'\cdot\iota(V)\cap \iota(V)= g'\cdot \gamma((j',0]) \cap \gamma((j',0]) \subseteq \gamma([0,i)) \cap \gamma((i',0])\stackrel{\gamma\text{ injective}}{=}\{\gamma(0)\}
\end{align*} 
which contradicts the assumption that the intersection on the left side is infinite.\qedhere
\end{itemize}
\endgroup
\end{proof}


\begin{thebibliography}{10}

\bibitem{BackLA}
  {A.~Ashtekar, J.~Lewandowski: {\it Background Independent Quantum Gravity: A Status Report}, Class.\ Quant.\ Grav.\ {\bf 21} (2004), R53-R152.
    {\sf e-print:\ arxiv0404018v2 (gr-qc)}.} 
    
\bibitem{Gale}
  {D.~Gale: {\it The Classification of 1-Manifolds: A Take-Home Exam}, The American Mathematical Monthly {\bf Vol.\ 94, No.\ 2} (Feb., 1987), 170-175.
    } 

\bibitem{FH}
  {Ch. Fleischhack: {\it Symmetries of Analytic Paths}.  
    {\sf e-print:\ arxiv:1503.06341 (math-ph)}.}

\bibitem{InvConLQG}
{M.~Hanusch: Invariant Connections in Loop Quantum Gravity. {\it Commun. Math. Phys.} {\bf 343(1)} (2016) 1--38. 
{\sf e-print:\ 1307.5303v1 (math-ph)}.}

\bibitem{MAX}
  {M. Hanusch: Invariant Connections and Symmetry Reduction in Loop Quantum Gravity (Dissertation). University of Paderborn, December 2014. \url{http://nbn-resolving.de/urn:nbn:de:hbz:466:2-15277}\\
  {\sf e-print:\ arxiv:1601.05531 (math-ph)}.} 


\bibitem{MAXDECO}
  {M. Hanusch: Decompositions of Analytic 1-Manifolds. {\sf e-print:\  	arXiv:1601.07139 (math.DG)}.}

\bibitem{PRIM} 
  {S.-G. Krantz, H.-R. Parks: {\it A Primer of Real Analytic Functions.} Birkh\"auser, 2002.
  }    
 
\bibitem{Polch} 
  {J. Polchinski: {\it String Theory Vol. 1,2.} Cambridge University Press, 2001.
  }  

\bibitem{Thiemann} 
  {Th. Thiemann: {\it Introduction to Modern Canonical Quantum General Relativity.} Cambridge University Press, 2008.
  }
    

\end{thebibliography}
\end{document}